\documentclass{article}

\usepackage[main,final,nonatbib]{neurips_2026}

\usepackage{amsmath,amssymb,amsthm}
\usepackage{amsfonts}
\usepackage[numbers,compress]{natbib}
\usepackage{mathtools}
\usepackage{mathrsfs}
\usepackage{enumitem}
\usepackage{bm}
\usepackage[hidelinks,bookmarks=false]{hyperref}
\usepackage{url}
\hypersetup{
  colorlinks,
  linkcolor={red!72!black},
  citecolor={blue!72!black},
  urlcolor={magenta!72!black}
}

\usepackage{subcaption}
\usepackage[table,xcdraw]{xcolor}
\usepackage{multirow}
\usepackage{booktabs,tabularx,threeparttable,makecell}
\usepackage{capt-of} 

\mathtoolsset{showonlyrefs=true}

\newtheorem{theorem}{Theorem}
\newtheorem{corollary}[theorem]{Corollary}
\newtheorem{lemma}[theorem]{Lemma}
\newtheorem{proposition}[theorem]{Proposition}
\newtheorem{external}[theorem]{External result}
\newtheorem*{restatedgaplemma}{Lemma~\ref{lem:uot_gap_sufficient}}
\newtheorem{assumption}{Assumption}
\theoremstyle{definition}
\newtheorem{definition}{Definition}
\newtheorem{remark}{Remark}

\newcommand{\R}{\mathbb{R}}
\newcommand{\N}{\mathbb{N}}
\newcommand{\KL}{\mathrm{KL}}

\renewcommand{\hat}{\widehat}
\renewcommand{\tilde}{\widetilde}
\newcommand{\Ep}{\mathbb{E}}
\newcommand{\mJ}{\mathcal{J}}
\newcommand{\mM}{\mathcal{M}}
\newcommand{\mP}{\mathcal{P}}

\DeclareMathOperator*{\argmin}{arg\,min}

\newcommand{\supp}{\mathrm{supp}}

\renewcommand{\epsilon}{\varepsilon}

\title{Minimax Optimal Estimation of Transport-Growth Pairs in Unbalanced Optimal Transport}

\author{
  Donlapark Ponnoprat \\
  Department of Statistics\\
  Chiang Mai University\\
  \texttt{donlapark.p@cmu.ac.th} \\
\And
Noboru Isobe \\
  RIKEN AIP \\
  \texttt{noboru.isobe@riken.jp} \\
\And
Masaaki Imaizumi \\
The University of Tokyo / RIKEN AIP / Kyoto University\\
\texttt{imaizumi@g.ecc.u-tokyo.ac.jp} \\
}

\usepackage{color}

\DeclareMathOperator{\clip}{clip}

\allowdisplaybreaks
\let\origmaketitle\maketitle
\def\maketitle{
  \begingroup
  \def\uppercasenonmath##1{} 
  \let\MakeUppercase\relax 
  \origmaketitle
  \endgroup
}

\begin{document}

\maketitle

\begin{abstract}
Unbalanced optimal transport extends classical optimal transport to measures with different total masses, but statistical guarantees for estimation of the transport map and the growth map remain limited. We study unbalanced optimal transport with quadratic cost and Kullback-Leibler marginal penalties and argue that the natural target of estimation is a transport-growth pair. Consequently, we develop two estimators for the transport-growth pair under two setups: a nearest-neighbor estimator for transporting between general densities, and a kernel-based estimator for transporting between smooth densities. We prove that the kernel-based estimator achieves the minimax optimal rate up to a logarithmic factor by deriving a matching lower bound on the minimax risk. Our main technical contribution is a reduction of estimation to a stability bound on the unbalanced optimal transport value. This bound controls the transport map and growth map risks through the duality gap of the true dual potentials for the estimated measures. These results provide a statistical foundation for estimation of the transport-growth pair in unbalanced optimal transport.
\end{abstract}

\section{Introduction}

\subsection{Background}

Optimal transport (OT) provides a principled geometric language for comparing probability distributions and has become a standard tool in statistics and machine learning. In statistical applications, however, the underlying measures are rarely observed directly; instead, one must estimate transport maps and transport plans from finite samples. For balanced OT, this question is now supported by a substantial statistical theory. Smooth transport maps admit minimax analysis via semi-dual curvature and growth arguments \cite{hutter2021minimax}, while plug-in, barycentric-projection, and entropic estimators have been analyzed in discrete, semi-discrete, and smooth regimes \cite{deb2021barycentric,Pooladian2021,pooladian2023semidiscrete,manole2024plugin}. More recently, the balanced theory has expanded to general function-space analyses and sharper stability reductions for plug-in estimators \cite{divol2025general,balakrishnan2025stability}.

Many datasets are not naturally balanced: the total masses may differ, unmatched observations may be present, and mass creation or destruction may be an intrinsic feature of the phenomenon under study. Unbalanced optimal transport (UOT) addresses this issue by relaxing the hard marginal constraints of OT and penalizing deviations of the transport plan marginals from the reference measures. This relaxation leads to the entropy-transport problems, which include the logarithmic entropy-transport problem and the Wasserstein-Fisher-Rao distance, also called the Hellinger-Kantorovich distance \cite{ChizatEtAl2018,LieroMielkeSavare2018,savare2024relaxation,gallouet2025regularity}. The relaxation has also generated a large computational and applied literature, including Sinkhorn algorithms and scalable parameterizations \cite{pham2020sinkhorn,sejourne2022faster,Gazdieva2024,yang2018scalable}, with applications ranging from single-cell dynamics to growth modeling and shape analysis \cite{schiebinger2019,sha2023cell,dimarino2020,bauer2022SRNF}.

Despite this rapid development, the statistical theory of UOT remains less developed, especially for transport-growth pairs. For the UOT model of this paper, \cite{vacher2022stability,vacher2023semi} derived semi-dual formulations and established global upper and lower bounds on the Bregman divergence of the semi-dual objective, yielding the first fast statistical rates for UOT semi-dual potentials. More recently, statistical properties of unbalanced Kantorovich-Rubinstein quantities have been analyzed for finitely supported measures and for spatio-temporal point-process models \cite{HundrieserEtAl2025,struleva2025sharp}. However, compared with the balanced case, minimax rates for estimating transport-growth pairs are still largely unknown.

The main difficulty is structural. In balanced OT, the specific form of the transport objective allows us to bound the transport error by exploiting the semi-dual form of the transport problem \cite{hutter2021minimax,manole2024plugin,balakrishnan2025stability}. UOT, by contrast, has a different objective where the hard marginal constraints are replaced with divergence penalties; this modification introduces a \emph{growth map} that specifies how the masses will contract or expand after the transport.
Consequently, a rigorous statistical theory for UOT must account for the risks of estimating the transport map and the growth map.

\subsection{Our contribution}

In this paper, we develop a statistical theory for estimation of the transport-growth pair in unbalanced optimal transport with KL marginal penalties. In particular, we develop two estimators for the transport-growth pair, then study their risks.

Concretely, our contributions are as follows.
\begin{itemize}
    \item \textbf{Estimators for the transport-growth pair: }We introduce the minimax estimation theory for the transport-growth pair in the Gaussian-Hellinger case of the UOT problem, and develop two estimators for this pair.
    \item \textbf{Proved convergence rates and optimality: } We derive upper bounds on the risks of our estimators and a lower bound on the minimax risk. In particular, one of our estimators achieves the minimax optimal convergence rate up to a logarithmic factor. This rate improves upon the existing rates, as summarized in Table~\ref{tab:main-rate-comparison-highd}.
    \item \textbf{Proof technique by stability bound:} To analyze the transport-growth pair, we establish a stability bound and a first-order expansion of the UOT value under perturbations of the marginals, which separates the step specific to UOT from the choice of the estimators of the source and target measures. 
\end{itemize}

\begin{minipage}{\textwidth}
\footnotesize
\setlength{\tabcolsep}{4pt}
\renewcommand{\arraystretch}{1.08}
\begin{threeparttable}
\captionof{table}[High-dimensional comparison of statistical rates]{Comparison of convergence rates of the risk.
\(d\ge 5\) is the dimension of samples, \(N = \min\{n,m\}\), where $n$ and $m$ are the sample sizes from the measures, \(\widetilde O(\cdot)\) hides polylogarithmic factors, and $\alpha-1$ is the smoothness of the densities of the measures.
For \cite{vacher2022stability},
their regularity argument yields
\((\alpha+1)\)-smooth semi-dual potentials with \(\alpha+1<d/2\).
}
\label{tab:main-rate-comparison-highd}
\begin{tabularx}{\textwidth}{
>{\raggedright\arraybackslash}p{2.0cm}
>{\centering\arraybackslash}p{2.55cm}
>{\centering\arraybackslash}p{1.7cm}
>{\centering\arraybackslash}p{1.5cm}
>{\raggedright\arraybackslash}X
>{\centering\arraybackslash}p{1.15cm}
}
\toprule
Reference & Setup & \makecell[c]{transport\\estimation} & \makecell[c]{growth\\ estimation} & Convergence rate & \makecell[c]{Lower\\bound} \\
\midrule \midrule

\makecell[l]{ \cite{hutter2021minimax} \\ \cite{manole2024plugin}}
&
\makecell[c]{ balanced \\ (w/ smoothness)}
&
\(\checkmark\)
&
N/A
&
\(\widetilde O \left(N^{-2\alpha/(2\alpha-2+d)}\right)\)
&
\(\checkmark\)
\\ \hline

\addlinespace[2pt]

\makecell[l]{\cite{vacher2022stability}}
&
\makecell[c]{unbalanced \\ (w/ smoothness)}
&
\(\checkmark\)
&
&
\(O \left(N^{-(\alpha+1)/d}\right)\)
&
\\

\addlinespace[2pt]

\makecell[l]{\cite{vacher2023semi}}
&
unbalanced
&
\(\checkmark\)
&
&
\(O \left(N^{-2/d}\right)\)
&
\\

\addlinespace[2pt]

\makecell[l]{\textbf{This paper}\\ \textbf{(nearest-neighbor)}}
&
unbalanced
&
\(\checkmark\)
&
\(\checkmark\)
&
\(\widetilde O(N^{-2/d})\)
&
\\

\addlinespace[2pt]

\makecell[l]{\textbf{This paper}\\ \textbf{(kernel-based)}}
&
\makecell[c]{unbalanced \\ (w/ smoothness)}
&
\(\checkmark\)
&
\(\checkmark\)
&
\(O \left(N^{-2\alpha/(2\alpha-2+d)}\right)\)
&
\(\checkmark\)
\\

\bottomrule
\end{tabularx}
\end{threeparttable}
\end{minipage}

\subsection{Notations}\label{subsec:notation}
Throughout, $\Omega\subset\R^d$ denotes the ambient domain. We write $C(\Omega)$ for the space of continuous real-valued functions on $\Omega$, and $C_b(\Omega)$ for the bounded continuous functions. The set of finite positive Radon measures on $\Omega$ is denoted by $\mM_+(\Omega)$, and $\mP(\Omega)$ denotes the Borel probability measures on $\Omega$. For a measurable map $T:\Omega\to\Omega$ and a measure $\mu\in\mM_+(\Omega)$, $T_{\#}\mu$ denotes the pushforward of $\mu$. For $\gamma\in\mM_+(\Omega\times\Omega)$, we write $\gamma_0$ and $\gamma_1$ for the first and second marginals of $\gamma$, respectively. If $\eta,\sigma\in\mM_+(\Omega)$ have the same total mass, then $\Pi(\eta,\sigma)$ denotes the set of couplings between $\eta$ and $\sigma$. We use $\|\cdot\|$ for the Euclidean norm, $\|\cdot\|_{\mathrm{op}}$ for the operator norm, and $\mathbf{1}\{\cdot\}$ for the indicator function. For $\eta,\sigma\in\mM_+(\Omega)$ with the same total mass, we write $W_2(\eta,\sigma)\coloneqq\bigl(\inf_{\gamma\in\Pi(\eta,\sigma)}\int_{\Omega\times\Omega}\|x-y\|^2\,d\gamma(x,y)\bigr)^{1/2}$ for the $2$-Wasserstein distance. For $\eta,\sigma\in\mM_+(\Omega)$ with $\eta\ll\sigma$, we write $D_\KL(\eta\mid\sigma)\coloneqq\int_\Omega F(d\eta/d\sigma)\,d\sigma$ for the KL divergence, where $F(r)\coloneqq r\log r-r+1$. Otherwise, $D_\KL(\eta\mid\sigma)\coloneqq+\infty$. We write $a\lesssim b$ or $b\gtrsim a$ if $a\le Cb$ for some constant $C>0$, and $a\asymp b$ if $a\lesssim b$ and $b\lesssim a$. For real numbers $a<b$, the clipping function is $\clip_{[a,b]}(t)\coloneqq\min\{b,\max\{a,t\}\}$.

\section{Unbalanced optimal transport with quadratic cost} \label{sec:setup}

We briefly review unbalanced optimal transport in the \textit{Gaussian-Hellinger} case. For a complete treatment of the general case, see \cite{LieroMielkeSavare2018,savare2024relaxation,gallouet2025regularity}.
Let $\Omega\subset\R^d$ be a bounded convex domain, and let $\mu,\nu\in\mM_+(\Omega)$ be finite positive Radon measures. We write $c(x,y)\coloneqq\frac12\|x-y\|^2$ for the quadratic cost. The \emph{$c$-transform} of a function $\varphi:\Omega \to \R$ is defined as $\varphi^c(x) =\inf_{y\in\Omega} \{ c(x,y) - \varphi(y) \}$. We say that $\varphi$ is \emph{$c$-concave} if $\varphi=\psi^c$ for some $\psi: \Omega \to \R$.

\subsection{Transport formulations}

\paragraph{Monge problem:} We first formulate UOT as an optimization over a transport map and a growth map.
A \textit{transport-growth pair} consists of a measurable map $T:\Omega\to\Omega$ and a measurable weight $\lambda:\Omega\to[0,\infty)$, and maps $\mu$ to the weighted pushforward
\begin{align}
    (T,\lambda)_{\#}\mu \coloneqq T_{\#}(\lambda^2\mu).
\end{align}
The map $\lambda$ is referred to as the \emph{growth map}, since $\lambda^2$ rescales mass after the transport.
In the {Gaussian-Hellinger} case,
we define the cone cost
\[
    C((x,r),(y,s))^2 \coloneqq r^2 + s^2 - 2rs\,e^{-\|x-y\|^2/4} ,\qquad x,y\in\Omega,\ r,s\ge 0.
\]
Using this cone cost, we obtain the
\emph{unbalanced Monge problem} for the Gaussian-Hellinger case:
\begin{align}
\label{eq:gh_monge}
    \mathrm{UM}(\mu,\nu)
    \coloneqq
    \inf_{(T,\lambda):\,(T,\lambda)_{\#}\mu=\nu}
    \int_\Omega
    \bigl(
        1 + \lambda(x)^2 - 2\lambda(x)\,e^{-\|x-T(x)\|^2/4}
    \bigr)\,d\mu(x).
\end{align}
In contrast to balanced OT, where only a transport map appears, UOT also requires a growth map to reconcile the discrepancy between the total masses of $\mu$ and $\nu$.

\paragraph{Kantorovich problem:} We consider another problem for unbalanced optimal transport by optimizing over transport plans $\gamma\in\mM_+(\Omega\times\Omega)$ instead of deterministic transport-growth pairs for the Monge problem. Specifically, the Kantorovich problem associated with \eqref{eq:gh_monge} is given by
\begin{align}
\label{eq:gh_kantorovich}
    \mathrm{UOT}(\mu,\nu)
    \coloneqq
    \inf_{\gamma'\in\mM_+(\Omega\times\Omega)}
    \left\{
        \int_{\Omega\times\Omega} \frac{\|x-y\|^2}{2}\,d\gamma'(x,y)
        + D_\KL(\gamma_0'\mid\mu)
        + D_\KL(\gamma_1'\mid\nu)
    \right\}.
\end{align}
The corresponding dual problem is as follows:
\begin{align}
\label{eq:gh_dual}
    \mathrm{UOT}(\mu,\nu)
    =
    \sup_{\substack{(\varphi,\psi)\in C_b(\Omega)^2\\ \varphi(x)+\psi(y)\le \|x-y\|^2/2}}
    \left\{
        \int_\Omega \bigl(1-e^{-\varphi(x)}\bigr)\,d\mu(x)
        + \int_\Omega \bigl(1-e^{-\psi(y)}\bigr)\,d\nu(y)
    \right\}.
\end{align}

Let $(\varphi_0,\psi_0)$ denote an optimal dual pair for \eqref{eq:gh_dual}, and we refer to $\varphi_0$ and $\psi_0$ as \emph{potentials}.  Existence of the solutions of \eqref{eq:gh_monge} and \eqref{eq:gh_dual} is guaranteed by the following theorem:
\begin{theorem}[{\cite{gallouet2025regularity}}]
\label{prop:monge-potential}
Let $\Omega\subset\R^d$ be bounded and convex, let $\mu=\rho_0\,dx$ with $\rho_0>0$ a.e., and let $\nu\ll dx$ be supported in $\Omega$. Then there exists a $\mu$-a.e. unique $c$-concave $\varphi_0$ such that $(\varphi_0,\psi_0)$ solves the dual problem \eqref{eq:gh_dual} with $\psi_0 = \varphi^c_0$, and $\varphi_0$ uniquely determines a solution $(T_0,\lambda_0)$ of the Monge problem \eqref{eq:gh_monge} via:
\begin{align}\label{eq:unique}
    T_0(x)=x-\nabla \varphi_0(x),
    \qquad
    \lambda_0(x)=\exp\left(-\varphi_0(x)+\frac1{4}\|\nabla \varphi_0(x)\|^2\right).
\end{align}
\end{theorem}
\paragraph{Useful form:} Following Theorem~\ref{prop:monge-potential}, we define the \emph{active marginal} $\gamma_0=e^{-\varphi_0}\mu$ and the function $a_0:\Omega\to(0,\infty)$, $a_0(x)\coloneqq e^{-\varphi_0(x)/2}$, which is the square root of the density $d\gamma_0/d\mu$.
It follows from complementary slackness that $\varphi_0(x) + \psi_0(T_0(x)) = \| x - T_0(x) \|^2/2$. We thus have the following gradient-free formulation of $T_0$ and $\lambda_0$:
\begin{align} \label{eq:convenient}
    T_0(x)
    &\in
    \argmin_{y\in\Omega}
    \left\{
        \frac{\|x-y\|^2}{2}-\psi_0(y)
    \right\},
    \qquad
    \lambda_0(x)
    =a_0(x)^2
    \exp\left(\frac14\|x-T_0(x)\|^2\right) .
\end{align}

\subsection{Statistical estimation problem of the transport-growth pair}

Given finite positive measures $\mu,\nu \in \mM_+(\Omega)$, we consider total masses $    M_\mu \coloneqq \mu(\Omega)$ and $    M_\nu \coloneqq \nu(\Omega)$ and define the normalized measures $ \bar\mu \coloneqq {\mu} / {M_\mu}, \bar\nu \coloneqq {\nu} / {M_\nu}$. Given sample sizes $n,m\in\N$, we observe mutually independent samples $X_1,\dots,X_n \sim \bar\mu$ and $Y_1,\dots,Y_m \sim \bar\nu$. Because the data are sampled from the normalized measures $\bar\mu$ and $\bar\nu$, the samples do not identify the total masses $M_\mu$ and $M_\nu$. We therefore assume either that these masses are known, or that estimated masses $\hat M_\mu$ and $\hat M_\nu$ are available from external information. We then define atomic weights $\hat\mu_i \coloneqq \hat M_\mu/n$, $\hat\nu_j \coloneqq \hat M_\nu/m$ and the empirical measures
\begin{align}
    \hat\mu_n
    &\coloneqq
    \frac{\hat M_\mu}{n}\sum_{i=1}^n \delta_{X_i},\qquad 
    \hat\nu_m
    \coloneqq
    \frac{\hat M_\nu}{m}\sum_{j=1}^m \delta_{Y_j}.
\end{align}

Our objective is to estimate the transport-growth pair $(T_0,\lambda_0)$ associated with $(\mu,\nu)$ from the samples $X_1,\dots,X_n$, $Y_1,\dots,Y_m$ and the estimated masses $\hat M_\mu,\hat M_\nu$. Our proofs also estimate $a_0$, since $\gamma_0=a_0^2\mu$ expresses the active marginal through $a_0$.

\begin{remark}

The construction of $\hat M_\mu$ and $\hat M_\nu$ depends on the observation model.
If external measurements of $M_\mu$ and $M_\nu$ are available, these measurements can serve directly as $\hat M_\mu$ and $\hat M_\nu$. In Poisson or more general point-process models, the observed counts contain information about the masses; see Appendix~\ref{sec:appendix_ppp} and the recent UOT analyses of \cite{HundrieserEtAl2025,struleva2025sharp}.
\end{remark}

\section{Estimator design}

We propose two estimators: (i) a nearest-neighbor estimator and (ii) a kernel-based estimator; the former can handle high-dimensional data, while the latter uses the smoothness of the densities to attain a faster rate.

We introduce some notation: for a nonnegative matrix $G=(G_{ij})\in\R_+^{n\times m}$, define the discrete marginals $(G_0)_i \coloneqq \sum_{j=1}^m G_{ij}$ and $(G_1)_j \coloneqq \sum_{i=1}^n G_{ij}$. The corresponding discrete KL penalty between mass vectors is $ D_\KL(G_0\mid\hat\mu_n) \coloneqq \sum_{i=1}^n \hat\mu_i\, F\bigl((G_0)_i / \hat\mu_i\bigr). $

\subsection{Nearest-neighbor estimator} \label{sec:plan}

This estimator is constructed by solving the Kantorovich problem under the empirical measures $\hat{\mu}_n$ and $\hat{\nu}_m$ and using the discrete optimal transport plan, based on the following steps.

\paragraph{(i) Transport plan estimation:} We define an estimator $\hat\gamma$ of the unbalanced optimal transport plan as any solution of \eqref{eq:gh_kantorovich} with $\mu$ and $\nu$ replaced by $\hat\mu_n$ and $\hat\nu_m$, that is,
\begin{equation}
\hat\gamma\in\argmin_{\gamma'\in\R_+^{n\times m}}
\left\{
\sum_{i=1}^n\sum_{j=1}^m c(X_i,Y_j)\gamma_{ij}'
+D_\KL(\gamma_0'\mid\hat\mu_n)+D_\KL(\gamma_1'\mid\hat\nu_m)
\right\}.
\end{equation}

\paragraph{(ii) Estimate discrete transport:} Define the row masses $\hat r_i \coloneqq \sum_{j=1}^m \hat\gamma_{ij}$, $i=1,\dots,n$. We then define a point estimator $\hat T_i$ of $T_0(X_i)$ by the barycentric projection: if $\hat r_i > 0$,
\begin{equation}\label{eq:disc-frechet}
\hat T_i = \argmin_{y\in\Omega}\ \sum_{j=1}^m \hat\gamma_{ij} c(y,Y_j) = \frac{1}{\hat r_i}\sum_{j=1}^m \hat\gamma_{ij} Y_j \in \R^d.
\end{equation}
If $\hat r_i = 0$, then $\sum_{j=1}^m\hat\gamma_{ij}c(y,Y_j)=0$ for every $y\in\Omega$, and $\hat T_i$ is not identifiable from $\hat\gamma$. In this case, we set $\hat T_i \coloneqq X_i$ by convention.
Motivated by the identity $\gamma_0=a_0^2\mu$ from Section~\ref{sec:setup}, we define an estimator $\hat a_i \coloneqq \sqrt{{\hat r_i}/{\hat\mu_i}}$ for $a_0(X_i)$.
To estimate the value $\lambda_0(X_i)$ of the growth map, we fix constants $0<w_-<w_+$ such that $w_0\coloneqq e^{-\varphi_0}$ satisfies $w_-\le w_0\le w_+$ on $\Omega$, and set
\begin{align}
\label{eq:disc-gh-growth}
    \hat\lambda_i
    \coloneqq
    \clip_{[w_-,w_+]}(\hat a_i^2)
    \exp\left(\frac14\|X_i-\hat T_i\|^2\right).
\end{align}
By Lemma~\ref{lem:dual_potential_bounds} in Appendix~\ref{sec:stability}, constants $0<w_-<w_+$ with $w_-\le e^{-\varphi_0}\le w_+$ on $\Omega$ always exist. In particular, if $M_\mu$ and $M_\nu$ are known, the clipping constants may be chosen with $w_+/w_-\le e^{\mathrm{diam}(\Omega)^2}(1+M_\mu/M_\nu)(1+M_\nu/M_\mu)$.

\paragraph{(iii) Estimate transport/growth map:}

Since $(\hat T_i,\hat a_i,\hat\lambda_i)$ are defined only at the sample points $X_1,\dots,X_n$, we extend $(\hat T_i,\hat a_i,\hat\lambda_i)_{i=1}^n$ to all of $\Omega$.
Here, we employ the nearest-neighbor method. Let $(V_i)_{i=1}^n$ be the Voronoi partition induced by $X_1,\dots,X_n$:
    $V_i \coloneqq \{x\in\Omega:\ \|x-X_i\|\le \|x-X_k\|\ \forall k\neq i\}$.
Setting $v_i(x) \coloneqq \mathbf{1}\{x\in V_i\}$ yields the nearest-neighbor estimators:
\begin{equation}
\label{eq:disc-1nn-map}
    \hat T^{\mathrm{1NN}}(x)
    =
    \sum_{i=1}^n v_i(x) \hat T_i,
    ~~
    \hat a^{\mathrm{1NN}}(x)
    =
    \sum_{i=1}^n v_i(x)\hat a_i,
    ~~
    \hat\lambda^{\mathrm{1NN}}(x)
    =
    \sum_{i=1}^n v_i(x)\hat\lambda_i,
\end{equation}
where $\hat a^{\mathrm{1NN}}$ and $\hat\lambda^{\mathrm{1NN}}$ estimate $a_0$ and the growth map $\lambda_0$, respectively.

In addition to the nearest-neighbor estimator, the plan-based kernel estimator can also be considered; see Definition~\ref{def:nw_extension_app} in Appendix~\ref{sec:appendix_nw_extension}.

\subsection{Kernel-based estimator} \label{sec:kernel}

We next present a kernel-based estimator for the transport-growth pair on the hypercube $\Omega=[0,1]^d$. This estimator is based on density estimates of the measures $\mu,\nu$ and attains a faster rate when these densities are smooth.

\paragraph{(i) Prepare kernels:}
We introduce a new kernel function based on a cosine basis for estimating density functions on $[0,1]^d$. This method makes statistical use \cite{efromovich2010orthogonal,tsybakov2009nonparametric} of the eigenfunctions of the Neumann Laplacian \cite{strang1999dct}. This kernel is useful to relax the constraint that a density function must lie on a torus, a requirement of the kernel estimators in the balanced case \cite{manole2024plugin}.

We now define $K_L$. In preparation, we define a function $f(s)\coloneqq e^{-1/s} \mathbf{1}\{s > 0\}$ for $s \in \R$, and the function $\tau(t) \coloneqq {f(2-t)} / ({f(2-t)+f(t-1)})$ for $t\ge 0$.
Then, $\tau\in C^\infty([0,\infty))$, $0\le \tau\le 1$, $\tau(t)=1$ for $t\in[0,1]$, and $\tau(t)=0$ for $t\ge 2$. With $c_0 \coloneqq 1, c_{\ell} \coloneqq \sqrt{2}\ \ (\ell\ge 1)$, and for $L\ge 1$, define the one-dimensional Neumann kernel
\begin{align}
\label{eq:uot_cube_kernel_1d}
    K_L^{(1)}(u,v)
    \coloneqq
    1+\sum_{\ell=1}^{\infty}
    \tau \left(\frac{\pi^2 \ell^2}{L^2}\right)c_\ell^2
    \cos(\pi \ell u)\cos(\pi \ell v),
    \qquad
    u,v\in[0,1].
\end{align}
Because $\tau$ has compact support, the sum in \eqref{eq:uot_cube_kernel_1d} is finite for each $L$. We then define
\begin{align}
\label{eq:uot_cube_kernel}
    K_L(x,y)
    &\coloneqq
    \prod_{r=1}^d K_L^{(1)}(x_r,y_r),
    \qquad x,y\in[0,1]^d.
\end{align}
Since each factor $K_L^{(1)}(x_r,y_r)$ contains only $O(L)$ nonzero cosine modes, one evaluation of $K_L(x,y)$ costs $O(dL)$ arithmetic operations. Moreover, we have $\int_{[0,1]^d} K_L(x,y)\,dy = 1$ for $ x\in[0,1]^d$. This is the hypercube analogue of a smooth spectral cutoff, but expressed in a form that is computationally tractable in moderate and high dimension: the cosine basis gives the boundary adaptation, while the product form of $K_L$ gives exact separability.

\paragraph{(ii) Estimate measures via densities:} We first estimate the normalized densities $p$ and $q$ of $\bar\mu$ and $\bar\nu$, respectively. Then, we scale the estimates by the estimated masses $\hat M_\mu$ and $\hat M_\nu$ to obtain the measure estimates. In particular, for frequency cutoffs $L_n,L_m\ge1$, we define the preliminary kernel density estimators
\begin{align}
\label{eq:tilde_ker_def}
    \tilde p_n^{\mathrm{ker}}(x)
    &\coloneqq
    \frac{1}{n}\sum_{i=1}^n K_{L_n}(x,X_i),
    \qquad
    \tilde q_m^{\mathrm{ker}}(y)
    \coloneqq
    \frac{1}{m}\sum_{j=1}^m K_{L_m}(y,Y_j).
\end{align}
We then define the measure estimates, which have total masses $\hat M_\mu$ and $\hat M_\nu$, as
\begin{align}
    \hat \mu_n^{\mathrm{ker}}
    &\coloneqq
    \frac{\hat{M}_\mu (\tilde p_n^{\mathrm{ker}}(x))_+}
    {\int_{[0,1]^d}(\tilde p_n^{\mathrm{ker}}(u))_+\,du}\,dx,
    \qquad
    \hat \nu_m^{\mathrm{ker}}
    \coloneqq
    \frac{ \hat{M}_\nu(\tilde q_m^{\mathrm{ker}}(y))_+}
    {\int_{[0,1]^d}(\tilde q_m^{\mathrm{ker}}(v))_+\,dv}\,dy.
\end{align}

\paragraph{(iii) Estimate maps:}
Let $(\hat\varphi_{nm}^{\mathrm{ker}},\hat\psi_{nm}^{\mathrm{ker}})$ be an optimal dual pair for $\mathrm{UOT}(\hat\mu_n^{\mathrm{ker}},\hat\nu_m^{\mathrm{ker}})$. We define the estimator of $a_0$ as $\hat a_{nm}^{\mathrm{ker}}(x) \coloneqq e^{-\hat\varphi_{nm}^{\mathrm{ker}}(x)/2}$. Then, we define estimators for the transport map and the growth map by
\begin{align} \label{eq:kernel_estimator}
    \hat T_{nm}^{\mathrm{ker}}(x)
    &\in
    \argmin_{y\in[0,1]^d}
    \left\{
        \|x-y\|^2/2 - \hat\psi_{nm}^{\mathrm{ker}}(y)
    \right\},\\
    \hat\lambda_{nm}^{\mathrm{ker}}(x)
    &\coloneqq
     \clip_{[w_-,w_+]}(\hat a_{nm}^{\mathrm{ker}}(x)^2)
    \exp\left(\frac14\|x-\hat T_{nm}^{\mathrm{ker}}(x)\|^2\right).
\end{align}

\section{Minimax optimal rate}

\subsection{Basic assumptions}

\begin{assumption}\label{assm:curvature}
The potential $\varphi_0$ is twice continuously differentiable on $\Omega$, and there exists a constant $\kappa \in (0, 1)$ such that
\[
    \left(1 - \frac{1}{\kappa}\right)I \preceq \nabla^2 \varphi_0(x) \preceq (1 - \kappa)I, \qquad x \in \Omega.
\]
\end{assumption}
The bounds in Assumption~\ref{assm:curvature} are equivalent to $\kappa I \preceq \nabla^2 f_0(x) \preceq \kappa^{-1} I$ for the Brenier potential $f_0(x) \coloneqq \frac{1}{2}\|x\|^2 - \varphi_0(x)$ commonly used in convergence analysis of balanced OT map estimators \cite{hutter2021minimax,manole2024plugin,divol2025general}.

\begin{assumption}
\label{assm:supp_global}
The domain $\Omega \subset \mathbb{R}^d$ is compact, convex, and satisfies the interior cone condition: there exist $\epsilon_0, \delta_0 > 0$ such that for all $x \in \Omega$ and $\epsilon \in (0, \epsilon_0)$, $ \mathrm{Vol}(B(x,\epsilon) \cap \Omega) \geq \delta_0 \mathrm{Vol}(B(x,\epsilon)). $
\end{assumption}

\begin{assumption}
\label{assm:positivity}
The positive finite measures $\mu$ and $\nu$ have total masses $M_\mu$ and $M_\nu$. Their normalized measures $\bar\mu=\mu/M_\mu$ and $\bar\nu=\nu/M_\nu$ admit Lebesgue densities $p$ and $q$ bounded away from zero and infinity on $\Omega$: $0<\beta_{\min}\le p(x),q(x)\le\beta_{\max}<\infty$. Equivalently, $\mu$ and $\nu$ have Lebesgue densities $M_\mu p$ and $M_\nu q$.
\end{assumption}
Assumptions~\ref{assm:curvature}, \ref{assm:supp_global}, and~\ref{assm:positivity} are standard in the statistical literature on transport map estimation: similar assumptions on the densities, on the transport potential or on the map itself appear in the balanced case \cite{hutter2021minimax,manole2024plugin,balakrishnan2025stability} and the unbalanced case \cite{gallouet2025regularity}.

Under Assumptions~\ref{assm:curvature}, \ref{assm:supp_global}, and~\ref{assm:positivity}, the gap in the dual constraint $\varphi_0(x)+\psi_0(y)\le\frac12\|x-y\|^2$ is bounded below by a quadratic function of $\|y-T_0(x)\|$:
\begin{lemma}
\label{lem:uot_gap_sufficient}
Suppose that Assumptions~\ref{assm:curvature}, \ref{assm:supp_global}, and \ref{assm:positivity} hold. Then the dual potentials $(\varphi_0,\psi_0)$ satisfy:
\begin{align}
    \frac{1}{2}\|x-y\|^2-\varphi_0(x)-\psi_0(y)
    \ge \frac{\kappa}{2}\|y-T_0(x)\|^2,
    \qquad x,y\in\Omega.
\end{align}
\end{lemma}

\begin{assumption}\label{assm:masses}
There exist positive real sequences $a_n, b_m \to 0$ as $n,m \to \infty$ and a constant $\bar c\in(0,1]$ such that $\hat M_\mu \ge \bar c M_\mu$ and $\hat M_\nu \ge \bar c M_\nu$ almost surely for all $n,m\in\mathbb N$, and $\Ep[|\hat M_\mu-M_\mu|] \le a_n$ and $\Ep[|\hat M_\nu-M_\nu|] \le b_m$ for all sufficiently large $n,m$.
\end{assumption}
Assumption~\ref{assm:masses} is mild and is satisfied, for example, when the masses are known or when additional measurements of $M_\mu$ and $M_\nu$ are available. A conditional version holds in point-process models, in which truncated counts estimate the masses \cite{HundrieserEtAl2025,struleva2025sharp}. For that reason we state Assumption~\ref{assm:masses} in an abstract form instead of focusing on any specific model of the masses.

\subsection{Nearest-neighbor estimator}

We study the risk of the nearest-neighbor estimator \eqref{eq:disc-1nn-map}.
In preparation, we define the following rate:
\begin{align}
    \mathfrak R_n^{\mathrm{emp}}(d)
    \coloneqq
    \begin{cases}
        n^{-1/2}, & d \le 3,\\
        (\log n)n^{-1/2}, & d = 4,\\
        n^{-2/d}, & d \ge 5,
    \end{cases}
\end{align}
which is the upper bound of \cite[Theorem~1]{fournier2015rate} on the convergence rate of the empirical distribution in the squared Wasserstein distance $W_2^2$. Then, we obtain the following rate:

\begin{theorem}[Error rate of nearest-neighbor estimator]
\label{thm:main_plan_based_rates}
Assume that Assumptions~\ref{assm:curvature}, \ref{assm:supp_global}, \ref{assm:positivity} and \ref{assm:masses} hold. Let $(\hat T^{\mathrm{1NN}},\hat a^{\mathrm{1NN}},\hat\lambda^{\mathrm{1NN}})$ be the nearest-neighbor estimators defined in Section~\ref{sec:plan}, where the values $\hat\lambda_i$ of $\hat\lambda^{\mathrm{1NN}}$ are defined by \eqref{eq:disc-gh-growth}. Let $a_n,b_m$ be the sequences from Assumption~\ref{assm:masses}. Then there exists a constant $C>0$ such that, for all sufficiently large $n,m$,
\begin{align}
\label{eq:section4_plan_map_rate_abstract}
    \Ep\Bigl[
        \int_\Omega \|\hat T^{\mathrm{1NN}}(x)-T_0(x)\|^2\,d\mu(x)
    \Bigr]
    &\le
    C\log n \Bigl(
        M_\mu \mathfrak R_n^{\mathrm{emp}}(d)
        + M_\nu \mathfrak R_m^{\mathrm{emp}}(d)
    \Bigr),
\\
\label{eq:section4_plan_growth_rate_abstract}
    \Ep\Bigl[
        \int_\Omega |\hat\lambda^{\mathrm{1NN}}(x)-\lambda_0(x)|^2\,d\mu(x)
    \Bigr]
    &\le
    C\log n \Bigl(
        M_\mu \mathfrak R_n^{\mathrm{emp}}(d)
        + M_\nu \mathfrak R_m^{\mathrm{emp}}(d)
        + a_n+b_m
    \Bigr).
\end{align}
\end{theorem}

These bounds show that, for $d\ge4$, the rates of the nearest-neighbor estimator nearly match the lower bound for discrete approximation of measures under the squared Wasserstein distance $W_2^2$, which is of order $n^{-2/t}$ for every $t<d$ (cf. Proposition~\ref{prop:weighted_empirical_wasserstein}). We note that the bound \eqref{eq:section4_plan_map_rate_abstract} does not have the mass estimation terms $a_n,b_m$ because $\hat T^{\mathrm{1NN}}$ does not depend on $\hat M_\mu,\hat M_\nu$. Rescaling the two masses multiplies the discrete optimal transport plan $\hat\gamma$ by a constant, which cancels in the barycentric projection \eqref{eq:disc-frechet}. In contrast, $\hat\lambda^{\mathrm{1NN}}$ depends on the masses through $\hat a_i^2=\hat r_i/\hat\mu_i$. Furthermore, since $a_n$ and $b_m$ achieve the parametric rate in many cases, $a_n$ and $b_m$ do not affect the convergence rate of the right-hand side of \eqref{eq:section4_plan_growth_rate_abstract}.

\subsection{Kernel-based estimator on the hypercube}

We study the risk of the kernel-based estimator \eqref{eq:kernel_estimator} on $[0,1]^d$. Here, we additionally introduce an assumption on the smoothness of the normalized densities of $\mu,\nu$.
\begin{definition}[Neumann-compatible H\"older class]
\label{def:neumann_holder_class}
Let $\mathbb{T}_2^d \coloneqq (\mathbb{R}/2\mathbb{Z})^d$. For $f:[0,1]^d \to \mathbb{R}$, define the even $2$-periodic reflection of $f$ as $\mathcal{E}f: \mathbb{T}_2^d \to \mathbb{R}$ by $ (\mathcal{E}f)(x_1,\dots,x_d) \coloneqq f\bigl(\vartheta(x_1),\dots,\vartheta(x_d)\bigr) $, where $\vartheta(t) \coloneqq \min_{m\in\mathbb{Z}} |t-2m| \in [0,1] $. For $s>0$, let $\|\cdot\|_{C^s(\mathbb{T}_2^d)}$ be the H\"older norm of order $s$ (cf. Appendix~\ref{sec:holder_norm}). For $M>0$, we write
\begin{align}
    \mathcal{C}_N^s([0,1]^d;M)
    \coloneqq
    \left\{
        f:[0,1]^d\to\mathbb{R}:
        \|\mathcal{E}f\|_{C^s(\mathbb{T}_2^d)} \le M
    \right\},
\end{align}
which consists of functions whose coordinatewise even reflection is $C^s$-smooth on the doubled torus.
\end{definition}
\begin{assumption}[Smooth density]
\label{assm:cube_kernel_smoothness}
Let $\mathcal{C}_N^s([0,1]^d;M)$ be the Neumann-compatible H\"older class. For some $\alpha>1$ and $M>0$, the normalized densities satisfy $p,q \in \mathcal{C}_N^{\alpha-1}([0,1]^d;M)$.
\end{assumption}
In preparation, we define a rate as follows:
\begin{align}
    \mathfrak R_n^{\mathrm{ker}}(\alpha)
    \coloneqq
    \begin{cases}
        n^{-1}, & d=1,\\
        (\log n)n^{-1}, & d=2,\\
        n^{-{2\alpha} / ({2(\alpha-1)+d})}, & d\ge 3
    \end{cases}.
\end{align}

\begin{theorem}[Error rate of kernel-based estimator]
\label{thm:main_cube_kernel_rates}
Assume that $\Omega=[0,1]^d$ and that Assumptions~\ref{assm:curvature}, \ref{assm:positivity}, \ref{assm:masses}, and \ref{assm:cube_kernel_smoothness} hold.
Assume that $L_n = n^{1/(d+2(\alpha-1))}$ and $L_m = m^{1/(d+2(\alpha-1))}$, and let $a_n,b_m$ be the sequences from Assumption~\ref{assm:masses}. Then there exist constants $C>0$ and $N_0\in\mathbb{N}$, depending only on $d,\alpha,M,\beta_{\min},\beta_{\max}$, the constants $\kappa$ and $\bar c$ of Assumptions~\ref{assm:curvature} and \ref{assm:masses}, the clipping constants $w_\pm$, and the function $\tau$, such that, for all $n,m\ge N_0$,
\begin{align}
    \label{eq:section4_cube_kernel_map_rate}
    \Ep\left[
        \int_{[0,1]^d}
        \|\hat T_{nm}^{\mathrm{ker}}(x)-T_0(x)\|^2\,d\mu(x)
    \right]
    \le
    C\Bigl(
        M_\mu \mathfrak R_n^{\mathrm{ker}}(\alpha)
        +
        M_\nu \mathfrak R_m^{\mathrm{ker}}(\alpha)
    \Bigr).
\end{align}
If, in addition, $n$ and $m$ are large enough that $\Ep[|\hat M_\mu-M_\mu|]\le a_n$ and $\Ep[|\hat M_\nu-M_\nu|]\le b_m$, then
\begin{align}
    \label{eq:section4_cube_kernel_lambda_rate}
    \Ep\left[
        \int_{[0,1]^d}
        |\hat\lambda_{nm}^{\mathrm{ker}}(x)-\lambda_0(x)|^2\,d\mu(x)
    \right]
    \le
    C\Bigl(
        M_\mu \mathfrak R_n^{\mathrm{ker}}(\alpha)
        +
        M_\nu \mathfrak R_m^{\mathrm{ker}}(\alpha)
        + a_n+b_m
    \Bigr).
\end{align}
\end{theorem}

This result demonstrates that, under Assumption~\ref{assm:cube_kernel_smoothness}, the kernel-based estimator of the transport map achieves the rate $\mathfrak R_n^{\mathrm{ker}}(\alpha)$, which is faster than the rate $O((\log n)n^{-2/d})$ for $d\ge5$ of the nearest-neighbor estimator in Theorem~\ref{thm:main_plan_based_rates}, potentially mitigating the curse of dimensionality through the smoothness. In fact, this rate is minimax optimal for estimation of smooth transport and growth maps, up to a logarithmic factor, as we will discuss in Section~\ref{sec:minimax_lower_bound}.

\subsection{Proof outline for the upper bounds}
\label{subsec:proof-outline}

We present the case of the nearest-neighbor estimator below, while the approach is similar for the kernel-based estimator (Appendix~\ref{sec:proof_cube_kernel_rates}).

Let $\hat\gamma$ be an optimizer of $\mathrm{UOT}(\hat\mu_n,\hat\nu_m)$. Define the row masses $\hat r_i \coloneqq \sum_{j=1}^m \hat\gamma_{ij}$ and column masses $\hat s_j \coloneqq \sum_{i=1}^n \hat\gamma_{ij}$ of $\hat\gamma$, $\hat r^\star_i \coloneqq e^{-\varphi_0(X_i)}\hat\mu_i$ and $\hat s^\star_j \coloneqq e^{-\psi_0(Y_j)}\hat\nu_j$, the barycentric projections $\hat T_i \coloneqq \frac{1}{\hat r_i}\sum_{j=1}^m \hat\gamma_{ij}Y_j$, and $\hat\Delta_{nm}^{\mathrm{bar}} \coloneqq \sum_{i=1}^n \hat r_i \|\hat T_i-T_0(X_i)\|^2$. The proof uses the following inequality, which bounds $\int_\Omega\|\hat T^{\mathrm{1NN}}-T_0\|^2\,d\gamma_0$, where $\gamma_0 \coloneqq e^{-\varphi_0}\mu$ is the active marginal, in terms of $\hat\Delta_{nm}^{\mathrm{bar}}$, ${D_\KL}(\hat r\mid\hat r^\star)$, and the Voronoi cells (proof of Theorem~\ref{thm:uot_complete_1nn_estimated_mass}):
\begin{equation}
\int_\Omega \|\hat T^{\mathrm{1NN}}(x)-T_0(x)\|^2\,d\gamma_0(x) \lesssim n M_n \bigl(\hat\Delta_{nm}^{\mathrm{bar}} + {D_\KL}(\hat r\mid \hat r^\star)\bigr) + R_n^2.
\end{equation}
Here, $M_n$ and $R_n$ denote the maximum Voronoi cell mass and radius. By the Vapnik-Chervonenkis inequality (Lemma~\ref{lem:uot_voronoi_mass_radius}), $M_n \lesssim (\log n) / n$ with high probability and $\Ep[R_n^2] \lesssim ((\log n)/n)^{2/d}$.

To bound $\hat\Delta_{nm}^{\mathrm{bar}}$ and ${D_\KL}(\hat r\mid \hat r^\star)$, we consider the empirical excess: 
\[\hat E_n \coloneqq \mathrm{UOT}(\hat\mu_n,\hat\nu_m) - \int (1-e^{-\varphi_0(x)})\,d\hat\mu_n - \int (1-e^{-\psi_0(y)})\,d\hat\nu_m. \]
We refer to this quantity as the \emph{duality gap} of $(\varphi_0,\psi_0)$ for the measures $\hat\mu_n,\hat\nu_m$, since $(\varphi_0,\psi_0)$ is feasible for \eqref{eq:gh_dual} with these measures. We show that $\hat E_n$ can be expressed as follows (Proposition~\ref{prop:uot_barycentric_and_1nn_estimated_mass}):
\begin{align}
\label{eq:excess_identity_summary}
\hat E_n = \sum_{i=1}^n\sum_{j=1}^m \hat\gamma_{ij} \Bigl( \tfrac{1}{2}\| X_i - Y_j\|^2-\varphi_0(X_i)-\psi_0(Y_j) \Bigr) + {D_\KL}(\hat r \mid \hat r^\star) + {D_\KL}(\hat s \mid \hat s^\star).
\end{align}
By dual feasibility, all terms on the right-hand side are nonnegative. Hence, ${D_\KL}(\hat r\mid \hat r^\star) \le \hat E_n$. By Jensen's inequality and Lemma~\ref{lem:uot_gap_sufficient}, we have $\hat\Delta_{nm}^{\mathrm{bar}} \le \sum_{i=1}^n\sum_{j=1}^m \hat\gamma_{ij}\|Y_j - T_0(X_i)\|^2 \le \frac{2}{\kappa} \hat E_n$.

Taking the expectation, we then employ our stability bound (Proposition~\ref{prop:estimated_mass_stability}), which yields
\begin{equation}
\begin{aligned}
\Ep[\hat E_n] &\lesssim \Ep[M_\mu W_2^2(\bar\mu_n,\bar\mu)] + \Ep[M_\nu W_2^2(\bar\nu_m,\bar\nu)] + a_n + b_m,
\end{aligned}
\end{equation}
where $a_n$ and $b_m$ are the sequences from Assumption~\ref{assm:masses} and $\bar\mu_n,\bar\nu_m$ are the normalized empirical measures. Using any existing bound on $\Ep[W_2^2(\bar\mu_n,\bar\mu)]$ and $\Ep[W_2^2(\bar\nu_m,\bar\nu)]$, e.g., \cite{fournier2015rate}, we obtain \eqref{eq:section4_plan_map_rate_abstract}. For $\hat T^{\mathrm{1NN}}$, the terms $a_n$ and $b_m$ do not appear, because the argument applies with $\hat M_\mu=M_\mu$ and $\hat M_\nu=M_\nu$ without changing $\hat T^{\mathrm{1NN}}$.

To bound $\Ep[\int_\Omega|\hat\lambda^{\mathrm{1NN}}-\lambda_0|^2\,d\mu]$, we relate the error of $\hat\lambda_i$ to the errors of $\hat a_i$ and $\hat T_i$ using the following bound (Lemma~\ref{lem:discrete_gh_growth_transfer}):
\begin{equation}
 |\hat\lambda_i-\lambda_0(X_i)|^2 \le C_{\mathrm{gr}}\bigl(|\hat a_i-a_0(X_i)|^2+\|\hat T_i-T_0(X_i)\|^2\bigr),
\end{equation}
for some constant $C_{\mathrm{gr}} > 0$, where $\hat a_i \coloneqq \sqrt{\hat r_i/\hat\mu_i}$ is a plug-in estimate of $a_0(X_i)$. Multiplying both sides by $\hat\mu_i$ and summing over $i$, we show in Lemma~\ref{lem:uot_complete_1nn_active_fitted} that $ \sum_{i=1}^n \hat\mu_i |\hat a_i-a_0(X_i)|^2 \le {D_\KL}(\hat r\mid \hat r^\star) $. The rest of the proof proceeds as for $\Ep[\int_\Omega\|\hat T^{\mathrm{1NN}}-T_0\|^2\,d\mu]$.

\subsection{Minimax lower bound on the hypercube}
\label{sec:minimax_lower_bound}

To match Theorem~\ref{thm:main_cube_kernel_rates} with a lower bound, we restrict to $\Omega=[0,1]^d$ and assume that the masses $M_\mu,M_\nu$ are known and that $m=n$.
We write $\mathbb P_{\mu,\nu}^n$ and $\mathbb E_{\mu,\nu}^n$ for the joint distribution and expectation, respectively.

Fix $\alpha>1$, parameters $M,B<\infty$ and $R>1$, and the constants $\kappa$, $\beta_{\min}$, and $\beta_{\max}$ in Assumptions~\ref{assm:curvature} and~\ref{assm:positivity}. Let $\mathcal U_\alpha(M,B,R)$ denote the class of pairs $(\mu,\nu)\in\mM_+([0,1]^d)^2$ with $\mu([0,1]^d)=1$ that satisfy the two conditions: (i) Assumptions~\ref{assm:curvature}, \ref{assm:positivity}, and~\ref{assm:cube_kernel_smoothness} hold with these constants, with smoothness radius $M$ in Assumption~\ref{assm:cube_kernel_smoothness}, and (ii) the unique $c$-concave solution $\varphi_0$ of the dual problem \eqref{eq:gh_dual} and the associated solution $(T_0,\lambda_0)$ of the Monge problem \eqref{eq:gh_monge} obtained via Theorem~\ref{prop:monge-potential} satisfy $\|T_0\|_{C^\alpha([0,1]^d)}+\|\lambda_0\|_{C^\alpha([0,1]^d)}\le B$ and $R^{-1}\le \lambda_0(x)\le R$ for all $x\in[0,1]^d$.

Our lower bound for the transport map follows the construction of the lower bound for balanced OT maps \cite{hutter2021minimax}. On the other hand, to the best of our knowledge, no minimax lower bound for the growth map has been proved before.

\begin{theorem}[Minimax lower bound over $\mathcal U_\alpha(M,B,R)$]
\label{thm:lb-ubot}
For every $\alpha>1$, there exist constants $M_0,B_0<\infty$, $R_0>1$, and $\kappa_0,b_0\in(0,1)$ such that, for every $M\ge M_0$, $B\ge B_0$, $R\ge R_0$, $\kappa\le\kappa_0$, $\beta_{\min}\le b_0$, and $\beta_{\max}\ge 1/b_0$, there is a constant $c_{\mathrm{lb}}>0$, depending only on $d,\alpha,M,B,R$ and the constants $\kappa$, $\beta_{\min}$, and $\beta_{\max}$ of Assumptions~\ref{assm:curvature} and \ref{assm:positivity}, such that for all sufficiently large $n$,
\begin{align}
    \inf_{\hat T}
    \sup_{(\mu,\nu)\in\mathcal U_\alpha(M,B,R)}
    \mathbb E_{\mu,\nu}^n\left[
        \int_{[0,1]^d}\|\hat T(x)-T_0(x)\|^2\,d\mu(x)
    \right]
    &\ge
    c_{\mathrm{lb}}\left(
        n^{ -2\alpha/(2\alpha-2+d)}
        \vee
        n^{-1}
    \right),\label{eq:lb_transport}
\\
    \inf_{\hat\lambda}
    \sup_{(\mu,\nu)\in\mathcal U_\alpha(M,B,R)}
    \mathbb E_{\mu,\nu}^n\left[
        \int_{[0,1]^d}|\hat\lambda(x)-\lambda_0(x)|^2\,d\mu(x)
    \right]
    &\ge
    c_{\mathrm{lb}}\left(
        n^{ -2\alpha/(2\alpha-2+d)}
        \vee
        n^{-1}
    \right),\label{eq:lb_growth}
\end{align}
where the infima are taken over all measurable functions of $(X_1,\dots,X_n,Y_1,\dots,Y_n,M_\mu,M_\nu)$.
\end{theorem}

The set of pairs $(\mu,\nu)\in\mM_+([0,1]^d)^2$ with $\mu([0,1]^d)=1$ that satisfy condition (ii) and the assumptions of Theorem~\ref{thm:main_cube_kernel_rates} on $\mu$ and $\nu$, with the constants $\kappa,\beta_{\min},\beta_{\max},M$ fixed, is exactly $\mathcal U_\alpha(M,B,R)$. Moreover, $w_\pm$ can be taken independent of the pair $(\mu,\nu)$ in this class (see Lemma~\ref{lem:dual_potential_bounds} in Appendix~\ref{sec:stability}, with $M_\mu=1$ and $M_\nu=\int\lambda_0^2\,d\mu\in[R^{-2},R^2]$), so the constants $C$ and $N_0$ in the upper bounds \eqref{eq:section4_cube_kernel_map_rate} and \eqref{eq:section4_cube_kernel_lambda_rate} do not depend on $(\mu,\nu)$. Hence, in view of Theorem~\ref{thm:main_cube_kernel_rates}, our lower bound indicates that our kernel-based estimator is minimax optimal up to the logarithmic factor. 

\section{Experiments}

\subsection{Simulation study}

\begin{figure}[t]
    \centering
    \begin{subfigure}{0.35\textwidth}
        \centering
        \includegraphics[width=\textwidth]{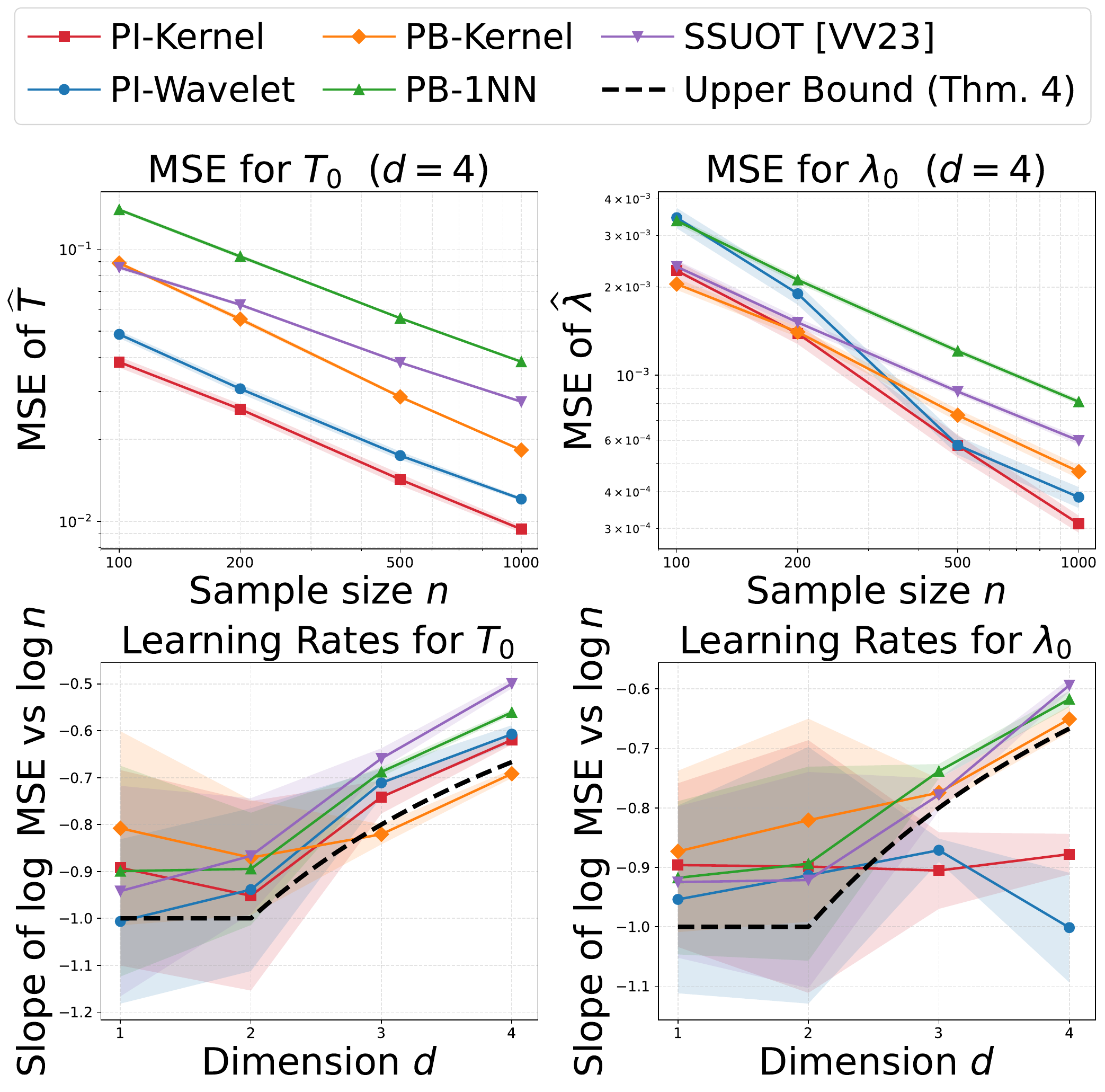}
        \caption{}
        \label{fig:simulation}
    \end{subfigure}
    \hfill
    \begin{subfigure}{0.64\textwidth}
        \centering
        \includegraphics[width=\textwidth]{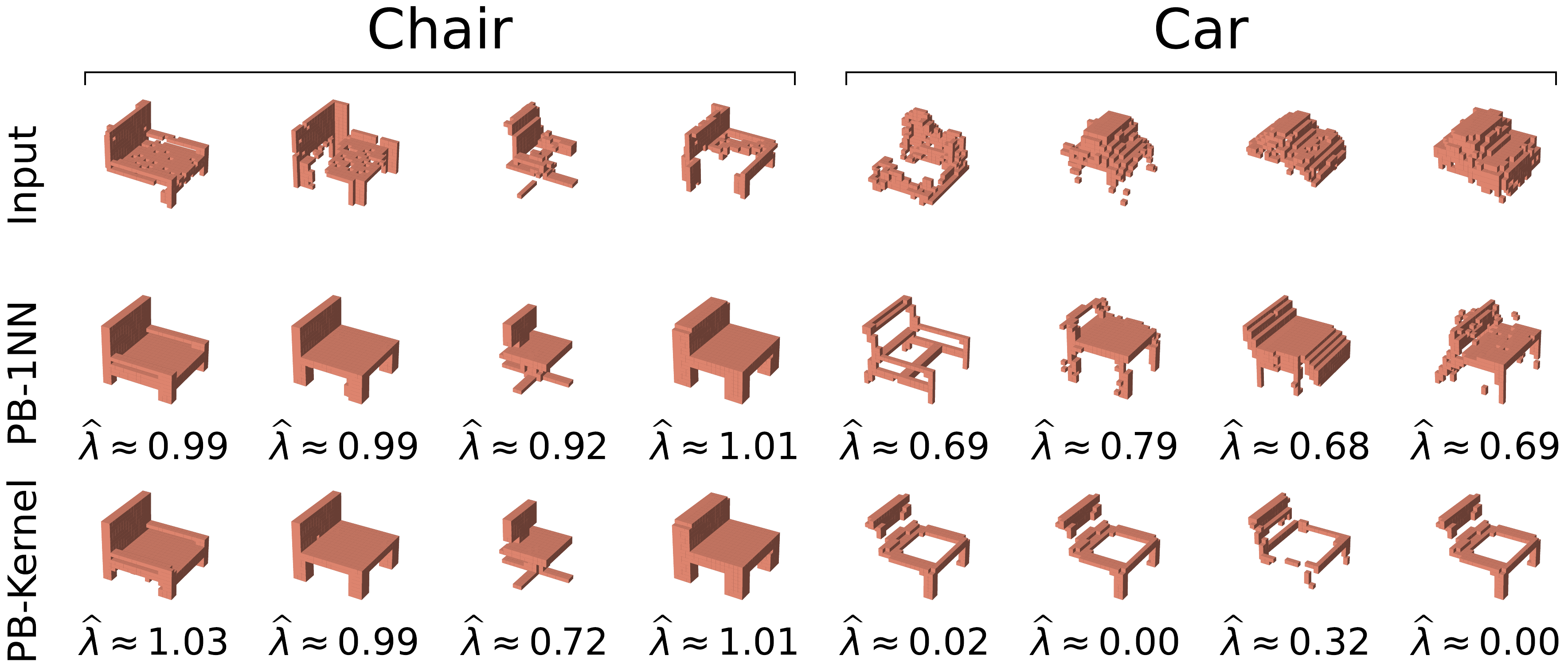}
        \caption{}
        \label{fig:3d}
    \end{subfigure}
    \caption{(a) MSE of the four UOT estimators and SSUOT. Top: The MSEs of estimating $T_0$ and $\lambda_0$ vs.\ $n$ at $d=4$, averaged over 10 seeds, with one standard error. Bottom: Learning rates for $T_0$ and $\lambda_0$ vs.\ $d$, that is, the slopes of the least-squares fit of the log mean MSE on $\log n$, with one standard error of the fitted slope. (b) Top: Incomplete 3D shapes. Middle: Complete 3D shapes predicted by the nearest-neighbor estimator. Bottom: Complete 3D shapes predicted by the plan-based kernel estimator.}
\end{figure}

We sample data from source and target measures $\mu,\nu \in \mM_+([0,1]^d)$ with $\mu([0,1]^d)=1$ and $\nu([0,1]^d)=2.5$. For $\theta=0.1$, the normalized source and target densities are the mixtures $p_\theta=(1-\theta)p+\theta$ and $q_\theta=(1-\theta)q+\theta$, where $p(x)\propto\prod_{i=1}^d|\sin(\pi(x_i-1/3))|$ and $q(x)\propto\prod_{i=1}^d|\sin(\pi(x_i-2/3))|$. Thus, $p_\theta$ and $q_\theta$ are bounded below by $\theta$, so Assumption~\ref{assm:positivity} holds with $\beta_{\min}=\theta$. Moreover, $p_\theta$ and $q_\theta$ are $1$-H\"older, which corresponds to $\alpha=2$ in Assumption~\ref{assm:cube_kernel_smoothness}. We benchmark several estimators of the pair $(T_0,\lambda_0)$: the nearest-neighbor estimator (PB-1NN, Section~\ref{sec:plan}), the plan-based kernel estimator (PB-Kernel, Appendix~\ref{sec:appendix_nw_extension}), the kernel-based estimator (PI-Kernel, Section~\ref{sec:kernel}), and the wavelet plug-in estimator (PI-Wavelet, Appendix~\ref{sec:wavelet}). We also include the SSUOT estimator of \cite{vacher2023semi} as a competing baseline.  Additional details are provided in Appendix~\ref{sec:sim_detail}.

Figure~\ref{fig:simulation} shows MSEs and empirical learning rates of the five methods. The bottom row shows the learning rates for $T_0$ and $\lambda_0$. For $d\le3$, all the estimators learn at rates faster than the rate $O((\log n)n^{-1/2})$ of Theorem~\ref{thm:main_plan_based_rates}. At $d=4$, the slowest estimator, SSUOT, learns $T_0$ at the rate $O(n^{-1/2})$. The dashed line is the rate $\mathfrak R_n^{\mathrm{ker}}(2)$ of Theorem~\ref{thm:main_cube_kernel_rates}. The learning rates of our plug-in estimators stay close to this line, except for $\lambda_0$ at $d\ge3$, where the learning rates are faster than this line. Moreover, the top row shows that, at $d=4$, both PI-Kernel and PI-Wavelet attain the smallest MSE for $T_0$ at every $n$. For $\lambda_0$, PI-Kernel and PI-Wavelet attain the smallest MSE for $n\ge500$.

\subsection{Application to 3D shape completion}

To showcase PB-1NN and PB-Kernel for high-dimensional tasks, we apply them to an unpaired 3D shape completion task using the Completion3D dataset~\cite{Chang2015,Yuan2018}. Specifically, we use the set of incomplete 3D point clouds of chairs as the source dataset and their complete versions as the target dataset. To test robustness, we introduce 30 incomplete cars into the source dataset as outliers. The point clouds are first transformed into a $16\times 16\times 16$ grid and subsequently projected via PCA into 1024 dimensions, yielding a source sample of size $n=5780$ and a target sample of size $m=5750$ in dimension $d=1024$. In this task, both source and target measures have unit mass. Nonetheless, we can leverage UOT to selectively reject outliers through the growth map assignment.

We fit PB-1NN and PB-Kernel to this dataset pair and use them to predict the complete shapes of four chairs and four cars in a held-out test set. As Figure~\ref{fig:3d} illustrates, both estimators successfully recover the shapes of the test chairs. However, their behavior diverges on the outlier test cars: while PB-1NN attempts to reconstruct chairs from the cars, PB-Kernel predicts the mean in the latent PCA space and yields values of $\hat\lambda$ that are close to zero. This indicates PB-Kernel implicitly discards the outliers, demonstrating that PB-Kernel can perform unpaired 3D shape completion in the presence of outliers.

We note that the plug-in estimators and SSUOT are omitted from this experiment due to computational intractability: the plug-in methods require a grid over $[0,1]^d$, which is not feasible for $d=1024$, while SSUOT requires solving, at each iteration, the quadratic program of \cite[Prop.~5.1]{vacher2023semi}, which has $(2n+1)(d+1)$ variables and $O(n^2)$ constraints. This intractability is a known phenomenon in estimation of smooth transport maps, as discussed in \cite{hutter2021minimax,manole2024plugin,divol2025general}.

\section{Conclusion}
We studied statistical estimation in unbalanced optimal transport with quadratic cost, where the target of estimation is the transport-growth pair, which consists of the transport map and the growth map. We developed the nearest-neighbor estimator and the kernel-based estimator, then derived convergence rates of their risks. We also showed the optimality of the kernel-based estimator by deriving the lower bound. Our main technical contribution is a reduction of estimation to a stability bound and the proof of this bound. Overall, our results provide a statistical foundation for estimation of the transport-growth pair in UOT and clarify the distinct roles of active marginals and growth estimation in this problem. A limitation is that our setup does not cover entropic regularization.

\section*{Acknowledgments and Disclosure of Funding}
D.P. was supported by JST CREST Grant Number JPMJCR21D2. N.I. was supported by JST ACT-X Grant Number JPMJAX25C2 and JSPS KAKENHI Grant Number JP25K24910.
M.I. was supported by JSPS KAKENHI Grant Number JP24K02904, JST CREST Grant Number JPMJCR21D2, JST FOREST Grant Number JPMJFR216I, JST BOOST Grant Number JPMJBY24A9, and JST ERATO Grant Number JPMJER2601.
The authors declare no competing interests.

\bibliography{main}
\bibliographystyle{alpha}

\appendix

\section{Related work}

\paragraph{Foundations of unbalanced optimal transport.} UOT is largely built on entropy-transport formulations, which relax the marginal constraints by convex divergences and include the logarithmic entropy-transport problem and the Wasserstein-Fisher-Rao distance \cite{ChizatEtAl2018,LieroMielkeSavare2018}. Recent work has studied the Monge formulation, the primal-dual optimality conditions, and the regularity theory of these models \cite{savare2024relaxation,gallouet2025regularity}. These results provide the structural background for our analysis.

\paragraph{Statistical map estimation in balanced OT.} In the balanced setting, OT map estimation is now supported by a fairly rich statistical theory. Minimax-optimal rates for smooth maps were established via semi-dual curvature arguments by \cite{hutter2021minimax}. On the plug-in side, barycentric-projection estimators based on empirical or smoothed couplings were analyzed by \cite{deb2021barycentric}, while entropic estimators were developed by \cite{Pooladian2021} and shown to be particularly effective in semi-discrete and discontinuous settings by \cite{pooladian2023semidiscrete}. Sharp smooth plug-in guarantees were obtained by \cite{manole2024plugin}, and the stability-based perspective has recently been broadened both to more general function classes and to sharper reductions from map estimation to distribution estimation \cite{divol2025general,balakrishnan2025stability,ponnoprat2025minimax}. Our work can be viewed as an unbalanced counterpart to this line of work on stability and plug-in estimation, but with a different estimation target: the transport-growth pair induced by the active marginals of the UOT transport plan.

\paragraph{Statistical theory for UOT.} The closest prior works are those of \cite{vacher2022stability,vacher2023semi}, who derive semi-dual formulations for quadratic UOT and global upper and lower bounds on the Bregman divergence of the semi-dual objective, leading to the first fast rates for UOT semi-dual potentials. Our contribution differs in both the estimation target and the technique. These works bound a pseudo-distance between the semi-dual potentials, which controls the $L^2$ error of the transport map, whereas we study the transport-growth pair $(T_0,\lambda_0)$ and relate its risk directly to perturbations of the UOT value and to Wasserstein errors of measure estimates. {Statistical results for other UOT models remain scarce.} In particular, \cite{HundrieserEtAl2025} study unbalanced Kantorovich-Rubinstein distances, transport plans, and barycenters on finite spaces, and \cite{struleva2025sharp} establish near-minimax rates of UOT under the Kantorovich-Rubinstein distance, which is used for comparing observations from spatio-temporal point processes. The statistical rates of entropically regularized UOT estimators are obtained by \cite{andrade2026complexityeot}. These works address different costs, observation models, and estimation targets from the estimation of the transport-growth pair considered here.

\paragraph{Algorithms and applications.} A large parallel literature develops scalable algorithms and applications for UOT; see \cite{sejourne2023theorynumerics} for a broad overview. On the computational side, generalized Sinkhorn methods and their complexity or acceleration for unbalanced problems were studied by \cite{pham2020sinkhorn,sejourne2022faster}, and lightweight or neural parameterizations have been proposed for large-scale settings \cite{Gazdieva2024,yang2018scalable,choi2024}. On the application side, UOT and the Wasserstein-Fisher-Rao distance have been used in single-cell trajectory inference and population dynamics \cite{schiebinger2019,sha2023cell}, tumor-growth and reaction-diffusion models \cite{dimarino2020}, gradient-flow learning \cite{yan2024}, and geometric shape analysis \cite{bauer2022SRNF}. These works motivate statistical guarantees for UOT. A nonparametric estimation theory for transport-growth pairs is still missing from this literature.

\section{Notations}\label{sec:notation_table}

Table~\ref{tab:notation} lists the recurring symbols of the paper, each with its definition and the place where the symbol is introduced. In the proofs, $C$ denotes a positive constant. The value of $C$ may change from line to line, or even within a line, but it does not depend on $n$ and $m$.

\begin{table}[h]
\centering
\small
\renewcommand{\arraystretch}{1.15}
\caption{Recurring symbols.}
\label{tab:notation}
\begin{tabularx}{\textwidth}{@{}lXl@{}}
\toprule
Symbol & Meaning & Defined in \\
\midrule
$\Omega$, $\mathrm{diam}(\Omega)$ & domain and its diameter & Sec.~\ref{sec:setup}, Assm.~\ref{assm:supp_global} \\
$c(x,y)$ & quadratic cost $\frac12\|x-y\|^2$ & Sec.~\ref{sec:setup} \\
$\mu,\nu$; $M_\mu,M_\nu$ & source and target measures; their total masses & Sec.~\ref{sec:setup} \\
$\bar\mu,\bar\nu$; $p,q$ & normalized measures $\mu/M_\mu$, $\nu/M_\nu$; their normalized densities & Sec.~\ref{sec:kernel}, Assm.~\ref{assm:positivity} \\
$\varphi^c$ & $c$-transform $\inf_{y}\{\frac12\|x-y\|^2-\varphi(y)\}$ & Sec.~\ref{sec:setup} \\
$(\varphi_0,\psi_0)$ & optimal dual pair, $\psi_0=\varphi_0^c$ & \eqref{eq:gh_dual}, Thm.~\ref{prop:monge-potential} \\
$(T_0,\lambda_0)$ & $T_0=\mathrm{id}-\nabla\varphi_0$, $\lambda_0=\exp(-\varphi_0+\frac14\|\nabla\varphi_0\|^2)$ & \eqref{eq:unique} \\
$f_0$ & Brenier potential $\frac12\|x\|^2-\varphi_0(x)$ & after Assm.~\ref{assm:curvature} \\
$a_0$, $w_0$ & $e^{-\varphi_0/2}=(d\gamma_0/d\mu)^{1/2}$; $w_0=a_0^2=e^{-\varphi_0}$ & Secs.~\ref{sec:setup} and~\ref{sec:plan} \\
$\gamma_0,\gamma_1$ & active marginals $e^{-\varphi_0}\mu$ and $e^{-\psi_0}\nu$ & Sec.~\ref{sec:setup}, App.~\ref{sec:proof_plan_based_rates} \\
$\zeta_0,\xi_0$ & $1-e^{-\varphi_0}$ and $1-e^{-\psi_0}$ & Prop.~\ref{prop:two_sample_stability} \\
$\kappa$ & constant in the Hessian bounds on $\varphi_0$ & Assm.~\ref{assm:curvature} \\
$\beta_{\min},\beta_{\max}$ & lower and upper bounds on $p,q$ & Assm.~\ref{assm:positivity} \\
$\hat M_\mu,\hat M_\nu$; $a_n,b_m$; $\bar c$ & estimated masses; bounds on $\Ep|\hat M_\mu-M_\mu|$, $\Ep|\hat M_\nu-M_\nu|$; constant in $\hat M_\mu\ge \bar cM_\mu$, $\hat M_\nu\ge \bar cM_\nu$ & Assm.~\ref{assm:masses} \\
$\rho_\mu,\rho_\nu$ & $\hat M_\mu/M_\mu$ and $\hat M_\nu/M_\nu$ & App.~\ref{sec:stability} \\
$\alpha$, $M$ & smoothness index, $p,q\in\mathcal C_N^{\alpha-1}([0,1]^d;M)$, and radius & Assm.~\ref{assm:cube_kernel_smoothness} \\
$\|\cdot\|_{C^s(\mathcal D)}$ & H\"older norm on $\mathcal D\in\{\mathbb{T}_2^d,[0,1]^d,\mathbb{R}\}$ & App.~\ref{sec:holder_norm} \\
$W_2$ & $2$-Wasserstein distance & Sec.~\ref{subsec:notation} \\
$w_\pm$ & clipping constants, $w_-\le e^{-\varphi_0}\le w_+$ & Sec.~\ref{sec:plan} \\
$M_\varphi,M_\psi$ & $\sup e^{-\varphi_0}$ and $\sup e^{-\psi_0}$ & proof of Prop.~\ref{prop:two_sample_stability} \\
$\Lambda$ & $\max(\mathrm{Lip}(\nabla\zeta_0),\mathrm{Lip}(\nabla\xi_0))$ & proof of Prop.~\ref{prop:two_sample_stability} \\
$C_\Lambda$ & stability constant $\max(M_\varphi,M_\psi)+\Lambda/2$ & Prop.~\ref{prop:two_sample_stability} \\
$\mathcal U_\alpha(M,B,R)$ & class of pairs $(\mu,\nu)$ in Theorem~\ref{thm:lb-ubot}, with $\mu([0,1]^d)=1$; $B$ bounds $\|T_0\|_{C^\alpha}+\|\lambda_0\|_{C^\alpha}$, and $R^{-1}\le\lambda_0\le R$ & Sec.~\ref{sec:minimax_lower_bound} \\
$n,m$; $X_i,Y_j$ & sample sizes; samples from $\bar\mu$ and $\bar\nu$ & Sec.~\ref{sec:setup} \\
$\hat\mu_n,\hat\nu_m$ & empirical measures with masses $\hat M_\mu,\hat M_\nu$ & Sec.~\ref{sec:setup} \\
$\bar\mu_n,\bar\nu_m$; $\tilde\mu_n,\tilde\nu_m$ & normalized empirical measures; empirical measures with masses $M_\mu,M_\nu$ & Sec.~\ref{subsec:proof-outline}, App.~\ref{sec:stability} \\
$\hat\gamma$; $\hat r_i,\hat s_j$; $\hat r_i^\star,\hat s_j^\star$ & discrete optimal transport plan; its row and column masses; $e^{-\varphi_0(X_i)}\hat\mu_i$, $e^{-\psi_0(Y_j)}\hat\nu_j$ & Secs.~\ref{sec:plan} and~\ref{subsec:proof-outline} \\
$\hat T^{\mathrm{1NN}},\hat\lambda^{\mathrm{1NN}}$ & nearest-neighbor estimators & \eqref{eq:disc-gh-growth}, \eqref{eq:disc-1nn-map} \\
$M_n,R_n$ & maximum Voronoi cell mass and radius & Sec.~\ref{subsec:proof-outline}, Lem.~\ref{lem:uot_voronoi_mass_radius} \\
$\hat T_{nm}^{\mathrm{ker}},\hat\lambda_{nm}^{\mathrm{ker}}$; $L_n$; $\tau$ & kernel-based estimators; frequency cutoff; function in $K_L$ & Sec.~\ref{sec:kernel} \\
$\mathfrak R_n^{\mathrm{emp}}(d)$, $\mathfrak R_n^{\mathrm{ker}}(\alpha)$ & rates & before Thms.~\ref{thm:main_plan_based_rates} and~\ref{thm:main_cube_kernel_rates} \\
$\mathfrak L_n(\alpha)$ & rate $n^{-(\alpha-1)/(2(\alpha-1)+d)}$ & \eqref{eq:frak_L_rate} \\
$\mathcal G_{nm}$ & event on which $\tilde p_n^{\mathrm{ker}}\ge\underline p/2$ and $\tilde q_m^{\mathrm{ker}}\ge\underline q/2$ & App.~\ref{sec:proof_cube_kernel_rates} \\
\bottomrule
\end{tabularx}
\end{table}
\section{Additional constructions and remarks}
\subsection{H\"older norms}
\label{sec:holder_norm}

In this subsection, we define the H\"older norms used in this paper. Let $s>0$ and $k\coloneqq\lceil s\rceil-1$, so that $s-k\in(0,1]$. Let $\mathcal D$ be one of the sets $\mathbb{T}_2^d$, $[0,1]^d$, and $\mathbb{R}$, with the distance $\mathrm{dist}_{\mathcal D}:\mathcal D\times\mathcal D\to[0,\infty)$ given by
\begin{align*}
    \mathrm{dist}_{\mathbb{T}_2^d}(x,y)\coloneqq\min_{z\in2\mathbb{Z}^d}\|x-y-z\|,
    \qquad
    \mathrm{dist}_{[0,1]^d}(x,y)\coloneqq\|x-y\|,
    \qquad
    \mathrm{dist}_{\mathbb{R}}(x,y)\coloneqq|x-y|.
\end{align*}
For $\flat\in\mathbb{N}_0^\ell$, where $\ell\coloneqq1$ if $\mathcal D=\mathbb{R}$ and $\ell\coloneqq d$ otherwise, write $|\flat|\coloneqq\flat_1+\dots+\flat_\ell$ and $\partial^\flat\coloneqq\partial_1^{\flat_1}\cdots\partial_\ell^{\flat_\ell}$. For a function $f:\mathcal D\to\mathbb{R}$ that is $k$ times continuously differentiable, with one-sided derivatives on the boundary of $[0,1]^d$,
\begin{align*}
    \|f\|_{C^s(\mathcal D)}
    \coloneqq
    \sum_{|\flat|\le k}\sup_{x\in\mathcal D}|\partial^\flat f(x)|
    +\sum_{|\flat|=k}\sup_{\substack{x,y\in\mathcal D\\ x\ne y}}\frac{|\partial^\flat f(x)-\partial^\flat f(y)|}{\mathrm{dist}_{\mathcal D}(x,y)^{s-k}}.
\end{align*}
For every other function $f:\mathcal D\to\mathbb{R}$, we set $\|f\|_{C^s(\mathcal D)}\coloneqq+\infty$. We write $C^s(\mathcal D)\coloneqq\{f:\mathcal D\to\mathbb{R}:\|f\|_{C^s(\mathcal D)}<\infty\}$. For a map $T=(T_1,\dots,T_d):\mathcal D\to\mathbb{R}^d$, we write $\|T\|_{C^s(\mathcal D)}\coloneqq\sum_{i=1}^d\|T_i\|_{C^s(\mathcal D)}$. For a set $\mathcal W\subset\mathbb{R}^d$, $C^s(\mathcal D;\mathcal W)$ denotes the set of maps $T:\mathcal D\to\mathcal W$ with $\|T\|_{C^s(\mathcal D)}<\infty$.

\subsection{Plan-based kernel estimator}
\label{sec:appendix_nw_extension}

In this subsection, we define the \emph{plan-based kernel estimator}, a second way to extend the discrete quantities $(\hat T_i,\hat a_i,\hat\lambda_i)_{i=1}^n$ of Section~\ref{sec:plan} from the sample points to all of $\Omega$. Step~(iii) of Section~\ref{sec:plan} performs this extension with the Voronoi weights $v_i(x)=\mathbf{1}\{x\in V_i\}$. The resulting nearest-neighbor estimator~\eqref{eq:disc-1nn-map} is piecewise constant on the cells $V_1,\dots,V_n$. In contrast, the plan-based kernel estimator uses kernel weights and therefore varies continuously with $x$ whenever the kernel is continuous. 

Fix a positive kernel $K:\R^d\to(0,\infty)$ and a bandwidth $h>0$. The \emph{Nadaraya-Watson weights} $v_{i,h}:\Omega\to[0,1]$ are
\begin{align}
\label{eq:nw_weights}
    v_{i,h}(x)
    \coloneqq
    \frac{K\left((x-X_i)/h\right)}{\sum_{k=1}^n K\left((x-X_k)/h\right)},
    \qquad i=1,\dots,n.
\end{align}
Like the Voronoi weights $v_i(x)$, these weights are nonnegative and sum to one. But unlike $v_i(x)$, which is positive for one sample point only, $v_{i,h}(x)$ is positive for every sample point $X_i$. We use these weights to define the plan-based kernel estimator.

\begin{definition}[Plan-based kernel estimator]
\label{def:nw_extension_app}
Let $(\hat T_i,\hat a_i,\hat\lambda_i)_{i=1}^n$ be the discrete quantities of Section~\ref{sec:plan}, with $\hat T_i$ given by \eqref{eq:disc-frechet} and $\hat\lambda_i$ by \eqref{eq:disc-gh-growth}. Let $v_{i,h}$ be the Nadaraya-Watson weights~\eqref{eq:nw_weights} of a positive kernel $K$ with bandwidth $h>0$. The plan-based kernel estimator is the triple of functions $\hat T_h^{\mathrm{NW}}:\Omega\to\R^d$, $\hat a_h^{\mathrm{NW}}:\Omega\to[0,\infty)$, and $\hat\lambda_h^{\mathrm{NW}}:\Omega\to(0,\infty)$ given by
\begin{align*}
    \hat T_h^{\mathrm{NW}}(x)
    \coloneqq
    \sum_{i=1}^n v_{i,h}(x)\hat T_i,
    \qquad
    \hat a_h^{\mathrm{NW}}(x)
    \coloneqq
    \sum_{i=1}^n v_{i,h}(x)\hat a_i,
    \qquad
    \hat\lambda_h^{\mathrm{NW}}(x)
    \coloneqq
    \sum_{i=1}^n v_{i,h}(x)\hat\lambda_i.
\end{align*}
\end{definition}
In other words, we use the kernel $K$ to smooth the estimates $\hat T_i$, $\hat a_i$, and $\hat\lambda_i$ of $T_0(X_i)$, $a_0(X_i)$, and $\lambda_0(X_i)$. The three smoothed functions have the form of the nearest-neighbor estimator~\eqref{eq:disc-1nn-map}, with the Voronoi weights $v_i(x)$ replaced by $v_{i,h}(x)$. As in Section~\ref{sec:plan}, the functions $\hat T_h^{\mathrm{NW}}$, $\hat a_h^{\mathrm{NW}}$, and $\hat\lambda_h^{\mathrm{NW}}$ estimate $T_0$, $a_0$, and $\lambda_0$, respectively.

For the Gaussian kernel $K(u)=e^{-\|u\|^2/2}$, the plan-based kernel estimator converges to the nearest-neighbor estimator as $h\to0$ at every $x$ in the interior of some $V_i$. For this kernel, the weights~\eqref{eq:nw_weights} satisfy
\begin{align*}
    \frac{v_{k,h}(x)}{v_{i,h}(x)}
    =
    \exp\left(-\frac{\|x-X_k\|^2-\|x-X_i\|^2}{2h^2}\right),
    \qquad x\in\Omega,\ i,k\in\{1,\dots,n\}.
\end{align*}
Hence, if $X_i$ is the unique sample point nearest to $x$, then $v_{i,h}(x)\to1=v_i(x)$ and $v_{k,h}(x)\to0=v_k(x)$ for every $k\in\{1,\dots,n\}\setminus\{i\}$ as $h\to0$.

\subsection{Obtaining samples through Poisson point processes}\label{sec:appendix_ppp}

For a parameter $t\ge1$, we consider observations from two independent Poisson point processes,
\begin{align}
    \mathcal{X} = \{X_i\}_{i=1}^{N_\mu} \sim \mathrm{PPP}(t\mu), \qquad \mathcal{Y} = \{Y_j\}_{j=1}^{N_\nu} \sim \mathrm{PPP}(t\nu).
\end{align}
Here, the sample sizes satisfy $N_\mu \sim \mathrm{Poisson}(tM_\mu)$ and $N_\nu \sim \mathrm{Poisson}(tM_\nu)$. Conditional on $N_\mu$, $X_i$ independently and identically follows $(M_\mu)^{-1}\mu$. Then, we define the empirical measures
\begin{align}
    \hat{\mu} \coloneqq \frac{1}{t}\sum_{i=1}^{ N_\mu} \delta_{X_i}, \qquad \hat{\nu} \coloneqq \frac{1}{t}\sum_{j=1}^{N_\nu} \delta_{Y_j}.
\end{align}
In this setup, we have $\Ep[\hat{\mu}] = \mu$. This setup with slight modification is studied by \cite{struleva2025sharp}.

The counts $N_\mu$ and $N_\nu$ are zero with positive probability, so we truncate them to estimate the masses. Let $0<\underline M_\mu\le M_\mu$ and $0<\underline M_\nu\le M_\nu$ be known constants, and define
\begin{align}
    \hat M_\mu \coloneqq \max\{N_\mu/t,\underline M_\mu\},
    \qquad
    \hat M_\nu \coloneqq \max\{N_\nu/t,\underline M_\nu\}.
\end{align}
Since $\underline M_\mu\le M_\mu$, the truncation does not increase the distance to $M_\mu$. Hence, Jensen's inequality yields
\begin{align}
    \Ep\bigl[|\hat M_\mu-M_\mu|\bigr]
    \le \Ep\bigl[|N_\mu/t-M_\mu|\bigr]
    \le \sqrt{M_\mu/t},
\end{align}
and the analogous bound holds for $\hat M_\nu$. Conditional on $N_\mu=n$ and $N_\nu=m$, the samples are independent with distributions $\bar\mu$ and $\bar\nu$, as in Section~\ref{sec:setup}, and $\hat M_\mu,\hat M_\nu$ are deterministic with $\hat M_\mu\ge\bar cM_\mu$ and $\hat M_\nu\ge\bar cM_\nu$, where $\bar c=\min\{\underline M_\mu/M_\mu,\underline M_\nu/M_\nu\}$. The bounds of Theorems~\ref{thm:main_plan_based_rates} and~\ref{thm:main_cube_kernel_rates} depend on $a_n,b_m$ only through $\Ep[|\hat M_\mu-M_\mu|]\le a_n$ and $\Ep[|\hat M_\nu-M_\nu|]\le b_m$. Hence, the bounds of Theorems~\ref{thm:main_plan_based_rates} and~\ref{thm:main_cube_kernel_rates} hold conditionally on $N_\mu=n$ and $N_\nu=m$ for all sufficiently large $n,m$, with $a_n$ and $b_m$ replaced by $|\hat M_\mu-M_\mu|$ and $|\hat M_\nu-M_\nu|$.

\subsection{Proof of Lemma~\ref{lem:uot_gap_sufficient}} \label{sec:uot_gap}

We first show that $T_0:\Omega\to\Omega$ is a bijection and that, for every $y\in\Omega$, the infimum defining $\psi_0(y)$ is attained at $T_0^{-1}(y)$.

\begin{lemma}[Bijectivity of $T_0$ and regularity of $\psi_0$]
\label{lem:uot_T0_bijection}
Suppose that Assumptions~\ref{assm:curvature}, \ref{assm:supp_global}, and \ref{assm:positivity} hold. Then $T_0=\nabla f_0$ for the Brenier potential $f_0(x)=\frac12\|x\|^2-\varphi_0(x)$, which satisfies $\kappa I\preceq\nabla^2 f_0(x)\preceq\kappa^{-1}I$ on $\Omega$, and $T_0:\Omega\to\Omega$ is a bijection. Moreover, for every $y\in\Omega$, the infimum in the $c$-transform $\psi_0(y)=\inf_{x\in\Omega}\{\frac12\|x-y\|^2-\varphi_0(x)\}$ is attained at $x=T_0^{-1}(y)$, and
\begin{align}
\label{eq:psi0_via_T0_inverse}
    \psi_0(y)
    =
    \frac12\|T_0^{-1}(y)-y\|^2-\varphi_0\bigl(T_0^{-1}(y)\bigr),
\end{align}
and
\begin{align}
\label{eq:psi0_hessian_bounds}
    (1-\kappa^{-1})I\preceq\nabla^2\psi_0(y)\preceq(1-\kappa)I,
    \qquad y\in\Omega.
\end{align}
\end{lemma}

\begin{proof}[Proof of Lemma \ref{lem:uot_T0_bijection}]
By Assumption~\ref{assm:curvature}, $\nabla^2 f_0=I-\nabla^2\varphi_0$ satisfies $\kappa I\preceq\nabla^2 f_0\preceq\kappa^{-1}I$ on $\Omega$. Since $\Omega$ is convex, $f_0$ is $\kappa$-strongly convex on $\Omega$. Strong convexity yields
\begin{align*}
    \langle T_0(x)-T_0(x'),x-x'\rangle \ge \kappa\|x-x'\|^2,
    \qquad x,x'\in\Omega,
\end{align*}
so $\|T_0(x)-T_0(x')\|\ge\kappa\|x-x'\|$ by the Cauchy-Schwarz inequality; in particular $T_0$ is injective.

For surjectivity, we combine the primal-dual optimality conditions for \eqref{eq:gh_kantorovich} and \eqref{eq:gh_dual} \cite{LieroMielkeSavare2018} with Theorem~\ref{prop:monge-potential}: the optimal transport plan is $\gamma=(\mathrm{id},T_0)_\#\gamma_0$ with $\gamma_0=e^{-\varphi_0}\mu$ and $\gamma_1=e^{-\psi_0}\nu$. The potential $\varphi_0$ is continuous by Assumption~\ref{assm:curvature} and $\psi_0=\varphi_0^c$ is continuous, so $\varphi_0$ and $\psi_0$ are bounded on the compact set $\Omega$. By Assumption~\ref{assm:positivity}, $\nu$ has a density bounded below on $\Omega$, so $\gamma_1=e^{-\psi_0}\nu$ has full support: $\supp(\gamma_1)=\Omega$. On the other hand, $\gamma_1=T_{0\#}\gamma_0$ is supported in the compact set $T_0(\Omega)$. Hence, $\Omega\subseteq T_0(\Omega)$. The reverse inclusion also holds: the transport plan $\gamma$ is supported in $\Omega\times\Omega$, so $T_0(x)\in\Omega$ for $\gamma_0$-a.e.\ $x$, and since $\gamma_0$ has full support, $T_0$ is continuous, and $\Omega$ is closed, $T_0(\Omega)\subseteq\Omega$. The map $T_0$ is therefore onto.

Fix $y\in\Omega$ and set $x'\coloneqq T_0^{-1}(y)$. The function
\begin{align}
\label{eq:uy_definition}
    u_y(x)\coloneqq\frac12\|x-y\|^2-\varphi_0(x)
    = f_0(x)-\langle x,y\rangle+\frac12\|y\|^2
\end{align}
is $\kappa$-strongly convex on $\Omega$ with $\nabla u_y(x')=T_0(x')-y=0$. A critical point of a convex function on a convex set is a global minimizer, so the infimum defining $\psi_0(y)=\varphi_0^c(y)$ is attained at $x'$, which is \eqref{eq:psi0_via_T0_inverse}. Since $\nabla^2 f_0\succeq\kappa I$ is invertible and $f_0$ is twice continuously differentiable, the inverse function theorem shows that $T_0^{-1}$ is continuously differentiable on $\Omega$ with $\nabla T_0^{-1}(y)=(\nabla^2 f_0(x'))^{-1}$. Differentiating $\psi_0(y)=u_y(T_0^{-1}(y))$ and using $\nabla u_y(x')=0$ yields $\nabla\psi_0(y)=y-T_0^{-1}(y)$, hence
\begin{align*}
    \nabla^2\psi_0(y)= I-\nabla T_0^{-1}(y)=I-\bigl(\nabla^2 f_0(x')\bigr)^{-1}.
\end{align*}
Because the eigenvalues of $\nabla^2 f_0(x')$ lie in $[\kappa,\kappa^{-1}]$, those of $(\nabla^2 f_0(x'))^{-1}$ lie in the same interval, which proves \eqref{eq:psi0_hessian_bounds}.
\end{proof}

We now prove Lemma~\ref{lem:uot_gap_sufficient}, which we restate below:
\begin{restatedgaplemma}[restated]
Suppose that Assumptions~\ref{assm:curvature}, \ref{assm:supp_global}, and \ref{assm:positivity} hold. Then the dual potentials $(\varphi_0,\psi_0)$ satisfy the following uniform lower bound on the gap:
\begin{align}
    \frac{1}{2}\|x-y\|^2-\varphi_0(x)-\psi_0(y)
    \ge \frac{\kappa}{2}\|y-T_0(x)\|^2,
    \qquad x,y\in\Omega.
\end{align}
\end{restatedgaplemma}
\begin{proof}[Proof of Lemma~\ref{lem:uot_gap_sufficient}]
Fix $x,y\in\Omega$ and set $y'\coloneqq T_0(x)$. The function
\begin{align*}
    g_x(z)\coloneqq\frac12\|x-z\|^2-\psi_0(z)
\end{align*}
is $\kappa$-strongly convex on $\Omega$ with $\nabla g_x(y')=T_0^{-1}(y')-x=0$, where we used $\nabla^2g_x=I-\nabla^2\psi_0\succeq\kappa I$ from the upper bound in \eqref{eq:psi0_hessian_bounds} and $\nabla\psi_0(z)=z-T_0^{-1}(z)$ from the proof of Lemma~\ref{lem:uot_T0_bijection}. Moreover, \eqref{eq:psi0_via_T0_inverse} yields $\psi_0(y')=\frac12\|x-y'\|^2-\varphi_0(x)$, so $g_x(y')=\varphi_0(x)$. Hence,
\begin{align*}
    \frac{1}{2}\|x-y\|^2-\varphi_0(x)-\psi_0(y)
    = g_x(y)-g_x(y')
    \ge \frac{\kappa}{2}\|y-T_0(x)\|^2,
\end{align*}
and this completes the proof.
\end{proof}

\section{Stability bound}
\label{sec:stability}

We introduce two variants of the empirical measures $\hat{\mu}_n$ and $\hat{\nu}_m$. The normalized empirical measures are defined as
\begin{align}
    \bar\mu_n \coloneqq {\hat\mu_n/\hat M_\mu} = \frac{1}{n}\sum_{i=1}^n \delta_{X_i}, \qquad \bar\nu_m \coloneqq {\hat\nu_m/\hat M_\nu} = \frac{1}{m}\sum_{j=1}^m \delta_{Y_j}.
\end{align}
We also define the empirical measures with masses $M_\mu$ and $M_\nu$ as
\begin{align}
    \tilde\mu_n \coloneqq \frac{M_\mu}{n}\sum_{i=1}^n \delta_{X_i} = M_\mu\bar\mu_n, \qquad \tilde\nu_m \coloneqq \frac{M_\nu}{m}\sum_{j=1}^m \delta_{Y_j} = M_\nu\bar\nu_m.
\end{align}

The following lemma bounds the ranges of $\varphi_0$ and $\psi_0$ in terms of $\mathrm{diam}(\Omega)$ and the mass ratio $M_\mu/M_\nu$.

\begin{lemma}[Two-sided bounds on the dual potentials]\label{lem:dual_potential_bounds}
Let $\Omega$ be a compact set, and let $(\varphi_0,\psi_0)$ be an optimal dual pair for \eqref{eq:gh_dual} with $\varphi_0$ $c$-concave and $\psi_0=\varphi_0^c$. Then
\begin{align}
\label{eq:dual_bounds_phi}
\begin{aligned}
\sup_{x\in\Omega}\varphi_0(x)&\le\tfrac12\mathrm{diam}(\Omega)^2+\log(1+M_\mu/M_\nu),\\
\inf_{x\in\Omega}\varphi_0(x)&\ge-\tfrac12\mathrm{diam}(\Omega)^2-\log(1+M_\nu/M_\mu),
\end{aligned}\\
\label{eq:dual_bounds_psi}
\begin{aligned}
\sup_{y\in\Omega}\psi_0(y)&\le\tfrac12\mathrm{diam}(\Omega)^2+\log(1+M_\nu/M_\mu),\\
\inf_{y\in\Omega}\psi_0(y)&\ge-\tfrac12\mathrm{diam}(\Omega)^2-\log(1+M_\mu/M_\nu).
\end{aligned}
\end{align}
\end{lemma}

\begin{proof}[Proof of Lemma \ref{lem:dual_potential_bounds}]
The pair $(0,0)$ is feasible for \eqref{eq:gh_dual}, so the dual value at the optimal pair $(\varphi_0,\psi_0)$ is nonnegative:
\begin{align*}
    \int_\Omega\bigl(1-e^{-\varphi_0}\bigr)\,d\mu+\int_\Omega\bigl(1-e^{-\psi_0}\bigr)\,d\nu\ge 0.
\end{align*}
The first integral is at most $M_\mu$. Suppose that $\sup_{y\in\Omega}\psi_0(y)<-\log(1+M_\mu/M_\nu)$. Then $1-e^{-\psi_0}<-M_\mu/M_\nu$ on $\Omega$, so the second integral is smaller than $-M_\mu$. Consequently, the dual value is negative. This contradiction shows that
\begin{align}
\label{eq:sup_psi_lower}
    \sup_{y\in\Omega}\psi_0(y)\ge-\log(1+M_\mu/M_\nu).
\end{align}
Exchanging the roles of $(\varphi_0,\mu)$ and $(\psi_0,\nu)$ yields $\sup_{x\in\Omega}\varphi_0(x)\ge-\log(1+M_\nu/M_\mu)$.

For the upper bounds in \eqref{eq:dual_bounds_phi} and \eqref{eq:dual_bounds_psi}, the dual constraint yields $\varphi_0(x)+\psi_0(y)\le\frac12\|x-y\|^2\le\frac12\mathrm{diam}(\Omega)^2$ for all $x,y\in\Omega$. Taking the supremum over $x$ and $y$ and using \eqref{eq:sup_psi_lower}, we obtain
\begin{align*}
    \sup_{x\in\Omega}\varphi_0(x)
    \le \tfrac12\mathrm{diam}(\Omega)^2-\sup_{y\in\Omega}\psi_0(y)
    \le \tfrac12\mathrm{diam}(\Omega)^2+\log(1+M_\mu/M_\nu).
\end{align*}
The upper bound in \eqref{eq:dual_bounds_psi} follows in the same way.

For the lower bounds in \eqref{eq:dual_bounds_phi} and \eqref{eq:dual_bounds_psi}, $\varphi_0$ is $c$-concave, so $\varphi_0=\psi_0^c$. Hence, for every $x\in\Omega$,
\begin{align*}
    \varphi_0(x)
    =\inf_{y\in\Omega}\bigl\{\tfrac12\|x-y\|^2-\psi_0(y)\bigr\}
    \ge-\sup_{y\in\Omega}\psi_0(y)
    \ge-\tfrac12\mathrm{diam}(\Omega)^2-\log(1+M_\nu/M_\mu),
\end{align*}
where the last step uses \eqref{eq:dual_bounds_psi}. Similarly, $\psi_0=\varphi_0^c$ yields the lower bound in \eqref{eq:dual_bounds_psi}.
\end{proof}

\subsection{Stability bound with known masses}

We are now ready to state the two-sample stability bound for the UOT problem  with $\tilde\mu_n$ and $\tilde\nu_m$, in the case that the masses $M_\mu$ and $M_\nu$ are known.
\begin{proposition}[Two-sample stability of $\mathrm{UOT}$]
\label{prop:two_sample_stability}
Let $\mu, \nu, \tilde\mu_n, \tilde\nu_m \in \mathcal{M}_+(\Omega)$. {Consider the UOT model introduced in Section~\ref{sec:setup}.} Let $\varphi_0, \psi_0$ be the optimal dual potentials for $(\mu,\nu)$, and define $\zeta_0\coloneqq 1-e^{-\varphi_0}$ and $\xi_0\coloneqq 1-e^{-\psi_0}$. Assume that the problem setting satisfies Assumptions~\ref{assm:curvature}, \ref{assm:supp_global}, and~\ref{assm:positivity}. Then there exists a constant $C_\Lambda>0$ such that:
\begin{align}
&\mathrm{UOT}(\tilde\mu_n, \tilde\nu_m) - \mathrm{UOT}(\mu, \nu) \\
&\leq \int \zeta_0\,d(\tilde\mu_n - \mu) + \int \xi_0\,d(\tilde\nu_m - \nu) + C_\Lambda \left( M_\mu W_2^2(\bar\mu_n, \bar\mu) + M_\nu W_2^2(\bar\nu_m, \bar\nu) \right).
\label{eq:uot_two_sample}
\end{align}
Moreover, the constant $C_\Lambda$ can be taken to satisfy
\begin{align}
\label{eq:C_Lambda_bound}
C_\Lambda\le e^{\mathrm{diam}(\Omega)^2/2}\Bigl(1+\frac{M_\mu}{M_\nu}+\frac{M_\nu}{M_\mu}\Bigr)\Bigl(1+\frac{\kappa^{-1}+\mathrm{diam}(\Omega)^2}{2}\Bigr).
\end{align}
\end{proposition}

\begin{proof}[Proof of Proposition \ref{prop:two_sample_stability}]
Let $\gamma$ be the optimal transport plan for $(\mu,\nu)$. By Assumption~\ref{assm:positivity}, the probability measures $\bar\mu$ and $\bar\nu$ are absolutely continuous. Hence, Brenier's theorem yields optimal transport maps $T_\mu,T_\nu:\Omega\to\Omega$ with $(T_\mu)_\#\bar\mu=\bar\mu_n$, $(T_\nu)_\#\bar\nu=\bar\nu_m$, and
\begin{align*}
    \int\|T_\mu(x)-x\|^2\,d\bar\mu(x)=W_2^2(\bar\mu_n,\bar\mu),
    \qquad
    \int\|T_\nu(y)-y\|^2\,d\bar\nu(y)=W_2^2(\bar\nu_m,\bar\nu).
\end{align*}
Since $\mu=M_\mu\bar\mu$ and $\tilde\mu_n=M_\mu\bar\mu_n$, we obtain $(T_\mu)_\#\mu=\tilde\mu_n$ and
\begin{align*}
    \int\|T_\mu(x)-x\|^2\,d\mu(x)=M_\mu W_2^2(\bar\mu_n,\bar\mu).
\end{align*}
The analogous identities hold for $T_\nu$, $\nu$, and $\tilde\nu_m$. Define the candidate transport plan $\gamma' = (T_\mu, T_\nu)_\# \gamma$. {By the data-processing inequality, $D_\KL$ contracts under deterministic maps:}
\[ {D_\KL}(\gamma_0' \mid \tilde\mu_n) = {D_\KL}((T_\mu)_\# \gamma_0 \mid (T_\mu)_\# \mu) \leq {D_\KL}(\gamma_0 \mid \mu). \]
Similarly, ${D_\KL}(\gamma_1' \mid \tilde\nu_m) \leq {D_\KL}(\gamma_1 \mid \nu)$. Thus, the UOT objective for the candidate $\gamma'$ is bounded:
\[ \mathrm{UOT}(\tilde\mu_n, \tilde\nu_m) \leq \int c\,d\gamma' + {D_\KL}(\gamma_0' \mid \tilde\mu_n) + {D_\KL}(\gamma_1' \mid \tilde\nu_m) \leq \int c\,d\gamma' + {D_\KL}(\gamma_0 \mid \mu) + {D_\KL}(\gamma_1 \mid \nu). \]
Subtracting $\mathrm{UOT}(\mu,\nu) = \int c\,d\gamma + {D_\KL}(\gamma_0 \mid \mu) + {D_\KL}(\gamma_1 \mid \nu)$ from this inequality, we get:
\[ \mathrm{UOT}(\tilde\mu_n, \tilde\nu_m) - \mathrm{UOT}(\mu,\nu) \leq \int c\,d\gamma' - \int c\,d\gamma. \]
By Lemma~\ref{lem:uot_cost_expansion}, we have
\begin{equation}
\label{eq:cost_diff_exact_0}
\begin{aligned}
\int c\,d\gamma' - \int c\,d\gamma 
&\le \int \|T_\mu(x)-x\|^2\,d\gamma_0(x) + \int \|T_\nu(y)-y\|^2\,d\gamma_1(y) \\
&\quad + \int \langle T_\mu(x)-x, x-y \rangle\,d\gamma(x,y) \\
&\quad + \int \langle T_\nu(y)-y, y-x \rangle\,d\gamma(x,y).
\end{aligned}
\end{equation}
We bound the two terms in \eqref{eq:cost_diff_exact_0} that involve $T_\mu(x)$. By Lemma~\ref{lem:dual_potential_bounds}, $\varphi_0$ and $\psi_0$ are bounded, so
\begin{align}
\label{eq:M_phi_psi}
\begin{aligned}
    M_\varphi &\coloneqq \sup_{x\in\Omega} e^{-\varphi_0(x)} \le e^{\mathrm{diam}(\Omega)^2/2}(1+M_\nu/M_\mu),\\
    M_\psi &\coloneqq \sup_{y\in\Omega} e^{-\psi_0(y)} \le e^{\mathrm{diam}(\Omega)^2/2}(1+M_\mu/M_\nu).
\end{aligned}
\end{align}
By the primal-dual optimality conditions \cite{LieroMielkeSavare2018}, the marginals of $\gamma$ are $\gamma_0=e^{-\varphi_0}\mu$ and $\gamma_1=e^{-\psi_0}\nu$. Consequently,
\begin{align}
    \int \|T_\mu(x)-x\|^2\,d\gamma_0(x)
    &=
    \int \|T_\mu(x)-x\|^2 e^{-\varphi_0(x)}\,d\mu(x)
    \\
    &\le
    M_\varphi \int \|T_\mu(x)-x\|^2\,d\mu(x)
    =
    M_\varphi M_\mu W_2^2(\bar\mu_n, \bar\mu).
\label{eq:Tmu_term_quadratic}
\end{align}
For $\int \langle T_\mu(x)-x, x-y \rangle\,d\gamma(x,y)$, we use the first-order optimality relation
\begin{align*}
    x-y = \nabla \varphi_0(x)
    \qquad\text{for $\gamma$-a.e.\ }(x,y),
\end{align*}
which follows from $c(x,y)=\tfrac{1}{2}\|x-y\|^2$ and the complementary slackness for the UOT dual problem: for $\gamma$-a.e.\ $(x,y)$, $\varphi_0(x)+\psi_0(y)=c(x,y)$, while $\varphi_0(x')+\psi_0(y)\le c(x',y)$ for all $x'\in\Omega$. Differentiating in $x$ at points where $\varphi_0$ is differentiable (which holds $\mu$-a.e., hence $\gamma_0$-a.e.\ since $\gamma_0\ll\mu$, by Theorem~\ref{prop:monge-potential} under Assumption~\ref{assm:curvature}) yields $\nabla c(x,y)=\nabla\varphi_0(x)$, i.e., $x-y=\nabla\varphi_0(x)$. Therefore,
\begin{align*}
    \int \langle T_\mu(x)-x, x-y \rangle\,d\gamma(x,y)
    &=
    \int \langle T_\mu(x)-x, \nabla \varphi_0(x) \rangle\,d\gamma_0(x)
    \\
    &=
    \int \langle T_\mu(x)-x, \nabla \varphi_0(x) \rangle e^{-\varphi_0(x)}\,d\mu(x)
    \\
    &=
    \int \langle T_\mu(x)-x, \nabla \zeta_0(x) \rangle\,d\mu(x),
\end{align*}
because $\nabla \zeta_0(x)=e^{-\varphi_0(x)}\nabla \varphi_0(x)$. Since both $\nabla \zeta_0$ and $\nabla \xi_0$ are Lipschitz on $\Omega$ by Lemma~\ref{lem:smoothness_effective}, by taking $\Lambda \coloneqq \max\bigl(\mathrm{Lip}(\nabla\zeta_0),\mathrm{Lip}(\nabla\xi_0)\bigr)>0$, we have the following inequality for all $x,y\in\Omega$:
\begin{align*}
    \bigl|\zeta_0(y)-\zeta_0(x)-\langle \nabla\zeta_0(x),y-x\rangle\bigr|
    \le
    \frac{\Lambda}{2}\|y-x\|^2.
\end{align*}
Applying this with $y=T_\mu(x)$ yields
\begin{align*}
    \langle \nabla \zeta_0(x), T_\mu(x)-x \rangle
    \le
    \zeta_0(T_\mu(x))-\zeta_0(x)+\frac{\Lambda}{2}\|T_\mu(x)-x\|^2.
\end{align*}
Integrating with respect to $\mu$ and using $(T_\mu)_\#\mu=\tilde\mu_n$, we obtain
\begin{align}
    \int \langle T_\mu(x)-x, \nabla \zeta_0(x) \rangle\,d\mu(x)
    &\le
    \int (\zeta_0(T_\mu(x))-\zeta_0(x))\,d\mu(x)
    + \frac{\Lambda}{2}\int \|T_\mu(x)-x\|^2\,d\mu(x)
    \\
    &=
    \int \zeta_0\,d(\tilde\mu_n-\mu)
    + \frac{\Lambda}{2}M_\mu W_2^2(\bar\mu_n,\bar\mu).
\label{eq:Tmu_term_linear}
\end{align}
The terms in \eqref{eq:cost_diff_exact_0} that involve $T_\nu(y)$ are handled in the same way, using $y-x=\nabla\psi_0(y)$ for $\gamma$-a.e.\ $(x,y)$ and noting that $\nabla\xi_0$ is also $\Lambda$-Lipschitz by our choice of $\Lambda$. Combining \eqref{eq:Tmu_term_quadratic} and \eqref{eq:Tmu_term_linear} with their analogues for $T_\nu(y)$, we may take
\begin{align}
\label{eq:C_Lambda_def}
    C_\Lambda \coloneqq \max(M_\varphi,M_\psi)+\frac{\Lambda}{2}.
\end{align}
The desired inequality~\eqref{eq:uot_two_sample} then follows from this choice of $C_\Lambda$.

It remains to prove \eqref{eq:C_Lambda_bound}. To this end, we first bound $\Lambda$. Differentiating $\nabla\zeta_0=e^{-\varphi_0}\nabla\varphi_0$ yields $\nabla^2\zeta_0=e^{-\varphi_0}\bigl(\nabla^2\varphi_0-\nabla\varphi_0\nabla\varphi_0^\top\bigr)$, where $\|\nabla^2\varphi_0\|_{\mathrm{op}}\le\kappa^{-1}$ by Assumption~\ref{assm:curvature} and $\|\nabla\varphi_0(x)\|=\|x-T_0(x)\|\le\mathrm{diam}(\Omega)$. Moreover, \eqref{eq:psi0_hessian_bounds} yields $\|\nabla^2\psi_0\|_{\mathrm{op}}\le\kappa^{-1}$, and $\|\nabla\psi_0(y)\|=\|y-T_0^{-1}(y)\|\le\mathrm{diam}(\Omega)$. Hence,
\begin{align*}
    \Lambda\le\max(M_\varphi,M_\psi)\bigl(\kappa^{-1}+\mathrm{diam}(\Omega)^2\bigr).
\end{align*}
Substituting this bound into \eqref{eq:C_Lambda_def} yields $C_\Lambda\le\max(M_\varphi,M_\psi)\bigl(1+(\kappa^{-1}+\mathrm{diam}(\Omega)^2)/2\bigr)$. Moreover, \eqref{eq:M_phi_psi} yields $\max(M_\varphi,M_\psi)\le e^{\mathrm{diam}(\Omega)^2/2}(1+M_\mu/M_\nu+M_\nu/M_\mu)$, which proves \eqref{eq:C_Lambda_bound}.
\end{proof}

\subsection{Extension to estimated masses}

We now consider the empirical measures $\hat\mu_n$ and $\hat\nu_m$, whose total masses are $\hat M_\mu$ and $\hat M_\nu$. We have $\hat\mu_n = \rho_{\mu} \tilde\mu_n$ and $ \hat\nu_m = \rho_{\nu} \tilde\nu_m,$ where $    \rho_{\mu} \coloneqq {\hat{M}_\mu}/  {M_\mu}$ and $ \rho_{\nu} \coloneqq {\hat{M}_\nu} / {M_\nu}$. By Assumption~\ref{assm:masses}, $\rho_{\mu},\rho_{\nu}\ge\bar c$ almost surely, and $\Ep[|\rho_{\mu}-1|]\le a_n/M_\mu$ and $\Ep[|\rho_{\nu}-1|]\le b_m/M_\nu$ for all sufficiently large $n,m$.

Thus, the following two lemmas yield, almost surely,
\begin{align}
\label{eq:empirical_mass_error_term}
\begin{aligned}
    &\mathrm{UOT}(\hat\mu_n,\hat\nu_m) - \mathrm{UOT}(\tilde\mu_n,\tilde\nu_m) \\
    &\le M_\mu|\rho_{\mu}-1| + M_\nu|\rho_{\nu}-1| + e\sqrt{M_\mu M_\nu}\bigl(|\log \rho_{\mu}| + |\log \rho_{\nu}|\bigr).
\end{aligned}
\end{align}

\begin{lemma}[$D_\KL$ under measure scaling]
\label{lem:kl_scaling_reference}
Let $\eta,\sigma \in \mathcal{M}_+(\Omega)$ with $\eta \ll \sigma$, and let $t>0$. Then
\begin{align}
    {D_\KL}(\eta\mid t\sigma) = {D_\KL}(\eta\mid \sigma) - \eta(\Omega)\log t + (t-1)\sigma(\Omega).
\end{align}

\end{lemma}

\begin{proof}[Proof of Lemma \ref{lem:kl_scaling_reference}]
Write $\eta = f\sigma$. Then $\eta = (f/t)(t\sigma)$, and therefore
\begin{align*}
    {D_\KL}(\eta\mid t\sigma)
    &= \int \left( \frac{f}{t}\log\frac{f}{t} - \frac{f}{t} + 1 \right)\,d(t\sigma) \\
    &= \int \bigl( f\log f - f + 1 \bigr)\,d\sigma - (\log t)\int f\,d\sigma + (t-1)\sigma(\Omega) \\
    &= {D_\KL}(\eta\mid \sigma) - \eta(\Omega)\log t + (t-1)\sigma(\Omega).
\end{align*}
\end{proof}

\begin{lemma}[Bound on $\gamma(\Omega\times\Omega)$]
\label{lem:uot_plan_mass_bound}
Let $\eta,\sigma \in \mathcal{M}_+(\Omega)$ with $\eta(\Omega)=A>0$ and $\sigma(\Omega)=B>0$, and let $\gamma$ be any optimal transport plan for $\mathrm{UOT}(\eta,\sigma)$ with quadratic cost. Then
\begin{align}
    \gamma(\Omega\times\Omega) \le e\sqrt{AB}.
\end{align}
\end{lemma}

\begin{proof}[Proof of Lemma \ref{lem:uot_plan_mass_bound}]
Let $s \coloneqq \gamma(\Omega\times\Omega) = \gamma_0(\Omega)=\gamma_1(\Omega)$. Since the zero transport plan is feasible, we have
\begin{align}
    \mathrm{UOT}(\eta,\sigma) \le \eta(\Omega) + \sigma(\Omega) = A+B.
\end{align}
Because $\int c\,d\gamma\ge0$, it follows that
\begin{align}
    {D_\KL}(\gamma_0\mid \eta) + {D_\KL}(\gamma_1\mid \sigma) \le A+B.
    \label{eq:kl_upper_sum}
\end{align}
Now let $f = d\gamma_0/d\eta$. Since $t\mapsto t\log t - t + 1$ is convex, Jensen's inequality yields
\begin{align}
    {D_\KL}(\gamma_0\mid \eta)
    &= \int \bigl( f\log f - f + 1 \bigr)\,d\eta\notag
\\
    &\ge A\left( \frac{s}{A}\log\frac{s}{A} - \frac{s}{A} + 1 \right)
     = s\log\frac{s}{A} - s + A.\label{eq:kl_lower_source}
\end{align}
Similarly,
\begin{align}
    {D_\KL}(\gamma_1\mid \sigma) \ge s\log\frac{s}{B} - s + B.
    \label{eq:kl_lower_target}
\end{align}
Summing \eqref{eq:kl_lower_source} and \eqref{eq:kl_lower_target} and using \eqref{eq:kl_upper_sum} yields
\begin{align}
    s\log\frac{s^2}{AB} - 2s + A+B \le A+B,
\end{align}
that is,
\begin{align}
    s\left( \log\frac{s^2}{AB} - 2 \right) \le 0.
\end{align}
If $s=0$ there is nothing to prove. Otherwise, $\log(s^2/(AB))\le 2$, hence $s \le e\sqrt{AB}$.
\end{proof}

\begin{proposition}[Stability of $\mathrm{UOT}$ under estimated masses]
\label{prop:estimated_mass_stability}
Assume that Assumptions~\ref{assm:curvature}, \ref{assm:supp_global}, \ref{assm:positivity}, and \ref{assm:masses} hold. Let $C_\Lambda$ be the stability constant from Proposition~\ref{prop:two_sample_stability}. Then, almost surely,
\begin{equation}
\label{eq:estimated_mass_stability}
\begin{aligned}
\mathrm{UOT}(\hat\mu_n,\hat\nu_m) - \mathrm{UOT}(\mu,\nu)
&\le \int \zeta_0\,d(\tilde\mu_n-\mu) + \int \xi_0\,d(\tilde\nu_m-\nu) \\
&\quad + C_\Lambda\Bigl( M_\mu W_2^2(\bar\mu_n,\bar\mu) + M_\nu W_2^2(\bar\nu_m,\bar\nu) \Bigr) \\
&\quad + \Bigl( M_\mu|\rho_{\mu}-1| + M_\nu|\rho_{\nu}-1| + e\sqrt{M_\mu M_\nu}\bigl(|\log \rho_{\mu}| + |\log \rho_{\nu}|\bigr) \Bigr).
\end{aligned}
\end{equation}
Consequently, there exists a constant $C_{M_\mu,M_\nu}>0$ such that, for all sufficiently large $n,m$, the following bound holds:
\begin{equation}
\label{eq:estimated_mass_stability_linearized}
\begin{aligned}
\Ep\bigl[\mathrm{UOT}(\hat\mu_n,\hat\nu_m)\bigr] - \mathrm{UOT}(\mu,\nu)
&\le C_\Lambda\Bigl( \Ep[M_\mu W_2^2(\bar\mu_n,\bar\mu)] + \Ep[M_\nu W_2^2(\bar\nu_m,\bar\nu)] \Bigr) \\
&\quad + C_{M_\mu,M_\nu}\bigl( a_n + b_m \bigr),
\end{aligned}
\end{equation}
where $a_n,b_m$ are the sequences from Assumption~\ref{assm:masses}.
\end{proposition}

\begin{proof}[Proof of Proposition \ref{prop:estimated_mass_stability}]
Let $\tilde\gamma_{nm}$ be an optimal transport plan for $\mathrm{UOT}(\tilde\mu_n,\tilde\nu_m)$, and define $s_{nm} \coloneqq \tilde\gamma_{nm}(\Omega\times\Omega)$. Since we have $\hat\mu_n = \rho_{\mu}\tilde\mu_n$ and $\hat\nu_m = \rho_{\nu}\tilde\nu_m$, the same transport plan $\tilde\gamma_{nm}$ is admissible for $\mathrm{UOT}(\hat\mu_n,\hat\nu_m)$. Hence, we obtain
\begin{align*}
    \mathrm{UOT}(\hat\mu_n,\hat\nu_m)
    &\le \int c\,d\tilde\gamma_{nm}
    + {D_\KL}((\tilde\gamma_{nm})_0\mid \hat\mu_n)
    + {D_\KL}((\tilde\gamma_{nm})_1\mid \hat\nu_m).
\end{align*}
Applying Lemma~\ref{lem:kl_scaling_reference} with $\eta=(\tilde\gamma_{nm})_0$ and $\sigma=\tilde\mu_n$, and then with $\eta=(\tilde\gamma_{nm})_1$ and $\sigma=\tilde\nu_m$, yields
\begin{align*}
    {D_\KL}((\tilde\gamma_{nm})_0\mid \hat\mu_n)
    &= {D_\KL}((\tilde\gamma_{nm})_0\mid \tilde\mu_n) - s_{nm}\log \rho_{\mu} + (\rho_{\mu}-1)M_\mu, \\
    {D_\KL}((\tilde\gamma_{nm})_1\mid \hat\nu_m)
    &= {D_\KL}((\tilde\gamma_{nm})_1\mid \tilde\nu_m) - s_{nm}\log \rho_{\nu} + (\rho_{\nu}-1)M_\nu.
\end{align*}
Therefore, we obtain
\begin{align}
    \mathrm{UOT}(\hat\mu_n,\hat\nu_m) - \mathrm{UOT}(\tilde\mu_n,\tilde\nu_m)
    \le -s_{nm}(\log \rho_{\mu} + \log \rho_{\nu}) + (\rho_{\mu}-1)M_\mu + (\rho_{\nu}-1)M_\nu.
\end{align}
By Lemma~\ref{lem:uot_plan_mass_bound}, we also obtain
\begin{align}
    s_{nm} \le e\sqrt{M_\mu M_\nu}.
\end{align}
We then obtain \eqref{eq:empirical_mass_error_term}.
Next, Proposition~\ref{prop:two_sample_stability} applies to $(\tilde\mu_n,\tilde\nu_m)$ because $\tilde\mu_n(\Omega)=M_\mu$ and $\tilde\nu_m(\Omega)=M_\nu$. Hence
\begin{align}
\label{eq:balanced_part_estimated_mass}
\begin{aligned}
    &\mathrm{UOT}(\tilde\mu_n,\tilde\nu_m) - \mathrm{UOT}(\mu,\nu) \\
    &\le \int \zeta_0\,d(\tilde\mu_n-\mu) + \int \xi_0\,d(\tilde\nu_m-\nu)  + C_\Lambda\Bigl( M_\mu W_2^2(\bar\mu_n,\bar\mu) + M_\nu W_2^2(\bar\nu_m,\bar\nu) \Bigr).
\end{aligned}
\end{align}
Adding \eqref{eq:empirical_mass_error_term} and \eqref{eq:balanced_part_estimated_mass} proves  \eqref{eq:estimated_mass_stability}.

We now derive \eqref{eq:estimated_mass_stability_linearized}. Taking expectations in \eqref{eq:estimated_mass_stability}, the linear terms $\Ep[\int\zeta_0 \,d(\tilde\mu_n-\mu)]$ and $\Ep[\int\xi_0 \,d(\tilde\nu_m-\nu)]$ vanish by unbiasedness of $\tilde\mu_n=M_\mu\bar\mu_n$ and $\tilde\nu_m=M_\nu\bar\nu_m$. By the a.s.\ lower bound $\hat M_\mu\ge \bar cM_\mu$ in Assumption~\ref{assm:masses}, $\rho_{\mu}\ge \bar c$ a.s., so the mean-value theorem applied to $\log$ on $[\bar c,\infty)$ yields \( |\log\rho_{\mu}| \le |\rho_{\mu}-1|/\bar c \) a.s., and analogously $|\log\rho_{\nu}|\le|\rho_{\nu}-1|/\bar c$ a.s. Hence, using the identities $M_\mu|\rho_{\mu}-1|=|\hat M_\mu-M_\mu|$ and $M_\nu|\rho_{\nu}-1|=|\hat M_\nu-M_\nu|$,
\begin{align}
    \Ep\!\bigl[ M_\mu|\rho_{\mu}-1| + e\sqrt{M_\mu M_\nu}\,|\log\rho_{\mu}| \bigr]
    &\le \Ep[|\hat M_\mu-M_\mu|] + \frac{e\sqrt{M_\mu M_\nu}}{\bar c\,M_\mu}\,\Ep[|\hat M_\mu-M_\mu|] \\
    &\le \left(1 + \bar c^{-1}e\sqrt{M_\nu/M_\mu}\right) a_n, \label{eq:mass_error_expectation}
\end{align}
and analogously $\Ep\!\bigl[ M_\nu|\rho_{\nu}-1| + e\sqrt{M_\mu M_\nu}\,|\log\rho_{\nu}| \bigr] \le \left(1 + \bar c^{-1}e\sqrt{M_\mu/M_\nu}\right) b_m$. Setting $C_{M_\mu,M_\nu}\coloneqq 1+\bar c^{-1}e\sqrt{M_\mu/M_\nu}+\bar c^{-1}e\sqrt{M_\nu/M_\mu}$ and bounding $1+\bar c^{-1}e\sqrt{M_\nu/M_\mu}$ and $1+\bar c^{-1}e\sqrt{M_\mu/M_\nu}$ by $C_{M_\mu,M_\nu}$ yields \eqref{eq:estimated_mass_stability_linearized}.
\end{proof}

\begin{remark}[Stability bound for general measures]\label{rem:stability_general_measures}
Assume Assumptions~\ref{assm:curvature}, \ref{assm:supp_global}, and \ref{assm:positivity}. Let $\tilde\mu,\tilde\nu\in\mM_+(\Omega)$ with $\tilde\mu(\Omega)=M_\mu$ and $\tilde\nu(\Omega)=M_\nu$, and let $\rho_{\mu},\rho_{\nu}>0$. Then, the proof of \eqref{eq:estimated_mass_stability} yields
\begin{align*}
    &\mathrm{UOT}(\rho_{\mu}\tilde\mu,\rho_{\nu}\tilde\nu) - \mathrm{UOT}(\mu,\nu) \\
    &\le \int \zeta_0 \,d(\tilde\mu-\mu) + \int \xi_0 \,d(\tilde\nu-\nu) \\
    &\quad + C_\Lambda\Bigl( M_\mu W_2^2(\tilde\mu/M_\mu,\bar\mu) + M_\nu W_2^2(\tilde\nu/M_\nu,\bar\nu) \Bigr) \\
    &\quad + M_\mu|\rho_{\mu}-1| + M_\nu|\rho_{\nu}-1| + e\sqrt{M_\mu M_\nu}\bigl(|\log \rho_{\mu}| + |\log \rho_{\nu}|\bigr).
\end{align*}
\end{remark}

\subsection{Supportive results}

The proof of the stability bound~\eqref{eq:uot_two_sample} uses the following two lemmas:
\begin{lemma}
\label{lem:smoothness_effective}
Assume Assumptions~\ref{assm:curvature}, \ref{assm:supp_global}, and \ref{assm:positivity}, and define
\begin{align*}
    \zeta_0(x) \coloneqq 1-e^{-\varphi_0(x)},
    \qquad
    \xi_0(y) \coloneqq 1-e^{-\psi_0(y)}.
\end{align*}
Then $\nabla \zeta_0$ and $\nabla \xi_0$ are Lipschitz on $\Omega$.
\end{lemma}

\begin{proof}[Proof of Lemma \ref{lem:smoothness_effective}]
We first prove this for $\zeta_0$. The gradient is $\nabla \zeta_0(x) = e^{-\varphi_0(x)} \nabla \varphi_0(x)$. By the product rule,
\begin{align*}
    \nabla^2 \zeta_0(x) = - e^{-\varphi_0(x)} (\nabla \varphi_0(x))(\nabla \varphi_0(x))^\top + e^{-\varphi_0(x)} \nabla^2 \varphi_0(x).
\end{align*}
By Assumptions~\ref{assm:curvature} and~\ref{assm:supp_global}, $\varphi_0$ and $\nabla \varphi_0(x)$ are bounded on the compact set $\Omega$, and $\|\nabla^2 \varphi_0(x)\|_{\mathrm{op}} \le \kappa^{-1}$. Thus $\nabla^2\zeta_0$ is bounded on $\Omega$, and $\nabla \zeta_0$ is Lipschitz on the convex set $\Omega$.

For $\xi_0$, Lemma~\ref{lem:uot_T0_bijection} yields $(1-\kappa^{-1})I \preceq \nabla^2 \psi_0 \preceq (1-\kappa)I$ and the boundedness of $\nabla\psi_0(y)=y-T_0^{-1}(y)$ on $\Omega$. The rest of the proof is identical to that for $\zeta_0$.
\end{proof}

\begin{lemma}
\label{lem:uot_cost_expansion}
Let $\gamma \in \mathcal{M}_+(\Omega \times \Omega)$ have marginals $\gamma_0$ and $\gamma_1$, and let $T_\mu, T_\nu: \Omega \to \Omega$ be measurable maps. Define the candidate transport plan $\gamma' = (T_\mu, T_\nu)_\# \gamma$. Then
\begin{equation}
\label{eq:cost_diff_exact}
\begin{aligned}
\int c\,d\gamma' - \int c\,d\gamma 
&\le \int \|T_\mu(x)-x\|^2\,d\gamma_0(x) + \int \|T_\nu(y)-y\|^2\,d\gamma_1(y) \\
&\quad + \int \langle T_\mu(x)-x, x-y \rangle\,d\gamma(x,y) \\
&\quad + \int \langle T_\nu(y)-y, y-x \rangle\,d\gamma(x,y).
\end{aligned}
\end{equation}
\end{lemma}

\begin{proof}[Proof of Lemma \ref{lem:uot_cost_expansion}]
By definition of the pushforward,
\begin{align*}
    \int c(x,y)\,d\gamma'(x,y)
    =
    \frac{1}{2}\int \|T_\mu(x) - T_\nu(y)\|^2\,d\gamma(x,y).
\end{align*}
For all $x,y\in\Omega$, expanding the square yields
\begin{align*}
\tfrac{1}{2}\|T_\mu(x) - T_\nu(y)\|^2
&=
\tfrac{1}{2}\|(T_\mu(x)-x) - (T_\nu(y)-y) + (x-y)\|^2 \\
&=
\tfrac{1}{2}\|x-y\|^2 + \tfrac{1}{2}\|T_\mu(x)-x\|^2 + \tfrac{1}{2}\|T_\nu(y)-y\|^2 \\
&\quad + \langle x-y, T_\mu(x)-x \rangle - \langle x-y, T_\nu(y)-y \rangle \\
&\quad - \langle T_\mu(x)-x, T_\nu(y)-y \rangle \\
&\le
\tfrac{1}{2}\|x-y\|^2 + \|T_\mu(x)-x\|^2 + \|T_\nu(y)-y\|^2 \\
&\quad + \langle T_\mu(x)-x, x-y \rangle + \langle T_\nu(y)-y, y-x \rangle.
\end{align*}
Integrating this inequality with respect to $d\gamma(x,y)$ and using the identities
\begin{align*}
    \int \|T_\mu(x)-x\|^2\,d\gamma(x,y)
    &=
    \int \|T_\mu(x)-x\|^2\,d\gamma_0(x),\\
    \int \|T_\nu(y)-y\|^2\,d\gamma(x,y)
    &=
    \int \|T_\nu(y)-y\|^2\,d\gamma_1(y),
\end{align*}
yields \eqref{eq:cost_diff_exact}.
\end{proof}

\section{Proof of Theorem \ref{thm:main_plan_based_rates}} \label{sec:proof_plan_based_rates}

We first decompose $\mathrm{UOT}(\tilde\mu_n,\tilde\nu_m)$, in the spirit of \cite[Proposition~13]{manole2024plugin} for the balanced case. In the unbalanced case, the decomposition contains more than the deviations of $Y_j$ from $T_0(X_i)$. The decomposition also measures how the row and column masses of an optimal transport plan for $\mathrm{UOT}(\tilde\mu_n,\tilde\nu_m)$ differ from the \emph{active} marginals $\gamma_0$ and $\gamma_1$ transported by the transport-growth pair $(T_0,\lambda_0)$. In Proposition~\ref{prop:uot_exact_empirical_excess}, we derive the decomposition for the empirical measures $\tilde\mu_n,\tilde\nu_m$, whose total masses are $M_\mu,M_\nu$. We then derive in Proposition~\ref{prop:uot_barycentric_and_1nn_estimated_mass} this decomposition for the measures $\hat\mu_n,\hat\nu_m$, whose total masses are $\hat M_\mu,\hat M_\nu$.

Let $\tilde\gamma=(\tilde\gamma_{ij})$ be any optimizer of $\mathrm{UOT}(\tilde\mu_n,\tilde\nu_m)$, and write
\begin{align}
    \tilde r_i \coloneqq \sum_{j=1}^m \tilde\gamma_{ij},
    \qquad
    \tilde s_j \coloneqq \sum_{i=1}^n \tilde\gamma_{ij}.
\end{align}
We also denote the atomic masses of $\tilde\mu_n$ and $\tilde\nu_m$ by $\tilde\mu_i \coloneqq {M_\mu} / {n}$ and $ \tilde\nu_j \coloneqq {M_\nu} / {m}$. Finally, recall that for the optimal dual pair $(\varphi_0,\psi_0)$ of $(\mu,\nu)$ we have $\gamma_0 = e^{-\varphi_0}\mu$ and $ \gamma_1 = e^{-\psi_0}\nu$, and $\zeta_0(x)=1-e^{-\varphi_0(x)}$, $\xi_0(y)=1-e^{-\psi_0(y)}$.

\begin{proposition}[Empirical UOT excess identity]
\label{prop:uot_exact_empirical_excess}
Define $r_i^\star \coloneqq e^{-\varphi_0(X_i)}\tilde\mu_i$ and $s_j^\star \coloneqq e^{-\psi_0(Y_j)}\tilde\nu_j$. Then, we have
\begin{equation}
\label{eq:uot_exact_empirical_excess}
\begin{aligned}
&\mathrm{UOT}(\tilde\mu_n,\tilde\nu_m)
- \int \zeta_0\,d\tilde\mu_n
- \int \xi_0\,d\tilde\nu_m \\
&\qquad=
\sum_{i=1}^n\sum_{j=1}^m
\tilde\gamma_{ij}
\Bigl(
c(X_i,Y_j)-\varphi_0(X_i)-\psi_0(Y_j)
\Bigr) 
+ {D_\KL}(\tilde r \mid r^\star)
+ {D_\KL}(\tilde s \mid s^\star).
\end{aligned}
\end{equation}
In particular, the right-hand side is nonnegative.
\end{proposition}

\begin{proof}[Proof of Proposition \ref{prop:uot_exact_empirical_excess}]
Since $\tilde\gamma$ is optimal for $\mathrm{UOT}(\tilde\mu_n,\tilde\nu_m)$, we have
\begin{align}
\label{eq:uot_empirical_objective_decomp_start}
\mathrm{UOT}(\tilde\mu_n,\tilde\nu_m)
=
\sum_{i=1}^n\sum_{j=1}^m c(X_i,Y_j)\tilde\gamma_{ij}
+ {D_\KL}(\tilde r \mid \tilde\mu_n)
+ {D_\KL}(\tilde s \mid \tilde\nu_m).
\end{align}
By the definition of $r_i^\star = e^{-\varphi_0(X_i)}\tilde\mu_i$, we obtain
\begin{align}
\label{eq:source_kl_rewrite}
\begin{aligned}
{D_\KL}(\tilde r \mid \tilde\mu_n)
&=
\sum_{i=1}^n
\left[
\tilde r_i \log \left(\frac{\tilde r_i}{r_i^\star}\right)
- \tilde r_i + r_i^\star
\right]
+ \sum_{i=1}^n \tilde r_i \log \left(\frac{r_i^\star}{\tilde\mu_i}\right)
+ \sum_{i=1}^n (\tilde\mu_i-r_i^\star) \\
&=
{D_\KL}(\tilde r \mid r^\star)
- \sum_{i=1}^n \tilde r_i   \varphi_0(X_i)
+ \sum_{i=1}^n \zeta_0(X_i)\tilde\mu_i,
\end{aligned}
\end{align}
where the second equality follows from
\begin{align}
    \log \left(\frac{r_i^\star}{\tilde\mu_i}\right)
    = -\varphi_0(X_i),
    \qquad
    \left(1-e^{-\varphi_0(X_i)}\right)
    = \zeta_0(X_i).
\end{align}
Similarly, we obtain
\begin{equation}
\label{eq:target_kl_rewrite}
{D_\KL}(\tilde s \mid \tilde\nu_m)
=
{D_\KL}(\tilde s \mid s^\star)
- \sum_{j=1}^m \tilde s_j   \psi_0(Y_j)
+ \sum_{j=1}^m \xi_0(Y_j)\tilde\nu_j .
\end{equation}
Substituting \eqref{eq:source_kl_rewrite} and \eqref{eq:target_kl_rewrite} into \eqref{eq:uot_empirical_objective_decomp_start}, and using the relations
\begin{align}
    \sum_{i=1}^n \tilde r_i \varphi_0(X_i)
    =
    \sum_{i=1}^n\sum_{j=1}^m \tilde\gamma_{ij}\varphi_0(X_i),
    \qquad
    \sum_{j=1}^m \tilde s_j \psi_0(Y_j)
    =
    \sum_{i=1}^n\sum_{j=1}^m \tilde\gamma_{ij}\psi_0(Y_j),
\end{align}
we obtain \eqref{eq:uot_exact_empirical_excess}. Nonnegativity follows because the dual constraint yields
\begin{align}
    c(X_i,Y_j)-\varphi_0(X_i)-\psi_0(Y_j)\ge 0
\end{align}
for every $i,j$, and {${D_\KL}(\tilde r\mid r^\star)$ and ${D_\KL}(\tilde s\mid s^\star)$ are nonnegative.}
\end{proof}

\subsection{From the discrete optimal transport plan to $\hat\Delta_{nm}^{\mathrm{tr}}$ and $\hat\Delta_{nm}^{\mathrm{bar}}$}

We now pass from the discrete optimal transport plan $\hat\gamma$ for $\mathrm{UOT}(\hat\mu_n,\hat\nu_m)$ to the barycentric projections $\hat T_i$ and the nearest-neighbor estimator $\hat T^{\mathrm{1NN}}$.

\begin{proposition}[Bounds on $\hat\Delta_{nm}^{\mathrm{tr}}$ and $\hat\Delta_{nm}^{\mathrm{bar}}$]
\label{prop:uot_barycentric_and_1nn_estimated_mass}
Assume that Assumptions~\ref{assm:curvature}, \ref{assm:supp_global}, \ref{assm:positivity}, and \ref{assm:masses} hold. Let $C_\Lambda$ be the stability constant from Proposition~\ref{prop:two_sample_stability}. Define
\begin{align}
    \hat r_i^\star \coloneqq e^{-\varphi_0(X_i)}\hat\mu_i,
    \qquad
    \hat s_j^\star \coloneqq e^{-\psi_0(Y_j)}\hat\nu_j,
\end{align}
and
\begin{align}
    \hat\Delta_{nm}^{\mathrm{tr}}
    &\coloneqq
    \sum_{i=1}^n\sum_{j=1}^m
    \hat\gamma_{ij}\|T_0(X_i)-Y_j\|^2, \\
    \hat\Delta_{nm}^{\mathrm{bar}}
    &\coloneqq
    \sum_{i=1}^n \hat r_i \|\hat T_i-T_0(X_i)\|^2.
\end{align}
Then, almost surely,
\begin{equation}
\label{eq:uot_hat_exact_empirical_excess}
\begin{aligned}
&\mathrm{UOT}(\hat\mu_n,\hat\nu_m)
- \int \zeta_0\,d\hat\mu_n
- \int \xi_0\,d\hat\nu_m \\
&\qquad=
\sum_{i=1}^n\sum_{j=1}^m
\hat\gamma_{ij}
\Bigl(
c(X_i,Y_j)-\varphi_0(X_i)-\psi_0(Y_j)
\Bigr)
+ {D_\KL}(\hat r \mid \hat r^\star)
+ {D_\KL}(\hat s \mid \hat s^\star).
\end{aligned}
\end{equation}
In particular,
\begin{align}
\label{eq:uot_hat_barycentric_delta_bound}
    \hat\Delta_{nm}^{\mathrm{bar}}
    \le
    \hat\Delta_{nm}^{\mathrm{tr}}
    \le
    \frac{2}{\kappa}
    \left[
        \mathrm{UOT}(\hat\mu_n,\hat\nu_m)
        - \int \zeta_0\,d\hat\mu_n
        - \int \xi_0\,d\hat\nu_m
    \right].
\end{align}
Moreover, there exists a constant $\widetilde C_{\mathrm{mass}}>0$ such that, for all sufficiently large $n,m$,
\begin{align}
\label{eq:uot_hat_barycentric_expectation_bound}
    \Ep[\hat\Delta_{nm}^{\mathrm{bar}}]
    &\le
    \frac{2C_\Lambda}{\kappa}
    \left(
        \Ep[M_\mu W_2^2(\bar\mu_n,\bar\mu)] + \Ep[M_\nu W_2^2(\bar\nu_m,\bar\nu)]
    \right)
    + \frac{2\widetilde C_{\mathrm{mass}}}{\kappa}\bigl(a_n+b_m\bigr).
\end{align}
\end{proposition}

\begin{proof}[Proof of Proposition \ref{prop:uot_barycentric_and_1nn_estimated_mass}]
The identities \eqref{eq:uot_empirical_objective_decomp_start}, \eqref{eq:source_kl_rewrite}, and \eqref{eq:target_kl_rewrite} hold for arbitrary positive masses $\tilde\mu_i$ and $\tilde\nu_j$. Therefore, replacing \( (\tilde\mu_n,\tilde\nu_m,\tilde\gamma,\tilde r,\tilde s,r^\star,s^\star) \) by \( (\hat\mu_n,\hat\nu_m,\hat\gamma,\hat r,\hat s,\hat r^\star,\hat s^\star) \) in these identities yields \eqref{eq:uot_hat_exact_empirical_excess}.

Lemma~\ref{lem:uot_gap_sufficient} and the nonnegativity of the {discrete $D_\KL$ terms} in \eqref{eq:uot_hat_exact_empirical_excess} imply
\begin{align}
    \frac{\kappa}{2}\hat\Delta_{nm}^{\mathrm{tr}}
    &\le
    \mathrm{UOT}(\hat\mu_n,\hat\nu_m)
    - \int \zeta_0\,d\hat\mu_n
    - \int \xi_0\,d\hat\nu_m.
\end{align}
Next, fix $i$ with $\hat r_i>0$. By \eqref{eq:disc-frechet} and $\sum_{j=1}^m\hat\gamma_{ij}=\hat r_i$,
\begin{align*}
    \hat T_i-T_0(X_i)
    =
    \frac{1}{\hat r_i}
    \sum_{j=1}^m \hat\gamma_{ij}\bigl(Y_j-T_0(X_i)\bigr).
\end{align*}
Since $z\mapsto \|z\|^2$ is convex, Jensen's inequality yields
\begin{align}
\label{eq:row_jensen}
    \hat r_i\|\hat T_i-T_0(X_i)\|^2 \le \sum_{j=1}^m \hat\gamma_{ij}\|Y_j-T_0(X_i)\|^2.
\end{align}
When $\hat r_i=0$, \eqref{eq:row_jensen} holds trivially. Summing over $i$ yields
\begin{align}
    \hat\Delta_{nm}^{\mathrm{bar}} \le \hat\Delta_{nm}^{\mathrm{tr}},
\end{align}
which proves \eqref{eq:uot_hat_barycentric_delta_bound}.

Next, by duality and \eqref{eq:estimated_mass_stability},
\begin{align}
\label{eq:uot_hat_excess_vs_oracle_bound}
\begin{aligned}
&\mathrm{UOT}(\hat\mu_n,\hat\nu_m)
- \int \zeta_0\,d\hat\mu_n
- \int \xi_0\,d\hat\nu_m \\
&\qquad=
\mathrm{UOT}(\hat\mu_n,\hat\nu_m)-\mathrm{UOT}(\mu,\nu)
- \int \zeta_0\,d(\hat\mu_n-\mu)
- \int \xi_0\,d(\hat\nu_m-\nu) \\
&\qquad\le
C_\Lambda\Bigl(
    M_\mu W_2^2(\bar\mu_n,\bar\mu)+M_\nu W_2^2(\bar\nu_m,\bar\nu)
\Bigr)
+ \bigl( |\hat M_\mu-M_\mu|+|\hat M_\nu-M_\nu| \bigr) \\
&\qquad\quad
+ e\sqrt{M_\mu M_\nu}\bigl(|\log\rho_{\mu}|+|\log\rho_{\nu}|\bigr) \\
&\qquad\quad
- \int \zeta_0\,d(\hat\mu_n-\tilde\mu_n)
- \int \xi_0\,d(\hat\nu_m-\tilde\nu_m),
\end{aligned}
\end{align}
where we used $\mathrm{UOT}(\mu,\nu)=\int\zeta_0\,d\mu+\int\xi_0\,d\nu$, the identity $\int\zeta_0\,d(\tilde\mu_n-\mu)-\int\zeta_0\,d(\hat\mu_n-\mu)=-\int\zeta_0\,d(\hat\mu_n-\tilde\mu_n)$, and the analogous identity for $\xi_0$. Since $\hat\mu_n-\tilde\mu_n$ and $\hat\nu_m-\tilde\nu_m$ are signed atomic measures supported on $\{X_1,\dots,X_n\}$ and $\{Y_1,\dots,Y_m\}$, respectively,
\begin{align}
    \left|
        \int \zeta_0\,d(\hat\mu_n-\tilde\mu_n)
    \right|
    \le
    \|\zeta_0\|_{L^\infty(\Omega)} |\hat M_\mu-M_\mu|,
    \qquad
    \left|
        \int \xi_0\,d(\hat\nu_m-\tilde\nu_m)
    \right|
    \le
    \|\xi_0\|_{L^\infty(\Omega)} |\hat M_\nu-M_\nu|.
\end{align}
Substituting these bounds into \eqref{eq:uot_hat_excess_vs_oracle_bound}, multiplying by $2/\kappa$ via \eqref{eq:uot_hat_barycentric_delta_bound}, and absorbing constants, we obtain constants $C_{\mathrm{mass}}>0$ and $C_{\mathrm{log}}\coloneqq e\sqrt{M_\mu M_\nu}/\kappa$, depending only on $M_\mu,M_\nu,\kappa,\|\zeta_0\|_{L^\infty(\Omega)},\|\xi_0\|_{L^\infty(\Omega)}$, such that, almost surely,
\begin{align}
\label{eq:uot_hat_barycentric_rate_bound}
\begin{aligned}
    \hat\Delta_{nm}^{\mathrm{bar}}
    &\le
    \frac{2C_\Lambda}{\kappa}
    \left(
        M_\mu W_2^2(\bar\mu_n,\bar\mu)+M_\nu W_2^2(\bar\nu_m,\bar\nu)
    \right)
    + \frac{2C_{\mathrm{mass}}}{\kappa}
    \left(
        |\hat M_\mu-M_\mu|+|\hat M_\nu-M_\nu|
    \right) \\
    &\quad+ 2C_{\mathrm{log}}\bigl(|\log\rho_{\mu}|+|\log\rho_{\nu}|\bigr).
\end{aligned}
\end{align}

It remains to take expectations of \eqref{eq:uot_hat_barycentric_rate_bound}. By Assumption~\ref{assm:masses}, $\hat M_\mu\ge \bar cM_\mu$ a.s.\ yields $|\log\rho_{\mu}|\le|\rho_{\mu}-1|/\bar c$ a.s., hence $M_\mu\Ep[|\log\rho_{\mu}|]\le \Ep[|\hat M_\mu-M_\mu|]/\bar c\le a_n/\bar c$, and analogously for the $\nu$-side, so \( e\sqrt{M_\mu M_\nu}\,\Ep[|\log\rho_{\mu}|]\le \bar c^{-1}e\sqrt{M_\nu/M_\mu}\,a_n \) and similarly $e\sqrt{M_\mu M_\nu}\Ep[|\log\rho_{\nu}|]\le \bar c^{-1}e\sqrt{M_\mu/M_\nu}\,b_m$. Setting $\widetilde C_{\mathrm{mass}}\coloneqq C_{\mathrm{mass}}+\bar c^{-1}e\sqrt{M_\mu/M_\nu}+\bar c^{-1}e\sqrt{M_\nu/M_\mu}$ proves \eqref{eq:uot_hat_barycentric_expectation_bound}.
\end{proof}

\subsection{From $\hat\Delta_{nm}^{\mathrm{bar}}$ to the transport map risk under $\gamma_0$}

To pass from $\hat\Delta_{nm}^{\mathrm{bar}}$ to $\int_\Omega\|\hat T^{\mathrm{1NN}}-T_0\|^2\,d\gamma_0$, we compare the masses $\gamma_0(V_i)$ of the Voronoi cells with $\hat r_i^\star$, keeping the term ${D_\KL}(\hat r\mid \hat r^\star)$ explicit in the argument.

\begin{lemma}[{\cite[Lemma~40]{manole2024plugin}}]
\label{lem:uot_voronoi_mass_radius}
Let $V_1,\dots,V_n$ be the Voronoi cells induced by $X_1,\dots,X_n \stackrel{\mathrm{i.i.d.}}{\sim}\bar\mu$, and define
\begin{align*}
    M_n \coloneqq \max_{1\le i\le n} \bar\mu(V_i),
    \qquad
    R_n \coloneqq \max_{1\le i\le n}\sup_{x\in V_i}\|x-X_i\|.
\end{align*}
Assume that the density $p$ of $\bar\mu$ satisfies $\beta_{\min}\le p\le\beta_{\max}$ on $\Omega$, and that $\Omega$ is compact, convex, and satisfies the interior cone condition (Assumption~\ref{assm:supp_global}). Then there exist constants $C_1,C_2>0$, depending only on $d$, $\beta_{\min}$, $\beta_{\max}$, $\epsilon_0$, and $\delta_0$, such that, for every integer $n\ge2$:
\begin{enumerate}[label=(\roman*)]
    \item for every $\delta\in(0,1)$,
    \begin{align}
\label{eq:uot_voronoi_maxmass_hp}
        \mathbb{P} \left(
            M_n
            \ge
            \frac{C_1}{n}\bigl[d\log n+\log(1/\delta)\bigr]
        \right)
        \le \delta;
    \end{align}
    \item
    \begin{align}
\label{eq:uot_voronoi_radius_moment}
        \Ep[R_n^2]
        \le
        C_2\left({(\log n)/n}\right)^{2/d}.
    \end{align}
\end{enumerate}
\end{lemma}

To compare $\hat r_i$ with $\hat r_i^\star$ through $D_\KL(\hat r\mid\hat r^\star)$, we bound each summand $\hat r_i\log(\hat r_i/\hat r_i^\star)-\hat r_i+\hat r_i^\star$ from below by $(\sqrt{\hat r_i}-\sqrt{\hat r_i^\star})^2$, using the following lemma.

\begin{lemma}
\label{lem:uot_kl_sqrt_lower_bound}
For every $a,b\ge 0$, with the conventions $0\log 0 = 0$ and $a\log(a/0)=+\infty$ for $a>0$, we have
\begin{align}
\label{eq:uot_kl_sqrt_lower_bound}
    a\log\left(\frac{a}{b}\right)-a+b
    \ge
    \left(\sqrt{a}-\sqrt{b}\right)^2.
\end{align}
\end{lemma}

\begin{proof}[Proof of Lemma \ref{lem:uot_kl_sqrt_lower_bound}]
If $b=0$, then \eqref{eq:uot_kl_sqrt_lower_bound} is immediate from the stated conventions, so we may assume that $b>0$. Writing $a=tb$, it suffices to prove that
\begin{align*}
    t\log t - t + 1 \ge (\sqrt{t}-1)^2,
    \qquad t\ge 0.
\end{align*}
Set $s=\sqrt{t}$. Then
\begin{align*}
    t\log t-t+1-(\sqrt t-1)^2
    =2s^2\log s-2s^2+2s
    =2s(s\log s-s+1)\ge0,
\end{align*}
because $u\mapsto u\log u-u+1$ is nonnegative on $[0,\infty)$.
\end{proof}

\begin{lemma}
\label{lem:uot_weighted_pinsker_rows_hat}
Assume that Assumptions~\ref{assm:curvature}, \ref{assm:supp_global}, and \ref{assm:positivity} hold and that $\hat M_\mu,\hat M_\nu>0$ almost surely. Define
\begin{align*}
    \hat b_i \coloneqq \|\hat T_i-T_0(X_i)\|^2,
    \qquad
    C_\Omega \coloneqq \mathrm{diam}(\Omega)^2.
\end{align*}
Then, almost surely,
\begin{align}
\label{eq:pinsker_rows_hat}
    \sum_{i=1}^n \hat r_i^\star \hat b_i
    \le
    3\hat\Delta_{nm}^{\mathrm{bar}}
    + 2C_\Omega {D_\KL}(\hat r\mid \hat r^\star).
\end{align}
\end{lemma}

\begin{proof}[Proof of Lemma \ref{lem:uot_weighted_pinsker_rows_hat}]

Set
\begin{align}
    B_n^\star \coloneqq \sum_{i=1}^n \hat r_i^\star \hat b_i,
    \qquad
    B_n \coloneqq \sum_{i=1}^n \hat r_i \hat b_i
    = \hat\Delta_{nm}^{\mathrm{bar}}.
\end{align}
By Lemma~\ref{lem:uot_kl_sqrt_lower_bound}, applied with $a=\hat r_i$ and $b=\hat r_i^\star$ for each $i$, we have
\begin{align}
    \sum_{i=1}^n \frac{(\hat r_i-\hat r_i^\star)^2}{\hat r_i+\hat r_i^\star}
    &=
    \sum_{i=1}^n \frac{(\sqrt{\hat r_i}+\sqrt{\hat r_i^\star})^2}{\hat r_i+\hat r_i^\star}\bigl(\sqrt{\hat r_i}-\sqrt{\hat r_i^\star}\bigr)^2
    \notag\\
    &\le
    2\sum_{i=1}^n \bigl(\sqrt{\hat r_i}-\sqrt{\hat r_i^\star}\bigr)^2
    \le
    2 {D_\KL}(\hat r\mid \hat r^\star).
    \label{eq:rows_weighted_pinsker_basic}
\end{align}
If $\hat r_i=0$, then $\hat T_i=X_i\in\Omega$, while if $\hat r_i>0$, then $\hat T_i$ is a convex combination of $Y_1,\dots,Y_m\in\Omega$ and $\Omega$ is convex. Since also $T_0(X_i)\in\Omega$, we have $0\le \hat b_i\le C_\Omega$. Hence, the Cauchy-Schwarz inequality and \eqref{eq:rows_weighted_pinsker_basic} imply
\begin{align}
\begin{aligned}
    B_n^\star-B_n
    &= \sum_{i=1}^n (\hat r_i^\star-\hat r_i)\hat b_i \\
    &\le
    \left(
        \sum_{i=1}^n \frac{(\hat r_i-\hat r_i^\star)^2}{\hat r_i+\hat r_i^\star}
    \right)^{1/2}
    \left(
        \sum_{i=1}^n (\hat r_i+\hat r_i^\star)\hat b_i^2
    \right)^{1/2} \\
    &\le
    \sqrt{2{D_\KL}(\hat r\mid \hat r^\star)}
    \left(
        C_\Omega \sum_{i=1}^n (\hat r_i+\hat r_i^\star)\hat b_i
    \right)^{1/2} \\
    &=
    \sqrt{2C_\Omega {D_\KL}(\hat r\mid \hat r^\star)}
      \sqrt{B_n^\star+B_n}.
\end{aligned}
\end{align}
Hence,
\begin{align}
    B_n^\star-B_n
    \le
    C_\Omega {D_\KL}(\hat r\mid \hat r^\star)
    + \tfrac{1}{2}(B_n^\star+B_n).
\end{align}
Rearranging yields $B_n^\star \le 3B_n + 2C_\Omega {D_\KL}(\hat r\mid \hat r^\star)$, which is \eqref{eq:pinsker_rows_hat}.

\end{proof}

\begin{theorem}[Transport map risk of $\hat T^{1\mathrm{NN}}$]
\label{thm:uot_complete_1nn_estimated_mass}
Assume that Assumptions~\ref{assm:curvature}, \ref{assm:supp_global}, and \ref{assm:positivity} hold and that $\hat M_\mu,\hat M_\nu>0$ almost surely. Then there exists a constant $C>0$, depending only on $d,\beta_{\min},\beta_{\max},\epsilon_0,\delta_0,M_\mu,M_\nu,\kappa,\Omega$ and the clipping constants $w_\pm$ of Section~\ref{sec:plan}, such that for all sufficiently large $n$,
\begin{align}
\label{eq:uot_complete_1nn_estimated_mass_expectation}
\begin{aligned}
    \Ep\Bigl[
        \int_\Omega \|\hat T^{\mathrm{1NN}}(x)-T_0(x)\|^2\,d\gamma_0(x)
    \Bigr]
    &\le
    C \log n
    \Bigl(
        \Ep[M_\mu W_2^2(\bar\mu_n,\bar\mu)]
        +
        \Ep[M_\nu W_2^2(\bar\nu_m,\bar\nu)]
    \Bigr) \\
    &\quad
    + C\left({(\log n)/n}\right)^{2/d}
    + C{(\log n)/n}.
\end{aligned}
\end{align}
\end{theorem}

\begin{proof}[Proof of Theorem \ref{thm:uot_complete_1nn_estimated_mass}]
By Assumptions~\ref{assm:curvature} and~\ref{assm:supp_global}, the map $T_0(x) = x - \nabla \varphi_0(x)$ is Lipschitz with constant $L_T \coloneqq \kappa^{-1}$. Let $0<w_-<w_+$ be the clipping constants of Section~\ref{sec:plan}, so that
\begin{align}
    w_- \le w_0(x) \le w_+,
    \qquad x\in\Omega.
\end{align}
Define the event
\begin{align}
    \mathcal{A}_{nm}
    \coloneqq
    \left\{
        |\hat M_\mu-M_\mu|\le \frac{M_\mu}{2},
        \ |\hat M_\nu-M_\nu|\le \frac{M_\nu}{2}
    \right\}.
\end{align}
We first show that, on $\mathcal{A}_{nm}$, almost surely,
\begin{align}
\label{eq:uot_complete_1nn_estimated_mass_as}
\begin{aligned}
    \int_\Omega \|\hat T^{\mathrm{1NN}}(x)-T_0(x)\|^2\,d\gamma_0(x)
    &\le
    4\frac{w_+}{w_-}  n M_n
    \left(
        3\hat\Delta_{nm}^{\mathrm{bar}}
        + 2C_\Omega {D_\KL}(\hat r\mid \hat r^\star)
    \right) \\
    &\quad + 2L_T^2\gamma_0(\Omega)R_n^2.
\end{aligned}
\end{align}
For $x\in V_i$, one has $\hat T^{\mathrm{1NN}}(x)=\hat T_i$, hence
\begin{align}
    \|\hat T^{\mathrm{1NN}}(x)-T_0(x)\|^2
    \le
    2\hat b_i + 2\|T_0(X_i)-T_0(x)\|^2
    \le
    2\hat b_i + 2L_T^2\|X_i-x\|^2.
\end{align}
Integrating over $V_i$ with respect to $\gamma_0$ and summing over $i$ yields
\begin{align}
\label{eq:uot_complete_1nn_estimated_mass_start}
    \int_\Omega \|\hat T^{\mathrm{1NN}}(x)-T_0(x)\|^2\,d\gamma_0(x)
    \le
    2\sum_{i=1}^n \gamma_0(V_i)\hat b_i
    + 2L_T^2\gamma_0(\Omega)R_n^2.
\end{align}
On $\mathcal{A}_{nm}$, we have $\hat M_\mu\ge M_\mu/2$. Moreover,
\begin{align}
    \gamma_0(V_i)
    = \int_{V_i} w_0\,d\mu
    \le
    w_+ \mu(V_i)
    = M_\mu w_+\bar\mu(V_i)
    \le
    M_\mu w_+ M_n,
\end{align}
whereas
\begin{align}
    \hat r_i^\star = w_0(X_i){\hat M_\mu/n}
    \ge
    w_-\frac{M_\mu}{2n}.
\end{align}
Hence, on $\mathcal{A}_{nm}$,
\begin{align}
    \gamma_0(V_i)
    \le
    2\frac{w_+}{w_-}  nM_n \hat r_i^\star.
\end{align}
Substituting this into \eqref{eq:uot_complete_1nn_estimated_mass_start} yields
\begin{align}
    \int_\Omega \|\hat T^{\mathrm{1NN}}(x)-T_0(x)\|^2\,d\gamma_0(x)
    \le
    4\frac{w_+}{w_-}  nM_n \sum_{i=1}^n \hat r_i^\star\hat b_i
    + 2L_T^2\gamma_0(\Omega)R_n^2.
\end{align}
Applying Lemma~\ref{lem:uot_weighted_pinsker_rows_hat} proves \eqref{eq:uot_complete_1nn_estimated_mass_as}.

For \eqref{eq:uot_complete_1nn_estimated_mass_expectation}, we first show that we may take $\hat M_\mu=M_\mu$ and $\hat M_\nu=M_\nu$. Let $G\in\R_+^{n\times m}$ have marginals $G_0\ll\tilde\mu_n$ and $G_1\ll\tilde\nu_m$. For $\eta\ll\sigma$, the definition of $D_\KL$ yields
\begin{align*}
    D_\KL(\sqrt{\rho_{\mu}\rho_{\nu}}\,\eta\mid\sqrt{\rho_{\mu}\rho_{\nu}}\,\sigma)=\sqrt{\rho_{\mu}\rho_{\nu}}\,D_\KL(\eta\mid\sigma).
\end{align*}
Hence, Lemma~\ref{lem:kl_scaling_reference} with $t=\sqrt{\rho_{\mu}/\rho_{\nu}}$ and with $t=\sqrt{\rho_{\nu}/\rho_{\mu}}$ yields
\begin{align}
    D_\KL(\sqrt{\rho_{\mu}\rho_{\nu}}\,G_0\mid\hat\mu_n)
    &=\sqrt{\rho_{\mu}\rho_{\nu}}\Bigl(D_\KL(G_0\mid\tilde\mu_n)-G_0(\Omega)\log\sqrt{\rho_{\mu}/\rho_{\nu}}
    \\
    &\qquad \qquad \qquad +\bigl(\sqrt{\rho_{\mu}/\rho_{\nu}}-1\bigr)M_\mu\Bigr),
    \\
    D_\KL(\sqrt{\rho_{\mu}\rho_{\nu}}\,G_1\mid\hat\nu_m)
    &=\sqrt{\rho_{\mu}\rho_{\nu}}\Bigl(D_\KL(G_1\mid\tilde\nu_m)-G_1(\Omega)\log\sqrt{\rho_{\nu}/\rho_{\mu}}
    \\
    &\qquad \qquad \qquad +\bigl(\sqrt{\rho_{\nu}/\rho_{\mu}}-1\bigr)M_\nu\Bigr).
\label{eq:kl_scaled_marginals}
\end{align}
Since $G_0(\Omega)=G_1(\Omega)$, the two logarithmic terms in \eqref{eq:kl_scaled_marginals} cancel. Thus, the objectives of $\mathrm{UOT}(\hat\mu_n,\hat\nu_m)$ and $\mathrm{UOT}(\tilde\mu_n,\tilde\nu_m)$ satisfy
\begin{align*}
    &\sum_{i=1}^n\sum_{j=1}^m c(X_i,Y_j)\sqrt{\rho_{\mu}\rho_{\nu}}\,G_{ij}+D_\KL(\sqrt{\rho_{\mu}\rho_{\nu}}\,G_0\mid\hat\mu_n)+D_\KL(\sqrt{\rho_{\mu}\rho_{\nu}}\,G_1\mid\hat\nu_m)
    \\
    &\qquad=\sqrt{\rho_{\mu}\rho_{\nu}}\Bigl(\sum_{i=1}^n\sum_{j=1}^m c(X_i,Y_j)G_{ij}+D_\KL(G_0\mid\tilde\mu_n)+D_\KL(G_1\mid\tilde\nu_m)\Bigr)
    \\
    &\qquad\quad+\bigl(\rho_{\mu}-\sqrt{\rho_{\mu}\rho_{\nu}}\bigr)M_\mu+\bigl(\rho_{\nu}-\sqrt{\rho_{\mu}\rho_{\nu}}\bigr)M_\nu.
\end{align*}
The map $G\mapsto\sqrt{\rho_{\mu}\rho_{\nu}}\,G$ is a bijection of $\R_+^{n\times m}$. Therefore, the transport plan $\hat\gamma/\sqrt{\rho_{\mu}\rho_{\nu}}$ is optimal for $\mathrm{UOT}(\tilde\mu_n,\tilde\nu_m)$. This transport plan has the same barycentric projections \eqref{eq:disc-frechet} as $\hat\gamma$ and the same indices $i$ with $\hat r_i=0$. Thus, $\hat T^{\mathrm{1NN}}$ is also the nearest-neighbor estimator of $T_0$ built from an optimal transport plan for $\mathrm{UOT}(\tilde\mu_n,\tilde\nu_m)$.

It therefore suffices to consider the case $\hat M_\mu=M_\mu$ and $\hat M_\nu=M_\nu$ with an arbitrary optimal transport plan $\hat\gamma$. In this case, the event $\mathcal{A}_{nm}$ always holds. Moreover, Assumption~\ref{assm:masses} holds with $\bar c=1$ and arbitrary positive sequences $a_n,b_m\to0$, so Proposition~\ref{prop:uot_barycentric_and_1nn_estimated_mass} applies. Now set
\begin{align*}
    \hat Z_n
    \coloneqq
    3\hat\Delta_{nm}^{\mathrm{bar}}
    +2C_\Omega {D_\KL}(\hat r\mid \hat r^\star),
    \qquad
    \hat E_n
    \coloneqq
    \mathrm{UOT}(\hat\mu_n,\hat\nu_m)
    - \int \zeta_0\,d\hat\mu_n
    - \int \xi_0\,d\hat\nu_m.
\end{align*}
By \eqref{eq:uot_hat_exact_empirical_excess} and dual feasibility,
\begin{align}
\label{eq:kl_le_E}
    {D_\KL}(\hat r\mid \hat r^\star)\le\hat E_n.
\end{align}
By \eqref{eq:uot_hat_barycentric_delta_bound},
\begin{align}
\label{eq:bar_le_E}
    \hat\Delta_{nm}^{\mathrm{bar}}\le\frac{2}{\kappa}\hat E_n.
\end{align}
Since $\mathrm{UOT}(\mu,\nu)=\int\zeta_0\,d\mu+\int\xi_0\,d\nu$, \eqref{eq:uot_two_sample} yields
\begin{align}
\label{eq:E_le_W}
    \hat E_n\le C_\Lambda\Bigl(M_\mu W_2^2(\bar\mu_n,\bar\mu)+M_\nu W_2^2(\bar\nu_m,\bar\nu)\Bigr).
\end{align}
Combining \eqref{eq:kl_le_E}, \eqref{eq:bar_le_E}, and \eqref{eq:E_le_W}, we obtain, with $C_3\coloneqq(6/\kappa+2C_\Omega)C_\Lambda$,
\begin{align}
\label{eq:uot_complete_1nn_estimated_mass_Z}
    \hat Z_n
    \le
    C_3\Bigl(M_\mu W_2^2(\bar\mu_n,\bar\mu)+M_\nu W_2^2(\bar\nu_m,\bar\nu)\Bigr)
    \le
    C_0
    \qquad\text{almost surely},
\end{align}
where $C_0\coloneqq C_3C_\Omega(M_\mu+M_\nu)$ and the last inequality uses $W_2^2\le C_\Omega$ on $\Omega$.

Let $C_1,C_2$ be the constants of Lemma~\ref{lem:uot_voronoi_mass_radius}, and set $t_n\coloneqq C_1(d+2)(\log n)/n$. By part (i) of Lemma~\ref{lem:uot_voronoi_mass_radius} with $\delta=n^{-2}$, we have $\mathbb{P}(M_n\ge t_n)\le n^{-2}$. On the event $\{M_n<t_n\}$, we have $nM_n\hat Z_n\le nt_n\hat Z_n$; on the event $\{M_n\ge t_n\}$, the bound $nM_n\le n$ and \eqref{eq:uot_complete_1nn_estimated_mass_Z} yield $nM_n\hat Z_n\le C_0n$. Thus, by splitting into these two events, we obtain
\begin{align}
\label{eq:uot_complete_1nn_estimated_mass_voronoi}
    \Ep[nM_n\hat Z_n]
    \le
    nt_n\Ep[\hat Z_n]+C_0n\,\mathbb{P}(M_n\ge t_n)
    \le
    C_1(d+2)\log n\,\Ep[\hat Z_n]+C_0/n.
\end{align}
Substituting $\gamma_0(\Omega)\le w_+M_\mu$, the bound $\Ep[R_n^2]\le C_2((\log n)/n)^{2/d}$ from Lemma~\ref{lem:uot_voronoi_mass_radius}(ii), \eqref{eq:uot_complete_1nn_estimated_mass_voronoi}, and \eqref{eq:uot_complete_1nn_estimated_mass_Z} into \eqref{eq:uot_complete_1nn_estimated_mass_as} proves \eqref{eq:uot_complete_1nn_estimated_mass_expectation}.
\end{proof}

\subsection{Active-source and growth estimation for the nearest-neighbor estimator}

The row masses of the discrete optimal transport plan \(\hat\gamma\) directly encode \(w_0=a_0^2=e^{-\varphi_0}\). The growth map \( \lambda_0=w_0\exp(\|x-T_0(x)\|^2/4) \) also depends on the displacement $x-T_0(x)$. Hence, we proceed in two steps: first we estimate \(a_0\) from the row masses \(\hat r_i\), and then we bound \(\int_\Omega|\hat\lambda^{\mathrm{1NN}}-\lambda_0|^2\,d\mu\) in terms of the errors of \(\hat a^{\mathrm{1NN}}\) and \(\hat T^{\mathrm{1NN}}\). Our bound relies on Lemma~\ref{lem:uot_kl_sqrt_lower_bound} and on the following pointwise bound.

\begin{lemma}[Pointwise bound on $|\hat\lambda_i-\lambda_0(X_i)|^2$]
\label{lem:discrete_gh_growth_transfer}
Let $x\in\Omega$, let $T,T^0\in\Omega$, and let $w^0\in[w_-,w_+]$. Define
\begin{align*}
    a^0=\sqrt{w^0},
    \qquad
    \lambda^0=w^0\exp\left(\frac14\|x-T^0\|^2\right).
\end{align*}
For any $u\ge0$, set
\begin{align*}
    \bar w=\clip_{[w_-,w_+]}(u),
    \qquad
    \bar\lambda=\bar w\exp\left(\frac14\|x-T\|^2\right).
\end{align*}
Then, with $C_{\mathrm{gr}}\coloneqq e^{\mathrm{diam}(\Omega)^2/2}\bigl(8w_++w_+^2\mathrm{diam}(\Omega)^2\bigr)$,
\begin{align}
\label{eq:discrete_gh_growth_transfer}
    |\bar\lambda-\lambda^0|^2
    \le
    C_{\mathrm{gr}}\left(|\sqrt{u}-a^0|^2+\|T-T^0\|^2\right).
\end{align}
\end{lemma}

\begin{proof}[Proof of Lemma \ref{lem:discrete_gh_growth_transfer}]
Since $w^0\in[w_-,w_+]$, clipping can only move $u$ closer to $w^0$; hence
\begin{align*}
    |\sqrt{\bar w}-\sqrt{w^0}|
    \le
    |\sqrt{u}-\sqrt{w^0}|.
\end{align*}
Moreover, $|\bar w-w^0|\le 2\sqrt{w_+}|\sqrt{\bar w}-\sqrt{w^0}|$. For $x,T\in\Omega$, the map $T\mapsto\exp(\|x-T\|^2/4)$ is bounded by $e^{\mathrm{diam}(\Omega)^2/4}$ and is Lipschitz with constant $\frac12\mathrm{diam}(\Omega)e^{\mathrm{diam}(\Omega)^2/4}$. Since $w^0\le w_+$, we therefore have
\begin{align*}
    |\bar\lambda-\lambda^0|
    &\le
    e^{\mathrm{diam}(\Omega)^2/4}|\bar w-w^0|+\tfrac12 w_+\mathrm{diam}(\Omega)e^{\mathrm{diam}(\Omega)^2/4}\|T-T^0\| \\
    &\le
    2\sqrt{w_+}\,e^{\mathrm{diam}(\Omega)^2/4}|\sqrt{u}-a^0|+\tfrac12 w_+\mathrm{diam}(\Omega)e^{\mathrm{diam}(\Omega)^2/4}\|T-T^0\|.
\end{align*}
Squaring yields
\begin{align*}
    |\bar\lambda-\lambda^0|^2
    &\le
    8w_+e^{\mathrm{diam}(\Omega)^2/2}|\sqrt{u}-a^0|^2+\tfrac12 w_+^2\mathrm{diam}(\Omega)^2e^{\mathrm{diam}(\Omega)^2/2}\|T-T^0\|^2 \\
    &\le
    C_{\mathrm{gr}}\left(|\sqrt{u}-a^0|^2+\|T-T^0\|^2\right).
\end{align*}
This proves \eqref{eq:discrete_gh_growth_transfer}.
\end{proof}

\begin{lemma}[Bound on $\int|\hat\lambda-\lambda_0|^2\,d\hat\eta$]
\label{lem:continuous_gh_growth_transfer}
Let $\hat\eta$ be a finite positive measure on a bounded set $\Omega$, let $\hat a,a_0:\Omega\to[0,\infty)$, and let $\hat T,T_0:\Omega\to\Omega$ be measurable. Suppose $a_0^2=w_0\in[w_-,w_+]$ and set $\hat\gamma_0=\hat a^2\hat\eta$. Define
\begin{align*}
    \hat\lambda(x)
    \coloneqq
    \clip_{[w_-,w_+]}\bigl(\hat a(x)^2\bigr)
    \exp\left(\frac14\|x-\hat T(x)\|^2\right),
    \qquad
    \lambda_0(x)
    \coloneqq
    w_0(x)\exp\left(\frac14\|x-T_0(x)\|^2\right).
\end{align*}
Then, with $C\coloneqq C_{\mathrm{gr}}\bigl(1+4\max(1,\mathrm{diam}(\Omega)^2)/w_-\bigr)$ and $C_{\mathrm{gr}}$ the constant of Lemma~\ref{lem:discrete_gh_growth_transfer},
\begin{align}
\label{eq:continuous_gh_growth_transfer}
    \int_\Omega |\hat\lambda(x)-\lambda_0(x)|^2\,d\hat\eta(x)
    \le
    C\int_\Omega |\hat a(x)-a_0(x)|^2\,d\hat\eta(x)
    + C\int_\Omega \|\hat T(x)-T_0(x)\|^2\,d\hat\gamma_0(x).
\end{align}
\end{lemma}

\begin{proof}[Proof of Lemma \ref{lem:continuous_gh_growth_transfer}]
The pointwise argument in Lemma~\ref{lem:discrete_gh_growth_transfer} yields
\begin{align}
\label{eq:gh_transfer_pointwise}
    |\hat\lambda(x)-\lambda_0(x)|^2
    \le
    C_{\mathrm{gr}}|\hat a(x)-a_0(x)|^2+C_{\mathrm{gr}}\|\hat T(x)-T_0(x)\|^2.
\end{align}
To integrate the last term in \eqref{eq:gh_transfer_pointwise} with respect to $\hat\gamma_0$ instead of $\hat\eta$, we prove the pointwise inequality
\begin{align}
\label{eq:gh_transfer_pointwise_helper}
    \|\hat T(x)-T_0(x)\|^2
    \le
    C'\hat a(x)^2\|\hat T(x)-T_0(x)\|^2+C'|\hat a(x)-a_0(x)|^2,
    \qquad x\in\Omega,
\end{align}
for $C'\coloneqq \max\!\left(\tfrac{4}{w_-},\,\tfrac{4\,\mathrm{diam}(\Omega)^2}{w_-}\right)$. Fix $x\in\Omega$ and split into two cases.
\smallskip\\
\emph{Case 1: $\hat a(x)^2\ge w_-/4$.} Then $\hat a(x)^2\|\hat T(x)-T_0(x)\|^2\ge (w_-/4)\|\hat T(x)-T_0(x)\|^2$, hence $\|\hat T(x)-T_0(x)\|^2\le (4/w_-)\hat a(x)^2\|\hat T(x)-T_0(x)\|^2$.
\smallskip\\
\emph{Case 2: $\hat a(x)^2<w_-/4$.} Since $a_0(x)^2=w_0(x)\ge w_-$, $a_0(x)\ge \sqrt{w_-}$, while $\hat a(x)<\sqrt{w_-}/2$, so $|\hat a(x)-a_0(x)|\ge a_0(x)-\hat a(x)\ge \sqrt{w_-}/2$, i.e.\ $|\hat a(x)-a_0(x)|^2\ge w_-/4$. Combined with the trivial bound $\|\hat T(x)-T_0(x)\|^2\le \mathrm{diam}(\Omega)^2$, we obtain $\|\hat T-T_0\|^2\le (4\,\mathrm{diam}(\Omega)^2/w_-)|\hat a-a_0|^2$.
\smallskip\\
Thus, \eqref{eq:gh_transfer_pointwise_helper} holds in either case. Substituting \eqref{eq:gh_transfer_pointwise_helper} into \eqref{eq:gh_transfer_pointwise} yields
\begin{align}
\label{eq:gh_transfer_pointwise_combined}
    |\hat\lambda(x)-\lambda_0(x)|^2
    \le
    C_{\mathrm{gr}}(1+C')|\hat a(x)-a_0(x)|^2+C_{\mathrm{gr}}C'\hat a(x)^2\|\hat T(x)-T_0(x)\|^2.
\end{align}
Since $C'=4\max(1,\mathrm{diam}(\Omega)^2)/w_-$, both coefficients are at most $C$. Integrating \eqref{eq:gh_transfer_pointwise_combined} against $\hat\eta$ and using $\hat a(x)^2\,d\hat\eta(x)=d\hat\gamma_0(x)$, which holds by definition of $\hat\gamma_0$, yields \eqref{eq:continuous_gh_growth_transfer}.
\end{proof}

\begin{lemma}
\label{lem:uot_fitted_row_kl_expect}
Assume that Assumptions~\ref{assm:curvature}, \ref{assm:supp_global}, \ref{assm:positivity}, and \ref{assm:masses} hold. Let $\hat\gamma$ be any optimizer of $\mathrm{UOT}(\hat\mu_n,\hat\nu_m)$, define
\begin{align*}
    \hat r_i \coloneqq \sum_{j=1}^m \hat\gamma_{ij},
    \qquad
    \hat\mu_i \coloneqq {\hat M_\mu/n},
    \qquad
    \hat r_i^\star \coloneqq e^{-\varphi_0(X_i)}\hat\mu_i,
\end{align*}
and write $\hat r=(\hat r_i)_{i=1}^n$ and $\hat r^\star=(\hat r_i^\star)_{i=1}^n$. Then there exists a constant $C>0$, depending only on \( M_\mu,M_\nu,\|\zeta_0\|_{L^\infty(\Omega)},\|\xi_0\|_{L^\infty(\Omega)} \), the constant $\bar c$ of Assumption~\ref{assm:masses}, and $C_\Lambda$ from Proposition~\ref{prop:two_sample_stability}, such that for all sufficiently large $n,m$,
\begin{align}
\label{eq:uot_fitted_row_kl_expect}
    \Ep[{D_\KL}(\hat r\mid \hat r^\star)]
    \le
    C
    \Bigl(
        \Ep[M_\mu W_2^2(\bar\mu_n,\bar\mu)]
        + \Ep[M_\nu W_2^2(\bar\nu_m,\bar\nu)]
        + a_n + b_m
    \Bigr),
\end{align}
where $a_n,b_m$ are the sequences from Assumption~\ref{assm:masses}.
\end{lemma}

\begin{proof}[Proof of Lemma \ref{lem:uot_fitted_row_kl_expect}]
By \eqref{eq:uot_hat_exact_empirical_excess} and the nonnegativity of $\sum_{i=1}^n\sum_{j=1}^m\hat\gamma_{ij}\bigl(c(X_i,Y_j)-\varphi_0(X_i)-\psi_0(Y_j)\bigr)$ and of ${D_\KL}(\hat s\mid \hat s^\star)$, the inequality
\begin{align}
    {D_\KL}(\hat r\mid \hat r^\star)
    \le
    \mathrm{UOT}(\hat\mu_n,\hat\nu_m)
    - \int \zeta_0\,d\hat\mu_n
    - \int \xi_0\,d\hat\nu_m
    \label{eq:fitted_row_kl_excess}
\end{align}
holds. By the stability bound \eqref{eq:estimated_mass_stability},
\begin{align*}
&\mathrm{UOT}(\hat\mu_n,\hat\nu_m)-\mathrm{UOT}(\mu,\nu) \\
&\qquad\le
\int \zeta_0\,d(\tilde\mu_n-\mu)
+ \int \xi_0\,d(\tilde\nu_m-\nu)
+ C_\Lambda\Bigl(
    M_\mu W_2^2(\bar\mu_n,\bar\mu)+M_\nu W_2^2(\bar\nu_m,\bar\nu)
\Bigr) \\
&\qquad\quad
+ \bigl( |\hat M_\mu-M_\mu|+|\hat M_\nu-M_\nu| \bigr)
+ e\sqrt{M_\mu M_\nu}\bigl(|\log\rho_{\mu}|+|\log\rho_{\nu}|\bigr),
\end{align*}
using $M_\mu|\rho_{\mu}-1|=|\hat M_\mu-M_\mu|$ and $M_\nu|\rho_{\nu}-1|=|\hat M_\nu-M_\nu|$. Subtracting \( \int \zeta_0\,d(\hat\mu_n-\mu)+\int \xi_0\,d(\hat\nu_m-\nu) \) from both sides and using \( \mathrm{UOT}(\mu,\nu)=\int \zeta_0\,d\mu+\int \xi_0\,d\nu \) yields
\begin{align}
    \mathrm{UOT}(\hat\mu_n,\hat\nu_m)
    - \int \zeta_0\,d\hat\mu_n
    - \int \xi_0\,d\hat\nu_m
    &\le
    C_\Lambda\Bigl(
        M_\mu W_2^2(\bar\mu_n,\bar\mu)+M_\nu W_2^2(\bar\nu_m,\bar\nu)
    \Bigr)\notag
\\
    &\quad +
    \bigl( |\hat M_\mu-M_\mu|+|\hat M_\nu-M_\nu| \bigr)\notag
\\
    &\quad +
    e\sqrt{M_\mu M_\nu}\bigl(|\log\rho_{\mu}|+|\log\rho_{\nu}|\bigr)\notag
\\
    &\quad
    - \int \zeta_0\,d(\hat\mu_n-\tilde\mu_n)
    - \int \xi_0\,d(\hat\nu_m-\tilde\nu_m).
    \label{eq:fitted_row_kl_pathwise}
\end{align}
The last two integrals are bounded in absolute value by
\begin{align}
    \|\zeta_0\|_{L^\infty(\Omega)} |\hat M_\mu-M_\mu|
    +
    \|\xi_0\|_{L^\infty(\Omega)} |\hat M_\nu-M_\nu|.
    \label{eq:fitted_row_kl_linear}
\end{align}
We now take expectations of \eqref{eq:fitted_row_kl_pathwise}. The a.s.\ lower bound $\hat M_\mu\ge \bar cM_\mu$ in Assumption~\ref{assm:masses} yields $|\log\rho_{\mu}|\le|\rho_{\mu}-1|/\bar c$ a.s. Since Assumption~\ref{assm:masses} also gives $\Ep[|\hat M_\mu-M_\mu|]\le a_n$, $M_\mu\Ep[|\log\rho_{\mu}|]\le a_n/\bar c$. Hence, \( e\sqrt{M_\mu M_\nu}\,\Ep[|\log\rho_{\mu}|]\le \bar c^{-1}e\sqrt{M_\nu/M_\mu}\,a_n. \) The analogous bounds hold on the $\nu$-side with $b_m$. Setting \( C\coloneqq C_\Lambda + \|\zeta_0\|_{L^\infty(\Omega)}+\|\xi_0\|_{L^\infty(\Omega)}+1 + \bar c^{-1}e\sqrt{M_\nu/M_\mu}+\bar c^{-1}e\sqrt{M_\mu/M_\nu} \) and combining \eqref{eq:fitted_row_kl_excess}, \eqref{eq:fitted_row_kl_pathwise}, and \eqref{eq:fitted_row_kl_linear} with these expectation bounds proves \eqref{eq:uot_fitted_row_kl_expect}.
\end{proof}

\begin{lemma}[Bound on $\sum_i\hat\mu_i|\hat a_i-a_0(X_i)|^2$]
\label{lem:uot_complete_1nn_active_fitted}
Under the assumptions of Lemma~\ref{lem:uot_fitted_row_kl_expect}, we have
\begin{align}
\label{eq:uot_complete_1nn_active_fitted_empirical}
    \sum_{i=1}^n \hat\mu_i
    |\hat a_i-a_0(X_i)|^2
    \le
    {D_\KL}(\hat r\mid \hat r^\star).
\end{align}
Consequently, with $a_n,b_m$ as in Assumption~\ref{assm:masses} and the constant $C>0$ of Lemma~\ref{lem:uot_fitted_row_kl_expect}, for all sufficiently large $n,m$,
\begin{align}
\label{eq:uot_complete_1nn_active_fitted_expectation}
    \Ep\left[
    \sum_{i=1}^n \hat\mu_i
    |\hat a_i-a_0(X_i)|^2
    \right]
    \le
    C
    \Bigl(
        \Ep[M_\mu W_2^2(\bar\mu_n,\bar\mu)]
        + \Ep[M_\nu W_2^2(\bar\nu_m,\bar\nu)]
        + a_n + b_m
    \Bigr).
\end{align}
\end{lemma}

\begin{proof}[Proof of Lemma \ref{lem:uot_complete_1nn_active_fitted}]
Since $\hat r_i=\hat\mu_i\hat a_i^2$ and $\hat r_i^\star=\hat\mu_i a_0(X_i)^2$, we have
\begin{align*}
    {D_\KL}(\hat r\mid \hat r^\star)
    =
    \sum_{i=1}^n\hat\mu_i
    \left[
        \hat a_i^2
        \log\left(\frac{\hat a_i^2}{a_0(X_i)^2}\right)
        -\hat a_i^2
        +a_0(X_i)^2
    \right].
\end{align*}
Lemma~\ref{lem:uot_kl_sqrt_lower_bound} yields \eqref{eq:uot_complete_1nn_active_fitted_empirical} termwise, and \eqref{eq:uot_complete_1nn_active_fitted_expectation} follows from Lemma~\ref{lem:uot_fitted_row_kl_expect}.
\end{proof}

\begin{lemma}[Bound on $\int|\hat\lambda^{1\mathrm{NN}}-\lambda_0|^2\,d\mu$]
\label{lem:lambda_1nn_population_extension}
Assume Assumptions~\ref{assm:curvature}, \ref{assm:supp_global}, and \ref{assm:positivity}. Let $V_i,M_n,R_n$ be as in Lemma~\ref{lem:uot_voronoi_mass_radius}, and let $\mathcal A_{nm}\coloneqq\{|\hat M_\mu-M_\mu|\le M_\mu/2,\ |\hat M_\nu-M_\nu|\le M_\nu/2\}$. Then there is a constant $C>0$, depending only on $M_\mu$ and $\mathrm{Lip}(\lambda_0)$, such that, on $\mathcal A_{nm}$,
\begin{align}
\label{eq:lambda_1nn_population_extension}
    \int_\Omega |\hat\lambda^{\mathrm{1NN}}(x)-\lambda_0(x)|^2\,d\mu(x)
    \le
    C nM_n\sum_{i=1}^n\hat\mu_i|\hat\lambda_i-\lambda_0(X_i)|^2
    + C R_n^2 .
\end{align}
\end{lemma}

\begin{proof}[Proof of Lemma \ref{lem:lambda_1nn_population_extension}]
We first check that $\lambda_0=w_0\exp(\|x-T_0(x)\|^2/4)$ is Lipschitz on $\Omega$. The potential $\varphi_0$ is $c$-concave, so $\varphi_0$ is an infimum of functions with a common Lipschitz constant on the bounded set $\Omega$. Hence, $\varphi_0$ is Lipschitz and bounded, and so is $w_0=e^{-\varphi_0}$. Moreover, $T_0$ is $\kappa^{-1}$-Lipschitz by Assumption~\ref{assm:curvature}, so $x\mapsto\exp(\|x-T_0(x)\|^2/4)$ is Lipschitz and bounded on the compact set $\Omega$. Consequently, the product $\lambda_0$ is Lipschitz. We write $L_\lambda\coloneqq\mathrm{Lip}(\lambda_0)$. For $x\in V_i$,
\begin{align}
\label{eq:lambda_pop_pointwise}
    |\hat\lambda_i-\lambda_0(x)|^2
    \le
    2|\hat\lambda_i-\lambda_0(X_i)|^2
    +
    2L_\lambda^2\|x-X_i\|^2 .
\end{align}
On $\mathcal A_{nm}$, we have $\hat\mu_i=\hat M_\mu/n\ge M_\mu/(2n)$. Combining this lower bound with $\mu(V_i)=M_\mu\bar\mu(V_i)\le M_\mu M_n$, we obtain $\mu(V_i)\le 2nM_n\hat\mu_i$. Integrating \eqref{eq:lambda_pop_pointwise} over $V_i$ with respect to $\mu$, we obtain
\begin{align*}
    \int_{V_i}|\hat\lambda^{\mathrm{1NN}}(x)-\lambda_0(x)|^2\,d\mu(x)
    \le
    4nM_n\hat\mu_i|\hat\lambda_i-\lambda_0(X_i)|^2
    +2L_\lambda^2\int_{V_i}\|x-X_i\|^2\,d\mu(x).
\end{align*}
Summing over $i$ and using $\sum_{i=1}^n\int_{V_i}\|x-X_i\|^2\,d\mu(x)\le M_\mu R_n^2$ yields \eqref{eq:lambda_1nn_population_extension} with $C=\max(4,2L_\lambda^2M_\mu)$.
\end{proof}

\begin{theorem}[Growth map risk of $\hat\lambda^{1\mathrm{NN}}$]
\label{thm:uot_complete_1nn_lambda_fitted}
Assume that Assumptions~\ref{assm:curvature}, \ref{assm:supp_global}, \ref{assm:positivity}, and \ref{assm:masses} hold. Let $\hat\lambda_i$ and $\hat\lambda^{\mathrm{1NN}}$ be defined by \eqref{eq:disc-gh-growth} and \eqref{eq:disc-1nn-map}. Then there exists a constant $C>0$ such that for all sufficiently large $n,m$,
\begin{align}
\label{eq:uot_complete_1nn_lambda_population}
\begin{aligned}
    \Ep\Bigl[
        \int_\Omega |\hat\lambda^{\mathrm{1NN}}(x)-\lambda_0(x)|^2\,d\mu(x)
    \Bigr]
    &\le
    C\log n
    \Bigl(
        \Ep[M_\mu W_2^2(\bar\mu_n,\bar\mu)]
        + \Ep[M_\nu W_2^2(\bar\nu_m,\bar\nu)]
        + a_n + b_m
    \Bigr) \\
    &\quad
    + C\left({(\log n)/n}\right)^{2/d}
    + C{(\log n)/n},
\end{aligned}
\end{align}
where $a_n,b_m$ are the sequences from Assumption~\ref{assm:masses}.
\end{theorem}

\begin{proof}[Proof of Theorem \ref{thm:uot_complete_1nn_lambda_fitted}]
We first prove that, for all sufficiently large $n,m$,
\begin{align}
\label{eq:uot_complete_1nn_lambda_fitted_expectation}
\begin{aligned}
    \Ep\Bigl[
        \int_\Omega |\hat\lambda^{\mathrm{1NN}}(x)-\lambda_0(x)|^2\,d\hat\mu_n(x)
    \Bigr]
    &\le
    C
    \Bigl(
        \Ep[M_\mu W_2^2(\bar\mu_n,\bar\mu)]
        + \Ep[M_\nu W_2^2(\bar\nu_m,\bar\nu)]
        + a_n + b_m
    \Bigr),
\end{aligned}
\end{align}
and then pass to \eqref{eq:uot_complete_1nn_lambda_population} through Lemma~\ref{lem:lambda_1nn_population_extension}. Apply Lemma~\ref{lem:discrete_gh_growth_transfer} with
\begin{align*}
    u_i=\hat r_i/\hat\mu_i,
    \qquad
    w_i^0=w_0(X_i)=e^{-\varphi_0(X_i)},
    \qquad
    T_i=\hat T_i,
    \qquad
    T_i^0=T_0(X_i).
\end{align*}
This yields
\begin{align}
\label{eq:lambda_transfer_sum_plan}
    \sum_{i=1}^n\hat\mu_i |\hat\lambda_i-\lambda_0(X_i)|^2
    \le
    C_{\mathrm{gr}}\sum_{i=1}^n\hat\mu_i |\hat a_i-a_0(X_i)|^2
    + C_{\mathrm{gr}}\sum_{i=1}^n\hat\mu_i\|\hat T_i-T_0(X_i)\|^2.
\end{align}
Lemma~\ref{lem:uot_complete_1nn_active_fitted} bounds the first term on the right-hand side of \eqref{eq:lambda_transfer_sum_plan}. For the second term, since $w_0\ge w_-$,
\begin{align*}
    \sum_{i=1}^n\hat\mu_i\|\hat T_i-T_0(X_i)\|^2
    \le
    w_-^{-1}
    \sum_{i=1}^n\hat r_i^\star\|\hat T_i-T_0(X_i)\|^2.
\end{align*}
Lemma~\ref{lem:uot_weighted_pinsker_rows_hat} gives
\begin{align*}
    \sum_{i=1}^n\hat r_i^\star\|\hat T_i-T_0(X_i)\|^2
    \le
    3\hat\Delta_{nm}^{\mathrm{bar}}
    +2C_\Omega {D_\KL}(\hat r\mid \hat r^\star).
\end{align*}
Since $\hat\lambda^{\mathrm{1NN}}(X_i)=\hat\lambda_i$, the left-hand side of \eqref{eq:lambda_transfer_sum_plan} is $\int_\Omega|\hat\lambda^{\mathrm{1NN}}-\lambda_0|^2\,d\hat\mu_n$. Taking expectations and using Proposition~\ref{prop:uot_barycentric_and_1nn_estimated_mass} and Lemma~\ref{lem:uot_fitted_row_kl_expect} proves \eqref{eq:uot_complete_1nn_lambda_fitted_expectation}.

We now prove the bound \eqref{eq:uot_complete_1nn_lambda_population} under $\mu$. Let $\mathcal A_{nm}$ be the event from Lemma~\ref{lem:lambda_1nn_population_extension}, and write
\begin{align*}
    S_n \coloneqq \sum_{i=1}^n\hat\mu_i|\hat\lambda_i-\lambda_0(X_i)|^2 .
\end{align*}
Since $\clip_{[w_-,w_+]}(\hat a_i^2)\le w_+$, $e^{-\varphi_0}\le w_+$, and $\|x-T\|\le\mathrm{diam}(\Omega)$ for all $x,T\in\Omega$, both $\hat\lambda_i$ and $\lambda_0$ take values in $[0,w_+e^{\mathrm{diam}(\Omega)^2/4}]$. Hence, $|\hat\lambda_i-\lambda_0(x)|\le C_\lambda\coloneqq 2w_+e^{\mathrm{diam}(\Omega)^2/4}$ for all $i$ and $x$. Since $\hat M_\mu\le \tfrac32 M_\mu$ on $\mathcal A_{nm}$, it follows that
\begin{align}
\label{eq:lambda_pop_uniform}
    \int_\Omega|\hat\lambda^{\mathrm{1NN}}-\lambda_0|^2\,d\mu\le C_\lambda^2M_\mu
    \quad\text{almost surely},
    \quad
    S_n\mathbf 1_{\mathcal A_{nm}}\le C_\lambda^2\hat M_\mu\mathbf 1_{\mathcal A_{nm}}\le \tfrac32C_\lambda^2M_\mu .
\end{align}
The complement $\mathcal A_{nm}^c$ is the union of the events $\{|\hat M_\mu-M_\mu|>M_\mu/2\}$ and $\{|\hat M_\nu-M_\nu|>M_\nu/2\}$. Hence, using the first bound in \eqref{eq:lambda_pop_uniform}, Markov's inequality, the union bound, and Assumption~\ref{assm:masses}, we obtain
\begin{align}
    \Ep\Bigl[\mathbf 1_{\mathcal A_{nm}^c}\int_\Omega|\hat\lambda^{\mathrm{1NN}}-\lambda_0|^2\,d\mu\Bigr]
    &\le
    C_\lambda^2M_\mu\mathbb{P}(\mathcal A_{nm}^c)
    \notag\\
    &\le
    C_\lambda^2M_\mu\Bigl(\frac{2}{M_\mu}\Ep[|\hat M_\mu-M_\mu|]+\frac{2}{M_\nu}\Ep[|\hat M_\nu-M_\nu|]\Bigr)
    \notag\\
    &\le
    C(a_n+b_m).
    \label{eq:lambda_pop_badmass}
\end{align}
On $\mathcal A_{nm}$, Lemma~\ref{lem:lambda_1nn_population_extension} gives
\begin{align}
    \mathbf 1_{\mathcal A_{nm}}\int_\Omega|\hat\lambda^{\mathrm{1NN}}-\lambda_0|^2\,d\mu
    \le
    C nM_n S_n\mathbf 1_{\mathcal A_{nm}} + CR_n^2 .
    \label{eq:lambda_pop_goodmass}
\end{align}
Let $C_1,C_2$ be the constants of Lemma~\ref{lem:uot_voronoi_mass_radius}. By part (i) of Lemma~\ref{lem:uot_voronoi_mass_radius} with $\delta=n^{-2}$, the threshold $t_n\coloneqq C_1(d+2)(\log n)/n$ satisfies $\mathbb{P}(M_n\ge t_n)\le n^{-2}$. On the event $\{M_n<t_n\}$, we have $nM_nS_n\mathbf 1_{\mathcal A_{nm}}\le nt_nS_n$; on the event $\{M_n\ge t_n\}$, the bound $nM_n\le n$ and the second bound in \eqref{eq:lambda_pop_uniform} yield $nM_nS_n\mathbf 1_{\mathcal A_{nm}}\le\tfrac32C_\lambda^2M_\mu n$. Thus, by splitting into these two events, we obtain
\begin{align}
    \Ep[nM_nS_n\mathbf 1_{\mathcal A_{nm}}]
    &\le
    C_1(d+2)\log n\,\Ep[S_n] + \tfrac32 C_\lambda^2M_\mu\, n\mathbb{P}(M_n\ge t_n) \notag
\\
    &\le
    C_1(d+2)\log n\,\Ep[S_n] + \tfrac32 C_\lambda^2M_\mu/n ,
    \label{eq:lambda_pop_voronoi}
\end{align}
where $\Ep[S_n]$ is the left-hand side of \eqref{eq:uot_complete_1nn_lambda_fitted_expectation}. Moreover, Lemma~\ref{lem:uot_voronoi_mass_radius}(ii) gives $\Ep[R_n^2]\le C_2((\log n)/n)^{2/d}$. Combining \eqref{eq:uot_complete_1nn_lambda_fitted_expectation} and this bound on $\Ep[R_n^2]$ with \eqref{eq:lambda_pop_badmass}, \eqref{eq:lambda_pop_goodmass}, and \eqref{eq:lambda_pop_voronoi} proves \eqref{eq:uot_complete_1nn_lambda_population}.
\end{proof}

\begin{proof}[Proof of Theorem \ref{thm:main_plan_based_rates}]
We first reduce the expectations in \eqref{eq:section4_plan_map_rate_abstract} and \eqref{eq:section4_plan_growth_rate_abstract} to \eqref{eq:uot_complete_1nn_estimated_mass_expectation} and \eqref{eq:uot_complete_1nn_lambda_population}. For $\Ep[\int_\Omega\|\hat T^{\mathrm{1NN}}-T_0\|^2\,d\mu]$, we have $d\gamma_0=w_0\,d\mu$, and $w_0=e^{-\varphi_0}\ge w_->0$ on $\Omega$ by the choice of the clipping constants in Section~\ref{sec:plan}. Hence,
\begin{align*}
    \int_\Omega \|\hat T^{\mathrm{1NN}}-T_0\|^2\,d\mu
    =
    \int_\Omega \|\hat T^{\mathrm{1NN}}-T_0\|^2 w_0^{-1}\,d\gamma_0
    \le
    w_-^{-1}
    \int_\Omega \|\hat T^{\mathrm{1NN}}-T_0\|^2\,d\gamma_0 ,
\end{align*}
so Theorem~\ref{thm:uot_complete_1nn_estimated_mass} bounds this expectation after absorbing $w_-^{-1}$ into the constant. The expectation $\Ep[\int_\Omega|\hat\lambda^{\mathrm{1NN}}-\lambda_0|^2\,d\mu]$ is bounded by \eqref{eq:uot_complete_1nn_lambda_population} in Theorem~\ref{thm:uot_complete_1nn_lambda_fitted}.

It remains to bound $\Ep[W_2^2(\bar\mu_n,\bar\mu)]$ and $\Ep[W_2^2(\bar\nu_m,\bar\nu)]$. By Assumptions~\ref{assm:supp_global} and \ref{assm:positivity}, the normalized measures $\bar\mu$ and $\bar\nu$ are supported on a bounded subset of $\mathbb{R}^d$ and have finite moments of every order. Theorem~1 of \cite{fournier2015rate}, applied with $p=2$, yields
\begin{align*}
    \Ep[W_2^2(\bar\mu_n,\bar\mu)]
    &\le C_\mu \mathfrak R_n^{\mathrm{emp}}(d),
    \qquad
    \Ep[W_2^2(\bar\nu_m,\bar\nu)]
    \le C_\nu \mathfrak R_m^{\mathrm{emp}}(d).
\end{align*}
Substituting these rates into \eqref{eq:uot_complete_1nn_estimated_mass_expectation}, we obtain a bound of the form
\begin{align}
\label{eq:section4_plan_form}
    C\log n\Bigl(M_\mu\mathfrak R_n^{\mathrm{emp}}(d)+M_\nu\mathfrak R_m^{\mathrm{emp}}(d)\Bigr)
    + C\bigl((\log n)/n\bigr)^{2/d} + C(\log n)/n .
\end{align}
Substituting these rates and Assumption~\ref{assm:masses} into \eqref{eq:uot_complete_1nn_lambda_population}, we obtain \eqref{eq:section4_plan_form} with the additional term $C\log n\,(a_n+b_m)$. Finally, $C((\log n)/n)^{2/d}$ and $C(\log n)/n$ are absorbed into the first term in \eqref{eq:section4_plan_form}, since $((\log n)/n)^{2/d}\le \log n\,\mathfrak R_n^{\mathrm{emp}}(d)$ and $(\log n)/n\le\log n\,\mathfrak R_n^{\mathrm{emp}}(d)$ for all sufficiently large $n$. This proves \eqref{eq:section4_plan_map_rate_abstract} and \eqref{eq:section4_plan_growth_rate_abstract}.
\end{proof}

\begin{remark}[Order of the constants in Theorem~\ref{thm:main_plan_based_rates}]
\label{rem:constants_plan}
Write $D\coloneqq\mathrm{diam}(\Omega)$ and $\mathfrak m\coloneqq 1+M_\mu/M_\nu+M_\nu/M_\mu$, and take the clipping constants $w_+\coloneqq e^{D^2/2}(1+M_\nu/M_\mu)$ and $w_-\coloneqq e^{-D^2/2}(1+M_\mu/M_\nu)^{-1}$, which satisfy $w_-\le e^{-\varphi_0}\le w_+$ on $\Omega$ by Lemma~\ref{lem:dual_potential_bounds}. By tracking the constants in the proofs of Theorems~\ref{thm:uot_complete_1nn_estimated_mass} and~\ref{thm:uot_complete_1nn_lambda_fitted}, the constant in~\eqref{eq:section4_plan_map_rate_abstract} is $O\bigl(e^{3D^2/2}\mathfrak m^2(C_\Lambda\kappa^{-1}+\kappa^{-2})\bigr)$. The constant in~\eqref{eq:section4_plan_growth_rate_abstract} is $O\bigl(e^{2D^2}\mathfrak m^3(C_\Lambda\kappa^{-1}+\kappa^{-2})\bigr)$.
\end{remark}

\section{Proof of Theorem \ref{thm:main_cube_kernel_rates}}\label{sec:proof_cube_kernel_rates}

Here, we additionally define
\begin{align}
    \widehat{\mathrm{UOT}}_{nm}^{\mathrm{ker}}
    &\coloneqq
    \mathrm{UOT}(\hat\mu_n^{\mathrm{ker}},\hat\nu_m^{\mathrm{ker}}).
\end{align}
Let $\hat p_n^{\mathrm{ker}}$ and $\hat q_m^{\mathrm{ker}}$ be the densities of the probability measures $\hat\mu_n^{\mathrm{ker}}/\hat M_\mu$ and $\hat\nu_m^{\mathrm{ker}}/\hat M_\nu$, that is,
\begin{align}
    \hat p_n^{\mathrm{ker}}(x)
    &\coloneqq
    \frac{(\tilde p_n^{\mathrm{ker}}(x))_+}
    {\int_{[0,1]^d}(\tilde p_n^{\mathrm{ker}}(u))_+\,du},
    \qquad
    \hat q_m^{\mathrm{ker}}(y)
    \coloneqq
    \frac{(\tilde q_m^{\mathrm{ker}}(y))_+}
    {\int_{[0,1]^d}(\tilde q_m^{\mathrm{ker}}(v))_+\,dv}.
\end{align}
We then define
\begin{align}
    \tilde\mu_n^{\mathrm{ker}}
    &\coloneqq
    M_\mu \hat p_n^{\mathrm{ker}}(x)\,dx,
    \qquad
    \tilde\nu_m^{\mathrm{ker}}
    \coloneqq
    M_\nu \hat q_m^{\mathrm{ker}}(y)\,dy.
\end{align}
We also define the rate
\begin{align}
\label{eq:frak_L_rate}
    \mathfrak L_n(\alpha)
    \coloneqq
    n^{{-(\alpha-1)/(2(\alpha-1)+d)}}.
\end{align}

The proof follows the stability argument used for Theorem~\ref{thm:main_plan_based_rates}, with the empirical measures $\hat\mu_n,\hat\nu_m$ replaced by the measure estimates $\hat\mu_n^{\mathrm{ker}},\hat\nu_m^{\mathrm{ker}}$. The overall idea is to decompose $\widehat{\mathrm{UOT}}_{nm}^{\mathrm{ker}}-\int\zeta_0\,d\hat\mu_n^{\mathrm{ker}}-\int\xi_0\,d\hat\nu_m^{\mathrm{ker}}$ into an integral against $\hat\gamma_{nm}^{\mathrm{ker}}$ and two ${D_\KL}$ terms relative to $e^{-\varphi_0}\hat\mu_n^{\mathrm{ker}}$ and $e^{-\psi_0}\hat\nu_m^{\mathrm{ker}}$, which is analogous to what we did in the proof of Theorem~\ref{thm:main_plan_based_rates}. We then bound the risks of $\hat T_{nm}^{\mathrm{ker}}$ and $\hat\lambda_{nm}^{\mathrm{ker}}$ in terms of $\Ep[W_2^2(\tilde\mu_n^{\mathrm{ker}},\mu)]$, $\Ep[W_2^2(\tilde\nu_m^{\mathrm{ker}},\nu)]$, and the sequences $a_n$ and $b_m$ of Assumption~\ref{assm:masses}. The upper bound therefore depends on the convergence rate of $\Ep[W_2^2(\tilde\mu_n^{\mathrm{ker}},\mu)]$ and $\Ep[W_2^2(\tilde\nu_m^{\mathrm{ker}},\nu)]$.

\subsection{Convergence rates of the kernel density estimators}

For $k\in\mathbb N_0^d$, let $\eta_k(x)\coloneqq\prod_{r=1}^d c_{k_r}\cos(\pi k_r x_r)$, with the constants $c_\ell$ of Section~\ref{sec:kernel}, and let $\varrho_k\coloneqq\pi^2\|k\|^2$. The functions $(\eta_k)_{k\in\mathbb N_0^d}$ form an orthonormal basis of $L^2([0,1]^d)$, and $\eta_k$ is the eigenfunction of the Neumann Laplacian $-\Delta$ with eigenvalue $\varrho_k$. For a zero-mean $h\in L^2([0,1]^d)$, we write
\begin{align}
\label{eq:neumann_hminus1_norm}
    \|h\|_{H^{-1}_N([0,1]^d)}^2
    \coloneqq
    \sum_{k\in\mathbb{N}_0^d\setminus\{0\}}
    \frac{|\langle h,\eta_k\rangle_{L^2([0,1]^d)}|^2}{\varrho_k}.
\end{align}

\begin{lemma}[{Bound on $W_2$ by the $H^{-1}$ norm \cite[Theorem~5.34]{santambrogio2015optimal}}]
\label{lem:cube_hminus1_to_w2}
Let $f$ and $g$ be bounded probability densities on $[0,1]^d$ satisfying
\begin{align}
    0<\underline\beta \le f(x),g(x),
    \qquad x\in[0,1]^d.
\end{align}
Then
\begin{align}
\label{eq:cube_hminus1_to_w2}
    W_2^2(f\,dx,g\,dx)
    \le
    \underline\beta^{-1}
    \|f-g\|_{H^{-1}_N([0,1]^d)}^2.
\end{align}
\end{lemma}

\begin{proof}[Proof of Lemma \ref{lem:cube_hminus1_to_w2}]
Set $h\coloneqq g-f$. Because $\int_{[0,1]^d}h=0$, there exists a unique weak solution $u$ of the Neumann problem
\begin{align*}
    -\Delta u = h,
    \qquad
    \partial_{\mathbf n} u = 0 \text{ on } \partial[0,1]^d,
    \qquad
    \int_{[0,1]^d}u = 0.
\end{align*}
Writing $u=\sum_{k\neq 0}u_k\eta_k$ yields $u_k = \langle h,\eta_k\rangle/\varrho_k$, and therefore
\begin{align*}
    \int_{[0,1]^d}\|\nabla u(x)\|^2\,dx
    =
    \sum_{k\neq 0}\varrho_k |u_k|^2
    =
    \sum_{k\neq 0}\frac{|\langle h,\eta_k\rangle|^2}{\varrho_k}
    =
    \|h\|_{H^{-1}_N([0,1]^d)}^2.
\end{align*}
Now set $f_t=(1-t)f+t g$ for $t\in[0,1]$ and define
\begin{align*}
    v_t(x) \coloneqq \frac{\nabla u(x)}{f_t(x)}.
\end{align*}
Since $f_t\ge \underline\beta$, the vector field $v_t$ is well defined. Moreover, $f_tv_t=\nabla u$ and $-\Delta u=g-f$ yield
\begin{align*}
    \partial_t f_t + \nabla \cdot(f_t v_t)
    =
    (g-f)+\nabla \cdot(\nabla u)
    =0,
\end{align*}
so $(f_t,v_t)$ is an admissible Benamou-Brenier path from $f\,dx$ to $g\,dx$. Hence, the Benamou-Brenier formula \cite[Theorem~5.28]{santambrogio2015optimal} yields
\begin{align*}
    W_2^2(f\,dx,g\,dx)
    &\le
    \int_0^1\int_{[0,1]^d} f_t(x)\|v_t(x)\|^2\,dx\,dt \\
    &=
    \int_0^1\int_{[0,1]^d}\frac{\|\nabla u(x)\|^2}{f_t(x)}\,dx\,dt \\
    &\le
    \underline\beta^{-1}\int_{[0,1]^d}\|\nabla u(x)\|^2\,dx
    =
    \underline\beta^{-1}\|f-g\|_{H^{-1}_N([0,1]^d)}^2,
\end{align*}
which proves \eqref{eq:cube_hminus1_to_w2}.
\end{proof}

\begin{proposition}[Kernel density rates on the hypercube]
\label{prop:cube_kernel_density_rates}
Set $s\coloneqq \alpha-1>0$. Assume that $\Omega=[0,1]^d$, that Assumption~\ref{assm:cube_kernel_smoothness} holds, and that
\begin{align*}
    \beta_{\min} \le p(x),q(x) \le \beta_{\max}
    \quad \text{for all } x\in[0,1]^d.
\end{align*}
Let
\begin{align}
    L_n = n^{1/(d+2s)},
    \qquad
    L_m = m^{1/(d+2s)}.
\end{align}
Then there exist constants $C>0$ and $N_0\in\mathbb{N}$, depending only on $d,\alpha,M,\beta_{\min},\beta_{\max}$, and the function $\tau$ such that, for all $n,m\ge N_0$,
\begin{align}
\label{eq:cube_kernel_density_rate_source_w2}
    \Ep\bigl[W_2^2(\tilde\mu_n^{\mathrm{ker}},\mu)\bigr]
    &\le
    C M_\mu\mathfrak R_n^{\mathrm{ker}}(\alpha),
    \\
\label{eq:cube_kernel_density_rate_target_w2}
    \Ep\bigl[W_2^2(\tilde\nu_m^{\mathrm{ker}},\nu)\bigr]
    &\le
    C M_\nu\mathfrak R_m^{\mathrm{ker}}(\alpha),
\end{align}
and
\begin{align}
\label{eq:cube_kernel_density_rate_source_l1}
    \Ep\bigl[\|\hat p_n^{\mathrm{ker}}-p\|_{L^1([0,1]^d)}\bigr]
    &\le
    C\mathfrak L_n(\alpha),
    \\
\label{eq:cube_kernel_density_rate_target_l1}
    \Ep\bigl[\|\hat q_m^{\mathrm{ker}}-q\|_{L^1([0,1]^d)}\bigr]
    &\le
    C\mathfrak L_m(\alpha).
\end{align}
Moreover, let $\underline p\coloneqq\inf_{x\in[0,1]^d}p(x)$ and $\underline q\coloneqq\inf_{y\in[0,1]^d}q(y)$. For every $A>0$, there exists a constant $C_A>0$, depending only on $A$ and the quantities above, such that, for all $n,m\ge N_0$,
\begin{align}
\label{eq:cube_kernel_lower_bound_event}
    \mathbb{P}\Bigl(\inf_{x\in[0,1]^d}\tilde p_n^{\mathrm{ker}}(x)<\frac{\underline p}{2}\Bigr)
    \le C_An^{-A},
    \qquad
    \mathbb{P}\Bigl(\inf_{y\in[0,1]^d}\tilde q_m^{\mathrm{ker}}(y)<\frac{\underline q}{2}\Bigr)
    \le C_Am^{-A}.
\end{align}
\end{proposition}

\begin{proof}[Proof of Proposition \ref{prop:cube_kernel_density_rates}]
We prove \eqref{eq:cube_kernel_density_rate_source_w2}, \eqref{eq:cube_kernel_density_rate_source_l1}, and the first bound in \eqref{eq:cube_kernel_lower_bound_event}. The same argument applied to $q$ and $Y_1,\dots,Y_m$ proves \eqref{eq:cube_kernel_density_rate_target_w2}, \eqref{eq:cube_kernel_density_rate_target_l1}, and the second bound in \eqref{eq:cube_kernel_lower_bound_event}. For brevity write
\begin{align*}
    \tau_n(k)
    \coloneqq
    \prod_{r=1}^d
    \tau\left(\frac{\pi^2 k_r^2}{L_n^2}\right),
    \qquad k\in\mathbb N_0^d.
\end{align*}
Let
\begin{align*}
    \theta_k \coloneqq \int_{[0,1]^d} p(x)\eta_k(x)\,dx,
    \qquad
    \hat\theta_k \coloneqq \frac{1}{n}\sum_{i=1}^n \eta_k(X_i),
\end{align*}
so that $\theta_0=1$ and
\begin{equation}\label{eq:tildeker}
    \tilde p_n^{\mathrm{ker}}(x)
    =
    \sum_{k\in\mathbb N_0^d}
    \tau_n(k)\hat\theta_k\eta_k(x).
\end{equation}
Write
\begin{align*}
    p_{L_n}(x)
    \coloneqq
    \Ep\bigl[\tilde p_n^{\mathrm{ker}}(x)\bigr]
    =
    \sum_{k\in\mathbb N_0^d}
    \tau_n(k)\theta_k\eta_k(x).
\end{align*}
Equivalently, $p_{L_n}$ is the restriction to $[0,1]^d$ of $\mathcal{K}_n\mathcal Ep$, where the multiplier operator $\mathcal{K}_n$ on $C(\mathbb T_2^d)$ is $\mathcal{K}_nf\coloneqq\sum_{k\in\mathbb Z^d}\tau_n(|k_1|,\dots,|k_d|)\hat f(k)e^{i\pi k\cdot(\,\cdot\,)}$, with $\hat f(k)$ the $k$th Fourier coefficient of $f$. By Assumption~\ref{assm:cube_kernel_smoothness} and Definition~\ref{def:neumann_holder_class}, the even reflection $\mathcal Ep$ belongs to $C^s(\mathbb T_2^d)$ with $\|\mathcal Ep\|_{C^s(\mathbb T_2^d)}\le M$.

We bound the bias in the sup norm through the moments of the kernel of $\mathcal{K}_n$. Let $\upsilon:\R\to[0,1]$ be $\upsilon(u)\coloneqq\tau(u^2)$. Since $\tau$ is smooth with $\tau\equiv1$ on $[0,1]$ and $\tau\equiv0$ on $[2,\infty)$, we have $\upsilon\in C_c^\infty(\R)$ and $\upsilon\equiv1$ on $[-1,1]$. Define the function $\varpi:\R\to\R$ by
\begin{align*}
    \varpi(y)\coloneqq\frac{1}{2\pi}\int_\R\upsilon(u)e^{iuy}\,du.
\end{align*}
Then $\varpi$ is a Schwartz function, and
\begin{align*}
    \int_\R\varpi=\upsilon(0)=1,
    \qquad
    \int_\R y^j\varpi(y)\,dy=0
    \quad\text{for every integer }j\ge1.
\end{align*}
Let $\varpi_d:\R^d\to\R$ be $\varpi_d(y)\coloneqq\prod_{r=1}^d\varpi(y_r)$. Since $\tau_n(k)=\prod_{r=1}^d\upsilon(\pi k_r/L_n)$, the operator $\mathcal{K}_n$ is the convolution
\begin{align}
\label{eq:cube_kernel_convolution}
    \mathcal{K}_nf(x)=\int_{\R^d}f\bigl(x-y/L_n\bigr)\varpi_d(y)\,dy,
    \qquad x\in\R^d,
\end{align}
where we identify $f\in C(\mathbb T_2^d)$ with its periodic extension to $\R^d$. Both sides of \eqref{eq:cube_kernel_convolution} are continuous in $f$ for the sup norm and agree on the functions $e^{i\pi k\cdot(\,\cdot\,)}$, whose span is dense in $C(\mathbb T_2^d)$.

Let $\ell\coloneqq\lceil s\rceil-1$. The derivatives of $\mathcal Ep$ of order $\ell$ are $(s-\ell)$-H\"older with constant at most $M$. Thus, Taylor's formula yields, for all $x,y\in\R^d$,
\begin{align*}
    \Bigl|(\mathcal Ep)\bigl(x-y/L_n\bigr)-\sum_{|\flat|\le\ell}\frac{\partial^\flat(\mathcal Ep)(x)}{\flat!}\Bigl(-\frac{y}{L_n}\Bigr)^\flat\Bigr|
    \le
    CM\,\|y\|^sL_n^{-s},
\end{align*}
where the sum is over multi-indices $\flat\in\mathbb N_0^d$. The moments of $\varpi_d$ of positive order vanish and $\int_{\R^d}\varpi_d=1$. Hence, integrating against $\varpi_d$ in \eqref{eq:cube_kernel_convolution} removes the Taylor polynomial except for $(\mathcal Ep)(x)$. Since $\int_{\R^d}\|y\|^s|\varpi_d(y)|\,dy<\infty$, we obtain
\begin{align}
\label{eq:cube_kernel_bias_sup}
    \|p_{L_n}-p\|_{L^\infty([0,1]^d)}
    \le
    \|\mathcal Ep-\mathcal{K}_n\mathcal Ep\|_{L^\infty(\mathbb T_2^d)}
    \le
    C L_n^{-s}.
\end{align}
Since $[0,1]^d$ has volume one, it follows that
\begin{align}
\label{eq:cube_kernel_bias_l2}
    \|p_{L_n}-p\|_{L^2([0,1]^d)}
    \le
    C L_n^{-s}.
\end{align}
Moreover, if the $k$th coefficient of $p_{L_n}-p$ is nonzero, then $\tau_n(k)\neq 1$, so for some coordinate $r$ one has $\pi^2k_r^2/L_n^2\ge 1$ and therefore $\varrho_k\ge \pi^2k_r^2\ge L_n^2$. Hence, by \eqref{eq:neumann_hminus1_norm},
\begin{align}
\label{eq:cube_kernel_bias_hminus1}
    \|p_{L_n}-p\|_{H^{-1}_N([0,1]^d)}^2
    &=
    \sum_{k\neq 0}
    \frac{|1-\tau_n(k)|^2|\theta_k|^2}{\varrho_k}
    \\
    &\le
    L_n^{-2}
    \sum_{k\neq 0}|1-\tau_n(k)|^2|\theta_k|^2
    \\
    &=
    L_n^{-2}\|p_{L_n}-p\|_{L^2([0,1]^d)}^2
    \le
    C L_n^{-2(s+1)}
    =
    C L_n^{-2\alpha},
\end{align}
where the last inequality uses \eqref{eq:cube_kernel_bias_l2}.

For the stochastic term $\tilde p_n^{\mathrm{ker}}-p_{L_n}$, orthonormality of $(\eta_k)$ and the bound $p\le\beta_{\max}$ imply
\begin{align}
\label{eq:cube_kernel_variance_l2}
\begin{aligned}
    \Ep\bigl[\|\tilde p_n^{\mathrm{ker}}-p_{L_n}\|_{L^2([0,1]^d)}^2\bigr]
    &=
    \sum_{k\neq 0}
    \tau_n(k)^2
    \mathrm{Var}(\hat\theta_k)
    \\
    &\le
    \frac{C}{n}
    \#\Bigl\{k\in\mathbb N_0^d:\tau_n(k)\neq 0\Bigr\}
    \\
    &\le
    C\frac{L_n^d}{n}.
\end{aligned}
\end{align}
Likewise,
\begin{align}
\label{eq:cube_kernel_variance_hminus1}
\begin{aligned}
    \Ep\bigl[\|\tilde p_n^{\mathrm{ker}}-p_{L_n}\|_{H^{-1}_N([0,1]^d)}^2\bigr]
    &=
    \sum_{k\neq 0}
    \frac{\tau_n(k)^2}{\varrho_k}
    \mathrm{Var}(\hat\theta_k)
    \\
    &\le
    \frac{C}{n}
    \sum_{\substack{k\in\mathbb N_0^d\setminus\{0\}\\0\le k_r\lesssim L_n}}
    \|k\|^{-2}
    \\
    &\le
    C
    \begin{cases}
        n^{-1}, & d=1,
        \\
        (\log L_n)n^{-1}, & d=2,
        \\
        L_n^{d-2}n^{-1}, & d\ge 3.
    \end{cases}
\end{aligned}
\end{align}
Combining \eqref{eq:cube_kernel_bias_l2}, \eqref{eq:cube_kernel_bias_hminus1}, \eqref{eq:cube_kernel_variance_l2}, and \eqref{eq:cube_kernel_variance_hminus1}, and using the Cauchy-Schwarz and Jensen inequalities, we obtain
\begin{align}
\label{eq:cube_kernel_preliminary_l1}
    \Ep\bigl[\|\tilde p_n^{\mathrm{ker}}-p\|_{L^1([0,1]^d)}\bigr]
    &\le
    C\left(
        L_n^{-s}
        +
        \sqrt{\frac{L_n^d}{n}}
    \right),
    \\
\label{eq:cube_kernel_preliminary_hminus1}
    \Ep\bigl[\|\tilde p_n^{\mathrm{ker}}-p\|_{H^{-1}_N([0,1]^d)}^2\bigr]
    &\le
    C\left(
        L_n^{-2\alpha}
        +
        \begin{cases}
            n^{-1}, & d=1,
            \\
            (\log L_n)n^{-1}, & d=2,
            \\
            L_n^{d-2}n^{-1}, & d\ge 3
        \end{cases}
    \right).
\end{align}
We now derive the corresponding $L^1$ bound for the density estimate $\hat p_n^{\mathrm{ker}}$. Because $\eta_0\equiv 1$ on $[0,1]^d$, the orthonormality of $(\eta_k)_{k\in\mathbb N_0^d}$ yields $\int_{[0,1]^d}\eta_k(x)\,dx=0$ for every $k\neq0$. Moreover, $\tau(0)=1$, so the spectral form \eqref{eq:tildeker} and $\hat\theta_0=\tfrac{1}{n}\sum_{i=1}^n\eta_0(X_i)=1$ yield
\begin{align*}
    \int_{[0,1]^d}\tilde p_n^{\mathrm{ker}}(x)\,dx
    = \tau_n(0)\hat\theta_0
    = 1.
\end{align*}
Set $\tilde p_n^{+}\coloneqq(\tilde p_n^{\mathrm{ker}})_+$ and $Z\coloneqq\int_{[0,1]^d}\tilde p_n^{+}$. Since $\int_{[0,1]^d}\tilde p_n^{\mathrm{ker}}=1$, we have
\begin{align*}
    Z-1 = \int_{[0,1]^d}(\tilde p_n^{\mathrm{ker}})_-.
\end{align*}
Moreover, on the set $\{\tilde p_n^{\mathrm{ker}}<0\}$ we have $p-\tilde p_n^{\mathrm{ker}}\ge -\tilde p_n^{\mathrm{ker}}=(\tilde p_n^{\mathrm{ker}})_-$, hence
\begin{align*}
    Z-1 = \int_{[0,1]^d}(\tilde p_n^{\mathrm{ker}})_- \le \|\tilde p_n^{\mathrm{ker}}-p\|_{L^1([0,1]^d)}.
\end{align*}
Also,
\begin{align*}
    \|\tilde p_n^{+}-\tilde p_n^{\mathrm{ker}}\|_{L^1([0,1]^d)}
    = \int_{[0,1]^d}(\tilde p_n^{\mathrm{ker}})_-
    = Z-1,
\end{align*}
so by the triangle inequality,
\begin{align}
\label{eq:cube_kernel_positive_part_triangle}
    \|\tilde p_n^{+}-p\|_{L^1([0,1]^d)}
    \le \|\tilde p_n^{+}-\tilde p_n^{\mathrm{ker}}\|_{L^1([0,1]^d)} + \|\tilde p_n^{\mathrm{ker}}-p\|_{L^1([0,1]^d)}
    \le 2\|\tilde p_n^{\mathrm{ker}}-p\|_{L^1([0,1]^d)}.
\end{align}
Finally,
\begin{align*}
    \left\|\frac{\tilde p_n^{+}}{Z}-\tilde p_n^{+}\right\|_{L^1([0,1]^d)}
    = |1-Z| = Z-1
    \le \|\tilde p_n^{\mathrm{ker}}-p\|_{L^1([0,1]^d)},
\end{align*}
which, combined with \eqref{eq:cube_kernel_positive_part_triangle}, implies
\begin{align}
\label{eq:cube_kernel_positive_part_l1}
    \|\hat p_n^{\mathrm{ker}}-p\|_{L^1([0,1]^d)}
    \le 3\|\tilde p_n^{\mathrm{ker}}-p\|_{L^1([0,1]^d)}.
\end{align}
Taking expectations and combining with \eqref{eq:cube_kernel_preliminary_l1} yields
\begin{align}
\label{eq:cube_kernel_l1_after_renormalization}
    \Ep\bigl[\|\hat p_n^{\mathrm{ker}}-p\|_{L^1([0,1]^d)}\bigr]
    \le
    C\left(
        L_n^{-s}
        +
        \sqrt{\frac{L_n^d}{n}}
    \right).
\end{align}

It remains to prove \eqref{eq:cube_kernel_density_rate_source_w2}. We have $\underline p\ge\beta_{\min}>0$. By \eqref{eq:cube_kernel_bias_sup}, there exists an integer $N_0\ge3$, depending only on $d,\alpha,M,\beta_{\min},\beta_{\max}$, and the function $\tau$, such that $\|p_{L_n}-p\|_{L^\infty([0,1]^d)}\le\beta_{\min}/4$ for all $n\ge N_0$. Moreover, the product form \eqref{eq:uot_cube_kernel} implies
\begin{align*}
    \sup_{x\in[0,1]^d} K_{L_n}(x,x)
    =
    \sup_{x\in[0,1]^d}
    \prod_{r=1}^d K_{L_n}^{(1)}(x_r,x_r)
    \le
    C L_n^d.
\end{align*}
Define
\begin{align*}
    P f \coloneqq \int f(z)p(z)\,dz,
    \qquad
    P_n f \coloneqq \frac{1}{n}\sum_{i=1}^n f(X_i).
\end{align*}
Then
\begin{align*}
    \|\tilde p_n^{\mathrm{ker}}-p_{L_n}\|_{L^\infty([0,1]^d)}
    =
    \sup_{x\in[0,1]^d}|(P_n-P)K_{L_n}(x,\cdot)|.
\end{align*}
For every $x,z\in[0,1]^d$, the Cauchy-Schwarz inequality in the cosine expansion yields
\begin{align*}
    |K_{L_n}(x,z)|
    \le
    K_{L_n}(x,x)^{1/2}K_{L_n}(z,z)^{1/2}
    \le
    C L_n^d,
\end{align*}
so the class $\mathcal F_n\coloneqq\{K_{L_n}(x,\cdot):x\in[0,1]^d\}$ has envelope $U_n\le C L_n^d$. Likewise, using $p\le \beta_{\max}$, orthonormality of $(\eta_k)$, and $0\le \tau_n(k)\le 1$,
\begin{align*}
    \mathrm{Var}(K_{L_n}(x,X_1))
    &\le
    \Ep[K_{L_n}(x,X_1)^2]
    \\
    &=
    \int K_{L_n}(x,z)^2 p(z)\,dz
    \\
    &\le
    \beta_{\max}\int K_{L_n}(x,z)^2\,dz
    \\
    &=
    \beta_{\max}\sum_{k\in\mathbb N_0^d} \tau_n(k)^2\eta_k(x)^2
    \le
    \beta_{\max}K_{L_n}(x,x)
    \le
    C L_n^d.
\end{align*}
Hence, $\sigma_n^2\coloneqq \sup_{x}\mathrm{Var}(K_{L_n}(x,X_1))\le C L_n^d$. We next pass from a fixed $x$ to the supremum through a finite grid. Since $|K_{L_n}^{(1)}(u,v)|\le CL_n$ and $|\partial_uK_{L_n}^{(1)}(u,v)|\le CL_n^2$ for $u,v\in[0,1]$, the product form \eqref{eq:uot_cube_kernel} yields
\begin{align*}
    \|K_{L_n}(x,\cdot)-K_{L_n}(x',\cdot)\|_{L^\infty([0,1]^d)}\le CL_n^{d+1}\|x-x'\|,
    \qquad x,x'\in[0,1]^d.
\end{align*}
Hence, there exists a finite set $\mathcal X_n\subset[0,1]^d$ with $|\mathcal X_n|\le CL_n^{d(d+1)}$ such that
\begin{align*}
    \|\tilde p_n^{\mathrm{ker}}-p_{L_n}\|_{L^\infty([0,1]^d)}
    =
    \sup_{x\in[0,1]^d}|(P_n-P)K_{L_n}(x,\cdot)|
    \le
    \max_{x\in\mathcal X_n}|(P_n-P)K_{L_n}(x,\cdot)|+\frac{\beta_{\min}}{8}.
\end{align*}
Consequently, the event $\{\|\tilde p_n^{\mathrm{ker}}-p_{L_n}\|_{L^\infty([0,1]^d)}>\beta_{\min}/4\}$ is contained in the union of the events $\{|(P_n-P)K_{L_n}(x,\cdot)|>\beta_{\min}/8\}$ over $x\in\mathcal X_n$. For each fixed $x$, the variables $K_{L_n}(x,X_i)-PK_{L_n}(x,\cdot)$ are i.i.d., centered, bounded by $2U_n$, and have variance at most $\sigma_n^2$. Therefore, Bernstein's inequality and the union bound over $\mathcal X_n$ yield, for a constant $c>0$,
\begin{align*}
    \mathbb{P}\Bigl(
        \|\tilde p_n^{\mathrm{ker}}-p_{L_n}\|_{L^\infty([0,1]^d)}
        >
        \frac{\beta_{\min}}{4}
    \Bigr)
    &\le
    2|\mathcal X_n|\exp\Bigl(-\frac{n(\beta_{\min}/8)^2}{2\sigma_n^2+(4/3)U_n(\beta_{\min}/8)}\Bigr)
    \\
    &\le
    CL_n^{d(d+1)}\exp\Bigl(-\frac{cn}{L_n^d}\Bigr).
\end{align*}
Since $L_n=n^{1/(d+2s)}$, we have $n/L_n^d=n^{2s/(d+2s)}$ with $s>0$. Thus, for every $A>0$, there exists a constant $C_A>0$ such that
\begin{equation}\label{eq:bern}
    \mathbb{P}\Bigl(
        \|\tilde p_n^{\mathrm{ker}}-p_{L_n}\|_{L^\infty([0,1]^d)}
        >
        \frac{\beta_{\min}}{4}
    \Bigr)
    \le
    C_A n^{-A}.
\end{equation}
For $n\ge N_0$, we have $\|p_{L_n}-p\|_{L^\infty([0,1]^d)}\le\beta_{\min}/4$, so the event $\{\|\tilde p_n^{\mathrm{ker}}-p\|_{L^\infty([0,1]^d)}>\beta_{\min}/2\}$ is contained in the event $\{\|\tilde p_n^{\mathrm{ker}}-p_{L_n}\|_{L^\infty([0,1]^d)}>\beta_{\min}/4\}$. Therefore, by \eqref{eq:bern}, for all $n\ge N_0$, the event
\begin{align*}
    \mathcal G_n
    \coloneqq
    \left\{
        \|\tilde p_n^{\mathrm{ker}}-p\|_{L^\infty([0,1]^d)}
        \le
        \frac{\beta_{\min}}{2}
    \right\}
\end{align*}
satisfies
\begin{align}
\label{eq:cube_kernel_Gn_complement}
    \mathbb{P}(\mathcal G_n^c)\le C_A n^{-A}.
\end{align}
Since $\beta_{\min}\le\underline p$, we have $\tilde p_n^{\mathrm{ker}}\ge\underline p-\beta_{\min}/2\ge\underline p/2$ on $\mathcal G_n$. This proves the first bound in \eqref{eq:cube_kernel_lower_bound_event}. On $\mathcal G_n$ we have $\tilde p_n^{\mathrm{ker}}\ge \underline p/2>0$, so $(\tilde p_n^{\mathrm{ker}})_+=\tilde p_n^{\mathrm{ker}}$. Since $\int_{[0,1]^d}\tilde p_n^{\mathrm{ker}}=1$, it follows that
\begin{align*}
    \hat p_n^{\mathrm{ker}}
    =
    \tilde p_n^{\mathrm{ker}}.
\end{align*}
Since $p\ge \underline p$ and $\tilde p_n^{\mathrm{ker}}\ge \underline p/2$ on $\mathcal G_n$, Lemma~\ref{lem:cube_hminus1_to_w2} yields
\begin{align*}
    W_2^2\bigl(\hat p_n^{\mathrm{ker}}\,dx,p\,dx\bigr)
    =
    W_2^2\bigl(\tilde p_n^{\mathrm{ker}}\,dx,p\,dx\bigr)
    \le
    \frac{2}{\underline p}
    \|\tilde p_n^{\mathrm{ker}}-p\|_{H^{-1}_N([0,1]^d)}^2
    \qquad \text{on }\mathcal G_n.
\end{align*}
On $\mathcal G_n^c$, both $\hat p_n^{\mathrm{ker}}\,dx$ and $p\,dx$ are supported in $[0,1]^d$, so
\begin{align*}
    W_2^2\bigl(\hat p_n^{\mathrm{ker}}\,dx,p\,dx\bigr)
    \le d.
\end{align*}
Taking expectations and using \eqref{eq:cube_kernel_preliminary_hminus1} and \eqref{eq:cube_kernel_Gn_complement} with $A=2$ therefore yields
\begin{align}
\label{eq:cube_kernel_w2_probability}
    \Ep\bigl[W_2^2(\hat p_n^{\mathrm{ker}}\,dx,p\,dx)\bigr]
    \le
    C\left(
        L_n^{-2\alpha}
        +
        \begin{cases}
            n^{-1}, & d=1,
            \\
            (\log L_n)n^{-1}, & d=2,
            \\
            L_n^{d-2}n^{-1}, & d\ge 3
        \end{cases}
    \right).
\end{align}
Since $L_n=n^{1/(d+2s)}$, \eqref{eq:cube_kernel_l1_after_renormalization} yields
\begin{align*}
    \Ep\bigl[\|\hat p_n^{\mathrm{ker}}-p\|_{L^1([0,1]^d)}\bigr]
    \le
    C n^{{-s/(2s+d)}}
    =
    C\mathfrak L_n(\alpha),
\end{align*}
while \eqref{eq:cube_kernel_w2_probability} becomes
\begin{align*}
    \Ep\bigl[W_2^2(\hat p_n^{\mathrm{ker}}\,dx,p\,dx)\bigr]
    \le
    C\mathfrak R_n^{\mathrm{ker}}(\alpha).
\end{align*}
Finally,
\begin{align*}
    W_2^2(\tilde\mu_n^{\mathrm{ker}},\mu)
    =
    M_\mu W_2^2(\hat p_n^{\mathrm{ker}}\,dx,p\,dx),
\end{align*}
which proves \eqref{eq:cube_kernel_density_rate_source_w2} and \eqref{eq:cube_kernel_density_rate_source_l1}. The same argument applied to $q$ and $Y_1,\dots,Y_m$ proves \eqref{eq:cube_kernel_density_rate_target_w2}, \eqref{eq:cube_kernel_density_rate_target_l1}, and the second bound in \eqref{eq:cube_kernel_lower_bound_event}.
\end{proof}

\subsection{Proof of the upper bound}

We are now in a position to prove \eqref{eq:section4_cube_kernel_map_rate} and \eqref{eq:section4_cube_kernel_lambda_rate}.

\begin{corollary}[Risk of the kernel-based estimator]
\label{cor:uot_cube_kernel_rate_transfer}
Assume that $\Omega=[0,1]^d$, $c(x,y)=\tfrac{1}{2}\|x-y\|^2$, that Assumptions~\ref{assm:curvature}, \ref{assm:positivity}, and \ref{assm:masses} hold, and that the hypotheses of Proposition~\ref{prop:cube_kernel_density_rates} hold. Let
\[
    \hat\mu_n^{\mathrm{ker}},\ \hat\nu_m^{\mathrm{ker}},\ \hat T_{nm}^{\mathrm{ker}},\ \hat\lambda_{nm}^{\mathrm{ker}},\ \widehat{\mathrm{UOT}}_{nm}^{\mathrm{ker}}
\]
be as in Section~\ref{sec:kernel} and Appendix~\ref{sec:proof_cube_kernel_rates}, and let $a_n,b_m$ be the sequences from Assumption~\ref{assm:masses}. Suppose that, for some deterministic sequences $r_n^\mu,r_m^\nu,\ell_n,\ell_m\ge 0$, where $r_n^\mu \ge M_\mu/n$ and $r_m^\nu \ge M_\nu/m$,
\begin{align}
\label{eq:uot_cube_kernel_assumed_w2}
    \Ep\bigl[W_2^2(\tilde\mu_n^{\mathrm{ker}},\mu)\bigr] \le r_n^\mu,
    \qquad
    \Ep\bigl[W_2^2(\tilde\nu_m^{\mathrm{ker}},\nu)\bigr] \le r_m^\nu,
\end{align}
\begin{align}
    \Ep\bigl[\|\hat p_n^{\mathrm{ker}}-p\|_{L^1([0,1]^d)}\bigr] \le \ell_n,
    \qquad
    \Ep\bigl[\|\hat q_m^{\mathrm{ker}}-q\|_{L^1([0,1]^d)}\bigr] \le \ell_m.
\end{align}
Then there exist constants $C>0$ and $N_0\in\mathbb{N}$, depending only on $d$, $\alpha$, $M$, $\beta_{\min}$, $\beta_{\max}$, $w_\pm$, the function $\tau$, and the constants $\kappa$ and $\bar c$ of Assumptions~\ref{assm:curvature} and \ref{assm:masses}, such that, for all $n,m\ge N_0$,
\begin{align}
\label{eq:uot_cube_kernel_rate_transfer_map}
    \Ep\left[
        \int_{[0,1]^d}
        \|\hat T_{nm}^{\mathrm{ker}}(x)-T_0(x)\|^2\,d\mu(x)
    \right]
    \le C(r_n^\mu+r_m^\nu).
\end{align}
If, in addition, $n$ and $m$ are large enough that $\Ep[|\hat M_\mu-M_\mu|]\le a_n$ and $\Ep[|\hat M_\nu-M_\nu|]\le b_m$, then
\begin{align}
\label{eq:uot_cube_kernel_rate_transfer_lambda}
    \Ep\left[
        \int_{[0,1]^d}
        |\hat\lambda_{nm}^{\mathrm{ker}}(x)-\lambda_0(x)|^2\,d\mu(x)
    \right]
    &\le C(r_n^\mu+r_m^\nu+a_n+b_m),
\end{align}
and
\begin{align}
\label{eq:uot_cube_kernel_rate_transfer_abs_value}
\begin{aligned}
    \Ep\left[
        \left|
            \widehat{\mathrm{UOT}}_{nm}^{\mathrm{ker}}
            -
            \mathrm{UOT}(\mu,\nu)
        \right|
    \right]
    &\le M_\mu\|\zeta_0\|_{L^\infty([0,1]^d)}\ell_n
    + M_\nu\|\xi_0\|_{L^\infty([0,1]^d)}\ell_m \\
    &\qquad + C(r_n^\mu+r_m^\nu+a_n+b_m).
\end{aligned}
\end{align}
\end{corollary}

\begin{proof}[Proof of Corollary \ref{cor:uot_cube_kernel_rate_transfer}]
Let $\hat\gamma_{nm}^{\mathrm{ker}}$ be an optimal transport plan for $\widehat{\mathrm{UOT}}_{nm}^{\mathrm{ker}}$, with marginals $\hat\gamma_{0,nm}^{\mathrm{ker}}$ and $\hat\gamma_{1,nm}^{\mathrm{ker}}$.

We first derive an algebraic identity that decomposes $\widehat{\mathrm{UOT}}_{nm}^{\mathrm{ker}}$ into an integral against $\hat\gamma_{nm}^{\mathrm{ker}}$ and two $D_\KL$ terms relative to $e^{-\varphi_0}\hat\mu_n^{\mathrm{ker}}$ and $e^{-\psi_0}\hat\nu_m^{\mathrm{ker}}$. By optimality of $\hat\gamma_{nm}^{\mathrm{ker}}$ for $\mathrm{UOT}(\hat\mu_n^{\mathrm{ker}},\hat\nu_m^{\mathrm{ker}})$,
\begin{align}
\label{eq:uot_cube_kernel_exact_excess_start}
    \widehat{\mathrm{UOT}}_{nm}^{\mathrm{ker}}
    =
    \int \tfrac{1}{2}\|x-y\|^2\,d\hat\gamma_{nm}^{\mathrm{ker}}(x,y)
    + {D_\KL}\bigl(\hat\gamma_{0,nm}^{\mathrm{ker}}\mid \hat\mu_n^{\mathrm{ker}}\bigr)
    + {D_\KL}\bigl(\hat\gamma_{1,nm}^{\mathrm{ker}}\mid \hat\nu_m^{\mathrm{ker}}\bigr).
\end{align}
We rewrite $D_\KL(\hat\gamma_{0,nm}^{\mathrm{ker}}\mid\hat\mu_n^{\mathrm{ker}})$ with $\gamma_0^{\mathrm{or}}\coloneqq e^{-\varphi_0}\hat\mu_n^{\mathrm{ker}}$. If $f=d\hat\gamma_{0,nm}^{\mathrm{ker}}/d\hat\mu_n^{\mathrm{ker}}$, then $d\hat\gamma_{0,nm}^{\mathrm{ker}}/d\gamma_0^{\mathrm{or}} = fe^{\varphi_0}$, and therefore
\begin{align*}
    {D_\KL}\bigl(\hat\gamma_{0,nm}^{\mathrm{ker}}\mid \gamma_0^{\mathrm{or}}\bigr)
    &= \int\Bigl[f\log\bigl(fe^{\varphi_0}\bigr)-f+e^{-\varphi_0}\Bigr]\,d\hat\mu_n^{\mathrm{ker}} \\
    &= {D_\KL}\bigl(\hat\gamma_{0,nm}^{\mathrm{ker}}\mid \hat\mu_n^{\mathrm{ker}}\bigr)
    + \int \varphi_0\,d\hat\gamma_{0,nm}^{\mathrm{ker}}
    - \int \zeta_0\,d\hat\mu_n^{\mathrm{ker}},
\end{align*}
where we used $\zeta_0=1-e^{-\varphi_0}$. Hence
\begin{align}
\label{eq:uot_cube_kernel_exact_excess_source}
    {D_\KL}\bigl(\hat\gamma_{0,nm}^{\mathrm{ker}}\mid \hat\mu_n^{\mathrm{ker}}\bigr)
    =
    {D_\KL}\bigl(\hat\gamma_{0,nm}^{\mathrm{ker}}\mid \gamma_0^{\mathrm{or}}\bigr)
    - \int \varphi_0\,d\hat\gamma_{0,nm}^{\mathrm{ker}}
    + \int \zeta_0\,d\hat\mu_n^{\mathrm{ker}}.
\end{align}
The same computation with $\gamma_1^{\mathrm{or}}\coloneqq e^{-\psi_0}\hat\nu_m^{\mathrm{ker}}$ yields
\begin{align}
\label{eq:uot_cube_kernel_exact_excess_target}
    {D_\KL}\bigl(\hat\gamma_{1,nm}^{\mathrm{ker}}\mid \hat\nu_m^{\mathrm{ker}}\bigr)
    =
    {D_\KL}\bigl(\hat\gamma_{1,nm}^{\mathrm{ker}}\mid \gamma_1^{\mathrm{or}}\bigr)
    - \int \psi_0\,d\hat\gamma_{1,nm}^{\mathrm{ker}}
    + \int \xi_0\,d\hat\nu_m^{\mathrm{ker}}.
\end{align}
Substituting \eqref{eq:uot_cube_kernel_exact_excess_source} and \eqref{eq:uot_cube_kernel_exact_excess_target} into \eqref{eq:uot_cube_kernel_exact_excess_start} and using
\begin{align*}
    \int \varphi_0\,d\hat\gamma_{0,nm}^{\mathrm{ker}} + \int \psi_0\,d\hat\gamma_{1,nm}^{\mathrm{ker}}
    =
    \int \bigl(\varphi_0(x)+\psi_0(y)\bigr)\,d\hat\gamma_{nm}^{\mathrm{ker}}(x,y),
\end{align*}
we obtain
\begin{align}
\label{eq:uot_cube_kernel_exact_excess}
\begin{aligned}
&\widehat{\mathrm{UOT}}_{nm}^{\mathrm{ker}}
- \int \zeta_0\,d\hat\mu_n^{\mathrm{ker}}
- \int \xi_0\,d\hat\nu_m^{\mathrm{ker}} \\
&\qquad=
\int_{[0,1]^d\times[0,1]^d}
\Bigl(
\tfrac{1}{2}\|x-y\|^2-\varphi_0(x)-\psi_0(y)
\Bigr)
 \,d\hat\gamma_{nm}^{\mathrm{ker}}(x,y) \\
&\qquad\quad
+ {D_\KL}\bigl(\hat\gamma_{0,nm}^{\mathrm{ker}}\mid e^{-\varphi_0}\hat\mu_n^{\mathrm{ker}}\bigr)
+ {D_\KL}\bigl(\hat\gamma_{1,nm}^{\mathrm{ker}}\mid e^{-\psi_0}\hat\nu_m^{\mathrm{ker}}\bigr).
\end{aligned}
\end{align}
We define
\begin{align}
\label{eq:uot_cube_kernel_value_gap}
    V_{nm}^{\mathrm{ker}}
    \coloneqq
    \widehat{\mathrm{UOT}}_{nm}^{\mathrm{ker}}
    -\mathrm{UOT}(\mu,\nu)
    -\int\zeta_0\,d(\hat\mu_n^{\mathrm{ker}}-\mu)
    -\int\xi_0\,d(\hat\nu_m^{\mathrm{ker}}-\nu).
\end{align}
Assumption~\ref{assm:supp_global} holds on $\Omega=[0,1]^d$. The measures $\hat\mu_n^{\mathrm{ker}}$ and $\hat\nu_m^{\mathrm{ker}}$ satisfy $\hat\mu_n^{\mathrm{ker}}=\rho_{\mu}\tilde\mu_n^{\mathrm{ker}}$ and $\hat\nu_m^{\mathrm{ker}}=\rho_{\nu}\tilde\nu_m^{\mathrm{ker}}$ with $\tilde\mu_n^{\mathrm{ker}}(\Omega)=M_\mu$ and $\tilde\nu_m^{\mathrm{ker}}(\Omega)=M_\nu$. Hence, \eqref{eq:estimated_mass_stability} holds almost surely with $\hat\mu_n,\hat\nu_m,\tilde\mu_n,\tilde\nu_m$ replaced by $\hat\mu_n^{\mathrm{ker}},\hat\nu_m^{\mathrm{ker}},\tilde\mu_n^{\mathrm{ker}},\tilde\nu_m^{\mathrm{ker}}$, by Remark~\ref{rem:stability_general_measures}. Moreover, $|\int\zeta_0\,d(\hat\mu_n^{\mathrm{ker}}-\tilde\mu_n^{\mathrm{ker}})|\le\|\zeta_0\|_{L^\infty([0,1]^d)}|\hat M_\mu-M_\mu|$. The analogous bound holds for $\xi_0$. Thus, subtracting $\int\zeta_0\,d(\hat\mu_n^{\mathrm{ker}}-\mu)+\int\xi_0\,d(\hat\nu_m^{\mathrm{ker}}-\nu)$ from both sides of \eqref{eq:estimated_mass_stability}, we obtain, almost surely,
\begin{align*}
    V_{nm}^{\mathrm{ker}}
    &\le C_\Lambda\Bigl(W_2^2(\tilde\mu_n^{\mathrm{ker}},\mu)+W_2^2(\tilde\nu_m^{\mathrm{ker}},\nu)\Bigr) \\
    &\quad+M_\mu|\rho_{\mu}-1|+e\sqrt{M_\mu M_\nu}\,|\log\rho_{\mu}|+\|\zeta_0\|_{L^\infty([0,1]^d)}|\hat M_\mu-M_\mu| \\
    &\quad+M_\nu|\rho_{\nu}-1|+e\sqrt{M_\mu M_\nu}\,|\log\rho_{\nu}|+\|\xi_0\|_{L^\infty([0,1]^d)}|\hat M_\nu-M_\nu|.
\end{align*}
For all $n,m$ sufficiently large such that $\Ep[|\hat M_\mu-M_\mu|]\le a_n$ and $\Ep[|\hat M_\nu-M_\nu|]\le b_m$, taking expectations in the previous display and applying \eqref{eq:uot_cube_kernel_assumed_w2} and \eqref{eq:mass_error_expectation}, with its analogue for $\hat M_\nu$, we obtain
\begin{equation}
\label{eq:uot_cube_kernel_value_expansion_expect}
\Ep\bigl[V_{nm}^{\mathrm{ker}}\bigr]
\le
C(r_n^\mu+r_m^\nu+a_n+b_m).
\end{equation}
The constant $C$ in \eqref{eq:uot_cube_kernel_value_expansion_expect} depends on the masses only through the ratio $M_\nu/M_\mu$. Since $M_\nu=\int\lambda_0^2\,d\mu$ and $w_-\le\lambda_0\le w_+e^{d/4}$, we have
\begin{align*}
    w_-^2\le \frac{M_\nu}{M_\mu}\le w_+^2e^{d/2}.
\end{align*}
Moreover, $\psi_0=\varphi_0^c\ge\log w_-$, so
\begin{align*}
    \|\zeta_0\|_{L^\infty([0,1]^d)}\le1+w_+,
    \qquad
    \|\xi_0\|_{L^\infty([0,1]^d)}\le1+\frac{1}{w_-}.
\end{align*}
Finally, \eqref{eq:C_Lambda_bound} bounds $C_\Lambda$.

Next, using the notation from Proposition~\ref{prop:cube_kernel_density_rates}, we define the event $\mathcal{G}_{nm}$ by
\begin{align}
    \mathcal{G}_{nm}
    \coloneqq
    \left\{
        \inf_{x \in [0,1]^d} \tilde p_n^{\mathrm{ker}}(x) \ge \frac{\underline p}{2}
        \quad \text{and} \quad
        \inf_{y \in [0,1]^d} \tilde q_m^{\mathrm{ker}}(y) \ge \frac{\underline q}{2}
    \right\},
\end{align}
where $\tilde p_n^{\mathrm{ker}}$ is defined in \eqref{eq:tilde_ker_def}, $\underline p = \inf_{x\in[0,1]^d} p(x)$, and $\tilde q_m^{\mathrm{ker}}$ and $\underline q$ are defined analogously. On $\mathcal{G}_{nm}$, the measure estimates have Lebesgue densities bounded below by $\hat M_\mu\underline p/2$ and $\hat M_\nu\underline q/2$. Hence, Theorem~\ref{prop:monge-potential} applies to the pair $(\hat\mu_n^{\mathrm{ker}},\hat\nu_m^{\mathrm{ker}})$. Moreover, replacing $\hat\psi_{nm}^{\mathrm{ker}}$ by $(\hat\varphi_{nm}^{\mathrm{ker}})^c\ge\hat\psi_{nm}^{\mathrm{ker}}$ preserves feasibility in \eqref{eq:gh_dual} and does not decrease the dual objective. Since $\hat\nu_m^{\mathrm{ker}}$ has a positive density, optimality and continuity yield
\begin{align*}
    \hat\psi_{nm}^{\mathrm{ker}}=(\hat\varphi_{nm}^{\mathrm{ker}})^c
    \qquad\text{on }[0,1]^d.
\end{align*}
The same argument yields $\hat\varphi_{nm}^{\mathrm{ker}}=(\hat\psi_{nm}^{\mathrm{ker}})^c$. Thus, $(\hat\varphi_{nm}^{\mathrm{ker}},\hat\psi_{nm}^{\mathrm{ker}})$ is the optimal dual pair given by Theorem~\ref{prop:monge-potential}. Using the same argument as in the proof of Proposition~\ref{prop:two_sample_stability}, we obtain
\begin{align*}
    \hat\gamma_{nm}^{\mathrm{ker}}=(\mathrm{id},\hat T_{nm}^{\mathrm{ker}})_\#\hat\gamma_{0,nm}^{\mathrm{ker}},
    \qquad
    \hat\gamma_{0,nm}^{\mathrm{ker}}=e^{-\hat\varphi_{nm}^{\mathrm{ker}}}\hat\mu_n^{\mathrm{ker}}=(\hat a_{nm}^{\mathrm{ker}})^2\hat\mu_n^{\mathrm{ker}}
    \qquad\text{on }\mathcal{G}_{nm}.
\end{align*}
Since $\mathrm{UOT}(\mu,\nu)=\int\zeta_0\,d\mu+\int\xi_0\,d\nu$, $V_{nm}^{\mathrm{ker}}$ is the left-hand side of \eqref{eq:uot_cube_kernel_exact_excess}. Applying Lemma~\ref{lem:uot_gap_sufficient} to the integrand in \eqref{eq:uot_cube_kernel_exact_excess}, we obtain, on $\mathcal{G}_{nm}$,
\begin{align}
\label{eq:uot_cube_kernel_active_marginal_bound}
    \frac{\kappa}{2}
    \int_{[0,1]^d}
    \|\hat T_{nm}^{\mathrm{ker}}(x)-T_0(x)\|^2\,d\hat\gamma_{0,nm}^{\mathrm{ker}}(x)
    \le V_{nm}^{\mathrm{ker}}.
\end{align}
Directly calculating ${D_\KL}\bigl(\hat\gamma_{0,nm}^{\mathrm{ker}}\mid e^{-\varphi_0}\hat\mu_n^{\mathrm{ker}}\bigr)$, applying Lemma~\ref{lem:uot_kl_sqrt_lower_bound} pointwise with $a=(\hat a_{nm}^{\mathrm{ker}})^2$ and $b=a_0^2$, and using \eqref{eq:uot_cube_kernel_exact_excess} yields the following bound on $\mathcal{G}_{nm}$ for $\hat a_{nm}^{\mathrm{ker}}$ and $a_0$:
\begin{align}
\label{eq:uot_cube_kernel_active_factor_bound}
    \int_{[0,1]^d}
    |\hat a_{nm}^{\mathrm{ker}}(x)-a_0(x)|^2\,d\hat\mu_n^{\mathrm{ker}}(x)
    \le
    {D_\KL}\bigl(\hat\gamma_{0,nm}^{\mathrm{ker}}\mid e^{-\varphi_0}\hat\mu_n^{\mathrm{ker}}\bigr)
    \le V_{nm}^{\mathrm{ker}}.
\end{align}

By the choice of the clipping constants in Section~\ref{sec:plan}, $a_0(x)^2 = e^{-\varphi_0(x)} \ge w_-$ for all $x \in [0,1]^d$. Also, since the images of both $\hat T_{nm}^{\mathrm{ker}}$ and $T_0$ are contained in $[0,1]^d$, we have $ \|\hat T_{nm}^{\mathrm{ker}}(x) - T_0(x)\|^2 \le \mathrm{diam}([0,1]^d)^2 = d. $ Consequently,
\begin{align}
\label{eq:helper_derivation_step1}
    \|\hat T_{nm}^{\mathrm{ker}}(x) - T_0(x)\|^2 
    &\le \frac{1}{w_-} a_0(x)^2 \|\hat T_{nm}^{\mathrm{ker}}(x) - T_0(x)\|^2 \\
    &\le \frac{2}{w_-} \hat a_{nm}^{\mathrm{ker}}(x)^2 \|\hat T_{nm}^{\mathrm{ker}}(x) - T_0(x)\|^2 \\
    &\quad + \frac{2}{w_-} | \hat a_{nm}^{\mathrm{ker}}(x) - a_0(x) |^2 \|\hat T_{nm}^{\mathrm{ker}}(x) - T_0(x)\|^2 \\
    &\le
    C_{\mathrm{tr}}\,\hat a_{nm}^{\mathrm{ker}}(x)^2\,\|\hat T_{nm}^{\mathrm{ker}}(x)-T_0(x)\|^2
    +
    C_{\mathrm{tr}}\,|\hat a_{nm}^{\mathrm{ker}}(x)-a_0(x)|^2,
\end{align}
with $C_{\mathrm{tr}}\coloneqq 2\max(1,d)/w_-$. We integrate this bound with respect to $d\hat\mu_n^{\mathrm{ker}}(x)$ and use $d\hat\gamma_{0,nm}^{\mathrm{ker}}(x) = \hat a_{nm}^{\mathrm{ker}}(x)^2\,d\hat\mu_n^{\mathrm{ker}}(x)$ on $\mathcal{G}_{nm}$. Hence, by \eqref{eq:uot_cube_kernel_active_marginal_bound} and \eqref{eq:uot_cube_kernel_active_factor_bound}, on $\mathcal{G}_{nm}$,
\begin{align}
\label{eq:uot_cube_kernel_T_to_mu_n}
\begin{aligned}
    \int_{[0,1]^d}
    \|\hat T_{nm}^{\mathrm{ker}}(x)-T_0(x)\|^2\,d\hat\mu_n^{\mathrm{ker}}(x)
    &\le
    C_{\mathrm{tr}} \int_{[0,1]^d} \|\hat T_{nm}^{\mathrm{ker}}(x)-T_0(x)\|^2\,d\hat\gamma_{0,nm}^{\mathrm{ker}}(x) \\
    &\quad + C_{\mathrm{tr}} \int_{[0,1]^d} |\hat a_{nm}^{\mathrm{ker}}(x)-a_0(x)|^2\,d\hat\mu_n^{\mathrm{ker}}(x) \\
    &\le C_{\mathrm{tr}} \bigl( 2/\kappa + 1 \bigr) V_{nm}^{\mathrm{ker}}.
\end{aligned}
\end{align}

On $\mathcal{G}_{nm}$, we have $\hat p_n^{\mathrm{ker}} = \tilde p_n^{\mathrm{ker}}$ and $\hat q_m^{\mathrm{ker}} = \tilde q_m^{\mathrm{ker}}$. Hence, $\hat p_n^{\mathrm{ker}}\ge\underline p/2$ and $\hat q_m^{\mathrm{ker}}\ge\underline q/2$. Moreover, Assumptions~\ref{assm:masses} and~\ref{assm:positivity} give $\hat M_\mu\ge \bar cM_\mu$ and $p\le\beta_{\max}$. Consequently, on $\mathcal{G}_{nm}$,
\begin{align}
\label{eq:kernel_density_domination}
    d\mu(x) = M_\mu p(x)\,dx
    \le \frac{2\beta_{\max}}{\bar c\,\underline p}\,\hat M_\mu\tilde p_n^{\mathrm{ker}}(x)\,dx
    = \frac{2\beta_{\max}}{\bar c\,\underline p}\,d\hat\mu_n^{\mathrm{ker}}(x).
\end{align}

Combining \eqref{eq:kernel_density_domination} with \eqref{eq:uot_cube_kernel_T_to_mu_n} yields, on $\mathcal{G}_{nm}$,
\begin{align}
\label{eq:uot_cube_w2_coupling}
\begin{aligned}
    \int_{[0,1]^d} \|\hat T_{nm}^{\mathrm{ker}}(x)-T_0(x)\|^2\,d\mu(x)
    &\le \frac{2\beta_{\max}}{\bar c\,\underline p}
    \int_{[0,1]^d} \|\hat T_{nm}^{\mathrm{ker}}(x)-T_0(x)\|^2\,d\hat\mu_n^{\mathrm{ker}}(x) \\
    &\le \frac{2\beta_{\max}}{\bar c\,\underline p}\, C_{\mathrm{tr}} \bigl( 2/\kappa + 1 \bigr) V_{nm}^{\mathrm{ker}}.
\end{aligned}
\end{align}
On the complement $\mathcal{G}_{nm}^c$, the density of $\hat\mu_n^{\mathrm{ker}}$ can vanish on a set of positive Lebesgue measure, so Theorem~\ref{prop:monge-potential} need not apply. We therefore construct an optimal dual pair for $\mathrm{UOT}(\hat\mu_n^{\mathrm{ker}},\hat\nu_m^{\mathrm{ker}})$. Let $(\varphi_j,\psi_j)_{j\ge1}$ be a maximizing sequence for the dual problem \eqref{eq:gh_dual} with the measures $\hat\mu_n^{\mathrm{ker}}$ and $\hat\nu_m^{\mathrm{ker}}$. Replacing $\psi_j$ by $\varphi_j^c$ and then $\varphi_j$ by $\psi_j^c$ preserves feasibility and does not decrease the dual objective, so we may assume that
\begin{align*}
    \psi_j=\varphi_j^c,
    \qquad
    \varphi_j=\psi_j^c.
\end{align*}
Since the pair $(0,0)$ is feasible, we may also assume that the dual objective at $(\varphi_j,\psi_j)$ is nonnegative. Assumption~\ref{assm:masses} yields $\hat M_\mu,\hat M_\nu>0$. Hence, the proof of Lemma~\ref{lem:dual_potential_bounds} applies to $(\varphi_j,\psi_j)$ and yields
\begin{align*}
    \sup_{j\ge1}\bigl(\|\varphi_j\|_{L^\infty([0,1]^d)}+\|\psi_j\|_{L^\infty([0,1]^d)}\bigr)<\infty.
\end{align*}
Moreover, the function $x\mapsto\frac12\|x-y\|^2$ is $\sqrt d$-Lipschitz on $[0,1]^d$ for every $y\in[0,1]^d$, so $\varphi_j$ and $\psi_j$ are $\sqrt d$-Lipschitz. By the Arzel\`a-Ascoli theorem, a subsequence of $(\varphi_j,\psi_j)_{j\ge1}$ converges uniformly to a pair $(\varphi,\psi)\in C_b([0,1]^d)^2$. This pair is feasible. Uniform convergence also shows that this pair attains the supremum in \eqref{eq:gh_dual}, so $(\varphi,\psi)$ is an optimal dual pair.

Thus, on the event $\mathcal{G}_{nm}^c$, the pair $(\hat\varphi_{nm}^{\mathrm{ker}},\hat\psi_{nm}^{\mathrm{ker}})$ exists. The set of minimizers in \eqref{eq:kernel_estimator} is nonempty because $\hat\psi_{nm}^{\mathrm{ker}}$ is continuous. Both $\hat T_{nm}^{\mathrm{ker}}$ and $T_0$ take values in the hypercube $[0,1]^d$, so $\|\hat T_{nm}^{\mathrm{ker}}(x)-T_0(x)\|^2\le\mathrm{diam}([0,1]^d)^2=d$ for every $x\in[0,1]^d$. We take $A=2$ in \eqref{eq:cube_kernel_lower_bound_event}. Hence, \eqref{eq:cube_kernel_lower_bound_event} and the union bound yield, for all $n,m\ge N_0$, where $N_0$ is the integer of Proposition~\ref{prop:cube_kernel_density_rates},
\begin{align}
\label{eq:uot_cube_kernel_complement}
\begin{aligned}
    \Ep\left[ \int_{[0,1]^d} \|\hat T_{nm}^{\mathrm{ker}}(x)-T_0(x)\|^2\,d\mu(x) \mathbf{1}_{\mathcal{G}_{nm}^c} \right]
    &\le
    dM_\mu C_A(n^{-2}+m^{-2})
    \\
    &\le
    dC_A\Bigl(1+\frac{M_\mu}{M_\nu}\Bigr)(r_n^\mu+r_m^\nu),
\end{aligned}
\end{align}
where the last inequality uses $r_n^\mu\ge M_\mu/n$ and $r_m^\nu\ge M_\nu/m$.

We now prove \eqref{eq:uot_cube_kernel_rate_transfer_map}, whose right-hand side has no terms $a_n,b_m$. Let $\varsigma_{nm}\coloneqq\frac12\log(\rho_{\mu}/\rho_{\nu})\in\R$, so that
\begin{align*}
    \rho_{\mu}e^{-\varsigma_{nm}}=\rho_{\nu}e^{\varsigma_{nm}}=\sqrt{\rho_{\mu}\rho_{\nu}}.
\end{align*}
A pair $(\varphi,\psi)\in C_b([0,1]^d)^2$ satisfies the constraint in \eqref{eq:gh_dual} if and only if $(\varphi+\varsigma_{nm},\psi-\varsigma_{nm})$ does. Since $\hat\mu_n^{\mathrm{ker}}=\rho_{\mu}\tilde\mu_n^{\mathrm{ker}}$ and $\hat\nu_m^{\mathrm{ker}}=\rho_{\nu}\tilde\nu_m^{\mathrm{ker}}$, the dual objectives satisfy
\begin{equation}
\label{eq:uot_cube_kernel_dual_shift}
\begin{aligned}
    &\int\bigl(1-e^{-\varphi-\varsigma_{nm}}\bigr)\,d\hat\mu_n^{\mathrm{ker}}+\int\bigl(1-e^{-\psi+\varsigma_{nm}}\bigr)\,d\hat\nu_m^{\mathrm{ker}}
    \\
    &\qquad=\sqrt{\rho_{\mu}\rho_{\nu}}\Bigl(\int\bigl(1-e^{-\varphi}\bigr)\,d\tilde\mu_n^{\mathrm{ker}}+\int\bigl(1-e^{-\psi}\bigr)\,d\tilde\nu_m^{\mathrm{ker}}\Bigr)
    \\
    &\qquad\quad+\bigl(\rho_{\mu}-\sqrt{\rho_{\mu}\rho_{\nu}}\bigr)M_\mu+\bigl(\rho_{\nu}-\sqrt{\rho_{\mu}\rho_{\nu}}\bigr)M_\nu.
\end{aligned}
\end{equation}
The map $(\varphi,\psi)\mapsto(\varphi+\varsigma_{nm},\psi-\varsigma_{nm})$ is a bijection of the feasible pairs. Thus, the pair $(\hat\varphi_{nm}^{\mathrm{ker}}-\varsigma_{nm},\hat\psi_{nm}^{\mathrm{ker}}+\varsigma_{nm})$ is an optimal dual pair for $\mathrm{UOT}(\tilde\mu_n^{\mathrm{ker}},\tilde\nu_m^{\mathrm{ker}})$. For every $x\in[0,1]^d$, the functions $y\mapsto\|x-y\|^2/2-\hat\psi_{nm}^{\mathrm{ker}}(y)$ and $y\mapsto\|x-y\|^2/2-\hat\psi_{nm}^{\mathrm{ker}}(y)-\varsigma_{nm}$ have the same minimizers. Therefore, $\hat T_{nm}^{\mathrm{ker}}$ is also the estimator of $T_0$ built from this pair.

Consequently, we may take $\hat M_\mu=M_\mu$ and $\hat M_\nu=M_\nu$ when bounding the risk of $\hat T_{nm}^{\mathrm{ker}}$. In this case, $\rho_{\mu}=\rho_{\nu}=1$. Moreover, Assumption~\ref{assm:masses} holds with $\bar c=1$ and arbitrary positive sequences $a_n,b_m\to0$. Let $\tilde V_{nm}^{\mathrm{ker}}$ denote the quantity defined in \eqref{eq:uot_cube_kernel_value_gap} with $\hat\mu_n^{\mathrm{ker}},\hat\nu_m^{\mathrm{ker}}$ replaced by $\tilde\mu_n^{\mathrm{ker}},\tilde\nu_m^{\mathrm{ker}}$. Hence, \eqref{eq:estimated_mass_stability} yields
\begin{align*}
    \tilde V_{nm}^{\mathrm{ker}}\le C_\Lambda\Bigl(W_2^2(\tilde\mu_n^{\mathrm{ker}},\mu)+W_2^2(\tilde\nu_m^{\mathrm{ker}},\nu)\Bigr).
\end{align*}
The functions $\tilde p_n^{\mathrm{ker}}$ and $\tilde q_m^{\mathrm{ker}}$ that define $\mathcal G_{nm}$ do not involve $\hat M_\mu$ and $\hat M_\nu$. Thus, \eqref{eq:uot_cube_w2_coupling}, with $\bar c=1$ and with $V_{nm}^{\mathrm{ker}}$ replaced by $\tilde V_{nm}^{\mathrm{ker}}$, and \eqref{eq:uot_cube_kernel_assumed_w2} yield
\begin{align*}
    \Ep\left[ \int_{[0,1]^d} \|\hat T_{nm}^{\mathrm{ker}}(x)-T_0(x)\|^2\,d\mu(x) \mathbf{1}_{\mathcal{G}_{nm}} \right]
    \le
    \frac{2\beta_{\max}}{\underline p}\,C_{\mathrm{tr}}\bigl(2/\kappa+1\bigr)C_\Lambda(r_n^\mu+r_m^\nu).
\end{align*}
Adding \eqref{eq:uot_cube_kernel_complement} yields~\eqref{eq:uot_cube_kernel_rate_transfer_map}.

For the growth map, Lemma~\ref{lem:continuous_gh_growth_transfer}, \eqref{eq:uot_cube_kernel_active_factor_bound}, and \eqref{eq:uot_cube_kernel_active_marginal_bound} provide the following bound under $\hat\mu_n^{\mathrm{ker}}$ on $\mathcal{G}_{nm}$, where $C>0$ is the constant of Lemma~\ref{lem:continuous_gh_growth_transfer}:
\begin{align}
\label{eq:uot_cube_kernel_growth_empirical}
    \int_{[0,1]^d}
    |\hat\lambda_{nm}^{\mathrm{ker}}(x)-\lambda_0(x)|^2\,d\hat\mu_n^{\mathrm{ker}}(x)
    &\le
    C {D_\KL}\bigl(\hat\gamma_{0,nm}^{\mathrm{ker}}\mid e^{-\varphi_0}\hat\mu_n^{\mathrm{ker}}\bigr) \\
    &\quad + C \int_{[0,1]^d}
    \|\hat T_{nm}^{\mathrm{ker}}(x)-T_0(x)\|^2\,d\hat\gamma_{0,nm}^{\mathrm{ker}}(x) \nonumber \\
    &\le C(1+2/\kappa) V_{nm}^{\mathrm{ker}}.
\end{align}
Combining this bound with \eqref{eq:kernel_density_domination} yields, on $\mathcal{G}_{nm}$,
\begin{align}
\begin{aligned}
    \int_{[0,1]^d} |\hat\lambda_{nm}^{\mathrm{ker}}(x)-\lambda_0(x)|^2\,d\mu(x)
    &\le \frac{2\beta_{\max}}{\bar c\,\underline p}
    \int_{[0,1]^d} |\hat\lambda_{nm}^{\mathrm{ker}}(x)-\lambda_0(x)|^2\,d\hat\mu_n^{\mathrm{ker}}(x) \\
    &\le \frac{2\beta_{\max}}{\bar c\,\underline p}\, C(1+2/\kappa) V_{nm}^{\mathrm{ker}}.
\end{aligned}
\end{align}
On $\mathcal{G}_{nm}^c$, both $\hat\lambda_{nm}^{\mathrm{ker}}$ and $\lambda_0$ are bounded by $w_+e^{d/4}$. Hence, the same union-bound argument as in \eqref{eq:uot_cube_kernel_complement} absorbs the contribution of $\mathcal{G}_{nm}^c$. Moreover, $V_{nm}^{\mathrm{ker}}\ge0$, because the right-hand side of \eqref{eq:uot_cube_kernel_exact_excess} is nonnegative and $\mathrm{UOT}(\mu,\nu)=\int\zeta_0\,d\mu+\int\xi_0\,d\nu$. Taking expectations and applying \eqref{eq:uot_cube_kernel_value_expansion_expect} therefore proves~\eqref{eq:uot_cube_kernel_rate_transfer_lambda}.

Finally, $V_{nm}^{\mathrm{ker}}\ge0$, the definition \eqref{eq:uot_cube_kernel_value_gap}, the bound on $|\int\zeta_0\,d(\hat\mu_n^{\mathrm{ker}}-\tilde\mu_n^{\mathrm{ker}})|$, and the analogous bound for $\xi_0$ yield
\begin{align*}
    \left|
        \widehat{\mathrm{UOT}}_{nm}^{\mathrm{ker}}
        -
        \mathrm{UOT}(\mu,\nu)
    \right|
    &\le
    \left|
        \int \zeta_0\,d(\tilde\mu_n^{\mathrm{ker}}-\mu)
    \right|
    +
    \left|
        \int \xi_0\,d(\tilde\nu_m^{\mathrm{ker}}-\nu)
    \right| \\
    &\qquad +
    \|\zeta_0\|_{L^\infty([0,1]^d)}|\hat M_\mu-M_\mu|+\|\xi_0\|_{L^\infty([0,1]^d)}|\hat M_\nu-M_\nu|
    + V_{nm}^{\mathrm{ker}}.
\end{align*}
Moreover,
\begin{align*}
    \left|\int \zeta_0\,d(\tilde\mu_n^{\mathrm{ker}}-\mu)\right|
    &\le
    M_\mu\|\zeta_0\|_{L^\infty([0,1]^d)}
    \|\hat p_n^{\mathrm{ker}}-p\|_{L^1([0,1]^d)},
    \\
    \left|\int \xi_0\,d(\tilde\nu_m^{\mathrm{ker}}-\nu)\right|
    &\le
    M_\nu\|\xi_0\|_{L^\infty([0,1]^d)}
    \|\hat q_m^{\mathrm{ker}}-q\|_{L^1([0,1]^d)},
\end{align*}
so taking expectations and applying \eqref{eq:uot_cube_kernel_value_expansion_expect} proves \eqref{eq:uot_cube_kernel_rate_transfer_abs_value}.
\end{proof}

\begin{proof}[Proof of Theorem \ref{thm:main_cube_kernel_rates}]
Apply Corollary~\ref{cor:uot_cube_kernel_rate_transfer} with
\begin{align*}
    r_n^\mu = C M_\mu\mathfrak R_n^{\mathrm{ker}}(\alpha),
    \qquad
    r_m^\nu = C M_\nu\mathfrak R_m^{\mathrm{ker}}(\alpha),
\end{align*}
and
\begin{align*}
    \ell_n = C\mathfrak L_n(\alpha),
    \qquad
    \ell_m = C\mathfrak L_m(\alpha),
\end{align*}
which are available from Proposition~\ref{prop:cube_kernel_density_rates}, whose hypotheses follow from Assumptions~\ref{assm:positivity} and~\ref{assm:cube_kernel_smoothness} and the choice of $L_n,L_m$. We enlarge $C$ so that $C\ge1$. Since $\mathfrak R_n^{\mathrm{ker}}(\alpha)\ge1/n$ for all $n\ge3$, this choice yields $r_n^\mu\ge M_\mu/n$ and $r_m^\nu\ge M_\nu/m$. For all sufficiently large $n,m$, Assumption~\ref{assm:masses} gives the bounds $\Ep[|\hat M_\mu-M_\mu|]\le a_n$ and $\Ep[|\hat M_\nu-M_\nu|]\le b_m$, so the corollary yields \eqref{eq:section4_cube_kernel_map_rate} and \eqref{eq:section4_cube_kernel_lambda_rate}.
\end{proof}

\begin{remark}[Order of the constants in Theorem~\ref{thm:main_cube_kernel_rates}]
\label{rem:constants_kernel}
Let $D=\sqrt d$, $\mathfrak m$, and $w_\pm$ be as in Remark~\ref{rem:constants_plan}. By tracking the constants in the proof of Corollary~\ref{cor:uot_cube_kernel_rate_transfer}, the constant in~\eqref{eq:section4_cube_kernel_map_rate} is $O\bigl(e^{D^2/2}\mathfrak m\,C_\Lambda\kappa^{-1}\bigr)$. The constant in~\eqref{eq:section4_cube_kernel_lambda_rate} is $O\bigl(e^{2D^2}\mathfrak m^3C_\Lambda\kappa^{-1}\bigr)$.
\end{remark}

\section{Wavelet plug-in estimator on the hypercube}\label{sec:wavelet}

\subsection{Definition}

In this section, we define the wavelet plug-in estimator, which uses a boundary-corrected wavelet basis on the hypercube $\Omega=[0,1]^d$. This basis allows the wavelet plug-in estimator to operate directly on nonperiodic data.

\begin{definition}[Boundary-corrected wavelet basis]
\label{def:bc_wavelet_basis}
Fix an integer $N\ge 2$ and a base scale $j_0\ge \lceil \log_2 N\rceil+1$. Let $\zeta^{\mathrm{db}},\xi^{\mathrm{db}}$ denote the compactly supported Daubechies scaling and wavelet functions on $\mathbb R$ such that $\xi^{\mathrm{db}}$ has $N$ vanishing moments, shifted so that their supports are $[-N+1,N]$. We follow the construction of \cite{Cohen1993}, which is also used in the analysis of plug-in wavelet estimators for balanced OT maps in \cite{manole2024plugin}.

For each level $j\ge j_0$, define the \emph{interior} one-dimensional translates
\begin{align} \label{eq:wavelets}
    \zeta_{j,k}^{\mathrm{int}}(t) \coloneqq 2^{j/2}\zeta^{\mathrm{db}}(2^jt-k),
    \qquad
    \xi_{j,k}^{\mathrm{int}}(t) \coloneqq 2^{j/2}\xi^{\mathrm{db}}(2^jt-k),
\end{align}
for $N\le k\le 2^j-N-1$. The boundary-corrected construction replaces the first and last $N$ translates at each level by the functions
\begin{align*}
    \zeta_{j,k}^{L},\ \xi_{j,k}^{L},
    \qquad
    \zeta_{j,k}^{R},\ \xi_{j,k}^{R},
    \qquad k=0,\ldots,N-1,
\end{align*}
satisfying
\begin{align*}
    \operatorname{supp}(\zeta_{j,k}^{L}),\operatorname{supp}(\xi_{j,k}^{L})
    &\subset [0,(2N-1)2^{-j}],\\
    \operatorname{supp}(\zeta_{j,k}^{R}),\operatorname{supp}(\xi_{j,k}^{R})
    &\subset [1-(2N-1)2^{-j},1].
\end{align*}
We then define the one-dimensional boundary-corrected scaling and wavelet families by
\begin{align}
    \zeta_{j,k}^{\mathrm{bc},1}(t)
    &\coloneqq
    \begin{cases}
        \zeta_{j,k}^{L}(t), & 0\le k\le N-1,\\
        \zeta_{j,k}^{\mathrm{int}}(t), & N\le k\le 2^j-N-1,\\
        \zeta_{j,2^j-1-k}^{R}(t), & 2^j-N\le k\le 2^j-1,
    \end{cases}\\
    \xi_{j,k}^{\mathrm{bc},1}(t)
    &\coloneqq
    \begin{cases}
        \xi_{j,k}^{L}(t), & 0\le k\le N-1,\\
        \xi_{j,k}^{\mathrm{int}}(t), & N\le k\le 2^j-N-1,\\
        \xi_{j,2^j-1-k}^{R}(t), & 2^j-N\le k\le 2^j-1.
    \end{cases}
\end{align}

Now write $K(j)\coloneqq \{0,\ldots,2^j-1\}^d$. For $k=(k_1,\ldots,k_d)$, $x=(x_1,\ldots,x_d)\in \Omega$, and $\ell\in\{0,1\}^d\setminus\{0\}$, define
\begin{equation}\label{eq:wavprods}
    \begin{split}
    \zeta_{j_0,k}^{\mathrm{bc}}(x)
    &\coloneqq
    \prod_{r=1}^d \zeta_{j_0,k_r}^{\mathrm{bc},1}(x_r),
    \qquad k\in K(j_0),\\
    \xi_{j,k,\ell}^{\mathrm{bc}}(x)
    &\coloneqq
    \prod_{r:\ell_r=0}\zeta_{j,k_r}^{\mathrm{bc},1}(x_r)
    \prod_{r:\ell_r=1}\xi_{j,k_r}^{\mathrm{bc},1}(x_r),
    \qquad j\ge j_0,\ k\in K(j).
    \end{split}
\end{equation}
Set
\begin{align*}
    \Phi^{\mathrm{bc}}
    &\coloneqq
    \left\{\zeta_{j_0,k}^{\mathrm{bc}}:k\in K(j_0)\right\},\\
    \Psi_j^{\mathrm{bc}}
    &\coloneqq
    \left\{\xi_{j,k,\ell}^{\mathrm{bc}}:k\in K(j),\ \ell\in\{0,1\}^d\setminus\{0\}\right\},
    \qquad j\ge j_0.
\end{align*}
The resulting family
\begin{align*}
    \Psi^{\mathrm{bc}}
    \coloneqq
    \Phi^{\mathrm{bc}}\cup\bigcup_{j\ge j_0}\Psi_j^{\mathrm{bc}}
\end{align*}
is the boundary-corrected wavelet basis on $\Omega=[0,1]^d$. By this construction, $\Psi^{\mathrm{bc}}$ is an orthonormal basis of $L^2(\Omega)$, and the scaling space $\mathrm{span}(\Phi^{\mathrm{bc}})$ contains all tensor-product polynomials of coordinatewise degree at most $N-1$; in particular, $\mathrm{span}(\Phi^{\mathrm{bc}})$ contains the constants.
\end{definition}

\begin{definition}[Wavelet plug-in estimator]
\label{def:uot_wavelet_plugin}
Let $\Psi^{\mathrm{bc}} = \Phi^{\mathrm{bc}}\cup \bigcup_{j\ge j_0}\Psi_j^{\mathrm{bc}}$ be the boundary-corrected wavelet basis from Definition~\ref{def:bc_wavelet_basis}. For $\zeta\in\Phi^{\mathrm{bc}}$ and $\xi\in\Psi_j^{\mathrm{bc}}$, define the empirical coefficients $ \hat\alpha_\zeta \coloneqq \tfrac{1}{n}\sum_{i=1}^n \zeta(X_i), \hat\beta_\zeta \coloneqq \tfrac{1}{m}\sum_{i=1}^m \zeta(Y_i), $ and $ \hat\alpha_\xi \coloneqq \tfrac{1}{n}\sum_{i=1}^n \xi(X_i), \hat\beta_\xi \coloneqq \tfrac{1}{m}\sum_{i=1}^m \xi(Y_i). $ Given resolution levels $J_n,J_m\ge j_0$, the preliminary boundary-corrected wavelet estimators of the normalized densities $p$ and $q$ are
\begin{align}
    \tilde p_n^{\mathrm{wav}}(x)
    &\coloneqq
    \sum_{\zeta\in\Phi^{\mathrm{bc}}}\hat\alpha_\zeta\zeta(x)
    + \sum_{j=j_0}^{J_n}\sum_{\xi\in\Psi_j^{\mathrm{bc}}}\hat\alpha_\xi\xi(x), \quad 
    \tilde q_m^{\mathrm{wav}}(y)
    \coloneqq
    \sum_{\zeta\in\Phi^{\mathrm{bc}}}\hat\beta_\zeta\zeta(y)
    + \sum_{j=j_0}^{J_m}\sum_{\xi\in\Psi_j^{\mathrm{bc}}}\hat\beta_\xi\xi(y).
\end{align}
Since these preliminary estimators may take negative values, we define the density estimates
\begin{align}
\label{eq:uot_wavelet_density_estimator}
    \hat p_n^{\mathrm{wav}}(x)
    &\coloneqq
    \frac{(\tilde p_n^{\mathrm{wav}}(x))_+}{\int_{\Omega}(\tilde p_n^{\mathrm{wav}}(u))_+\,du},
    \qquad
    \hat q_m^{\mathrm{wav}}(y)
    \coloneqq
    \frac{(\tilde q_m^{\mathrm{wav}}(y))_+}{\int_{\Omega}(\tilde q_m^{\mathrm{wav}}(v))_+\,dv}.
\end{align}
We then define the measure estimates, which have total masses $\hat M_\mu$ and $\hat M_\nu$, as $\hat\mu_n^{\mathrm{wav}} \coloneqq \hat M_\mu\hat p_n^{\mathrm{wav}}\,dx$ and $\hat\nu_m^{\mathrm{wav}} \coloneqq \hat M_\nu\hat q_m^{\mathrm{wav}}\,dx$. Let $(\hat\varphi_{nm}^{\mathrm{wav}},\hat\psi_{nm}^{\mathrm{wav}})$ be an optimal dual pair for $\mathrm{UOT}(\hat\mu_n^{\mathrm{wav}},\hat\nu_m^{\mathrm{wav}})$. As for the kernel-based estimator in Section~\ref{sec:kernel}, the \emph{wavelet plug-in estimator} consists of
\begin{align}
    \hat a_{nm}^{\mathrm{wav}}(x)
    &\coloneqq
    e^{-\hat\varphi_{nm}^{\mathrm{wav}}(x)/2},
    \qquad
    \hat T_{nm}^{\mathrm{wav}}(x)
    \in
    \argmin_{y\in \Omega}
    \left\{
        \tfrac{1}{2}\|x-y\|^2-\hat\psi_{nm}^{\mathrm{wav}}(y)
    \right\},
    \\
    \hat\lambda_{nm}^{\mathrm{wav}}(x)
    &\coloneqq
    \clip_{[w_-,w_+]}\bigl(\hat a_{nm}^{\mathrm{wav}}(x)^2\bigr)
    \exp\left(\frac14\|x-\hat T_{nm}^{\mathrm{wav}}(x)\|^2\right),
\end{align}
where $w_\pm$ are the clipping constants of Section~\ref{sec:plan}. We also write $\widehat{\mathrm{UOT}}_{nm}^{\mathrm{wav}}\coloneqq\mathrm{UOT}(\hat\mu_n^{\mathrm{wav}},\hat\nu_m^{\mathrm{wav}})$ for the plug-in estimator of $\mathrm{UOT}(\mu,\nu)$, and $\hat\gamma_{0,nm}^{\mathrm{wav}}\coloneqq(\hat a_{nm}^{\mathrm{wav}})^2\hat\mu_n^{\mathrm{wav}}=e^{-\hat\varphi_{nm}^{\mathrm{wav}}}\hat\mu_n^{\mathrm{wav}}$ for the active marginal of $\mathrm{UOT}(\hat\mu_n^{\mathrm{wav}},\hat\nu_m^{\mathrm{wav}})$. By the primal-dual optimality conditions \cite{LieroMielkeSavare2018}, $\hat\gamma_{0,nm}^{\mathrm{wav}}$ is the first marginal of every optimal transport plan for $\mathrm{UOT}(\hat\mu_n^{\mathrm{wav}},\hat\nu_m^{\mathrm{wav}})$. We say that the problem $\mathrm{UOT}(\hat\mu_n^{\mathrm{wav}},\hat\nu_m^{\mathrm{wav}})$ is \emph{Monge-type} if $(\mathrm{id},\hat T_{nm}^{\mathrm{wav}})_\#\hat\gamma_{0,nm}^{\mathrm{wav}}$ is an optimal transport plan for this problem. By Theorem~\ref{prop:monge-potential} and the primal-dual optimality conditions, $\mathrm{UOT}(\hat\mu_n^{\mathrm{wav}},\hat\nu_m^{\mathrm{wav}})$ is Monge-type whenever $\hat M_\mu,\hat M_\nu>0$ and the density estimates $\hat p_n^{\mathrm{wav}}$ and $\hat q_m^{\mathrm{wav}}$ are bounded below by a positive constant.
\end{definition}

The density estimates in Definition~\ref{def:uot_wavelet_plugin} are well defined. By Definition~\ref{def:bc_wavelet_basis}, the scaling space $\mathrm{span}(\Phi^{\mathrm{bc}})$ contains the constants, so there exist coefficients $(c_\zeta)_{\zeta\in\Phi^{\mathrm{bc}}}$ such that $1=\sum_{\zeta\in\Phi^{\mathrm{bc}}}c_\zeta\zeta$. The orthonormality of $\Psi^{\mathrm{bc}}$ yields $\int_\Omega\zeta(x)\,dx=c_\zeta$ for every $\zeta\in\Phi^{\mathrm{bc}}$ and $\int_\Omega\xi(x)\,dx=0$ for every $\xi\in\Psi_j^{\mathrm{bc}}$ with $j\ge j_0$. Hence,
\begin{align*}
    \int_{\Omega}\tilde p_n^{\mathrm{wav}}(x)\,dx
    &= \sum_{\zeta\in\Phi^{\mathrm{bc}}} c_\zeta\hat\alpha_\zeta
    = \frac{1}{n}\sum_{i=1}^n\sum_{\zeta\in\Phi^{\mathrm{bc}}}c_\zeta\zeta(X_i)
    = \frac{1}{n}\sum_{i=1}^n 1 = 1.
\end{align*}
The analogous identity holds for $\tilde q_m^{\mathrm{wav}}$. Since $(\tilde p_n^{\mathrm{wav}})_+\ge\tilde p_n^{\mathrm{wav}}$ and $(\tilde q_m^{\mathrm{wav}})_+\ge\tilde q_m^{\mathrm{wav}}$, the denominators of $\hat p_n^{\mathrm{wav}}$ and $\hat q_m^{\mathrm{wav}}$ are at least $1$.

\subsection{Theoretical analysis}

Write
\begin{align}
    \mathfrak R_n^{\mathrm{wav}}(\alpha)
    \coloneqq
    \begin{cases}
        n^{-1}, & d=1,\\
        (\log n)^2 n^{-1}, & d=2,\\
        n^{{-2\alpha/(2(\alpha-1)+d)}}, & d\ge 3,
    \end{cases}
\end{align}
and let $\mathfrak L_n(\alpha)$ be as in \eqref{eq:frak_L_rate}.

\begin{theorem}[Error rate of wavelet plug-in estimator]
\label{thm:main_wavelet_rates}
Assume that $\Omega=[0,1]^d$, and let
\begin{align*}
    \hat\mu_n^{\mathrm{wav}},\qquad \hat\nu_m^{\mathrm{wav}},\qquad \hat T_{nm}^{\mathrm{wav}},\qquad \hat\lambda_{nm}^{\mathrm{wav}},\qquad \widehat{\mathrm{UOT}}_{nm}^{\mathrm{wav}}
\end{align*}
be as in Definition~\ref{def:uot_wavelet_plugin}. Assume Assumptions~\ref{assm:curvature}, \ref{assm:positivity}, and \ref{assm:masses}. Assume moreover that $\alpha>1$, that $p$ and $q$ satisfy
\begin{align}
    \|p\|_{C^{\alpha-1}(\Omega)}+\|q\|_{C^{\alpha-1}(\Omega)} \le M,
\end{align}
and that the functions $\zeta^{\mathrm{db}}$ and $\xi^{\mathrm{db}}$ in Definition~\ref{def:bc_wavelet_basis} belong to $C^r(\mathbb R)$ for some $r>\lceil\alpha-1\rceil$. Assume also that $2^{J_n} \asymp n^{{1} / ({d+2(\alpha-1)})}$ and $2^{J_m} \asymp m^{{1} / ({d+2(\alpha-1)})}$. Let $a_n,b_m$ be the sequences from Assumption~\ref{assm:masses}. Then there exists a constant $C>0$ such that, for all sufficiently large $n,m$,
\begin{align}
    \Ep\left[
        \int_{\Omega}
        \|\hat T_{nm}^{\mathrm{wav}}(x)-T_0(x)\|^2\,d\mu(x)
    \right]
    &\le
    C\Bigl(M_\mu \mathfrak R_n^{\mathrm{wav}}(\alpha) + M_\nu \mathfrak R_m^{\mathrm{wav}}(\alpha)\Bigr),
\end{align}
\begin{align}
    \Ep\left[
        \int_{\Omega}
        |\hat\lambda_{nm}^{\mathrm{wav}}(x)-\lambda_0(x)|^2\,d\mu(x)
    \right]
    &\le
    C\Bigl(M_\mu \mathfrak R_n^{\mathrm{wav}}(\alpha) + M_\nu \mathfrak R_m^{\mathrm{wav}}(\alpha) + a_n + b_m\Bigr),
\end{align}
and
\begin{align}
\label{eq:main_wavelet_value_rate}
\begin{aligned}
    \Ep\left[
        \left|
            \widehat{\mathrm{UOT}}_{nm}^{\mathrm{wav}} - \mathrm{UOT}(\mu,\nu)
        \right|
    \right]
    &\le
    C\Bigl(
        \mathfrak L_n(\alpha)+\mathfrak L_m(\alpha) + M_\mu \mathfrak R_n^{\mathrm{wav}}(\alpha) \\
    &\qquad\qquad + M_\nu \mathfrak R_m^{\mathrm{wav}}(\alpha) + a_n + b_m
    \Bigr).
\end{aligned}
\end{align}
\end{theorem}

\section{Proof of Theorem \ref{thm:main_wavelet_rates}} \label{sec:proof_wavelet}

\subsection{Convergence rates of the wavelet density estimators}

Throughout this subsection we write
\begin{align}
    s \coloneqq \alpha-1 > 0,
\end{align}
and we assume that the densities $p$ and $q$ satisfy
\begin{align}
    \beta_{\min} \le p(x), q(x) \le \beta_{\max},
    \qquad
    \|p\|_{C^{s}(\Omega)} + \|q\|_{C^{s}(\Omega)} \le M,
\end{align}
for some $M>0$ and $0<\beta_{\min}\le\beta_{\max}<\infty$. The Wasserstein distance is defined for measures of equal total mass, so we compare $\mu$ and $\nu$ with
\begin{align}
    \tilde\mu_n^{\mathrm{wav}}
    &\coloneqq
    M_\mu \hat p_n^{\mathrm{wav}}(x)\,dx,
    \qquad
    \tilde\nu_m^{\mathrm{wav}}
    \coloneqq
    M_\nu \hat q_m^{\mathrm{wav}}(y)\,dy.
\end{align}
The following theorem provides bounds for the errors of density estimation by the plug-in wavelet estimators. For balanced OT, \cite{manole2024plugin} use only the $W_2$ bound to derive their rate for one-sample map estimation. In contrast, we also use the two $L^1$ bounds to derive our rate~\eqref{eq:main_wavelet_value_rate} for estimation of $\mathrm{UOT}(\mu,\nu)$.
\begin{theorem}[Wavelet density rates]
\label{thm:uot_plugin_density_rates}
Assume $\Omega=[0,1]^d$ and $\alpha>1$. Assume moreover that
\begin{align*}
    \beta_{\min} \le p(x),q(x) \le \beta_{\max},
    \qquad
    \|p\|_{C^{\alpha-1}(\Omega)}+\|q\|_{C^{\alpha-1}(\Omega)} \le M,
\end{align*}
that the functions $\zeta^{\mathrm{db}}$ and $\xi^{\mathrm{db}}$ in Definition~\ref{def:bc_wavelet_basis} belong to $C^r(\mathbb R)$ for some $r>\lceil\alpha-1\rceil$, and choose
\begin{align*}
    2^{J_n} \asymp n^{{1/(d+2(\alpha-1))}},
    \qquad
    2^{J_m} \asymp m^{{1/(d+2(\alpha-1))}}.
\end{align*}
Then there exists a constant $C>0$, depending only on $d,\alpha,M,\beta_{\min},\beta_{\max}$ and the basis $\Psi^{\mathrm{bc}}$, such that, for all sufficiently large $n,m$,
\begin{align}
\label{eq:uot_plugin_wavelet_w2_source}
    \Ep\bigl[W_2^2(\tilde\mu_n^{\mathrm{wav}},\mu)\bigr]
    &\le C M_\mu\mathfrak R_n^{\mathrm{wav}}(\alpha),
    \\
    \Ep\bigl[W_2^2(\tilde\nu_m^{\mathrm{wav}},\nu)\bigr]
    &\le C M_\nu\mathfrak R_m^{\mathrm{wav}}(\alpha),
\end{align}
and
\begin{align}
    \Ep\bigl[\|\hat p_n^{\mathrm{wav}}-p\|_{L^1(\Omega)}\bigr]
    &\le C\mathfrak L_n(\alpha),
    \\
\label{eq:uot_plugin_wavelet_l1_target}
    \Ep\bigl[\|\hat q_m^{\mathrm{wav}}-q\|_{L^1(\Omega)}\bigr]
    &\le C\mathfrak L_m(\alpha).
\end{align}
\end{theorem}

\begin{proof}[Proof of Theorem \ref{thm:uot_plugin_density_rates}]
We first prove the $L^1$ bounds. Let
\begin{align*}
    p_{J_n}^{\mathrm{bc}}(x)
    \coloneqq
    \sum_{\zeta\in\Phi^{\mathrm{bc}}}\alpha_\zeta\zeta(x)
    + \sum_{j=j_0}^{J_n}\sum_{\xi\in\Psi_j^{\mathrm{bc}}}\alpha_\xi\xi(x),
\end{align*}
where $\alpha_\zeta=\int_{\Omega}\zeta(x)p(x)\,dx$ and $\alpha_\xi=\int_{\Omega}\xi(x)p(x)\,dx$. Let
\begin{align*}
    V_{J_n}^{\mathrm{bc}}
    \coloneqq
    \mathrm{span}(\Phi^{\mathrm{bc}})
    \oplus
    \bigoplus_{j=j_0}^{J_n}\mathrm{span}(\Psi_j^{\mathrm{bc}}),
\end{align*}
so that $p_{J_n}^{\mathrm{bc}}$ is exactly the $L^2$-orthogonal projection of $p$ onto $V_{J_n}^{\mathrm{bc}}$. Fix $j\ge j_0$ and $\xi\in\Psi_j^{\mathrm{bc}}$, let $I_\xi\subset\Omega$ be a rectangle of diameter at most $C2^{-j}$ that contains the support of $\xi$, and let $\pi_\xi$ be the Taylor polynomial of $p$ of degree $\lceil s\rceil-1$ at a point of $I_\xi$. Since $\xi^{\mathrm{db}}\in C^r(\mathbb R)$ with $r>\lceil s\rceil\ge s$, the wavelet $\xi^{\mathrm{db}}$ has more than $s$ vanishing moments \cite[Corollary~5.5.2]{daubechies1992ten}, so $N>s$. By \cite[Lemma 24(iv)]{manole2024plugin}, every polynomial of degree at most $N-1$ over $\Omega$, including $\pi_\xi$, lies in $\mathrm{span}(\Phi^{\mathrm{bc}})$. Thus, by the orthonormality of the wavelet basis, $\xi$ is orthogonal to $\pi_\xi$. Therefore,
\begin{align}
\label{eq:wav_coef_bound}
    |\alpha_\xi|
    =
    \Bigl|\int_{I_\xi}\xi(x)\bigl(p(x)-\pi_\xi(x)\bigr)\,dx\Bigr|
    \le
    \|\xi\|_{L^1(\Omega)}\|p-\pi_\xi\|_{L^\infty(I_\xi)}
    \le
    C2^{-j(s+d/2)}\|p\|_{C^s(\Omega)},
\end{align}
where the last inequality uses $\|p-\pi_\xi\|_{L^\infty(I_\xi)}\le C2^{-js}\|p\|_{C^s(\Omega)}$, by Taylor's theorem on the rectangle $I_\xi$ of diameter at most $C2^{-j}$, and $\|\xi\|_{L^1(\Omega)}\le|I_\xi|\,\|\xi\|_{L^\infty(\Omega)}\le C2^{-jd/2}$, by the bound $\|\xi\|_{L^\infty(\Omega)}\le C2^{jd/2}$ of \cite[Lemma 24(ii) and (v)]{manole2024plugin}. By \cite[Lemma 24(ii) and (v)]{manole2024plugin}, we have $\|\xi\|_{L^\infty(\Omega)}\le C2^{jd/2}$ for $\xi\in\Psi_j^{\mathrm{bc}}$, while every $x\in\Omega$ belongs to at most $C$ of the rectangles $I_\xi$. Hence, for every $j>J_n$ and every $x\in\Omega$,
\begin{align*}
    \Bigl|\sum_{\xi\in\Psi_j^{\mathrm{bc}}}\alpha_\xi\xi(x)\Bigr|
    &\le
    \sum_{\xi\in\Psi_j^{\mathrm{bc}}}|\alpha_\xi||\xi(x)|
    \le
    C\max_{\xi\in\Psi_j^{\mathrm{bc}}}|\alpha_\xi|\sup_{\xi\in\Psi_j^{\mathrm{bc}}}\|\xi\|_{L^\infty(\Omega)}
    \\
    &\le
    C2^{-j(s+d/2)}\|p\|_{C^s(\Omega)}2^{jd/2}
    =
    C2^{-js}\|p\|_{C^s(\Omega)},
\end{align*}
so, since $s>0$ and the expansion converges to $p$ in $L^2(\Omega)$, the wavelet expansion of $p$ converges to $p$ absolutely and uniformly on $\Omega$. In particular, the boundary-corrected projection $p_{J_n}^{\mathrm{bc}}=P_{V_{J_n}^{\mathrm{bc}}}p$ satisfies
\begin{align}\label{eq:Linfbdd}
    \|p-p_{J_n}^{\mathrm{bc}}\|_{L^\infty(\Omega)}
    =
    \|p-P_{V_{J_n}^{\mathrm{bc}}}p\|_{L^\infty(\Omega)}
    \le
    C\|p\|_{C^s(\Omega)}\sum_{j>J_n}2^{-js}
    \le
    C2^{-J_ns},
\end{align}
after absorbing $\|p\|_{C^s(\Omega)}\le M$ into $C$, and therefore also, since $|\Omega|=1$,
\begin{align*}
    \|p-p_{J_n}^{\mathrm{bc}}\|_{L^1(\Omega)}
    \le
    \|p-p_{J_n}^{\mathrm{bc}}\|_{L^\infty(\Omega)}
    \le
    C2^{-J_ns}.
\end{align*}

For the stochastic part, the orthonormality of $\Psi^{\mathrm{bc}}$ and Parseval's identity imply
\begin{align*}
    \Ep\bigl[\|\tilde p_n^{\mathrm{wav}}-p_{J_n}^{\mathrm{bc}}\|_{L^2(\Omega)}^2\bigr]
    &= \sum_{\zeta\in\Phi^{\mathrm{bc}}}\mathrm{Var}(\hat\alpha_\zeta)
    + \sum_{j=j_0}^{J_n}\sum_{\xi\in\Psi_j^{\mathrm{bc}}}\mathrm{Var}(\hat\alpha_\xi) \\
    &\le \frac{C}{n}\left(\#\Phi^{\mathrm{bc}} + \sum_{j=j_0}^{J_n}\#\Psi_j^{\mathrm{bc}}\right) \\
    &= \frac{C}{n}\left(2^{dj_0} + (2^d-1)\sum_{j=j_0}^{J_n}2^{dj}\right)
    \le C\frac{2^{dJ_n}}{n}.
\end{align*}
Since $|\Omega|=1$, Jensen's inequality yields
\begin{align*}
    \Ep\bigl[\|\tilde p_n^{\mathrm{wav}}-p_{J_n}^{\mathrm{bc}}\|_{L^1(\Omega)}\bigr]
    \le
    \left(\Ep\bigl[\|\tilde p_n^{\mathrm{wav}}-p_{J_n}^{\mathrm{bc}}\|_{L^2(\Omega)}^2\bigr]\right)^{1/2}
    \le C\sqrt{\frac{2^{dJ_n}}{n}}.
\end{align*}
Therefore
\begin{align}
\label{eq:uot_plugin_wavelet_l1_preliminary}
    \Ep\bigl[\|\tilde p_n^{\mathrm{wav}}-p\|_{L^1(\Omega)}\bigr]
    \le C\left(2^{-J_ns}+\sqrt{\frac{2^{dJ_n}}{n}}\right).
\end{align}
Balancing the two terms in \eqref{eq:uot_plugin_wavelet_l1_preliminary} with $2^{J_n}\asymp n^{1/(d+2s)}$ yields
\begin{align*}
    \Ep\bigl[\|\tilde p_n^{\mathrm{wav}}-p\|_{L^1(\Omega)}\bigr]
    \le Cn^{{-s/(2s+d)}} = C\mathfrak L_n(\alpha).
\end{align*}
Since $\int_\Omega\tilde p_n^{\mathrm{wav}}(x)\,dx=1$, as shown after Definition~\ref{def:uot_wavelet_plugin}, and $p(x)\ge0$ for all $x \in \Omega$, the derivation of \eqref{eq:cube_kernel_positive_part_l1} in the proof of Proposition~\ref{prop:cube_kernel_density_rates}, with $\tilde p_n^{\mathrm{ker}}$ replaced by $\tilde p_n^{\mathrm{wav}}$, yields
\begin{align}
\label{eq:wav_hat_l1}
    \Ep\bigl[\|\hat p_n^{\mathrm{wav}}-p\|_{L^1(\Omega)}\bigr]
    \le 3\Ep\bigl[\|\tilde p_n^{\mathrm{wav}}-p\|_{L^1(\Omega)}\bigr]
    \le C\mathfrak L_n(\alpha).
\end{align}
The same argument applied to $\tilde q_m^{\mathrm{wav}}$ and $q$ yields the analogue of \eqref{eq:wav_hat_l1} for $\hat q_m^{\mathrm{wav}}$.

For the $W_2$ bounds, we use \cite[Lemma 30]{manole2024plugin}, which requires $p,q\ge\beta_{\min}$ and $\|p\|_{\mathcal B^s_{\infty,\infty}(\Omega)},\|q\|_{\mathcal B^s_{\infty,\infty}(\Omega)}<\infty$, where $\|\cdot\|_{\mathcal B^s_{\infty,\infty}(\Omega)}$ is a Besov norm (cf. \cite[Appendix A.2.4]{manole2024plugin}). By \eqref{eq:wav_coef_bound} and $|\alpha_\zeta|\le\|\zeta\|_{L^1(\Omega)}\|p\|_{L^\infty(\Omega)}$, we have $\|p\|_{\mathcal B^s_{\infty,\infty}(\Omega)}\le C'\|p\|_{C^s(\Omega)}$, where $C'\coloneqq C+\max_{\zeta\in\Phi^{\mathrm{bc}}}\|\zeta\|_{L^1(\Omega)}$ and $C$ is the constant in \eqref{eq:wav_coef_bound}. The same argument yields $\|q\|_{\mathcal B^s_{\infty,\infty}(\Omega)}\le C'\|q\|_{C^s(\Omega)}$. Therefore,
\begin{align*}
    \Ep\bigl[W_2^2(\hat p_n^{\mathrm{wav}}\,dx,p\,dx)\bigr]
    \le C\mathfrak R_n^{\mathrm{wav}}(\alpha),
    \qquad
    \Ep\bigl[W_2^2(\hat q_m^{\mathrm{wav}}\,dx,q\,dx)\bigr]
    \le C\mathfrak R_m^{\mathrm{wav}}(\alpha).
\end{align*}
Thus, since $M_\mu$ and $M_\nu$ are deterministic,
\begin{align*}
    W_2^2(\tilde\mu_n^{\mathrm{wav}},\mu) = M_\mu W_2^2(\hat p_n^{\mathrm{wav}}\,dx,p\,dx),
    \qquad
    W_2^2(\tilde\nu_m^{\mathrm{wav}},\nu) = M_\nu W_2^2(\hat q_m^{\mathrm{wav}}\,dx,q\,dx),
\end{align*}
which proves \eqref{eq:uot_plugin_wavelet_w2_source}--\eqref{eq:uot_plugin_wavelet_l1_target}.
\end{proof}

\begin{lemma}[High-probability lower bound on $\tilde p_n^{\mathrm{wav}}$ and $\tilde q_m^{\mathrm{wav}}$]
\label{lem:uot_wavelet_lower_bound_event}
Under the assumptions of Theorem~\ref{thm:uot_plugin_density_rates}, write $\underline p\coloneqq \inf_{x\in\Omega}p(x)>0$ and $\underline q\coloneqq \inf_{y\in\Omega}q(y)>0$, and define the event
\begin{align*}
    \mathcal H_{nm}
    \coloneqq
    \left\{
        \inf_{x\in\Omega}\tilde p_n^{\mathrm{wav}}(x)\ge \frac{\underline p}{2}
        \quad \text{and}\quad
        \inf_{y\in\Omega}\tilde q_m^{\mathrm{wav}}(y)\ge \frac{\underline q}{2}
    \right\}.
\end{align*}
Then for every $A>0$ there exists $C_A>0$ such that, for all sufficiently large $n,m$,
\begin{align}
\label{eq:uot_wavelet_lower_bound_event}
    \mathbb{P}\bigl(\mathcal H_{nm}^c\bigr) \le C_A\bigl(n^{-A}+m^{-A}\bigr).
\end{align}
On $\mathcal H_{nm}$, we have $\hat p_n^{\mathrm{wav}}=\tilde p_n^{\mathrm{wav}}$, $\hat q_m^{\mathrm{wav}}=\tilde q_m^{\mathrm{wav}}$, and
\begin{align}
\label{eq:wavelet_density_domination}
    d\mu \le \frac{2\beta_{\max}}{\underline p}\,d\tilde\mu_n^{\mathrm{wav}},
    \qquad
    d\nu \le \frac{2\beta_{\max}}{\underline q}\,d\tilde\nu_m^{\mathrm{wav}}.
\end{align}
\end{lemma}

\begin{proof}[Proof of Lemma \ref{lem:uot_wavelet_lower_bound_event}]
With the rectangles $I_\xi$ from the proof of Theorem~\ref{thm:uot_plugin_density_rates}, let $C_{\Psi}\ge1$ be a constant, depending only on $d$ and $\Psi^{\mathrm{bc}}$, such that, for all $j \ge j_0$ and all $x\in\Omega$,
\begin{align}
    \#\Phi^{\mathrm{bc}}\le C_{\Psi},
    \qquad
    \#\Psi_j^{\mathrm{bc}}\le C_{\Psi}2^{dj},
    \qquad
    \max_{\zeta\in\Phi^{\mathrm{bc}}}\|\zeta\|_{L^\infty(\Omega)}\le C_{\Psi},
    \label{eq:C_Psi_bounds_1}
    \\
    \max_{\xi\in\Psi_j^{\mathrm{bc}}}\|\xi\|_{L^\infty(\Omega)}\le C_{\Psi}2^{jd/2},
    \qquad
    \sum_{\xi\in\Psi_j^{\mathrm{bc}}}\mathbf{1}\{x\in I_\xi\}\le C_{\Psi}.
    \label{eq:C_Psi_bounds_2}
\end{align}
Recall the coefficients $\hat\alpha_\xi,\hat\alpha_\zeta$ from Definition~\ref{def:uot_wavelet_plugin} and $\alpha_\xi,\alpha_\zeta$ from the proof of Theorem~\ref{thm:uot_plugin_density_rates}:
\begin{align*}
    \hat\alpha_\xi
    =
    \frac1n\sum_{i=1}^n\xi(X_i),
    \quad
    \alpha_\xi
    =
    \int_\Omega\xi(x)p(x)\,dx,
    \quad
        \hat\alpha_\zeta
    =
    \frac1n\sum_{i=1}^n\zeta(X_i),
    \quad
    \alpha_\zeta
    =
    \int_\Omega\zeta(x)p(x)\,dx.
\end{align*}
Note that $\beta_{\max} \ge \int_\Omega p(x)\,dx = 1$. Since $\|\xi\|_{L^2(\Omega)}=1$ and $\beta_{\min}\le p(x)\le\beta_{\max}$ for all $x \in \Omega$,
\begin{align*}
    \mathrm{Var}(\xi(X_1))
    \le
    \Ep\bigl[\xi(X_1)^2\bigr]
    =
    \int_\Omega\xi(x)^2p(x)\,dx
    \le
    \beta_{\max}\int_\Omega\xi(x)^2\,dx
    =
    \beta_{\max}\|\xi\|_{L^2(\Omega)}^2
    =
    \beta_{\max},
\end{align*}
and $|\xi(X_1)-\alpha_\xi|\le2\|\xi\|_{L^\infty(\Omega)}\le2C_{\Psi}2^{jd/2}$. Similarly, we have $|\zeta(X_1)-\alpha_\zeta|\le 2C_{\Psi}$. Hence, applying Bernstein's inequality, we obtain for all $t>0$,
\begin{align} \label{eq:bern1}
    \mathbb{P}\bigl(|\hat\alpha_\xi-\alpha_\xi|\ge t\bigr)
    \le
    2\exp\Bigl(-\frac{nt^2}{2\beta_{\max}+\tfrac43C_{\Psi}2^{jd/2}t}\Bigr),
    \qquad
    \xi\in\Psi_j^{\mathrm{bc}},
\end{align}
and
\begin{align}\label{eq:bern2}
    \mathbb{P}\bigl(|\hat\alpha_\zeta-\alpha_\zeta|\ge t\bigr)
    \le
    2\exp\Bigl(-\frac{nt^2}{2\beta_{\max}+\tfrac43C_{\Psi}t}\Bigr),
    \qquad
    \zeta\in\Phi^{\mathrm{bc}}.
\end{align}
Since $\int_\Omega p(x)\,dx=1$, we have $\underline p\le1$. Let $u\in(0,1]$, which we will specify later, and let $t_j\coloneqq u2^{-jd/2}/(4C_{\Psi}^2J_n)$ for $j_0\le j\le J_n$ and $t_\Phi\coloneqq u/(4C_{\Psi}^2)$. Since $\tfrac43C_{\Psi}2^{jd/2}t_j=u/(3C_{\Psi}J_n)\le1$, $\tfrac43C_{\Psi}t_\Phi=u/(3C_{\Psi})\le1$ and $\beta_{\max}\ge1$, setting $t=t_j$ and $t=t_\Phi$ in~\eqref{eq:bern1} and~\eqref{eq:bern2}, respectively, yields the following bounds:
\begin{align}
    \mathbb{P}\Bigl(\max_{\xi\in\Psi_j^{\mathrm{bc}}}|\hat\alpha_\xi-\alpha_\xi|\ge t_j\Bigr)
    \le
    2C_{\Psi}2^{dj}\exp\Bigl(-\frac{nu^22^{-jd}}{48\beta_{\max} C_{\Psi}^4J_n^2}\Bigr),
    \label{eq:max_dev_wavelet}
\end{align}
and
\begin{align}
    \mathbb{P}\Bigl(\max_{\zeta\in\Phi^{\mathrm{bc}}}|\hat\alpha_\zeta-\alpha_\zeta|\ge t_\Phi\Bigr)
    \le
    2C_{\Psi}\exp\Bigl(-\frac{nu^2}{48\beta_{\max} C_{\Psi}^4}\Bigr).
    \label{eq:max_dev_scaling}
\end{align}
Recall that $\tilde p_n^{\mathrm{wav}}(x)-p_{J_n}^{\mathrm{bc}}(x)=\sum_{\zeta\in\Phi^{\mathrm{bc}}}(\hat\alpha_\zeta-\alpha_\zeta)\zeta(x)+\sum_{j=j_0}^{J_n}\sum_{\xi\in\Psi_j^{\mathrm{bc}}}(\hat\alpha_\xi-\alpha_\xi)\xi(x)$. On the event that $|\hat\alpha_\zeta-\alpha_\zeta|<t_\Phi$ for every $\zeta\in\Phi^{\mathrm{bc}}$ and $|\hat\alpha_\xi-\alpha_\xi|<t_j$ for every $\xi\in\Psi_j^{\mathrm{bc}}$ and every $j_0\le j\le J_n$, the bounds~\eqref{eq:C_Psi_bounds_1} and~\eqref{eq:C_Psi_bounds_2} in the definition of $C_{\Psi}$ yield, for every $x\in\Omega$,
\begin{align*}
    \bigl|\tilde p_n^{\mathrm{wav}}(x)-p_{J_n}^{\mathrm{bc}}(x)\bigr|
    \le
    C_{\Psi}^2t_\Phi+\sum_{j=j_0}^{J_n}C_{\Psi}^22^{jd/2}t_j
    =
    \frac{u}{4}+\frac{(J_n-j_0+1)u}{4J_n}
    <
    u,
\end{align*}
where the last inequality uses $j_0\ge1$. Hence, the event $\{\|\tilde p_n^{\mathrm{wav}}-p_{J_n}^{\mathrm{bc}}\|_{L^\infty(\Omega)}\ge u\}$ is contained in the union of the events $\{\max_{\zeta\in\Phi^{\mathrm{bc}}}|\hat\alpha_\zeta-\alpha_\zeta|\ge t_\Phi\}$ and $\{\max_{\xi\in\Psi_j^{\mathrm{bc}}}|\hat\alpha_\xi-\alpha_\xi|\ge t_j\}$ for $j_0\le j\le J_n$. Thus, the union bound, \eqref{eq:max_dev_scaling}, and \eqref{eq:max_dev_wavelet} yield
\begin{align*}
    \mathbb{P}\bigl(\|\tilde p_n^{\mathrm{wav}}-p_{J_n}^{\mathrm{bc}}\|_{L^\infty(\Omega)}\ge u\bigr)
    &\le
    2C_{\Psi}\exp\Bigl(-\frac{nu^2}{48\beta_{\max} C_{\Psi}^4}\Bigr)
    +
    \sum_{j=j_0}^{J_n}2C_{\Psi}2^{dj}\exp\Bigl(-\frac{nu^22^{-jd}}{48\beta_{\max} C_{\Psi}^4J_n^2}\Bigr) \\
    &\le 2(J_n-j_0+2)C_{\Psi}2^{dJ_n}\exp\Bigl(-\frac{nu^22^{-dJ_n}}{48\beta_{\max} C_{\Psi}^4J_n^2}\Bigr).
\end{align*}
Since $2^{J_n}\asymp n^{1/(d+2s)}$, the exponent is at least $cn^{2s/(d+2s)}/\log^2n$ for a constant $c>0$ depending only on $d$, $\beta_{\min}$, $\beta_{\max}$, and $C_{\Psi}$, while the prefactor is at most $Cn^{d/(d+2s)}\log n$ for a constant $C>0$. Hence, letting $u = \underline{p} / 4$, we have for every $A>0$ and all sufficiently large $n$,
\begin{align*}
    \mathbb{P}\left(
        \|\tilde p_n^{\mathrm{wav}}-p_{J_n}^{\mathrm{bc}}\|_{L^\infty(\Omega)}
        \ge \frac{\underline p}{4}
    \right)
    \le
    C_An^{-A}.
\end{align*}
By~\eqref{eq:Linfbdd}, we have
$
\|p_{J_n}^{\mathrm{bc}}-p\|_{L^\infty(\Omega)}\le C 2^{-J_ns}
$
for some $C>0$. Thus, $\|p_{J_n}^{\mathrm{bc}}-p\|_{L^\infty(\Omega)}\le\underline p/4$ for all sufficiently large $n$. Therefore, for such $n$, the event $\{\|\tilde p_n^{\mathrm{wav}}-p\|_{L^\infty(\Omega)}>\underline p/2\}$ is contained in the event $\{\|\tilde p_n^{\mathrm{wav}}-p_{J_n}^{\mathrm{bc}}\|_{L^\infty(\Omega)}\ge\underline p/4\}$. Hence, $\mathbb P(\|\tilde p_n^{\mathrm{wav}}-p\|_{L^\infty(\Omega)}>\underline p/2)\le C_An^{-A}$. The complement of this event implies $\inf_{x\in\Omega}\tilde p_n^{\mathrm{wav}}(x)\ge\underline p/2$. The same argument applied to $q$ and $Y_1,\dots,Y_m$ yields the corresponding bound for $\tilde q_m^{\mathrm{wav}}$. Hence, the union bound proves \eqref{eq:uot_wavelet_lower_bound_event}.

On $\mathcal H_{nm}$, we have $\tilde p_n^{\mathrm{wav}}(x)\ge \underline p/2>0$ for all $x \in \Omega$, so $(\tilde p_n^{\mathrm{wav}})_+ = \tilde p_n^{\mathrm{wav}}$ and consequently $\hat p_n^{\mathrm{wav}}=\tilde p_n^{\mathrm{wav}}$. Hence, $\tilde\mu_n^{\mathrm{wav}}=M_\mu\hat p_n^{\mathrm{wav}}\,dx\ge M_\mu(\underline p/2)\,dx$, while $\mu = M_\mu p\,dx\le M_\mu \beta_{\max}\,dx$, which combine to yield the first bound in \eqref{eq:wavelet_density_domination}. The same argument applied to $\tilde\nu_m^{\mathrm{wav}}$ and $\nu$ yields the second bound in \eqref{eq:wavelet_density_domination}.
\end{proof}

\subsection{Proof of the upper bound}

Throughout this section we work on the hypercube $\Omega=[0,1]^d$ and write
\begin{align*}
    \hat\mu_n^{\mathrm{wav}},\qquad \hat\nu_m^{\mathrm{wav}},\qquad \hat T_{nm}^{\mathrm{wav}},\qquad \hat\lambda_{nm}^{\mathrm{wav}},\qquad \widehat{\mathrm{UOT}}_{nm}^{\mathrm{wav}}
\end{align*}
for the objects of Definition~\ref{def:uot_wavelet_plugin}.

\begin{lemma}[Plug-in UOT excess identity]
\label{lem:uot_plugin_exact_excess}
Work on $\Omega=[0,1]^d$, and let
\[
    \hat\mu_n^{\mathrm{wav}},\ \hat\nu_m^{\mathrm{wav}},\ \widehat{\mathrm{UOT}}_{nm}^{\mathrm{wav}}
\]
be as in Definition~\ref{def:uot_wavelet_plugin}. Let $(\varphi_0,\psi_0)$ be an optimal dual pair for $\mathrm{UOT}(\mu,\nu)$, define
\begin{align*}
    \zeta_0(x) \coloneqq 1-e^{-\varphi_0(x)},
    \qquad
    \xi_0(y) \coloneqq 1-e^{-\psi_0(y)},
\end{align*}
and let $\hat\gamma_{nm}^{\mathrm{wav}}$ be any optimal transport plan for $\widehat{\mathrm{UOT}}_{nm}^{\mathrm{wav}}$ with marginals $(\hat\gamma_{0,nm}^{\mathrm{wav}},\hat\gamma_{1,nm}^{\mathrm{wav}})$. Define
\begin{align*}
    \gamma_{0,nm}^{\mathrm{wav},\mathrm{or}}
    &\coloneqq e^{-\varphi_0}\hat\mu_n^{\mathrm{wav}},
    \\
    \gamma_{1,nm}^{\mathrm{wav},\mathrm{or}}
    &\coloneqq e^{-\psi_0}\hat\nu_m^{\mathrm{wav}}.
\end{align*}
Then
\begin{equation}
\label{eq:uot_plugin_exact_excess}
\begin{aligned}
&\widehat{\mathrm{UOT}}_{nm}^{\mathrm{wav}}
- \int \zeta_0\,d\hat\mu_n^{\mathrm{wav}}
- \int \xi_0\,d\hat\nu_m^{\mathrm{wav}} \\
&\qquad=
\int_{\Omega\times\Omega}
\Bigl(
\tfrac{1}{2}\|x-y\|^2-\varphi_0(x)-\psi_0(y)
\Bigr)
  \,d\hat\gamma_{nm}^{\mathrm{wav}}(x,y) \\
&\qquad\quad
+ {D_\KL}(\hat\gamma_{0,nm}^{\mathrm{wav}}\mid \gamma_{0,nm}^{\mathrm{wav},\mathrm{or}})
+ {D_\KL}(\hat\gamma_{1,nm}^{\mathrm{wav}}\mid \gamma_{1,nm}^{\mathrm{wav},\mathrm{or}}).
\end{aligned}
\end{equation}
In particular,
\begin{align}
\label{eq:uot_plugin_exact_excess_nonnegative}
    \widehat{\mathrm{UOT}}_{nm}^{\mathrm{wav}}
    - \mathrm{UOT}(\mu,\nu)
    - \int \zeta_0\,d(\hat\mu_n^{\mathrm{wav}}-\mu)
    - \int \xi_0\,d(\hat\nu_m^{\mathrm{wav}}-\nu)
    \ge 0.
\end{align}
\end{lemma}

\begin{proof}[Proof of Lemma \ref{lem:uot_plugin_exact_excess}]
Since $\hat\gamma_{nm}^{\mathrm{wav}}$ is optimal for $\mathrm{UOT}(\hat\mu_n^{\mathrm{wav}},\hat\nu_m^{\mathrm{wav}})$,
\begin{align}
\label{eq:uot_plugin_exact_excess_start}
    \widehat{\mathrm{UOT}}_{nm}^{\mathrm{wav}}
    =
    \int \tfrac{1}{2}\|x-y\|^2\,d\hat\gamma_{nm}^{\mathrm{wav}}(x,y)
    + {D_\KL}\bigl(\hat\gamma_{0,nm}^{\mathrm{wav}}\mid \hat\mu_n^{\mathrm{wav}}\bigr)
    + {D_\KL}\bigl(\hat\gamma_{1,nm}^{\mathrm{wav}}\mid \hat\nu_m^{\mathrm{wav}}\bigr).
\end{align}
The computation of \eqref{eq:uot_cube_kernel_exact_excess_source} and \eqref{eq:uot_cube_kernel_exact_excess_target} in the proof of Corollary~\ref{cor:uot_cube_kernel_rate_transfer}, with $(\hat\mu_n^{\mathrm{ker}},\hat\nu_m^{\mathrm{ker}},\hat\gamma_{nm}^{\mathrm{ker}})$ replaced by $(\hat\mu_n^{\mathrm{wav}},\hat\nu_m^{\mathrm{wav}},\hat\gamma_{nm}^{\mathrm{wav}})$, yields
\begin{align}
\label{eq:uot_plugin_exact_excess_kl}
\begin{aligned}
    {D_\KL}\bigl(\hat\gamma_{0,nm}^{\mathrm{wav}}\mid \hat\mu_n^{\mathrm{wav}}\bigr)
    &= {D_\KL}\bigl(\hat\gamma_{0,nm}^{\mathrm{wav}}\mid \gamma_{0,nm}^{\mathrm{wav},\mathrm{or}}\bigr)
    - \int \varphi_0\,d\hat\gamma_{0,nm}^{\mathrm{wav}}
    + \int \zeta_0\,d\hat\mu_n^{\mathrm{wav}}, \\
    {D_\KL}\bigl(\hat\gamma_{1,nm}^{\mathrm{wav}}\mid \hat\nu_m^{\mathrm{wav}}\bigr)
    &= {D_\KL}\bigl(\hat\gamma_{1,nm}^{\mathrm{wav}}\mid \gamma_{1,nm}^{\mathrm{wav},\mathrm{or}}\bigr)
    - \int \psi_0\,d\hat\gamma_{1,nm}^{\mathrm{wav}}
    + \int \xi_0\,d\hat\nu_m^{\mathrm{wav}}.
\end{aligned}
\end{align}
Substituting \eqref{eq:uot_plugin_exact_excess_kl} into \eqref{eq:uot_plugin_exact_excess_start} and using $\int \varphi_0\,d\hat\gamma_{0,nm}^{\mathrm{wav}} + \int \psi_0\,d\hat\gamma_{1,nm}^{\mathrm{wav}} = \int (\varphi_0(x)+\psi_0(y))\,d\hat\gamma_{nm}^{\mathrm{wav}}(x,y)$ proves \eqref{eq:uot_plugin_exact_excess}.

Finally, duality yields
\begin{align*}
    \mathrm{UOT}(\mu,\nu)
    = \int \zeta_0\,d\mu + \int \xi_0\,d\nu,
\end{align*}
and \eqref{eq:uot_plugin_exact_excess_nonnegative} follows from $\tfrac12\|x-y\|^2-\varphi_0(x)-\psi_0(y)\ge0$ for all $x,y\in\Omega$ and the nonnegativity of the {$D_\KL$ terms}.
\end{proof}

\begin{theorem}[Plug-in UOT stability bound]
\label{thm:uot_plugin_stability_value}
Assume Assumptions~\ref{assm:curvature}, \ref{assm:positivity}, and \ref{assm:masses}, and work on $\Omega=[0,1]^d$. Let $C_\Lambda$ be the stability constant from Proposition~\ref{prop:two_sample_stability}, write $\rho_{\mu}\coloneqq\hat M_\mu/M_\mu$ and $\rho_{\nu}\coloneqq\hat M_\nu/M_\nu$, and define
\begin{equation}
\label{eq:uot_plugin_remainder}
\begin{aligned}
E_{nm}^{\mathrm{wav}}
&\coloneqq
C_\Lambda\Bigl(
W_2^2(\tilde\mu_n^{\mathrm{wav}},\mu)
+
W_2^2(\tilde\nu_m^{\mathrm{wav}},\nu)
\Bigr) \\
&\quad+
\bigl(1+\|\zeta_0\|_{L^\infty(\Omega)}\bigr)|\hat M_\mu-M_\mu|
+
\bigl(1+\|\xi_0\|_{L^\infty(\Omega)}\bigr)|\hat M_\nu-M_\nu| \\
&\quad+
e\sqrt{M_\mu M_\nu}\bigl(|\log\rho_{\mu}|+|\log\rho_{\nu}|\bigr).
\end{aligned}
\end{equation}
Then, almost surely,
\begin{equation}
\label{eq:uot_plugin_value_expansion}
\begin{aligned}
0
&\le
\widehat{\mathrm{UOT}}_{nm}^{\mathrm{wav}}
- \mathrm{UOT}(\mu,\nu)
- \int \zeta_0\,d(\hat\mu_n^{\mathrm{wav}}-\mu)
- \int \xi_0\,d(\hat\nu_m^{\mathrm{wav}}-\nu) \\
&\le
E_{nm}^{\mathrm{wav}}.
\end{aligned}
\end{equation}
In particular, almost surely,
\begin{equation}
\label{eq:uot_plugin_value_absolute}
\begin{aligned}
\left|
\widehat{\mathrm{UOT}}_{nm}^{\mathrm{wav}} - \mathrm{UOT}(\mu,\nu)
\right|
&\le
M_\mu\|\zeta_0\|_{L^\infty(\Omega)}
\|\hat p_n^{\mathrm{wav}}-p\|_{L^1(\Omega)} \\
&\quad+
M_\nu\|\xi_0\|_{L^\infty(\Omega)}
\|\hat q_m^{\mathrm{wav}}-q\|_{L^1(\Omega)} \\
&\quad+
\|\zeta_0\|_{L^\infty(\Omega)}|\hat M_\mu-M_\mu|
+
\|\xi_0\|_{L^\infty(\Omega)}|\hat M_\nu-M_\nu|
+
E_{nm}^{\mathrm{wav}}.
\end{aligned}
\end{equation}
\end{theorem}

\begin{proof}[Proof of Theorem \ref{thm:uot_plugin_stability_value}]
The lower bound in \eqref{eq:uot_plugin_value_expansion} is exactly \eqref{eq:uot_plugin_exact_excess_nonnegative} from Lemma~\ref{lem:uot_plugin_exact_excess}. Assumption~\ref{assm:supp_global} holds on $\Omega=[0,1]^d$. The measures $\hat\mu_n^{\mathrm{wav}}$ and $\hat\nu_m^{\mathrm{wav}}$ satisfy $\hat\mu_n^{\mathrm{wav}}=\rho_{\mu}\tilde\mu_n^{\mathrm{wav}}$ and $\hat\nu_m^{\mathrm{wav}}=\rho_{\nu}\tilde\nu_m^{\mathrm{wav}}$ with $\tilde\mu_n^{\mathrm{wav}}(\Omega)=M_\mu$ and $\tilde\nu_m^{\mathrm{wav}}(\Omega)=M_\nu$. Hence, \eqref{eq:estimated_mass_stability} holds almost surely with $\hat\mu_n,\hat\nu_m,\tilde\mu_n,\tilde\nu_m,\bar\mu_n,\bar\nu_m$ replaced by $\hat\mu_n^{\mathrm{wav}},\hat\nu_m^{\mathrm{wav}},\tilde\mu_n^{\mathrm{wav}},\tilde\nu_m^{\mathrm{wav}},\hat p_n^{\mathrm{wav}}\,dx,\hat q_m^{\mathrm{wav}}\,dx$, by Remark~\ref{rem:stability_general_measures}. Moreover, the proof of Theorem~\ref{thm:uot_plugin_density_rates} yields $M_\mu W_2^2(\hat p_n^{\mathrm{wav}}\,dx,p\,dx)=W_2^2(\tilde\mu_n^{\mathrm{wav}},\mu)$. We also have $M_\mu|\rho_{\mu}-1|=|\hat M_\mu-M_\mu|$. Since $\int_\Omega \hat p_n^{\mathrm{wav}}(x)\,dx=1$,
\begin{align}
\label{eq:uot_plugin_mass_term}
    \left|\int \zeta_0\,d(\hat\mu_n^{\mathrm{wav}}-\tilde\mu_n^{\mathrm{wav}})\right|
    = |\hat M_\mu-M_\mu|\left|\int_{\Omega} \zeta_0(x)\hat p_n^{\mathrm{wav}}(x)\,dx\right|
    \le \|\zeta_0\|_{L^\infty(\Omega)}|\hat M_\mu-M_\mu|.
\end{align}
The analogous bounds hold for $\hat\nu_m^{\mathrm{wav}}$, $\tilde\nu_m^{\mathrm{wav}}$, and $\xi_0$. Subtracting $\int\zeta_0\,d(\hat\mu_n^{\mathrm{wav}}-\mu)+\int\xi_0\,d(\hat\nu_m^{\mathrm{wav}}-\nu)$ from both sides of \eqref{eq:estimated_mass_stability} yields the upper bound in \eqref{eq:uot_plugin_value_expansion}.

For \eqref{eq:uot_plugin_value_absolute}, write
\begin{align*}
    A_{nm}^{\mathrm{wav}}
    &\coloneqq
    \int \zeta_0\,d(\hat\mu_n^{\mathrm{wav}}-\mu)
    + \int \xi_0\,d(\hat\nu_m^{\mathrm{wav}}-\nu),
    \\
    V_{nm}^{\mathrm{wav}}
    &\coloneqq
    \widehat{\mathrm{UOT}}_{nm}^{\mathrm{wav}} - \mathrm{UOT}(\mu,\nu) - A_{nm}^{\mathrm{wav}}.
\end{align*}
By \eqref{eq:uot_plugin_value_expansion}, $0\le V_{nm}^{\mathrm{wav}}\le E_{nm}^{\mathrm{wav}}$. Therefore
\begin{align*}
    \left|\widehat{\mathrm{UOT}}_{nm}^{\mathrm{wav}} - \mathrm{UOT}(\mu,\nu)\right|
    \le |A_{nm}^{\mathrm{wav}}| + E_{nm}^{\mathrm{wav}}.
\end{align*}
Since $\tilde\mu_n^{\mathrm{wav}}=M_\mu\hat p_n^{\mathrm{wav}}\,dx$ and $\mu=M_\mu p\,dx$, \eqref{eq:uot_plugin_mass_term} yields
\begin{align*}
    \left|\int \zeta_0\,d(\hat\mu_n^{\mathrm{wav}}-\mu)\right|
    &\le \left|\int \zeta_0\,d(\hat\mu_n^{\mathrm{wav}}-\tilde\mu_n^{\mathrm{wav}})\right| + M_\mu\left|\int_{\Omega} \zeta_0(x)(\hat p_n^{\mathrm{wav}}(x)-p(x))\,dx\right| \\
    &\le \|\zeta_0\|_{L^\infty(\Omega)}|\hat M_\mu-M_\mu| + M_\mu\|\zeta_0\|_{L^\infty(\Omega)}\|\hat p_n^{\mathrm{wav}}-p\|_{L^1(\Omega)}.
\end{align*}
The term $|\int\xi_0\,d(\hat\nu_m^{\mathrm{wav}}-\nu)|$ is treated in the same way, yielding \eqref{eq:uot_plugin_value_absolute}.
\end{proof}

\begin{theorem}[Plug-in transport map risk]
\label{thm:uot_plugin_map_risk}
Assume Assumptions~\ref{assm:curvature}, \ref{assm:positivity}, and \ref{assm:masses}, and work on $\Omega=[0,1]^d$. Suppose that $\mathrm{UOT}(\hat\mu_n^{\mathrm{wav}},\hat\nu_m^{\mathrm{wav}})$ is Monge-type, with optimal transport plan
\begin{align*}
    \hat\gamma_{nm}^{\mathrm{wav}} = (\mathrm{id},\hat T_{nm}^{\mathrm{wav}})_\#\hat\gamma_{0,nm}^{\mathrm{wav}},
    \qquad
    \hat\gamma_{0,nm}^{\mathrm{wav}} = (\hat a_{nm}^{\mathrm{wav}})^2\hat\mu_n^{\mathrm{wav}}.
\end{align*}
Then
\begin{align}
\label{eq:uot_plugin_map_risk_rate}
    \int_{\Omega}
    \|\hat T_{nm}^{\mathrm{wav}}(x)-T_0(x)\|^2\,d\hat\gamma_{0,nm}^{\mathrm{wav}}(x)
    \le
    \frac{2}{\kappa}E_{nm}^{\mathrm{wav}}.
\end{align}
Moreover, we have
\begin{equation}
\label{eq:uot_plugin_active_mass_mismatch}
\begin{aligned}
&{D_\KL}(\hat\gamma_{0,nm}^{\mathrm{wav}}\mid e^{-\varphi_0}\hat\mu_n^{\mathrm{wav}})
+
{D_\KL}(\hat\gamma_{1,nm}^{\mathrm{wav}}\mid e^{-\psi_0}\hat\nu_m^{\mathrm{wav}}) \le
E_{nm}^{\mathrm{wav}},
\end{aligned}
\end{equation}
and $\hat a_{nm}^{\mathrm{wav}}$ satisfies
\begin{align}
\label{eq:uot_plugin_active_factor_bound}
    \int_{\Omega}
    |\hat a_{nm}^{\mathrm{wav}}(x)-a_0(x)|^2\,d\hat\mu_n^{\mathrm{wav}}(x)
    \le
    {D_\KL}(\hat\gamma_{0,nm}^{\mathrm{wav}}\mid e^{-\varphi_0}\hat\mu_n^{\mathrm{wav}}).
\end{align}
Finally, there exists a constant $C_{\mathrm{tr}}>0$, depending only on $w_-$ and $\mathrm{diam}(\Omega)$, such that
\begin{align}
\label{eq:uot_plugin_T_to_mu_n}
    \int_{\Omega}
    \|\hat T_{nm}^{\mathrm{wav}}(x)-T_0(x)\|^2\,d\hat\mu_n^{\mathrm{wav}}(x)
    \le
    C_{\mathrm{tr}}(2/\kappa+1) E_{nm}^{\mathrm{wav}}.
\end{align}
\end{theorem}

\begin{proof}[Proof of Theorem \ref{thm:uot_plugin_map_risk}]
Since $\mathrm{UOT}(\mu,\nu)=\int\zeta_0\,d\mu+\int\xi_0\,d\nu$ and $\hat\gamma_{nm}^{\mathrm{wav}}=(\mathrm{id},\hat T_{nm}^{\mathrm{wav}})_\#\hat\gamma_{0,nm}^{\mathrm{wav}}$, Lemma~\ref{lem:uot_plugin_exact_excess} yields
\begin{align}
\label{eq:uot_plugin_map_risk_proof_start}
\begin{aligned}
&\widehat{\mathrm{UOT}}_{nm}^{\mathrm{wav}}
- \mathrm{UOT}(\mu,\nu)
- \int \zeta_0\,d(\hat\mu_n^{\mathrm{wav}}-\mu)
- \int \xi_0\,d(\hat\nu_m^{\mathrm{wav}}-\nu) \\
&\qquad=
\int
\Bigl(
\tfrac{1}{2}\|x-\hat T_{nm}^{\mathrm{wav}}(x)\|^2
- \varphi_0(x)-\psi_0(\hat T_{nm}^{\mathrm{wav}}(x))
\Bigr)
  \,d\hat\gamma_{0,nm}^{\mathrm{wav}}(x) \\
&\qquad\quad + {D_\KL}(\hat\gamma_{0,nm}^{\mathrm{wav}}\mid e^{-\varphi_0}\hat\mu_n^{\mathrm{wav}})
+ {D_\KL}(\hat\gamma_{1,nm}^{\mathrm{wav}}\mid e^{-\psi_0}\hat\nu_m^{\mathrm{wav}}).
\end{aligned}
\end{align}
Assumption~\ref{assm:supp_global} holds on $\Omega=[0,1]^d$. Hence, Lemma~\ref{lem:uot_gap_sufficient} yields
\begin{align*}
    \tfrac{1}{2}\|x-\hat T_{nm}^{\mathrm{wav}}(x)\|^2
    - \varphi_0(x)-\psi_0(\hat T_{nm}^{\mathrm{wav}}(x))
    \ge \frac{\kappa}{2}\|\hat T_{nm}^{\mathrm{wav}}(x)-T_0(x)\|^2
\end{align*}
for all $x\in\Omega$. Substituting this lower bound into \eqref{eq:uot_plugin_map_risk_proof_start} and discarding ${D_\KL}(\hat\gamma_{0,nm}^{\mathrm{wav}}\mid e^{-\varphi_0}\hat\mu_n^{\mathrm{wav}})+{D_\KL}(\hat\gamma_{1,nm}^{\mathrm{wav}}\mid e^{-\psi_0}\hat\nu_m^{\mathrm{wav}})\ge0$ yields the first inequality below. The second inequality is exactly Theorem~\ref{thm:uot_plugin_stability_value}:
\begin{equation}
\begin{aligned}
\frac{\kappa}{2}
\int_{\Omega}
\|\hat T_{nm}^{\mathrm{wav}}(x)-T_0(x)\|^2\,d\hat\gamma_{0,nm}^{\mathrm{wav}}(x)
&\le
\widehat{\mathrm{UOT}}_{nm}^{\mathrm{wav}}
- \mathrm{UOT}(\mu,\nu) \\
&\quad - \int \zeta_0\,d(\hat\mu_n^{\mathrm{wav}}-\mu)
- \int \xi_0\,d(\hat\nu_m^{\mathrm{wav}}-\nu) \\
&\le
E_{nm}^{\mathrm{wav}}.
\end{aligned}
\end{equation}
Hence, \eqref{eq:uot_plugin_map_risk_rate} follows immediately. Keeping the {$D_\KL$ terms} in \eqref{eq:uot_plugin_map_risk_proof_start} and using once again the upper bound in \eqref{eq:uot_plugin_value_expansion} proves \eqref{eq:uot_plugin_active_mass_mismatch}.

For \eqref{eq:uot_plugin_active_factor_bound}, the identities $\hat\gamma_{0,nm}^{\mathrm{wav}}=(\hat a_{nm}^{\mathrm{wav}})^2\hat\mu_n^{\mathrm{wav}}$ and $e^{-\varphi_0}\hat\mu_n^{\mathrm{wav}}=a_0^2\hat\mu_n^{\mathrm{wav}}$ allow us to apply Lemma~\ref{lem:uot_kl_sqrt_lower_bound} pointwise with $a=(\hat a_{nm}^{\mathrm{wav}})^2$ and $b=a_0^2$. Integrating against $\hat\mu_n^{\mathrm{wav}}$ yields
\begin{align*}
    \int_{\Omega}|\hat a_{nm}^{\mathrm{wav}}(x)-a_0(x)|^2\,d\hat\mu_n^{\mathrm{wav}}(x)
    \le {D_\KL}(\hat\gamma_{0,nm}^{\mathrm{wav}}\mid e^{-\varphi_0}\hat\mu_n^{\mathrm{wav}}).
\end{align*}

To prove \eqref{eq:uot_plugin_T_to_mu_n}, we bound $\int\|\hat T_{nm}^{\mathrm{wav}}-T_0\|^2\,d\hat\mu_n^{\mathrm{wav}}$ in terms of $\int\|\hat T_{nm}^{\mathrm{wav}}-T_0\|^2\,d\hat\gamma_{0,nm}^{\mathrm{wav}}$. By the choice of the clipping constants in Section~\ref{sec:plan}, $a_0(x)^2=e^{-\varphi_0(x)}\ge w_-$ for all $x\in\Omega$. Because $\hat T_{nm}^{\mathrm{wav}}$ and $T_0$ both take values in $\Omega=[0,1]^d$, we also have $\|\hat T_{nm}^{\mathrm{wav}}(x)-T_0(x)\|^2\le \mathrm{diam}(\Omega)^2 = d$. Consequently,
\begin{align*}
    \|\hat T_{nm}^{\mathrm{wav}}(x)-T_0(x)\|^2
    &\le \frac{1}{w_-} a_0(x)^2 \|\hat T_{nm}^{\mathrm{wav}}(x)-T_0(x)\|^2 \\
    &\le \frac{2}{w_-}\hat a_{nm}^{\mathrm{wav}}(x)^2\|\hat T_{nm}^{\mathrm{wav}}(x)-T_0(x)\|^2 \\
    &\quad + \frac{2}{w_-}|\hat a_{nm}^{\mathrm{wav}}(x)-a_0(x)|^2\|\hat T_{nm}^{\mathrm{wav}}(x)-T_0(x)\|^2 \\
    &\le
    C_{\mathrm{tr}}\,\hat a_{nm}^{\mathrm{wav}}(x)^2\,\|\hat T_{nm}^{\mathrm{wav}}(x)-T_0(x)\|^2
    + C_{\mathrm{tr}}\,|\hat a_{nm}^{\mathrm{wav}}(x)-a_0(x)|^2,
\end{align*}
for a constant $C_{\mathrm{tr}}>0$ depending only on $w_-$ and $\mathrm{diam}(\Omega)$. Integrating against $d\hat\mu_n^{\mathrm{wav}}$ and using $d\hat\gamma_{0,nm}^{\mathrm{wav}}=\hat a_{nm}^{\mathrm{wav}}(x)^2\,d\hat\mu_n^{\mathrm{wav}}(x)$, \eqref{eq:uot_plugin_map_risk_rate}, \eqref{eq:uot_plugin_active_mass_mismatch}, and \eqref{eq:uot_plugin_active_factor_bound} yields \eqref{eq:uot_plugin_T_to_mu_n}.
\end{proof}

\begin{corollary}[Plug-in growth map risk]
\label{cor:uot_plugin_growth_risk}
Assume that $\Omega=[0,1]^d$ and that Assumptions~\ref{assm:curvature}, \ref{assm:positivity}, and \ref{assm:masses} hold. Let
\[
    \hat\mu_n^{\mathrm{wav}},\ \hat\nu_m^{\mathrm{wav}},\ \hat T_{nm}^{\mathrm{wav}},\ \hat a_{nm}^{\mathrm{wav}},\ \hat\lambda_{nm}^{\mathrm{wav}}
\]
be as in Definition~\ref{def:uot_wavelet_plugin}. Suppose that $\mathrm{UOT}(\hat\mu_n^{\mathrm{wav}},\hat\nu_m^{\mathrm{wav}})$ is Monge-type, with optimal transport plan
\begin{align*}
    \hat\gamma_{nm}^{\mathrm{wav}} = (\mathrm{id},\hat T_{nm}^{\mathrm{wav}})_\#\hat\gamma_{0,nm}^{\mathrm{wav}},
    \qquad
    \hat\gamma_{0,nm}^{\mathrm{wav}} = (\hat a_{nm}^{\mathrm{wav}})^2\hat\mu_n^{\mathrm{wav}}.
\end{align*}
Then, with $C>0$ the constant of Lemma~\ref{lem:continuous_gh_growth_transfer},
\begin{align}
\label{eq:uot_plugin_growth_risk}
    \int_{\Omega}
    |\hat\lambda_{nm}^{\mathrm{wav}}(x)-\lambda_0(x)|^2\,d\hat\mu_n^{\mathrm{wav}}(x)
    \le
    C(1+2/\kappa)E_{nm}^{\mathrm{wav}}.
\end{align}
\end{corollary}

\begin{proof}[Proof of Corollary \ref{cor:uot_plugin_growth_risk}]
By Lemma~\ref{lem:continuous_gh_growth_transfer} applied with $\hat\eta=\hat\mu_n^{\mathrm{wav}}$, $\hat a=\hat a_{nm}^{\mathrm{wav}}$, $\hat T=\hat T_{nm}^{\mathrm{wav}}$, and $\hat\gamma_0=\hat\gamma_{0,nm}^{\mathrm{wav}}$,
\begin{align}
\label{eq:uot_plugin_growth_decomp}
    \int_{\Omega}|\hat\lambda_{nm}^{\mathrm{wav}}-\lambda_0|^2\,d\hat\mu_n^{\mathrm{wav}}
    \le
    C\int_{\Omega}|\hat a_{nm}^{\mathrm{wav}}-a_0|^2\,d\hat\mu_n^{\mathrm{wav}}
    +C\int_{\Omega}\|\hat T_{nm}^{\mathrm{wav}}-T_0\|^2\,d\hat\gamma_{0,nm}^{\mathrm{wav}}.
\end{align}
The first term on the right-hand side of \eqref{eq:uot_plugin_growth_decomp} is at most $CE_{nm}^{\mathrm{wav}}$ by \eqref{eq:uot_plugin_active_factor_bound} and \eqref{eq:uot_plugin_active_mass_mismatch}. The second term is at most $(2C/\kappa)E_{nm}^{\mathrm{wav}}$ by \eqref{eq:uot_plugin_map_risk_rate}, as claimed.
\end{proof}

\begin{corollary}[Risk of the wavelet plug-in estimator]
\label{cor:uot_plugin_rate_transfer}
Assume that $\Omega=[0,1]^d$, that $c(x,y)=\tfrac{1}{2}\|x-y\|^2$, that Assumptions~\ref{assm:curvature}, \ref{assm:positivity}, and \ref{assm:masses} hold, and that the hypotheses of Theorem~\ref{thm:uot_plugin_density_rates} hold. Let
\[
    \hat\mu_n^{\mathrm{wav}},\ \hat\nu_m^{\mathrm{wav}},\ \hat T_{nm}^{\mathrm{wav}},\ \hat\lambda_{nm}^{\mathrm{wav}},\ \widehat{\mathrm{UOT}}_{nm}^{\mathrm{wav}}
\]
be as in Definition~\ref{def:uot_wavelet_plugin}, and let $a_n,b_m$ be the sequences from Assumption~\ref{assm:masses}. Assume that for some deterministic sequences $r_{n,\mathrm{wav}}^\mu,r_{m,\mathrm{wav}}^\nu,\ell_{n,\mathrm{wav}},\ell_{m,\mathrm{wav}}\ge 0$, where $r_{n,\mathrm{wav}}^\mu\ge M_\mu/n$ and $r_{m,\mathrm{wav}}^\nu\ge M_\nu/m$,
\begin{align}
\label{eq:uot_plugin_assumed_w2}
    \Ep\bigl[W_2^2(\tilde\mu_n^{\mathrm{wav}},\mu)\bigr] \le r_{n,\mathrm{wav}}^\mu,
    \qquad
    \Ep\bigl[W_2^2(\tilde\nu_m^{\mathrm{wav}},\nu)\bigr] \le r_{m,\mathrm{wav}}^\nu.
\end{align}
Then there exists a constant $C>0$ such that, for all sufficiently large $n,m$,
\begin{align}
\label{eq:uot_plugin_rate_transfer_map}
    \Ep\left[
        \int_{\Omega}
        \|\hat T_{nm}^{\mathrm{wav}}(x)-T_0(x)\|^2\,d\mu(x)
    \right]
    \le
    C(r_{n,\mathrm{wav}}^\mu+r_{m,\mathrm{wav}}^\nu),
\end{align}
\begin{align}
\label{eq:uot_plugin_rate_transfer_lambda}
    \Ep\left[
        \int_{\Omega}
        |\hat\lambda_{nm}^{\mathrm{wav}}(x)-\lambda_0(x)|^2\,d\mu(x)
    \right]
    \le
    C(r_{n,\mathrm{wav}}^\mu+r_{m,\mathrm{wav}}^\nu+a_n+b_m),
\end{align}
and
\begin{align}
\label{eq:uot_plugin_rate_transfer_value}
\begin{aligned}
    &\Ep\left[
        \left|
        \widehat{\mathrm{UOT}}_{nm}^{\mathrm{wav}}
        - \mathrm{UOT}(\mu,\nu)
        - \int \zeta_0\,d(\hat\mu_n^{\mathrm{wav}}-\mu)
        - \int \xi_0\,d(\hat\nu_m^{\mathrm{wav}}-\nu)
        \right|
    \right] \\
    &\qquad\le
    C(r_{n,\mathrm{wav}}^\mu+r_{m,\mathrm{wav}}^\nu+a_n+b_m).
\end{aligned}
\end{align}
If, in addition,
\begin{align}
\label{eq:uot_plugin_assumed_l1}
    \Ep\bigl[\|\hat p_n^{\mathrm{wav}}-p\|_{L^1(\Omega)}\bigr] \le \ell_{n,\mathrm{wav}},
    \qquad
    \Ep\bigl[\|\hat q_m^{\mathrm{wav}}-q\|_{L^1(\Omega)}\bigr] \le \ell_{m,\mathrm{wav}},
\end{align}
then, for all sufficiently large $n,m$,
\begin{align}
\label{eq:uot_plugin_rate_transfer_abs_value}
\begin{aligned}
    \Ep\left[
        \left|
        \widehat{\mathrm{UOT}}_{nm}^{\mathrm{wav}} - \mathrm{UOT}(\mu,\nu)
        \right|
    \right]
    &\le M_\mu\|\zeta_0\|_{L^\infty(\Omega)}\ell_{n,\mathrm{wav}}
    + M_\nu\|\xi_0\|_{L^\infty(\Omega)}\ell_{m,\mathrm{wav}} \\
    &\quad + C(r_{n,\mathrm{wav}}^\mu+r_{m,\mathrm{wav}}^\nu+a_n+b_m).
\end{aligned}
\end{align}
\end{corollary}

\begin{proof}[Proof of Corollary \ref{cor:uot_plugin_rate_transfer}]
On the event $\mathcal H_{nm}$ of Lemma~\ref{lem:uot_wavelet_lower_bound_event}, the density estimates $\hat p_n^{\mathrm{wav}}$ and $\hat q_m^{\mathrm{wav}}$ are bounded below by $\underline p/2$ and $\underline q/2$. Assumption~\ref{assm:masses} also yields $\hat M_\mu,\hat M_\nu>0$. As in the proof of Corollary~\ref{cor:uot_cube_kernel_rate_transfer}, $(\hat\varphi_{nm}^{\mathrm{wav}},\hat\psi_{nm}^{\mathrm{wav}})$ is therefore the optimal dual pair given by Theorem~\ref{prop:monge-potential}. Hence, $\mathrm{UOT}(\hat\mu_n^{\mathrm{wav}},\hat\nu_m^{\mathrm{wav}})$ is Monge-type by the last sentence of Definition~\ref{def:uot_wavelet_plugin}, so \eqref{eq:uot_plugin_T_to_mu_n} and \eqref{eq:uot_plugin_growth_risk} hold on $\mathcal H_{nm}$. Moreover, Assumption~\ref{assm:masses} yields $\rho_{\mu},\rho_{\nu}\ge\bar c$, so $d\tilde\mu_n^{\mathrm{wav}}=\rho_{\mu}^{-1}\,d\hat\mu_n^{\mathrm{wav}}\le\bar c^{-1}\,d\hat\mu_n^{\mathrm{wav}}$. Combining this bound with \eqref{eq:wavelet_density_domination}, we obtain, on $\mathcal H_{nm}$,
\begin{align}
\label{eq:wavelet_density_domination_hat}
    d\mu \le \frac{2\beta_{\max}}{\bar c\,\underline p}\,d\hat\mu_n^{\mathrm{wav}},
    \qquad
    d\nu \le \frac{2\beta_{\max}}{\bar c\,\underline q}\,d\hat\nu_m^{\mathrm{wav}}.
\end{align}

We next bound the expectation of $E_{nm}^{\mathrm{wav}}$. Since $\rho_{\mu}\ge\bar c$ almost surely, the mean-value theorem applied to $\log$ on $[\bar c,\infty)$ yields $|\log\rho_{\mu}|\le|\rho_{\mu}-1|/\bar c$ almost surely, and therefore $e\sqrt{M_\mu M_\nu}\,\Ep[|\log\rho_{\mu}|]\le\bar c^{-1}e\sqrt{M_\nu/M_\mu}\,a_n$. The analogous bound holds for $\rho_{\nu}$ and $b_m$. Combining these bounds with \eqref{eq:uot_plugin_assumed_w2} and Assumption~\ref{assm:masses}, we obtain, for all sufficiently large $n,m$,
\begin{align}
\label{eq:uot_plugin_remainder_expect}
    \Ep\bigl[E_{nm}^{\mathrm{wav}}\bigr]
    \le
    C_\Lambda(r_{n,\mathrm{wav}}^\mu+r_{m,\mathrm{wav}}^\nu)+C(a_n+b_m).
\end{align}

We now prove \eqref{eq:uot_plugin_rate_transfer_map}, whose right-hand side has no terms $a_n,b_m$. We have $\hat\mu_n^{\mathrm{wav}}=\rho_{\mu}\tilde\mu_n^{\mathrm{wav}}$, $\hat\nu_m^{\mathrm{wav}}=\rho_{\nu}\tilde\nu_m^{\mathrm{wav}}$, $\tilde\mu_n^{\mathrm{wav}}(\Omega)=M_\mu$, and $\tilde\nu_m^{\mathrm{wav}}(\Omega)=M_\nu$. Thus, the identity~\eqref{eq:uot_cube_kernel_dual_shift} in the proof of Corollary~\ref{cor:uot_cube_kernel_rate_transfer}, with $(\hat\mu_n^{\mathrm{ker}},\hat\nu_m^{\mathrm{ker}},\tilde\mu_n^{\mathrm{ker}},\tilde\nu_m^{\mathrm{ker}})$ replaced by $(\hat\mu_n^{\mathrm{wav}},\hat\nu_m^{\mathrm{wav}},\tilde\mu_n^{\mathrm{wav}},\tilde\nu_m^{\mathrm{wav}})$, shows that the pair $(\hat\varphi_{nm}^{\mathrm{wav}}-\frac12\log(\rho_{\mu}/\rho_{\nu}),\hat\psi_{nm}^{\mathrm{wav}}+\frac12\log(\rho_{\mu}/\rho_{\nu}))$ is an optimal dual pair for $\mathrm{UOT}(\tilde\mu_n^{\mathrm{wav}},\tilde\nu_m^{\mathrm{wav}})$.

For every $x\in\Omega$, the functions $y\mapsto\frac12\|x-y\|^2-\hat\psi_{nm}^{\mathrm{wav}}(y)$ and $y\mapsto\frac12\|x-y\|^2-\hat\psi_{nm}^{\mathrm{wav}}(y)-\frac12\log(\rho_{\mu}/\rho_{\nu})$ have the same minimizers. Therefore, $\hat T_{nm}^{\mathrm{wav}}$ is also the estimator of $T_0$ built from this pair.

Consequently, we may take $\hat M_\mu=M_\mu$ and $\hat M_\nu=M_\nu$ when bounding the risk of $\hat T_{nm}^{\mathrm{wav}}$. In this case, Assumption~\ref{assm:masses} holds with $\bar c=1$ and arbitrary positive sequences $a_n,b_m\to0$. Moreover, $\hat\mu_n^{\mathrm{wav}}=\tilde\mu_n^{\mathrm{wav}}$. With these masses, the remainder \eqref{eq:uot_plugin_remainder} reduces to
\begin{align*}
    \tilde E_{nm}^{\mathrm{wav}}\coloneqq C_\Lambda\Bigl(W_2^2(\tilde\mu_n^{\mathrm{wav}},\mu)+W_2^2(\tilde\nu_m^{\mathrm{wav}},\nu)\Bigr).
\end{align*}
On $\mathcal H_{nm}$, \eqref{eq:wavelet_density_domination} and \eqref{eq:uot_plugin_T_to_mu_n}, with $E_{nm}^{\mathrm{wav}}$ replaced by $\tilde E_{nm}^{\mathrm{wav}}$, yield
\begin{align*}
    \int_{\Omega}\|\hat T_{nm}^{\mathrm{wav}}(x)-T_0(x)\|^2\,d\mu(x)
    &\le \frac{2\beta_{\max}}{\underline p}
    \int_{\Omega}\|\hat T_{nm}^{\mathrm{wav}}(x)-T_0(x)\|^2\,d\tilde\mu_n^{\mathrm{wav}}(x) \\
    &\le \frac{2\beta_{\max}}{\underline p}\,C_{\mathrm{tr}}(2/\kappa+1)\tilde E_{nm}^{\mathrm{wav}}.
\end{align*}
The functions $\tilde p_n^{\mathrm{wav}}$ and $\tilde q_m^{\mathrm{wav}}$ that define $\mathcal H_{nm}$ do not involve $\hat M_\mu$ and $\hat M_\nu$. Since $\tilde E_{nm}^{\mathrm{wav}}\ge0$, taking expectations and using \eqref{eq:uot_plugin_assumed_w2} therefore yields
\begin{align}
\label{eq:uot_plugin_map_good_event}
    \Ep\left[
        \int_{\Omega}\|\hat T_{nm}^{\mathrm{wav}}(x)-T_0(x)\|^2\,d\mu(x)
        \mathbf 1_{\mathcal H_{nm}}
    \right]
    \le C(r_{n,\mathrm{wav}}^\mu+r_{m,\mathrm{wav}}^\nu).
\end{align}
On $\mathcal H_{nm}^c$, the argument used on $\mathcal{G}_{nm}^c$ in the proof of Corollary~\ref{cor:uot_cube_kernel_rate_transfer}, applied to $\hat\mu_n^{\mathrm{wav}}$ and $\hat\nu_m^{\mathrm{wav}}$, yields an optimal dual pair. Thus, $\hat T_{nm}^{\mathrm{wav}}$ is defined on this event. Both $\hat T_{nm}^{\mathrm{wav}}$ and $T_0$ take values in $\Omega=[0,1]^d$, so $\|\hat T_{nm}^{\mathrm{wav}}(x)-T_0(x)\|^2\le d$, and \eqref{eq:uot_wavelet_lower_bound_event} yields
\begin{align}
\label{eq:uot_plugin_complement}
    \Ep\left[
        \int_{\Omega}\|\hat T_{nm}^{\mathrm{wav}}(x)-T_0(x)\|^2\,d\mu(x)
        \mathbf 1_{\mathcal H_{nm}^c}
    \right]
    \le d M_\mu\mathbb{P}(\mathcal H_{nm}^c)
    \le dM_\mu C_A(n^{-A}+m^{-A}).
\end{align}
We take $A=2$. Since $r_{n,\mathrm{wav}}^\mu\ge M_\mu/n$ and $r_{m,\mathrm{wav}}^\nu\ge M_\nu/m$, the right-hand side of \eqref{eq:uot_plugin_complement} is at most $dC_A(1+M_\mu/M_\nu)(r_{n,\mathrm{wav}}^\mu+r_{m,\mathrm{wav}}^\nu)$. Adding \eqref{eq:uot_plugin_map_good_event} and \eqref{eq:uot_plugin_complement} proves \eqref{eq:uot_plugin_rate_transfer_map}.

For the growth map, \eqref{eq:wavelet_density_domination_hat}, combined with \eqref{eq:uot_plugin_growth_risk}, yields
\begin{align*}
    \int_{\Omega}|\hat\lambda_{nm}^{\mathrm{wav}}(x)-\lambda_0(x)|^2\,d\mu(x)
    \le \frac{2\beta_{\max}}{\bar c\,\underline p}\,
    C(1+2/\kappa)E_{nm}^{\mathrm{wav}}
    \quad \text{on }\mathcal H_{nm}.
\end{align*}
On $\mathcal H_{nm}^c$, both $\hat\lambda_{nm}^{\mathrm{wav}}$ and $\lambda_0$ are bounded by $w_+e^{\mathrm{diam}(\Omega)^2/4}$, so $|\hat\lambda_{nm}^{\mathrm{wav}}-\lambda_0|^2$ is uniformly bounded. Hence, the same argument as in \eqref{eq:uot_plugin_complement} absorbs the contribution of the complement of $\mathcal H_{nm}$ into $C(r_{n,\mathrm{wav}}^\mu+r_{m,\mathrm{wav}}^\nu)$. Taking expectations on $\mathcal H_{nm}$ and applying \eqref{eq:uot_plugin_remainder_expect} proves \eqref{eq:uot_plugin_rate_transfer_lambda}.

For \eqref{eq:uot_plugin_rate_transfer_value}, \eqref{eq:uot_plugin_value_expansion} implies
\begin{align*}
    0
    &\le
    \widehat{\mathrm{UOT}}_{nm}^{\mathrm{wav}}
    - \mathrm{UOT}(\mu,\nu)
    - \int \zeta_0\,d(\hat\mu_n^{\mathrm{wav}}-\mu)
    - \int \xi_0\,d(\hat\nu_m^{\mathrm{wav}}-\nu)\\
    &\le
    E_{nm}^{\mathrm{wav}},
\end{align*}
so taking expectations and applying \eqref{eq:uot_plugin_remainder_expect} proves \eqref{eq:uot_plugin_rate_transfer_value}. Finally, taking expectations in \eqref{eq:uot_plugin_value_absolute} and using \eqref{eq:uot_plugin_assumed_l1}, Assumption~\ref{assm:masses}, and \eqref{eq:uot_plugin_remainder_expect} yields \eqref{eq:uot_plugin_rate_transfer_abs_value}.
\end{proof}

\begin{proof}[Proof of Theorem \ref{thm:main_wavelet_rates}]
By Assumption~\ref{assm:positivity} and the hypotheses of Theorem~\ref{thm:main_wavelet_rates}, the hypotheses of Theorem~\ref{thm:uot_plugin_density_rates} hold. Hence, we apply Corollary~\ref{cor:uot_plugin_rate_transfer} with
\begin{align*}
    r_{n,\mathrm{wav}}^\mu = C M_\mu\mathfrak R_n^{\mathrm{wav}}(\alpha),
    \quad
    r_{m,\mathrm{wav}}^\nu = C M_\nu\mathfrak R_m^{\mathrm{wav}}(\alpha),
    \quad
    \ell_{n,\mathrm{wav}} = C\mathfrak L_n(\alpha),
    \quad
    \ell_{m,\mathrm{wav}} = C\mathfrak L_m(\alpha),
\end{align*}
as supplied by Theorem~\ref{thm:uot_plugin_density_rates}, where we enlarge $C$ so that $C\ge1$. Since $\mathfrak R_n^{\mathrm{wav}}(\alpha)\ge1/n$ for all $n\ge3$, we obtain $r_{n,\mathrm{wav}}^\mu\ge M_\mu/n$ and $r_{m,\mathrm{wav}}^\nu\ge M_\nu/m$ for all $n,m\ge3$. The constants $M_\mu\|\zeta_0\|_{L^\infty(\Omega)}$ and $M_\nu\|\xi_0\|_{L^\infty(\Omega)}$ are absorbed into the generic constant $C$.
\end{proof}

\section{Proof of Theorem~\ref{thm:lb-ubot}}

\subsection{Statistical model and the class $\mathcal U_\alpha(M,B,R)$}

Let $\Omega=[0,1]^d$. We assume that $M_\mu,M_\nu$ are known and observe two independent samples $X_1,\dots,X_n\overset{\mathrm{i.i.d.}}{\sim}\bar\mu\coloneqq\mu/M_\mu$ and $Y_1,\dots,Y_n\overset{\mathrm{i.i.d.}}{\sim}\bar\nu\coloneqq\nu/M_\nu$. We write $\mathbb P_{\mu,\nu}^n$ and $\mathbb E_{\mu,\nu}^n$ for the joint distribution and expectation. The KL divergence between the joint distributions induced by two pairs of source and target measures $(\mu,\nu)$ and $(\mu',\nu')$ is
\begin{align}
    \label{eq:iid_kl}
    D_\KL(\mathbb P_{\mu,\nu}^n\mid \mathbb P_{\mu',\nu'}^n)
    =
    n\bigl[D_\KL(\bar\mu\mid \bar\mu')+D_\KL(\bar\nu\mid \bar\nu')\bigr].
\end{align}

For the appendix proof, write $\mathcal U_\alpha$ for $\mathcal U_\alpha(M,B,R)$ from Section~\ref{sec:minimax_lower_bound}. For $(\mu,\nu)\in\mathcal U_\alpha$, let $(T_{\mu,\nu},\lambda_{\mu,\nu})$ denote the transport-growth pair $(T_0,\lambda_0)$ given by Theorem~\ref{prop:monge-potential} for $(\mu,\nu)$. We also write the maximum risks of estimators $\hat T$ and $\hat\lambda$ as
\begin{align}
    \mathfrak M_n^T(\hat T)
    &\coloneqq
    \sup_{(\mu,\nu)\in\mathcal U_\alpha}
    \mathbb E_{\mu,\nu}^n\left[
        \int_{[0,1]^d}\|\hat T(x)-T_{\mu,\nu}(x)\|^2\,d\mu(x)
    \right],\\
    \mathfrak M_n^\lambda(\hat\lambda)
    &\coloneqq
    \sup_{(\mu,\nu)\in\mathcal U_\alpha}
    \mathbb E_{\mu,\nu}^n\left[
        \int_{[0,1]^d}|\hat\lambda(x)-\lambda_{\mu,\nu}(x)|^2\,d\mu(x)
    \right].
\end{align}

\subsection{Tools for minimax lower bounds}

Throughout, we use standard tools for minimax lower bounds. See, e.g., \cite[Theorems~2.2 and~2.5]{tsybakov2009nonparametric}, restated here for convenience.

\begin{external}[Lower bound from two hypotheses, \cite{tsybakov2009nonparametric}]\label{thm:tsy-twopoint}
    If probability measures $P_0,P_1$ satisfy $D_\KL(P_1\mid P_0)\le \varrho<\infty$ and $d(\theta_0,\theta_1)\ge 2s$, then there exists a constant $c_\varrho>0$ such that $\inf_{\hat\theta}\max_{i=0,1}P_i(d(\hat\theta,\theta_i)\ge s)\ge c_\varrho$.
\end{external}
\begin{external}[Lower bound from multiple hypotheses, \cite{tsybakov2009nonparametric}]\label{thm:tsy-fano}
    Let $(\Theta,d)$ be a pseudometric space and $\{P_\theta\}_{\theta\in\Theta}$ a family of probability measures. Suppose there exist $\theta_0,\theta_1,\dots,\theta_K\in\Theta$ and $s>0$ such that $P_{\theta_k} \ll P_{\theta_0}$ for $k=1,\ldots,K$, $d(\theta_k,\theta_{k'})\ge 2s$ for $0\le k\ne k'\le K$, and $K^{-1}\sum_{k=1}^KD_\KL(P_{\theta_k}\mid P_{\theta_0})\le (\log K)/9$, with $K\ge2$. Then there exists a universal constant $c>0$ such that $\inf_{\hat\theta}\sup_k P_{\theta_k}(d(\hat\theta,\theta_k)\ge s)\ge c$.
\end{external}
\begin{external}[Varshamov-Gilbert bound, {\cite[Lemma~2.9]{tsybakov2009nonparametric}}]\label{lem:vg}
    For every integer $N_J\ge 8$ there exist binary strings $\omega^{(0)},\omega^{(1)},\dots,\omega^{(K)}\in\{0,1\}^{N_J}$ with $\omega^{(0)}=\mathbf 0$, $K\ge 2^{N_J/8}$, and pairwise Hamming separation
    \begin{align}
        \|\omega^{(k)}-\omega^{(k')}\|^2 = \#\{i:\omega^{(k)}_i\ne\omega^{(k')}_i\} \ge N_J/8 \qquad \text{for all }0\le k\ne k'\le K.
    \end{align}
\end{external}

In the proof of Theorem~\ref{thm:lb-ubot}, we define each hypothesis through its potential, so we also record a lemma that identifies the transport-growth pair associated with a given $c$-concave potential.
\begin{lemma}[Identifiability of the transport-growth pair]
\label{lem:verification}
Let $\Omega\subset\R^d$ be compact and convex, and let $\mu=\rho_0\,dx$ with $\rho_0>0$ a.e.\ on $\Omega$. Let $\varphi:\R^d\to\R$ be twice continuously differentiable and such that $f\coloneqq\frac12\|\cdot\|^2-\varphi$ satisfies $\nabla^2 f(x)\succ0$ for all $x\in\R^d$ and $T\coloneqq\mathrm{id}-\nabla\varphi$ maps $\Omega$ into $\Omega$. Define $\psi\coloneqq\varphi^c$ and
\begin{align*}
    \lambda\coloneqq\exp\Bigl(-\varphi+\tfrac14\|\nabla\varphi\|^2\Bigr),
    \qquad
    \nu\coloneqq T_\#(\lambda^2\mu).
\end{align*}
Then, the potential $\varphi$ is $c$-concave, the pair $(\varphi,\psi)$ is optimal for the dual problem \eqref{eq:gh_dual} associated with $(\mu,\nu)$, the transport plan $\gamma\coloneqq(\mathrm{id},T)_\#(e^{-\varphi}\mu)$ is optimal for the Kantorovich problem \eqref{eq:gh_kantorovich}, and $(T,\lambda)$ coincides $\mu$-a.e.\ with the solution $(T_0,\lambda_0)$ of the Monge problem \eqref{eq:gh_monge} provided by Theorem~\ref{prop:monge-potential} for the pair $(\mu,\nu)$.
\end{lemma}

\begin{proof}[Proof of Lemma \ref{lem:verification}]
We first show that, for every $x\in\Omega$,
\begin{align}
\label{eq:verification_contact}
    \psi(T(x)) = \tfrac12\|x-T(x)\|^2-\varphi(x).
\end{align}
Fix $x\in\Omega$ and write $y=T(x)$. By definition of the $c$-transform, $\psi(y)=\inf_{x'\in\Omega}\{\frac12\|x'-y\|^2-\varphi(x')\}$, and the objective is $f(x')-\langle x',y\rangle+\frac12\|y\|^2$, a strictly convex function of $x'$ whose gradient vanishes at $x'=x$ because $\nabla f(x)=x-\nabla\varphi(x)=T(x)=y$. A critical point of a convex function is a global minimizer, so the infimum is attained at $x'=x$, which is \eqref{eq:verification_contact}. Since $\varphi$ is twice continuously differentiable, $\varphi$ is continuous and bounded on the compact set $\Omega$. The function $\psi=\varphi^c$ is an infimum of functions with a common Lipschitz constant on the bounded set $\Omega$. Thus, $\psi$ is Lipschitz and bounded on $\Omega$. Hence, the pair $(\varphi,\psi)$ belongs to $C_b(\Omega)^2$. Since $\psi=\varphi^c$, the inequality $\varphi(x)+\psi(y)\le\frac12\|x-y\|^2$ holds for all $x,y\in\Omega$. Therefore, the pair $(\varphi,\psi)$ is feasible for \eqref{eq:gh_dual}. Moreover, \eqref{eq:verification_contact} yields equality at $y=T(x)$. The potential $\varphi$ is also $c$-concave: $\psi=\varphi^c$ implies $\varphi^{cc}\ge\varphi$ pointwise, while \eqref{eq:verification_contact} yields $\varphi^{cc}(x)\le \frac12\|x-T(x)\|^2-\psi(T(x))=\varphi(x)$.

Next, we verify the marginal relations of $\gamma=(\mathrm{id},T)_\#(e^{-\varphi}\mu)$. The first marginal is $\gamma_0=e^{-\varphi}\mu$ by construction. For the second, \eqref{eq:verification_contact} and $\|x-T(x)\|=\|\nabla\varphi(x)\|$ yield
\begin{align*}
    e^{-\psi(T(x))}\lambda(x)^2
    =
    \exp\Bigl(\varphi(x)-\tfrac12\|\nabla\varphi(x)\|^2\Bigr)
    \exp\Bigl(-2\varphi(x)+\tfrac12\|\nabla\varphi(x)\|^2\Bigr)
    = e^{-\varphi(x)},
\end{align*}
so that, for every bounded measurable $h$,
\begin{align*}
    \int h\,e^{-\psi}\,d\nu
    = \int h(T(x))e^{-\psi(T(x))}\lambda(x)^2\,d\mu(x)
    = \int h(T(x))e^{-\varphi(x)}\,d\mu(x)
    = \int h\,d\gamma_1 .
\end{align*}
Hence, $\gamma_1=e^{-\psi}\nu$.

It remains to show that the objective of \eqref{eq:gh_kantorovich} at $\gamma$ is the same as the objective of \eqref{eq:gh_dual} at $(\varphi,\psi)$. With $F(r)=r\log r-r+1$ evaluated at $r=e^{-\varphi}$ and $r=e^{-\psi}$,
\begin{align}
    D_\KL(\gamma_0\mid\mu)
    &= \int\bigl(-\varphi e^{-\varphi}-e^{-\varphi}+1\bigr)\,d\mu
    = -\int\varphi\,d\gamma_0+\int(1-e^{-\varphi})\,d\mu,\label{eq:verif_kl0}
\\
    D_\KL(\gamma_1\mid\nu)
    &= -\int\psi\,d\gamma_1+\int(1-e^{-\psi})\,d\nu,\label{eq:verif_kl1}
\end{align}
while \eqref{eq:verification_contact} yields
\begin{align}
    \int\tfrac12\|x-y\|^2\,d\gamma(x,y)
    = \int\bigl(\varphi(x)+\psi(T(x))\bigr)\,d\gamma_0(x)
    = \int\varphi\,d\gamma_0+\int\psi\,d\gamma_1 .
    \label{eq:verif_cost}
\end{align}
Summing \eqref{eq:verif_kl0}, \eqref{eq:verif_kl1}, and \eqref{eq:verif_cost}, the potential terms cancel and
\begin{align*}
    \int\tfrac12\|x-y\|^2\,d\gamma+D_\KL(\gamma_0\mid\mu)+D_\KL(\gamma_1\mid\nu)
    = \int(1-e^{-\varphi})\,d\mu+\int(1-e^{-\psi})\,d\nu .
\end{align*}
The left-hand side is an upper bound for $\mathrm{UOT}(\mu,\nu)$ and, since $(\varphi,\psi)$ is feasible for \eqref{eq:gh_dual}, the right-hand side is a lower bound. Hence, both sides coincide with $\mathrm{UOT}(\mu,\nu)$, the transport plan $\gamma$ is optimal, and $(\varphi,\psi)$ is an optimal dual pair.

Finally, $\nu=T_\#(\lambda^2\mu)\ll dx$ because $T=\nabla f$ is injective with $\det\nabla T=\det\nabla^2 f>0$ by assumption, and $\lambda^2$ is bounded on $\Omega$. Since $T$ maps $\Omega$ into $\Omega$, $\nu$ is supported in $\Omega$. Theorem~\ref{prop:monge-potential} therefore applies to $(\mu,\nu)$, and the $\mu$-a.e.\ unique $c$-concave dual optimizer $\varphi_0$ coincides with $\varphi$ $\mu$-a.e. Hence, the pair $(T_0,\lambda_0)$ in \eqref{eq:unique} satisfies $(T_0,\lambda_0)=(T,\lambda)$ $\mu$-a.e., as claimed.
\end{proof}

\subsection{Proof of Theorem~\ref{thm:lb-ubot}}
\begin{proof}
To establish the minimax lower bound, we will appeal to both External results~\ref{thm:tsy-twopoint} and~\ref{thm:tsy-fano}. Both results involve constructing hypotheses $(\mu,\nu_k)\in\mathcal U_\alpha$ for which estimation is hard. Throughout, $O_\infty(r)$ denotes a function $f:\Omega\to\R$ with $\|f\|_{L^\infty(\Omega)}\le C r$, where $C$ is a constant independent of $k$. For any matrix-valued function $A:[0,1]^d \to \R^{d\times d}$, we denote $\| A \|_{\mathrm{op},\infty} \coloneqq \sup_{x\in [0,1]^d} \| A(x) \|_{\mathrm{op}}$.

We start off by fixing $\mu\coloneqq\mathrm{Unif}([0,1]^d)$ throughout the entire proof, so that $d\mu(x)=dx$ on $[0,1]^d$.

\subsubsection*{Setting up multiple hypotheses}

We construct a single family of hypotheses $(\mu,\nu_k)\in\mathcal U_\alpha$, $k=0,\dots,K$, whose transport maps and growth maps are both well separated in $L^2(\mu)$ while the distributions $\mathbb P_{\mu,\nu_k}^n$ stay close in KL divergence. 

Let $\xi\in C^\infty(\R)$ be a nonzero smooth bump function whose support is a compact subset of $(0,1)$, and let $x_*\in(0,1)$ satisfy $\xi(x_*)\ne 0$; define
\begin{align}
    g(x)=\prod_{i=1}^d\xi(x_i), \qquad x=(x_1,\dots,x_d).
\end{align}
Fix a small constant $a>0$ and a smooth cutoff $\chi\in C^\infty_c((0,1)^d)$ with $\chi\equiv 1$ on $[\delta_0,1-\delta_0]^d$ for some $\delta_0\in(0,1/4)$ such that $\supp(\xi)\subset[\delta_0,1-\delta_0]$. Define
\begin{align}
    \label{eq:base-potential}
    \varphi_0(x) \coloneqq a \chi(x) x_1.
\end{align}
In particular, we have $\nabla \varphi_0(x)=(a,0,\dots,0)$ for all $x \in [\delta_0,1-\delta_0]^d$. Since $\varphi_0\in C^\infty_c((0,1)^d)$, the transport-growth pair $(T_0,\lambda_0)$ via \eqref{eq:unique} satisfies $T_0=\mathrm{id}$ and $\lambda_0=1$ on a neighborhood of $\partial[0,1]^d$. Since $\|\nabla^2 \varphi_0\|_{\mathrm{op},\infty}\le Ca$, the function $\frac1{2}\|x\|^2-\varphi_0$ remains strictly convex and $T_0$ is a diffeomorphism for a sufficiently small $a>0$. 

Let $J=\lceil\theta  n^{1/(2\alpha-2+d)}\rceil$ for a constant $\theta>0$ to be determined later, and let $\{x^{(j)}\}_{j\in\mJ}$ be a grid with $x^{(j)}_i=(j_i-1)/J$, indexed by the set $\mJ$ of all $j\in\{1,\dots,J\}^d$ such that $\prod_{i=1}^d[x^{(j)}_i,x^{(j)}_i+1/J]\subset[\delta_0,1-\delta_0]^d$; the cardinality $N_J\coloneqq|\mJ|\asymp J^d$. Define
\begin{align}
    g_j(x) \coloneqq \frac{\varepsilon_{\mathrm b}}{J^{\alpha+1}} g(J(x-x^{(j)})),
\end{align}
so that $\{\supp (g_j)\}_{j\in\mJ}$ are pairwise disjoint and contained in $[\delta_0,1-\delta_0]^d$. Since $\partial^\flat g_j(\cdot)=\varepsilon_{\mathrm b}  J^{|\flat|-\alpha-1}\partial^\flat g(J(\cdot-x^{(j)}))$ for any multi-index $\flat$, and $\alpha>1$, one may fix $\varepsilon_{\mathrm b}$ small enough and $J\ge J_0$ large enough that uniformly in $j,J$:
\begin{align}
    \label{eq:bump-bounds}
    \|g_j\|_\infty\le J^{-\alpha-1}, \quad \|\nabla g_j\|_\infty\le J^{-\alpha}, \quad \|\nabla^2 g_j\|_{\mathrm{op},\infty}\le 1/2, \quad \|g_j\|_{C^{\alpha+1}}\le C_g,
\end{align}
for some constant $C_g>0$. Applying External result~\ref{lem:vg} with $N_J\ge 8$, we obtain binary strings $\omega^{(0)},\dots,\omega^{(K)}\in\{0,1\}^{\mJ}$ with $\omega^{(0)}=\mathbf{0}$, $K\ge 2^{N_J/8}$, and $\|\omega^{(k)}-\omega^{(k')}\|^2\ge N_J/8$ for $k\ne k'$. Define
\begin{equation}\label{eq:varphidef}
    \begin{split}
    \varphi_k(x) &\coloneqq \varphi_0(x) + \sum_{j\in\mJ}\omega^{(k)}_j g_j(x),\qquad k=0,\dots,K,\\
    T_k(x) &\coloneqq x - \nabla \varphi_k(x), \quad \lambda_k(x)\coloneqq\exp \Big( -\varphi_k(x)+\frac1{4}\|\nabla \varphi_k(x)\|^2\Big),\\
    \nu_k &\coloneqq (T_k)_\#(\lambda_k^2\mu).
    \end{split}
\end{equation}
Write $f_k(x)\coloneqq\frac1{2}\|x\|^2-\varphi_k(x)$, so that $T_k=\nabla f_k$. 

Notice that the masses of $\nu_k$ are not the same across $k=0,\ldots,K$, and we shall denote:
\[ H_k\coloneqq\nu_k(\Omega). \] 
Since $\nu_k=(T_k)_\#(\lambda_k^2\mu)$, we have $H_k=\int\lambda_k^2 \,d\mu$. 

\subsubsection*{Preliminary analysis of the construction}

Since $\nabla^2 \varphi_k = \nabla^2 \varphi_0 + \nabla^2(\varphi_k-\varphi_0)$ and $\|\nabla^2 \varphi_0\|_{\mathrm{op},\infty}\le Ca$, $\|\nabla^2(\varphi_k-\varphi_0)\|_{\mathrm{op},\infty}\le 1/2$ by \eqref{eq:bump-bounds}, we have $\nabla^2 f_k\succeq(1/2-Ca)I\succ 0$ for $a$ small enough, so $f_k$ is strictly convex on $\R^d$.

Moreover, the bumps $g_j$ are translates of the function $x\mapsto\frac{\varepsilon_{\mathrm b}}{J^{\alpha+1}}g(Jx)$ with disjoint supports. Every derivative of $g_j$ vanishes on the boundary of $\supp(g_j)$. Hence, $\varphi_k\in C^{\alpha+1}(\Omega)$ with
\begin{align}
\label{eq:phi_k_holder}
    \|\varphi_k\|_{C^{\alpha+1}(\Omega)}\le\|\varphi_0\|_{C^{\alpha+1}}+2C_g,
\end{align}
so $T_k\in C^\alpha(\Omega;\Omega)$ and $\lambda_k\in C^\alpha(\Omega)$ with uniform bounds. Because $\varphi_0\in C^\infty_c((0,1)^d)$ and $\varphi_k-\varphi_0$ is supported strictly inside $[\delta_0,1-\delta_0]^d$, we have $\nabla \varphi_k\equiv 0$ on a neighborhood of $\partial[0,1]^d$, hence $T_k\equiv\mathrm{id}$ there. Since $T_k=\nabla f_k$ is a diffeomorphism of $\R^d$, it maps $\partial[0,1]^d$ onto itself and $(0,1)^d$ onto $(0,1)^d$, so $T_k([0,1]^d)=[0,1]^d$. The assumptions of Lemma~\ref{lem:verification} are therefore satisfied, and the lemma identifies $(T_k,\lambda_k)$ as the $\mu$-a.e.\ unique transport-growth pair associated with the pair $(\mu,\nu_k)$.

Because $\mu=\mathrm{Unif}([0,1]^d)$ and each perturbation is supported strictly inside $[0,1]^d$, the densities of $\nu_k$ remain uniformly positive and uniformly bounded and are identically $1$ near the boundary. By the change-of-variables formula, the density of $\nu_k$ is
\begin{align*}
    h_k\coloneqq\frac{d\nu_k}{d\mu}=\Bigl(\frac{\lambda_k^2}{\det\nabla T_k}\Bigr)\circ T_k^{-1}.
\end{align*}
Thus, \eqref{eq:phi_k_holder} yields a uniform bound on $\|h_k\|_{C^{\alpha-1}(\Omega)}$. Since $h_k\equiv1$ near the boundary of $[0,1]^d$, we have $\|\mathcal Eh_k\|_{C^{\alpha-1}(\mathbb T_2^d)}\le C\|h_k\|_{C^{\alpha-1}(\Omega)}$ for the even reflection $\mathcal E$ of Definition~\ref{def:neumann_holder_class}. Moreover, $\nabla^2 f_k=I-\nabla^2\varphi_k$ satisfies $(1/2-Ca)I\preceq\nabla^2 f_k\preceq(3/2+Ca)I$. Combined with the uniform $C^\alpha$ bounds on $T_k$ and $\lambda_k$ from \eqref{eq:phi_k_holder}, these properties give constants $M_0,B_0,R_0$ and $\kappa_0,b_0$ such that each pair $(\mu,\nu_k)$ belongs to $\mathcal U_\alpha$ for every $M\ge M_0$, $B\ge B_0/2$, $R\ge R_0/2$, $\kappa\le\kappa_0$, $\beta_{\min}\le b_0$, and $\beta_{\max}\ge 1/b_0$.

Throughout these estimates we compare the hypothesis $k$ with the base hypothesis $k=0$. Denote $\ell_0 \coloneqq \log \lambda_0 = -\varphi_0+\frac1{4}\|\nabla \varphi_0\|^2$ and $\ell_k \coloneqq \log \lambda_k = -\varphi_k+\frac1{4}\|\nabla \varphi_k\|^2$. Since $\nabla \varphi_0(x)=(a,0,\dots,0)$ for all $x\in[\delta_0,1-\delta_0]^d$, we compute
\begin{align}
    \ell_k - \ell_0 &= \bigl(-\varphi_k+\frac1{4}\|\nabla \varphi_k\|^2\bigr) - \bigl(-\varphi_0+\frac1{4}\|\nabla \varphi_0\|^2\bigr) \notag\\
    &= -(\varphi_k-\varphi_0) + \frac1{2}\nabla \varphi_0\cdot\nabla(\varphi_k-\varphi_0) + \frac1{4}\|\nabla(\varphi_k-\varphi_0)\|^2 \\
    &= \frac{a}{2} \partial_1(\varphi_k-\varphi_0) - (\varphi_k-\varphi_0) + \frac1{4}\|\nabla(\varphi_k-\varphi_0)\|^2. \label{eq:u-expand}
\end{align}
Due to~\eqref{eq:bump-bounds} and the definitions of $\varphi_k$ and $\varphi_0$ in~\eqref{eq:varphidef}, the three terms on the right-hand side of \eqref{eq:u-expand} have pointwise values of $O(a J^{-\alpha})$, $O(J^{-\alpha-1})$, and $O(J^{-2\alpha})$; here, ${ (a/2)}\partial_1(\varphi_k-\varphi_0)$ is the largest of the three. By $\lambda_k-\lambda_0=\lambda_0(e^{\ell_k-\ell_0}-1)$ and $|e^u-1-u|\le u^2 e^{|u|}$,
\begin{align}
    \lambda_k - \lambda_0 = \lambda_0\Big(\frac{a}{2} \partial_1(\varphi_k-\varphi_0) + \mathcal R_k\Big), \label{eq:lambda-expand}
\end{align}
where $|\mathcal R_k|\le |\varphi_k-\varphi_0|+\tfrac14\|\nabla(\varphi_k-\varphi_0)\|^2+e^{|\ell_k-\ell_0|}(\ell_k-\ell_0)^2\le CJ^{-\alpha-1}$ for $J\ge J_0$.

Set $\lambda_{\min}\coloneqq\inf_{x \in [0,1]^d}\lambda_0(x) >0$ and $\lambda_{\max}\coloneqq\sup_{x \in [0,1]^d}\lambda_0(x)<\infty$. We next record the total masses, which are needed for the comparison of the distributions. 

By \eqref{eq:u-expand} and the bump bounds \eqref{eq:bump-bounds}, $\|\log\lambda_k-\log\lambda_0\|_\infty=O(aJ^{-\alpha})\le\log 2$ for all sufficiently large $J$, so the growth map $\lambda_k$ satisfies $\lambda_{\min}/2\le\lambda_k\le2\lambda_{\max}$ uniformly in $k$. Combined with $\mu(\Omega)=1$ and $H_k=\int\lambda_k^2 \,d\mu$, this yields
\begin{align}
    \lambda_{\min}^2/4 \le  H_k \le 4\lambda_{\max}^2\qquad\text{for all }k=0,\dots,K.
\end{align}

We shall also need the stability of the inverse maps. To bound $T_k^{-1}(y)-T_0^{-1}(y)$ for a fixed $y\in\Omega$, we set $z_k\coloneqq T_k^{-1}(y)$, $z_0\coloneqq T_0^{-1}(y)$. From $y=T_k(z_k)=z_k-\nabla \varphi_k(z_k)$ and $y=T_0(z_0)=z_0-\nabla \varphi_0(z_0)$, we have
\begin{align}
    (z_k-z_0)  -  \bigl(\nabla \varphi_0(z_k)-\nabla \varphi_0(z_0)\bigr)
     =  \nabla \varphi_k(z_k)-\nabla \varphi_0(z_k).
\end{align}
By the fundamental theorem of calculus, $\nabla \varphi_0(z_k)-\nabla \varphi_0(z_0)=A_y(z_k-z_0)$ with $A_y\coloneqq\int_0^1 \nabla^2 \varphi_0(z_0+t(z_k-z_0))\,dt$ and $\|A_y\|_{\mathrm{op}}\le\|\nabla^2 \varphi_0\|_{\mathrm{op},\infty}\le Ca$. Hence
\begin{align}
    (I-A_y)(z_k-z_0) = \nabla \varphi_k(z_k)-\nabla \varphi_0(z_k).
\end{align}
For a sufficiently small $a>0$, $I-A_y$ is invertible with $\|(I-A_y)^{-1}\|_{\mathrm{op}}\le(1-Ca)^{-1}$. Consequently, using $\|\nabla(\varphi_k-\varphi_0)\|_\infty\le J^{-\alpha}$ from \eqref{eq:bump-bounds}, we obtain:
\begin{align}\label{eq:T-inv-diff}
    \|T_k^{-1}(y)-T_0^{-1}(y)\|  \le  \frac{\|\nabla(\varphi_k-\varphi_0)\|_\infty}{1-Ca}  =  O(J^{-\alpha}).
\end{align}

\subsubsection*{Reduction to hypotheses of equal total mass}

Another preliminary step is to rescale each $\nu_k$ to the total mass $\nu_0(\Omega) = H_0$. For $k=0,\dots,K$, we set $c_k\coloneqq\frac12\log(H_k/H_0)\in\R$ and define
\begin{align}
\label{eq:equal_mass_rescaling}
    \check\varphi_k\coloneqq\varphi_k+c_k,
    \quad
    \check\lambda_k\coloneqq e^{-c_k}\lambda_k,
    \quad
    \check\nu_k\coloneqq(T_k)_\#(\check\lambda_k^2\mu)=\frac{H_0}{H_k}\,\nu_k, 
    \quad
    \bar\nu_k \coloneqq \frac{\nu_k}{H_k}.
\end{align}
Thus, $\check\nu_k$ has the total mass $H_0$, and $\bar\nu_k$ is its normalized measure. Hence, $(M_\mu,M_\nu)=(1,H_0)$ under every hypothesis $(\mu,\check\nu_k)$. Under these hypotheses, each measurable function of $(X_1,\dots,X_n,\allowbreak Y_1,\dots,Y_n,\allowbreak M_\mu,M_\nu)$ is a measurable function of $(X_1,\dots,X_n,\allowbreak Y_1,\dots,Y_n)$, whose distribution is $\mathbb P_{\mu,\check\nu_k}^n$.

Since $\check\varphi_k-\varphi_k = c_k$ is a constant, the function $\check\varphi_k$ satisfies the assumptions of Lemma~\ref{lem:verification} with the same map $T_k$, and
\begin{align*}
    \check\lambda_k=\exp\Bigl(-\check\varphi_k+\frac14\|\nabla\check\varphi_k\|^2\Bigr).
\end{align*}
Thus, $(T_k,\check\lambda_k)$ is the $\mu$-a.e.\ unique transport-growth pair associated with the pair $(\mu,\check\nu_k)$.

We next bound $c_k$, which determines the difference between $\check\lambda_k$ and $\lambda_k$. On $[\delta_0,1-\delta_0]^d$, we have $\lambda_0^2=\exp(-2ax_1+a^2/2)$, so $\partial_1\lambda_0^2=-2a\lambda_0^2$. Since $\varphi_k-\varphi_0$ is supported in this set, integration by parts yields
\begin{align}
    \label{eq:mass-by-parts}
    a\int\lambda_0^2\,\partial_1(\varphi_k-\varphi_0)\,d\mu
    =
    2a^2\int\lambda_0^2(\varphi_k-\varphi_0)\,d\mu.
\end{align}
By \eqref{eq:bump-bounds}, the right-hand side of \eqref{eq:mass-by-parts} is $O(J^{-\alpha-1})$, whereas the integrand on the left-hand side is of order $J^{-\alpha}$. Moreover, $\|\ell_k-\ell_0\|_\infty\le\log 2$ for all sufficiently large $J$. Since $|e^{2u}-1-2u|\le Cu^2$ for $|u|\le\log 2$, \eqref{eq:u-expand} and \eqref{eq:mass-by-parts} yield
\begin{align}
    |H_k-H_0|
    &=
    \Bigl|\int\lambda_0^2\bigl(e^{2(\ell_k-\ell_0)}-1\bigr)\,d\mu\Bigr|
    \\
    &\le
    2|1-a^2|\Bigl|\int\lambda_0^2(\varphi_k-\varphi_0)\,d\mu\Bigr|+CJ^{-2\alpha}
    \le
    CJ^{-\alpha-1},
\end{align}
where $C>0$ is a constant independent of $k$ and $J$. The last inequality uses \eqref{eq:bump-bounds} and $\alpha>1$. Since $H_k\ge\lambda_{\min}^2/4$, it follows that $|c_k|\le CJ^{-\alpha-1}$.

Consequently, $e^{|c_k|}\le2$ for all sufficiently large $J$. Hence,
\begin{align*}
    \|\check\lambda_k\|_{C^\alpha([0,1]^d)}\le2\|\lambda_k\|_{C^\alpha([0,1]^d)},
    \qquad
    \frac{\lambda_k}{2}\le\check\lambda_k\le2\lambda_k.
\end{align*}
Since $\nabla^2\check\varphi_k=\nabla^2\varphi_k$ and $\check\nu_k$ and $\nu_k$ have the same normalized measure $\bar\nu_k$, the pair $(\mu,\check\nu_k)$ satisfies condition (i) of $\mathcal U_\alpha$ with the same constants as $(\mu,\nu_k)$. For condition (ii), the transport-growth pair of $(\mu,\check\nu_k)$ is $(T_k,\check\lambda_k)$. By the choice of $B_0$ and $R_0$, we have $\|T_k\|_{C^\alpha([0,1]^d)}+\|\lambda_k\|_{C^\alpha([0,1]^d)}\le B_0/2$ and $2R_0^{-1}\le\lambda_k\le R_0/2$. Combining these bounds with the previous display yields
\begin{align*}
    \|T_k\|_{C^\alpha([0,1]^d)}+\|\check\lambda_k\|_{C^\alpha([0,1]^d)}
    &\le 2\bigl(\|T_k\|_{C^\alpha([0,1]^d)}+\|\lambda_k\|_{C^\alpha([0,1]^d)}\bigr)
    \le B_0,
    \\
    R_0^{-1}
    &\le\frac{\lambda_k}{2}\le\check\lambda_k\le2\lambda_k\le R_0.
\end{align*}
Hence, each pair $(\mu,\check\nu_k)$ belongs to $\mathcal U_\alpha$ for every $M\ge M_0$, $B\ge B_0$, $R\ge R_0$, $\kappa\le\kappa_0$, $\beta_{\min}\le b_0$, and $\beta_{\max}\ge 1/b_0$.

\subsubsection*{Separation bounds for the transport and growth maps}

To utilize the multiple hypotheses in External result~\ref{thm:tsy-fano}, we calculate the $L^2(\mu)$-separation between the $\check\lambda_k$'s. For the separation of the growth maps, we write
\begin{align}
    \check\lambda_k-\check\lambda_{k'}
    =
    (\lambda_k-\lambda_{k'})+(e^{-c_k}-1)\lambda_k-(e^{-c_{k'}}-1)\lambda_{k'},
\end{align}
where the last two terms are $O_\infty(J^{-\alpha-1})$ because $|c_k|\le CJ^{-\alpha-1}$ and $\lambda_k\le2\lambda_{\max}$.

Combined with \eqref{eq:lambda-expand}, applied to $k$ and to $k'$ and subtracted, this yields, with $\check{\mathcal R}_k\coloneqq\mathcal R_k+(e^{-c_k}-1)\lambda_k/\lambda_0$,
\begin{align}
    \label{eq:check-lambda-diff}
    \check\lambda_k-\check\lambda_{k'}
    =
    \lambda_0\Bigl(\frac{a}{2}\partial_1(\varphi_k-\varphi_{k'})+\check{\mathcal R}_k-\check{\mathcal R}_{k'}\Bigr),
\end{align}
where $\|\check{\mathcal R}_k\|_\infty\le CJ^{-\alpha-1}$ for $J\ge J_0$, because $|e^{-c_k}-1|\le 2|c_k|\le CJ^{-\alpha-1}$ and $\lambda_k\le 2\lambda_0$.
We use the cross term $(a/2)\partial_1(\varphi_k-\varphi_{k'})$ in~\eqref{eq:check-lambda-diff} and the disjointness of $\{\supp(\nabla g_j)\}_{j\in\mJ}$. We claim
\begin{align}
    \label{eq:cg-positive}
    c_g\;\coloneqq\;\int_{[0,1]^d}|\partial_1 g(y)|^2\,dy \;>\; 0.
\end{align}
Recall that $g(x)=\prod_i\xi(x_i)$. If $\partial_1 g\equiv 0$ on $[0,1]^d$, then $g$ is constant in $x_1$. Since $\supp(g)\subset[\delta_0,1-\delta_0]^d$, it follows that $g\equiv0$, contradicting $g(x_*,\dots,x_*)=\xi(x_*)^d\ne 0$. Hence, $\partial_1 g$ is nonzero on a set of positive Lebesgue measure, which proves~\eqref{eq:cg-positive}. By definition, $\partial_1 g_j(x)=\varepsilon_{\mathrm b} J^{-\alpha}(\partial_1 g)(J(x-x^{(j)}))$. Thus, the substitution $y=J(x-x^{(j)})$ yields
\begin{align}
    \int|\partial_1 g_j|^2\,d\mu = \varepsilon_{\mathrm b}^2\,c_g\,J^{-2\alpha-d}.
\end{align}
Since $N_J\asymp J^d$, there is a constant $c_l>0$ with $N_J\ge c_lJ^d$ for all $J\ge J_0$. Set $c_\lambda\coloneqq(1/128) a^2\lambda_{\min}^2\varepsilon_{\mathrm b}^2 c_g c_l>0$. By~\eqref{eq:check-lambda-diff}, the disjoint supports of $\{\partial_1 g_j\}_{j\in\mJ}$, and $\|\omega^{(k)}-\omega^{(k')}\|^2\ge N_J/8\ge c_l J^d/8$, we obtain
\begin{align}
    \label{eq:lambda-sep}
    \|\check\lambda_k-\check\lambda_{k'}\|_{L^2(\mu)}^2
    &\ge \frac{a^2\lambda_{\min}^2}{8}\sum_{j\in\mJ}(\omega_j^{(k)}-\omega_j^{(k')})^2\!\int|\partial_1 g_j|^2\,d\mu - \lambda_{\max}^2\!\int(\check{\mathcal R}_k-\check{\mathcal R}_{k'})^2\,d\mu \notag\\
    &\ge \frac{a^2\lambda_{\min}^2}{8}\cdot\frac{c_l J^d}{8}\cdot\varepsilon_{\mathrm b}^2 c_g J^{-2\alpha-d} - C J^{-2\alpha-2} \notag\\
    &\ge c_\lambda J^{-2\alpha},
\end{align}
for $J$ large enough that the remainder $C J^{-2\alpha-2}$ is absorbed; the constant $c_\lambda$ is independent of $J,k,k'$.

The hypothesis $(\mu,\check\nu_k)$ has the transport map $T_k$, because Lemma~\ref{lem:verification} applies to $\check\varphi_k$ with the map $T_k$. Since $T_k-T_{k'}=-\nabla(\varphi_k-\varphi_{k'})=-\sum_{j\in\mJ}(\omega_j^{(k)}-\omega_j^{(k')})\nabla g_j$ and the supports of $\{\nabla g_j\}_{j\in\mJ}$ are pairwise disjoint, we have
\begin{align}
    \|T_k-T_{k'}\|_{L^2(\mu)}^2
    =
    \sum_{j\in\mJ}(\omega_j^{(k)}-\omega_j^{(k')})^2\int\|\nabla g_j\|^2\,d\mu .
\end{align}
By definition, we have $\nabla g_j(x)=\varepsilon_{\mathrm b}J^{-\alpha}(\nabla g)(J(x-x^{(j)}))$. Thus, the substitution $y=J(x-x^{(j)})$ yields
\begin{align}
    \int\|\nabla g_j\|^2\,d\mu = \varepsilon_{\mathrm b}^2\, c_{\nabla}\,J^{-2\alpha-d},
    \qquad
    c_{\nabla}\coloneqq\int_{[0,1]^d}\|\nabla g(y)\|^2\,dy \ge c_g>0,
\end{align}
where the last inequality follows from $\|\nabla g\|^2\ge|\partial_1 g|^2$ and \eqref{eq:cg-positive}. Combining this identity with $\|\omega^{(k)}-\omega^{(k')}\|^2\ge N_J/8\ge c_lJ^d/8$, we obtain the separation
\begin{align}
\label{eq:T-sep}
    \|T_k-T_{k'}\|_{L^2(\mu)}^2
    \ge
    \frac{c_l\varepsilon_{\mathrm b}^2c_\nabla}{8}\,J^{-2\alpha}
    \qquad\text{for all }0\le k\ne k'\le K.
\end{align}

\subsubsection*{KL bounds for the hypotheses}

Next, we bound the KL divergence $D_\KL(\bar\nu_k\mid \bar\nu_0)$ between the normalized measures. We do so by the $\chi^2$-divergence in two steps. First, the Cauchy-Schwarz inequality applied to $(H_k-H_0)=\int(h_k-h_0)\,d\mu=\int\frac{h_k-h_0}{\sqrt{h_0}}\cdot\sqrt{h_0}\,d\mu$ yields
\begin{align}
    \label{eq:CS-mass}
    (H_k-H_0)^2 \le \Bigl(\int\!\frac{(h_k-h_0)^2}{h_0}\,d\mu\Bigr)\Bigl(\int h_0\,d\mu\Bigr) = H_0\!\int\!\frac{(h_k-h_0)^2}{h_0}\,d\mu.
\end{align}
Writing $\bar h_k\coloneqq h_k/H_k$ for the density of $\bar\nu_k$ with respect to $\mu$, and splitting $\bar h_k-\bar h_0=\frac{h_k-h_0}{H_k}-\frac{h_0(H_k-H_0)}{H_0 H_k}$, we have
\begin{align}
    D_\KL(\bar\nu_k\mid \bar\nu_0)
    \le \int\!\frac{(\bar h_k-\bar h_0)^2}{\bar h_0}\,d\mu
    &\le \frac{2H_0}{H_k^2}\!\int\!\frac{(h_k-h_0)^2}{h_0}\,d\mu + \frac{2(H_k-H_0)^2}{H_k^2 H_0}\!\int h_0\,d\mu \notag\\
    &= \frac{2H_0}{H_k^2}\!\int\!\frac{(h_k-h_0)^2}{h_0}\,d\mu + \frac{2(H_k-H_0)^2}{H_k^2} \notag\\
    &\stackrel{\eqref{eq:CS-mass}}{\le}\frac{4H_0}{H_k^2}\!\int\!\frac{(h_k-h_0)^2}{h_0}\,d\mu \;\le\; C\!\int\!\frac{(h_k-h_0)^2}{h_0}\,d\mu,
    \label{eq:KL-chi-square}
\end{align}
where the last inequality uses $\lambda_{\min}^2/4\le H_k\le4\lambda_{\max}^2$ and $H_0=\int\lambda_0^2\,d\mu\le\lambda_{\max}^2$ to bound $4H_0/H_k^2\le 64\lambda_{\max}^2/\lambda_{\min}^4$. To bound the integral, we consider $r_k(y)\coloneqq h_k(y)/h_0(y)$. If $\|\log r_k\|_\infty\le 1/2$ for a sufficiently large $J$, we have:
\begin{align}
    \label{eq:hk-expand}
    h_k(y)-h_0(y)  =  h_0(y) \log r_k(y) + Q_k(y),
    \qquad |Q_k(y)|\le h_0(y) (\log r_k(y))^2.
\end{align}
Consequently,
\begin{align}
    \label{eq:I-reduction}
    \int \frac{(h_k-h_0)^2}{h_0}\,d\mu  \le  \frac94\|h_0\|_\infty \int(\log r_k)^2\,d\mu.
\end{align}
It thus suffices to bound $\int(\log r_k)^2\,d\mu$. By the change-of-variables formula for $\nu_k=(T_k)_\#(\lambda_k^2\mu)$, $\lambda_k=\exp(-\varphi_k+\frac1{4}\|\nabla \varphi_k\|^2)$, and $\nabla T_k=I-\nabla^2\varphi_k$,
\begin{align}
    \log h_k(y)
    & =  \Bigl[ 2\log\lambda_k(x)-\log\det \nabla T_k(x) \Bigr]\Big\vert_{x=T_k^{-1}(y)} \\
    & =  \Bigl[ \underbrace{-2\varphi_k(x)+\frac1{2}\|\nabla \varphi_k(x)\|^2-\log\det(I-\nabla^2\varphi_k(x))}_{L_k(x)} \Bigr]\Big\vert_{x=T_k^{-1}(y)}.
\end{align}
Thus $\log r_k(y)=L_k(T_k^{-1}(y))-L_0(T_0^{-1}(y))$, which we split as
\begin{align}
    \label{eq:logr-split}
    \log r_k(y)  =  (L_k-L_0)(T_k^{-1}(y)) + \bigl[L_0(T_k^{-1}(y))-L_0(T_0^{-1}(y))\bigr].
\end{align}
We first bound $L_k(x) - L_0(x)$. By definition,
\begin{align}
    L_k(x)-L_0(x)
    &= -2(\varphi_k-\varphi_0)(x) + \frac1{2}\bigl(\|\nabla \varphi_k(x)\|^2-\|\nabla \varphi_0(x)\|^2\bigr) \notag\\
    &\quad + \bigl[-\log\det(I-\nabla^2 \varphi_k(x))+\log\det(I-\nabla^2 \varphi_0(x))\bigr].
\label{eq:deltaL_split}
\end{align}
We bound the three terms on the right-hand side of \eqref{eq:deltaL_split} in turn, using \eqref{eq:base-potential}, \eqref{eq:bump-bounds}, and \eqref{eq:varphidef}.
\begin{enumerate}
\item[(i)] \emph{Zeroth-order term.}  $\varphi_k-\varphi_0=O_\infty(J^{-\alpha-1})$.

\item[(ii)] \emph{Gradient-squared term.} Since $\nabla \varphi_0=(a,0,\dots,0)$ on $[\delta_0,1-\delta_0]^d$,
\begin{align}
    \frac1{2}\|\nabla \varphi_k\|^2-\frac1{2}\|\nabla \varphi_0\|^2
    = a \partial_1(\varphi_k-\varphi_0)+\frac1{2}\|\nabla(\varphi_k-\varphi_0)\|^2
    = O_\infty(a J^{-\alpha}) + O_\infty(J^{-2\alpha}).
\end{align}

\item[(iii)] \emph{Log-determinant term.} We claim that pointwise on $\Omega$,
\begin{align}
    \label{eq:logdet-collapse}
    -\log\det(I-\nabla^2 \varphi_k(x)) + \log\det(I-\nabla^2 \varphi_0(x))  =  -\log\det(I-\nabla^2(\varphi_k-\varphi_0)(x)).
\end{align}
To see this, let $\mathcal S\coloneqq\bigcup_{j\in\mJ}\supp(g_j)$ and split $\Omega$ into $\mathcal S$ and $\Omega\setminus\mathcal S$. For $x\in\mathcal S$, that is $x\in\supp(g_j)$ for some $j\in\mJ$, we have $\chi(x)= 1$, so $\varphi_0(x)=ax_1$ is affine and $\nabla^2 \varphi_0(x)= 0$; hence $\nabla^2 \varphi_k(x)=\nabla^2(\varphi_k-\varphi_0)(x)$ and $\log\det(I-\nabla^2 \varphi_0(x))=0$, which yields \eqref{eq:logdet-collapse}. For $x\in\Omega\setminus\mathcal S$, $\varphi_k(x)-\varphi_0(x)=\sum_{j\in\mJ}\omega_j^{(k)}g_j(x)=0$, so $\nabla^2 \varphi_k(x)=\nabla^2 \varphi_0(x)$ and both sides of \eqref{eq:logdet-collapse} vanish. The claim thus follows.

Denote $A \coloneqq \nabla^2(\varphi_k-\varphi_0)$. Since $\|\nabla^2(\varphi_k-\varphi_0)\|_{\mathrm{op},\infty}\le 1/2$ by \eqref{eq:bump-bounds}, we can apply the Taylor expansion to obtain $\log\det(I-A)=-\mathrm{tr}(A)+O(\|A\|_{\mathrm{op}}^2)$. Since $\|\nabla^2(\varphi_k-\varphi_0)\|_{\mathrm{op},\infty}=O(J^{1-\alpha})$, we have
\begin{align}
    -\log\det(I-\nabla^2(\varphi_k-\varphi_0))  =  \Delta(\varphi_k-\varphi_0) + O_\infty(J^{2-2\alpha}).
\end{align}
\end{enumerate}
Combining (i)--(iii) with $\varphi_k-\varphi_0=O_\infty(J^{-\alpha-1})$, we obtain, for $J\ge 1/a$,
\begin{align}
    \label{eq:deltaL-expand}
    L_k(x)-L_0(x)  =  \Delta \varphi_k(x)-\Delta \varphi_0(x)+O_\infty(a J^{-\alpha})+O_\infty(J^{2-2\alpha}).
\end{align}
For the second term in \eqref{eq:logr-split}, the function $L_0$ is Lipschitz on $[0,1]^d$ with constant at most $Ca$, since $\varphi_0(x)=a\chi(x)x_1$ and $a$ is small. Hence, for every $y\in\Omega$,
\begin{align}\label{eq:L0-expand}
    \bigl|L_0(T_k^{-1}(y))-L_0(T_0^{-1}(y))\bigr|  \le  Ca\,\|T_k^{-1}(y)-T_0^{-1}(y)\|.
\end{align}
Substituting \eqref{eq:deltaL-expand}, \eqref{eq:L0-expand} and \eqref{eq:T-inv-diff} into~\eqref{eq:logr-split} yields
\begin{align}
\label{eq:logr_expansion}
    \log r_k(y)  =  \Delta \varphi_k(T_k^{-1}(y))-\Delta \varphi_0(T_k^{-1}(y))+O_\infty(a J^{-\alpha})+O_\infty(J^{2-2\alpha}).
\end{align}
In particular $\|\log r_k\|_\infty=O(J^{1-\alpha})\le 1/2$ for a sufficiently large $J$, justifying the condition used in~\eqref{eq:hk-expand}.

Combining \eqref{eq:KL-chi-square} and \eqref{eq:I-reduction} with \eqref{eq:logr_expansion} yields
\begin{align}\label{eq:KL-bound-master}
D_\KL(\bar\nu_k\mid \bar\nu_0)\le C \int(\Delta \varphi_k\circ T_k^{-1}-\Delta \varphi_0\circ T_k^{-1})^2\,d\mu + O(a^2 J^{-2\alpha}) + O(J^{4-4\alpha}).
\end{align}
Since $N_J\asymp J^d$, there is a constant $c_u>0$ with $N_J\le c_uJ^d$ for all $J\ge J_0$. We now bound the leading term $\int(\Delta(\varphi_k-\varphi_0)\circ T_k^{-1})^2\,d\mu$. Since $\Delta(\varphi_k-\varphi_0)$ is supported in $\bigcup_{j\in\mJ}\supp(g_j)\subset[\delta_0,1-\delta_0]^d$, where $\nabla^2\varphi_0\equiv0$, the change of variables $x=T_k^{-1}(y)$ has Jacobian $\det \nabla T_k(x)=\det(I-\nabla^2(\varphi_k-\varphi_0))=1+O(J^{1-\alpha})$, so $\|\Delta(\varphi_k-\varphi_0)\circ T_k^{-1}\|_{L^2(\mu)}^2\le C\|\Delta(\varphi_k-\varphi_0)\|_{L^2(\mu)}^2$. Using disjoint supports of $\{g_j\}_{j\in\mJ}$, $(\omega_j^{(k)})^2=\omega_j^{(k)}\in\{0,1\}$, and the rescaling $\Delta g_j(x)=\varepsilon_{\mathrm b} J^{1-\alpha}(\Delta g)(J(x-x^{(j)}))$,
\begin{align}
    \int(\Delta(\varphi_k-\varphi_0))^2\,d\mu
     = \sum_{j\in\mJ}\omega_j^{(k)}\cdot\frac{\varepsilon_{\mathrm b}^2}{J^{2\alpha-2+d}} \int(\Delta g)^2\,dy
     \le  C \varepsilon_{\mathrm b}^2 J^{2-2\alpha},
\end{align}
using $\sum_{j\in\mJ}\omega_j^{(k)}\le N_J\le c_u J^d$. The two remainders in~\eqref{eq:KL-bound-master} contribute $O(a^2 J^{-2\alpha})$ and $O(J^{4-4\alpha})$. Both remainders are $O(J^{2-2\alpha})$ for $\alpha\ge 1$. Hence,
\begin{align}
    D_\KL(\bar\nu_k\mid \bar\nu_0)  \le  C J^{2-2\alpha}.
\end{align}
By \eqref{eq:iid_kl} and $\mu$ fixed across hypotheses, for a constant $C_{\KL}>0$ independent of $n$ and $J$,
\begin{align}
    D_\KL(\mathbb P_{\mu,\nu_k}^n\mid \mathbb P_{\mu,\nu_0}^n)
    = n D_\KL(\bar\nu_k\mid \bar\nu_0)
    \le C_{\KL} n J^{2-2\alpha}.
\end{align}
From External result~\ref{lem:vg} and $N_J\ge c_l J^d$, $\log K\ge (N_J/8)\log 2\ge (c_l J^d\log 2)/8$, so
\begin{align}
    (\log K)/9 \ge (c_l J^d\log 2)/72.
\end{align}
Using $J\ge\theta n^{1/(2\alpha-2+d)}$ from the definition $J=\lceil\theta n^{1/(2\alpha-2+d)}\rceil$,
\begin{align}
    \theta^{2\alpha-2+d} \ge 72 C_{\KL}/(c_l\log 2)
    \quad\Longrightarrow\quad
    C_{\KL} n J^{2-2\alpha} \le (\log K)/9.
\end{align}
Therefore, by choosing
\begin{align}
    \theta  >  \left(72 C_{\KL}/(c_l\log 2)\right)^{ 1/(2\alpha-2+d)},
\end{align}
we have $D_\KL(\mathbb P_{\mu,\nu_k}^n\mid \mathbb P_{\mu,\nu_0}^n)\le(\log K)/9$, as required by External result~\ref{thm:tsy-fano}.

Moreover, the right-hand side of \eqref{eq:iid_kl} depends only on the normalized measures. Therefore,
\begin{align}
\label{eq:kl_equal_mass}
    D_\KL(\mathbb P_{\mu,\check\nu_k}^n\mid \mathbb P_{\mu,\check\nu_0}^n)\le(\log K)/9.
\end{align}

\subsubsection*{Applying Fano's inequality}

We now apply External result~\ref{thm:tsy-fano} to the hypotheses $(\mu,\check\nu_k)$. Let $\hat\lambda$ be an estimator, and let $s>0$. For each $k$, the pair $(\mu,\check\nu_k)$ belongs to $\mathcal U_\alpha$, and $\lambda_{\mu,\check\nu_k}=\check\lambda_k$. Hence, the definition of $\mathfrak M_n^\lambda$ and Markov's inequality yield
\begin{align}
    \mathfrak M_n^\lambda(\hat\lambda)
    &\ge \max_{0\le k\le K}\mathbb E_{\mu,\check\nu_k}^n\bigl[\|\hat\lambda-\check\lambda_k\|_{L^2(\mu)}^2\bigr]
    \notag\\
    &\ge s^2\max_{0\le k\le K}\mathbb P_{\mu,\check\nu_k}^n\bigl(\|\hat\lambda-\check\lambda_k\|_{L^2(\mu)}\ge s\bigr).
    \label{eq:markov_reduction}
\end{align}
Since $T_{\mu,\check\nu_k}=T_k$, the same argument yields the analogue of \eqref{eq:markov_reduction} for every estimator $\hat T$, with $\mathfrak M_n^\lambda$, $\hat\lambda$, and $\check\lambda_k$ replaced by $\mathfrak M_n^T$, $\hat T$, and $T_k$. By \eqref{eq:kl_equal_mass}, \eqref{eq:lambda-sep}, and $K\ge2^{N_J/8}\ge2$, External result~\ref{thm:tsy-fano}, applied with $d_\lambda(\lambda,\lambda')\coloneqq\|\lambda-\lambda'\|_{L^2(\mu)}$ and $2s=\sqrt{c_\lambda}J^{-\alpha}$, yields $\inf_{\hat\lambda}\max_{0\le k\le K}\mathbb P_{\mu,\check\nu_k}^n(\|\hat\lambda-\check\lambda_k\|_{L^2(\mu)}\ge s)\ge c$ for a constant $c>0$. Hence, \eqref{eq:markov_reduction} yields
\begin{align}
    \label{eq:lb-lambda-nonpar}
    \inf_{\hat\lambda} \mathfrak M_n^\lambda(\hat\lambda)  \gtrsim  J^{-2\alpha}  \asymp  n^{-2\alpha/(2\alpha-2+d)}.
\end{align}

Similarly, by \eqref{eq:kl_equal_mass}, \eqref{eq:T-sep}, and $K\ge2$, External result~\ref{thm:tsy-fano}, applied with $d_T(T,T')\coloneqq\|T-T'\|_{L^2(\mu)}$ and $2s=(c_l\varepsilon_{\mathrm b}^2c_\nabla/8)^{1/2}J^{-\alpha}$, yields $\inf_{\hat T}\max_{0\le k\le K}\mathbb P_{\mu,\check\nu_k}^n(\|\hat T-T_k\|_{L^2(\mu)}\ge s)\ge c$ for a constant $c>0$. Hence, the analogue of \eqref{eq:markov_reduction} for $\hat T$ yields
\begin{align}
    \label{eq:lb-T-nonpar}
    \inf_{\hat T} \mathfrak M_n^T(\hat T)  \gtrsim  J^{-2\alpha}  \asymp  n^{-2\alpha/(2\alpha-2+d)}.
\end{align}

\subsubsection*{Lower bounds based on two hypotheses}

Define $\varphi_0$ as in \eqref{eq:base-potential} and $\mu=\mathrm{Unif}([0,1]^d)$. For $\tilde\theta>0$ small, define
\begin{align}
    \varphi_{\ast}(x)\coloneqq \varphi_0(x)+(\tilde\theta/\sqrt n) g(x),
\end{align}
and let $(T_{\ast},\lambda_{\ast},\nu_{\ast})$ be defined by \eqref{eq:varphidef} with $\varphi_k$ replaced by $\varphi_{\ast}$. For $\tilde\theta$ small and $n$ large, $f_{\ast}=\frac1{2}\|\cdot\|^2-\varphi_{\ast}$ satisfies $\nabla^2 f_{\ast}\succ0$ on $\R^d$. Moreover, $T_{\ast}$ maps $[0,1]^d$ onto itself, by the same argument as for $T_k$.

As in the multiple hypotheses proof, we rescale $\nu_{\ast}$ to the total mass $H_0=\nu_0(\Omega)$. Let $H_{\ast}\coloneqq\nu_{\ast}(\Omega)=\int\lambda_{\ast}^2\,d\mu$ and $c_{\ast}\coloneqq\frac12\log(H_{\ast}/H_0)\in\R$, and define
\begin{align*}
    \check\lambda_{\ast}\coloneqq e^{-c_{\ast}}\lambda_{\ast},
    \qquad
    \check\nu_{\ast}\coloneqq(T_{\ast})_\#(\check\lambda_{\ast}^2\mu)=\frac{H_0}{H_{\ast}}\,\nu_{\ast},
    \qquad
    \bar\nu_{\ast}\coloneqq \frac{\nu_{\ast}}{H_{\ast}}.
\end{align*}
Thus, $\check\nu_{\ast}$ has the normalized measure $\bar\nu_{\ast}$ and the total mass $H_0$. Moreover, Lemma~\ref{lem:verification} applied to $\varphi_{\ast}+c_{\ast}$ identifies $(T_{\ast},\check\lambda_{\ast})$ as the $\mu$-a.e.\ unique transport-growth pair associated with the pair $(\mu,\check\nu_{\ast})$.

For all sufficiently large $n$, the function $\varphi_{\ast}-\varphi_0=(\tilde\theta/\sqrt n)g$ satisfies the last two bounds in \eqref{eq:bump-bounds}:
\begin{align*}
    \|\nabla^2((\tilde\theta/\sqrt n)g)\|_{\mathrm{op},\infty}\le1/2,
    \qquad
    \|(\tilde\theta/\sqrt n)g\|_{C^{\alpha+1}}\le C_g.
\end{align*}
Since $c_{\ast}\to0$ as $n\to\infty$, we have $\check\lambda_{\ast}\to\lambda_0$ in $C^\alpha(\Omega)$ and the density of $\check\nu_{\ast}$ converges to $h_0$ in $C^{\alpha-1}(\Omega)$, so the constants $M_0,B_0,R_0,\kappa_0,b_0$ chosen as in the multiple hypotheses proof also place the pair $(\mu,\check\nu_{\ast})$ in $\mathcal U_\alpha$ for all sufficiently large $n$.

For $\|T_0-T_{\ast}\|_{L^2(\mu)}$, we have
\begin{align}
\label{eq:T_ast_sep}
    \|T_0-T_{\ast}\|_{L^2(\mu)}^2
    = (\tilde\theta^2/n) \int \|\nabla g\|^2\,d\mu  \gtrsim  1/n.
\end{align}

For the growth maps, let $v:[0,1]^d\to\R$ be $v\coloneqq(a/2)\partial_1 g-g$. Since $\nabla \varphi_0=(a,0,\dots,0)$ on $\supp(g)\subset[\delta_0,1-\delta_0]^d$, the computation in~\eqref{eq:u-expand} with $\varphi_k-\varphi_0$ replaced by $(\tilde\theta/\sqrt n) g$ yields
\begin{align}
    \label{eq:lambda-par-expand}
    \lambda_{\ast}(x)-\lambda_0(x) = \lambda_0(x)\,\frac{\tilde\theta}{\sqrt n}\,v(x) + \mathcal R(x),
\end{align}
with $\|\mathcal R\|_\infty=O(n^{-1})$. Let $\bar v\coloneqq H_0^{-1}\int\lambda_0^2v\,d\mu\in\R$. By \eqref{eq:lambda-par-expand},
\begin{align*}
    H_{\ast}-H_0=2\frac{\tilde\theta}{\sqrt n}\int\lambda_0^2v\,d\mu+O(n^{-1}),
    \qquad\text{so}\qquad
    c_{\ast}=\frac{\tilde\theta}{\sqrt n}\,\bar v+O(n^{-1}).
\end{align*}
Hence,
\begin{align}
    \label{eq:check-lambda-ast-diff}
    \check\lambda_{\ast}-\lambda_0=\lambda_0\,\frac{\tilde\theta}{\sqrt n}\,(v-\bar v)+\check{\mathcal R},
    \qquad
    \|\check{\mathcal R}\|_\infty=O(n^{-1}).
\end{align}
We claim that the following quantity is strictly positive:
\begin{align}
    \label{eq:cgpar-positive}
    c_g^{\mathrm{par}} \coloneqq \int_{[0,1]^d}(v-\bar v)^2\,dx.
\end{align}
Since $\supp(g)\subset[\delta_0,1-\delta_0]^d$, $v$ vanishes on $\partial[0,1]^d$. Thus, it suffices to show that $v$ is not identically zero. Suppose for contradiction that $v\equiv 0$ on $[0,1]^d$. Then $\frac{a}{2}\partial_1 g=g$ on $[0,1]^d$. For every $(x_2,\dots,x_d)\in[0,1]^{d-1}$, solving the linear ODE $\frac{a}{2}\partial_1 g=g$ in $x_1$ yields
\begin{align*}
    g(x_1,x_2,\dots,x_d)=g(0,x_2,\dots,x_d)\,e^{2x_1/a} = 0,
    \qquad x_1\in[0,1],
\end{align*}
contradicting $g(x_*,\dots,x_*)=\xi(x_*)^d\ne 0$. Thus, $c_g^{\mathrm{par}}>0$ as claimed.

Squaring~\eqref{eq:check-lambda-ast-diff} and integrating against $d\mu=dx$ yields
\begin{align}
\label{eq:lambda_ast_sep}
    \|\lambda_0-\check\lambda_{\ast}\|_{L^2(\mu)}^2 \ge \lambda_{\min}^2\,\frac{\tilde\theta^2}{n}\,c_g^{\mathrm{par}} - O(n^{-3/2}) \;\gtrsim\; 1/n.
\end{align}

We now bound $D_\KL(\mathbb P_{\mu,\check\nu_{\ast}}^n\mid \mathbb P_{\mu,\nu_0}^n)=nD_\KL(\bar\nu_{\ast}\mid \bar\nu_0)$. Let $r_{\ast}\coloneqq h_{\ast}/h_0$, where $h_{\ast}$ is the density of $\nu_{\ast}$ with respect to $\mu$, and let $L_{\ast}$ be the function $L_k$ with $\varphi_k$ replaced by $\varphi_{\ast}$. Then \eqref{eq:logr-split} holds with $(r_k,L_k,T_k)$ replaced by $(r_{\ast},L_{\ast},T_{\ast})$. Since $\|\varphi_{\ast}-\varphi_0\|_{C^2([0,1]^d)}\le C\tilde\theta/\sqrt n$ and the eigenvalues of $I-\nabla^2\varphi_0$ and $I-\nabla^2\varphi_{\ast}$ are at least $1/4$, we have
\begin{align*}
    \|L_{\ast}-L_0\|_\infty\le \frac{C\tilde\theta}{\sqrt n}.
\end{align*}
Moreover, \eqref{eq:L0-expand} holds with $T_k$ replaced by $T_{\ast}$. The argument for \eqref{eq:T-inv-diff} also yields
\begin{align*}
    \|T_{\ast}^{-1}(y)-T_0^{-1}(y)\|\le \frac{C\tilde\theta}{\sqrt n},
    \qquad y\in\Omega.
\end{align*}
Thus, for all sufficiently large $n$,
\begin{align*}
    \|\log r_{\ast}\|_\infty\le \frac{C\tilde\theta}{\sqrt n}\le\frac12.
\end{align*}
Since $\|\log\lambda_{\ast}-\log\lambda_0\|_\infty\le C\tilde\theta/\sqrt n\le\log2$ for all sufficiently large $n$ by \eqref{eq:lambda-par-expand}, we have $\lambda_{\min}/2\le\lambda_{\ast}\le 2\lambda_{\max}$, and $H_{\ast} = \int\lambda_{\ast}^2\,d\mu$ implies $\lambda_{\min}^2/4\le H_{\ast}\le 4\lambda_{\max}^2$. Hence, \eqref{eq:KL-chi-square} and \eqref{eq:I-reduction} hold with $(\bar\nu_k,h_k,H_k,r_k)$ replaced by $(\bar\nu_{\ast},h_{\ast},H_{\ast},r_{\ast})$ and imply that
\begin{align}
    D_\KL(\bar\nu_{\ast}\mid \bar\nu_0)
    \le
    C\int(\log r_{\ast})^2\,d\mu
    \le
    \frac{C\tilde\theta^2}{n}.
\end{align}
Hence, $D_\KL(\mathbb P_{\mu,\check\nu_{\ast}}^n\mid \mathbb P_{\mu,\nu_0}^n)=n D_\KL(\bar\nu_{\ast}\mid \bar\nu_0)\le C\tilde\theta^2$. For all sufficiently large $n$, the pairs $(\mu,\nu_0)$ and $(\mu,\check\nu_{\ast})$ belong to $\mathcal U_\alpha$, with the transport-growth pairs $(T_0,\lambda_0)$ and $(T_{\ast},\check\lambda_{\ast})$, respectively. Thus, \eqref{eq:markov_reduction} and the analogue of \eqref{eq:markov_reduction} for $\hat T$ hold with the maximum taken over the hypotheses $(\mu,\nu_0)$ and $(\mu,\check\nu_{\ast})$. By External result~\ref{thm:tsy-twopoint}, applied with $\varrho=C\tilde\theta^2$, we obtain
\begin{align*}
    \inf_{\hat T}\max\Bigl\{\mathbb P_{\mu,\nu_0}^n\bigl(\|\hat T-T_0\|_{L^2(\mu)}\ge s_T\bigr),\,\mathbb P_{\mu,\check\nu_{\ast}}^n\bigl(\|\hat T-T_{\ast}\|_{L^2(\mu)}\ge s_T\bigr)\Bigr\}&\ge c_\varrho,
    \\
    \inf_{\hat\lambda}\max\Bigl\{\mathbb P_{\mu,\nu_0}^n\bigl(\|\hat\lambda-\lambda_0\|_{L^2(\mu)}\ge s_\lambda\bigr),\,\mathbb P_{\mu,\check\nu_{\ast}}^n\bigl(\|\hat\lambda-\check\lambda_{\ast}\|_{L^2(\mu)}\ge s_\lambda\bigr)\Bigr\}&\ge c_\varrho,
\end{align*}
where $s_T\coloneqq \tfrac{1}{2}\|T_0-T_{\ast}\|_{L^2(\mu)}$ and $s_\lambda\coloneqq \tfrac{1}{2}\|\lambda_0-\check\lambda_{\ast}\|_{L^2(\mu)}$, which,  by \eqref{eq:T_ast_sep} and \eqref{eq:lambda_ast_sep}, satisfy $s_T^2\gtrsim1/n$ and $s_\lambda^2\gtrsim1/n$ . Combining the previous display with the analogue of \eqref{eq:markov_reduction} for $\hat T$ at $s=s_T$ and with \eqref{eq:markov_reduction} at $s=s_\lambda$, we obtain
\begin{align}
    \label{eq:lb-T-par}
    \inf_{\hat T} \mathfrak M_n^T(\hat T) & \gtrsim  1/n,\\
    \label{eq:lb-lambda-par}
    \inf_{\hat\lambda} \mathfrak M_n^\lambda(\hat\lambda) & \gtrsim  1/n.
\end{align}

Combining \eqref{eq:lb-T-nonpar}, \eqref{eq:lb-lambda-nonpar}, \eqref{eq:lb-T-par}, and \eqref{eq:lb-lambda-par} yields Theorem~\ref{thm:lb-ubot}.

\end{proof}

\section{Extensions to other costs}\label{sec:other_costs}

The analysis of this paper uses the quadratic cost $c(x,y)=\frac12\|x-y\|^2$ of Section~\ref{sec:setup}. In this section, we identify the steps of the analysis that use the quadratic form of $c$ and state what a general cost would need at each step, with the Hellinger-Kantorovich cost as an example.

\begin{remark}[General costs]\label{rem:general_costs}
Let $\tilde c:\Omega\times\Omega\to\R$ be a cost function, and consider the problem~\eqref{eq:gh_kantorovich} with $c$ replaced by $\tilde c$. An extension of the upper bounds of Theorems~\ref{thm:main_plan_based_rates} and~\ref{thm:main_cube_kernel_rates} to $\tilde c$ would require at least five ingredients at the steps that use the quadratic form of $c$:
\begin{enumerate}[label=(\roman*)]
    \item $\tilde c\in C^2(\Omega\times\Omega)$ with bounded second derivatives, for an analogue of the bound~\eqref{eq:cost_diff_exact} in Lemma~\ref{lem:uot_cost_expansion}, which is used in the proof of Proposition~\ref{prop:two_sample_stability};
    \item an analogue of Theorem~\ref{prop:monge-potential} for $\tilde c$ under the twist condition, for the transport-growth pair $(T_0,\lambda_0)$;
    \item for every $x\in\Omega$, the strong convexity of the map $y\mapsto\tilde c(x,y)-\psi_0(y)$ on $\Omega$, for the gap bound of Lemma~\ref{lem:uot_gap_sufficient};
    \item constants $0<\sigma\le L<\infty$ such that $\sigma I\preceq\nabla_y^2\tilde c(y,y')\preceq LI$ and $\nabla_y\tilde c(y,y')|_{y=y'}=0$ for all $y,y'\in\Omega$, for the first inequality in~\eqref{eq:uot_hat_barycentric_delta_bound} in Proposition~\ref{prop:uot_barycentric_and_1nn_estimated_mass};
    \item an analogue of \eqref{eq:psi0_hessian_bounds}, whose proof uses $T_0=\nabla f_0$, for the Lipschitz gradient of $\xi_0$ in Lemma~\ref{lem:smoothness_effective}, which is used in the proof of Proposition~\ref{prop:two_sample_stability}.
\end{enumerate}
Under (iv), let $\hat T_i$ be the minimizer in \eqref{eq:disc-frechet} with $c$ replaced by $\tilde c$. From (iv), we have $0\le\tilde c(y,Y_j)-\tilde c(Y_j,Y_j)\le\frac L2\|y-Y_j\|^2$ for all $y\in\Omega$ and $j\in\{1,\dots,m\}$. For every $i\in\{1,\dots,n\}$ with $\hat r_i>0$, combining the upper bound at $y=T_0(X_i)$ and the lower bound at $y=\hat T_i$ with the $\sigma\hat r_i$-strong convexity of the map $y\mapsto\sum_{j=1}^m\hat\gamma_{ij}\tilde c(y,Y_j)$ on $\Omega$ yields
\begin{align}
\label{eq:general_costs_row}
    \hat r_i\|\hat T_i-T_0(X_i)\|^2
    \le
    \frac{L}{\sigma}\sum_{j=1}^m\hat\gamma_{ij}\|Y_j-T_0(X_i)\|^2.
\end{align}
For the quadratic cost, $\sigma=L=1$. Hence, summing \eqref{eq:general_costs_row} over the indices $i$ with $\hat r_i>0$ yields the first inequality in~\eqref{eq:uot_hat_barycentric_delta_bound} in Proposition~\ref{prop:uot_barycentric_and_1nn_estimated_mass}. In contrast, the lower bound of Theorem~\ref{thm:lb-ubot} needs more than (i)--(v). The construction behind this lower bound uses the explicit forms~\eqref{eq:unique} of $T_0$ and $\lambda_0$, which hold for the quadratic cost. For a general cost $\tilde c$, both forms change, so the separation estimates in the proof of Theorem~\ref{thm:lb-ubot} would have to be derived again.
\end{remark}

\begin{remark}[Hellinger-Kantorovich cost]\label{rem:hk}
The Hellinger-Kantorovich problem~\cite{LieroMielkeSavare2018} is the UOT problem with KL penalties and the cost $c_{\mathrm{HK}}:\Omega\times\Omega\to[0,\infty]$, $c_{\mathrm{HK}}(x,y)\coloneqq-\log\cos^2(\|x-y\|\wedge\pi/2)$. Suppose that $D\coloneqq\mathrm{diam}(\Omega)<\pi/2$. Then $c_{\mathrm{HK}}$ is finite and smooth on $\Omega\times\Omega$ and satisfies
\begin{align*}
    \|x-y\|^2\le c_{\mathrm{HK}}(x,y)\le C_D\|x-y\|^2,
    \qquad x,y\in\Omega,
\end{align*}
with the constant $C_D\coloneqq-2\log\cos(D)/D^2>1$. These bounds hold because $-\log\cos^2 t=t^2+t^4/6+\cdots$ has nonnegative Taylor coefficients for $t\in[0,\pi/2)$. Moreover, $2I\preceq\nabla_y^2c_{\mathrm{HK}}(y,y')\preceq2\sec^2(D)\,I$ and $\nabla_yc_{\mathrm{HK}}(y,y')|_{y=y'}=0$ for all $y,y'\in\Omega$. At $t=\|y-y'\|\in[0,D]$, the eigenvalues of $\nabla_y^2c_{\mathrm{HK}}(y,y')$ are $2\sec^2t$ in the direction $y-y'$ and $2\tan(t)/t$ in the orthogonal directions. Both eigenvalues lie in $[2,2\sec^2(D)]$, which yields the Hessian bounds. Thus, $c_{\mathrm{HK}}$ satisfies (i) and (iv) in Remark~\ref{rem:general_costs} with $\sigma=2$ and $L=2\sec^2(D)$. In addition, \cite{gallouet2025regularity} proves the analogue of Theorem~\ref{prop:monge-potential} for $c_{\mathrm{HK}}$ on a smooth compact Riemannian manifold, which is (ii) in that setting. On the other hand, (iii) and (v) in Remark~\ref{rem:general_costs} remain open. We expect these two properties to follow from the Hellinger-Kantorovich analogue of Assumption~\ref{assm:curvature}, in which case the upper bounds of Theorems~\ref{thm:main_plan_based_rates} and~\ref{thm:main_cube_kernel_rates} would hold with the same rates. For the lower bound, the forms~\eqref{eq:unique} change for $c_{\mathrm{HK}}$ as well. However, the construction behind Theorem~\ref{thm:lb-ubot} moves mass over distances much smaller than $\pi/2$, where $c_{\mathrm{HK}}(x,y)=\|x-y\|^2+O(\|x-y\|^4)$, so we expect the same minimax rate.
\end{remark}

\section{The Wasserstein distance between measures of equal total mass} \label{sec:empirical_wasserstein}

We present here the definition of the Wasserstein distance between finite positive measures of equal total mass and the effect of mass scaling on the Wasserstein distances. 

\begin{definition}
Let $p\in[1,\infty)$ and let $\eta,\sigma\in \mM_+(\Omega)$ satisfy $\eta(\Omega)=\sigma(\Omega)<\infty$. We define the $p$-Wasserstein distance between $\eta$ and $\sigma$ by
\begin{align}
    W_p(\eta,\sigma) \coloneqq \left( \inf_{\gamma\in\Pi(\eta,\sigma)} \int_{\Omega\times\Omega} \|x-y\|^p\,d\gamma(x,y) \right)^{1/p},
\end{align}
where $\Pi(\eta,\sigma)$ denotes the set of couplings of $\eta$ and $\sigma$. When $\eta$ and $\sigma$ are probability measures, this is the usual Wasserstein distance.
\end{definition}

\begin{lemma}[Scaling in the total mass] \label{lem:mass_scaling_wasserstein}
Let $m>0$ and let $\eta,\sigma$ be probability measures on $\Omega$. Then
\begin{align}
\label{eq:mass_scaling_wasserstein}
    W_p^p(m\eta,m\sigma) = m   W_p^p(\eta,\sigma),
    \qquad
    W_p(m\eta,m\sigma) = m^{1/p} W_p(\eta,\sigma).
\end{align}
\end{lemma}

\begin{proof}[Proof of Lemma \ref{lem:mass_scaling_wasserstein}]
The map $\gamma \mapsto m\gamma$ is a bijection from $\Pi(\eta,\sigma)$ onto $\Pi(m\eta,m\sigma)$. Therefore,
\begin{align}
    \inf_{\gamma\in\Pi(m\eta,m\sigma)} \int \|x-y\|^p\,d\gamma(x,y)
    = m \inf_{\gamma\in\Pi(\eta,\sigma)} \int \|x-y\|^p\,d\gamma(x,y).
\end{align}
Taking the $p$th root yields the second identity in \eqref{eq:mass_scaling_wasserstein}.
\end{proof}

To give some context on the rates of our nearest-neighbor estimator in Theorem~\ref{thm:main_plan_based_rates}, we provide below a Wasserstein lower bound for discrete approximation of measures, which is expressed in the notion of the \emph{lower Wasserstein dimension} of \cite{Young1982,weed2019sharp}: for a probability measure $\eta$ on $\Omega$ and $\tau\in[0,1)$, let $N_\epsilon(S)$ be the minimum number of closed balls of diameter $\epsilon$ needed to cover $S$, and define
\begin{align}
    N_\epsilon(\eta,\tau) \coloneqq \inf\{N_\epsilon(S): & \ S\subset \Omega,\ \eta(S)\ge 1-\tau\},
    \qquad
    d_\epsilon(\eta,\tau) \coloneqq \frac{\log N_\epsilon(\eta,\tau)}{-\log \epsilon}, \\
    &d_*(\eta)\coloneqq \lim_{\tau\downarrow 0} \liminf_{\epsilon\downarrow 0} d_\epsilon(\eta,\tau).
\end{align}
Here, $d_*(\bar\mu)$ is referred to as the lower Wasserstein dimension of $\bar\mu$.

\begin{proposition}[Lower bound for discrete approximations]
\label{prop:weighted_empirical_wasserstein}
Let $\mu\in\mM_+(\Omega)$ with $\mu(\Omega)=M_\mu\in(0,\infty)$, and let $\bar\mu\coloneqq\mu/M_\mu$. If $t<d_*(\bar\mu)$ and $\eta_n\in\mM_+(\Omega)$ is any discrete measure with $\eta_n(\Omega)=M_\mu$ and $|\mathrm{supp}(\eta_n)|\le n$, then
\begin{align}
\label{eq:discrete_lower_bound}
    W_p(\mu,\eta_n) \gtrsim (M_\mu)^{1/p} n^{-1/t}.
\end{align}
In particular, the same lower bound holds for the empirical measure $\eta_n=\tilde\mu_n$ with mass $M_\mu$. Under Assumption~\ref{assm:positivity}, \eqref{eq:discrete_lower_bound} holds for every $t<d$.
\end{proposition}

\begin{proof}[Proof of Proposition \ref{prop:weighted_empirical_wasserstein}]
Let $\bar\eta_n \coloneqq \eta_n/M_\mu$. Then $\bar\eta_n$ is a probability measure supported on at most $n$ points. The lower bound in Theorem~1 of \cite{weed2019sharp} and the remark after the theorem imply that for every $t<d_*(\bar\mu)$,
\begin{align*}
    W_p(\bar\mu,\bar\eta_n) \gtrsim n^{-1/t}.
\end{align*}
Applying Lemma~\ref{lem:mass_scaling_wasserstein} yields
\begin{align*}
    W_p(\mu,\eta_n)
    =
    (M_\mu)^{1/p} W_p(\bar\mu,\bar\eta_n)
    \gtrsim
    (M_\mu)^{1/p} n^{-1/t}.
\end{align*}
Taking $\eta_n=\tilde\mu_n$ proves the claim for $\tilde\mu_n$. For the final claim, Assumption~\ref{assm:positivity} gives $\bar\mu$ a Lebesgue density on $\Omega$. Hence, every Borel set $S\subset\Omega$ with $\bar\mu(S)>0$ has positive Lebesgue measure, so the Hausdorff dimension of $S$ is $d$. Consequently, the Hausdorff dimension of $\bar\mu$ is $d$, and \cite[Proposition~2]{weed2019sharp} yields $d_*(\bar\mu)\ge d$. Hence, \eqref{eq:discrete_lower_bound} holds for every $t<d$.
\end{proof}

\section{Details of the simulation} 

\subsection{Computational resources}\label{sec:comp}
The experiments were run on a Linux computing node with two 4-core Intel(R) Xeon(R) Gold 5222 processors and a single NVIDIA Tesla V100-SXM2 GPU, utilizing 32 GB of VRAM.

\subsection{Simulation and implementation details}\label{sec:sim_detail}
We generate source and target samples of size $n\in \{100,200,500,1000\}$ over 10 random seeds. For PB-1NN and PB-Kernel, we solve the discrete UOT problem, framed as nonnegative regularized linear regression as described in \cite{Chapel2021}, using the UOT solver from the Python Optimal Transport library \cite{flamary2021pot}. For PI-Kernel and PI-Wavelet, we solve the dual problem \eqref{eq:gh_dual} for $(\hat\mu_n^{\mathrm{ker}},\hat\nu_m^{\mathrm{ker}})$ and $(\hat\mu_n^{\mathrm{wav}},\hat\nu_m^{\mathrm{wav}})$ using the L-BFGS-B algorithm \cite{Liu1989} over the discretized grid with resolution $128, 64, 32, 32$ for $d=1,2,3,4$, respectively. For each $n\in \{100,200,500,1000\}$, we select the frequency cutoff $L_n$ by 10-fold least-squares cross-validation on a fresh sample of size $n$, over the range $[6,L_{\max}]$ with $L_{\max}=100,60,30,20$ for $d=1,2,3,4$. In contrast, we set the wavelet resolution level by the choice $2^{J_n}\asymp n^{1/(d+2(\alpha-1))}$ as specified in Theorem~\ref{thm:main_wavelet_rates}, namely $J_n=\lceil\log_2 n/(d+2(\alpha-1))\rceil$ with $\alpha=2$. We implement PI-Wavelet with the db4 stationary wavelet transform, even reflection at the boundary, and soft thresholding of the detail coefficients $\hat\alpha_\xi$ and $\hat\beta_\xi$. Finally, the cone program in the SSUOT estimator was solved using the CVXPY library \cite{diamond2016cvxpy} with MOSEK solver \cite{mosek}.

\end{document}